\documentclass[10pt,reqno]{amsart}
\usepackage{needspace}
\usepackage{amsfonts,amssymb,amsmath}

\parskip        \baselineskip
\usepackage{enumitem}
\usepackage{aliascnt}
\usepackage[mathscr]{eucal}
\usepackage{graphicx}
\usepackage{latexsym}
\usepackage{psfrag}
\usepackage{epsfig}
\usepackage[utf8]{inputenc}
\usepackage[english]{babel}
\usepackage{tikz}
\usepackage[foot]{amsaddr}
\usepackage{lineno, blindtext}
\usepackage[scr=boondox,  
cal=esstix]   
{mathalpha}
\usetikzlibrary{decorations.pathreplacing}

\newcommand{\suchthat}{\;\ifnum\currentgrouptype=16 \middle\fi|\;}
\def\tr{\operatorname{Tr}}
\def\eR{\mathbb{R}}
\def\eN{\mathbb{N}}

\newcommand{\ds}{\mathrm{d}s}
\newcommand{\dt}{\mathrm{d}t}
\newcommand{\dvr}{\operatorname{div}}
\newcommand{\supp}{\operatorname{supp}}
\newcommand{\dist}{\operatorname{dist}}
\newcommand{\tder}{\partial_t}

\newcommand{\essinf}{\operatorname{ess\,inf}}
\newcommand{\esssup}{\operatorname{ess\,sup}}
\newcommand{\sgnp}{\operatorname{sgn}^+}

\numberwithin{equation}{section}
\newtheorem{theorem}{Theorem}
\numberwithin{theorem}{section}

\newaliascnt{lemma}{theorem}
\newtheorem{lemma}[lemma]{Lemma}
\aliascntresetthe{lemma}

\newaliascnt{proposition}{theorem}

\aliascntresetthe{proposition}

\newaliascnt{corollary}{theorem}

\aliascntresetthe{corollary}

\newaliascnt{definition}{theorem}
\newtheorem{mydef}[definition]{Definition}
\aliascntresetthe{definition}

\newaliascnt{remark}{theorem}
\newtheorem{remark}[remark]{Remark}
\aliascntresetthe{remark}

\usepackage[colorlinks=true]{hyperref}

\newcommand{\boundellipse}[3]
{(#1) ellipse (#2 and #3)
}
\title{On existence of weak solutions to a Baer--Nunziato type system}
\author{Martin Kalousek}
\author{\v S\' arka Ne\v casov\'a}
\address{Institute of Mathematics, Czech Academy of Sciences, \v{Z}itn\'a 25, 11567 Prague, Czech Republic}
\email{kalousek@math.cas.cz; matus@math.cas.cz}
\begin{document}
	
	\begin{abstract}
		In this paper, we consider
  a compressible one velocity Baer--Nunziato type system with dissipation  describing the evolution of a mixture of two compressible heat conducting fluids. The complete existence proof for weak solutions to this system was addressed as an open problem in \cite[Section 5]{KNAC}. The purpose of this paper is to prove the global in time existence of weak solutions to the one velocity Baer--Nunziato type system for arbitrary large initial data. The goal is achieved in three steps. Firstly, the given system is transformed into a new one which possesses the "Navier-Stokes-Fourier" structure. Secondly, the new system is solved by an adaptation of the Feireisl--Lions approach for solving the compressible full system applying also the almost compactness property introduced by Vasseur et al. \cite{VaWeYu19}. Finally, the existence of a weak solution to the original one velocity Baer--Nunziato system is shown using the almost uniqueness property of renormalized solutions to pure transport equations.
	\end{abstract}
	
	\maketitle
\section{Introduction}\label{Sec:Intr}
	Based on the continuum theory of mixtures, the authors in \cite{BaNu} established a model describing combustion processes associated with so called deflagration--to--detonation transition. 
	A dissipative version of the above mentioned model is the Baer--Nunziato model for two compressible heat conducting fluids, which reads
	\begin{align}
		\tder\alpha_{\pm}+v_I\cdot\nabla\alpha_{\pm}&=0,\label{AlphaTE}\\
		\tder (\alpha_{\pm}\rho_{\pm})+\dvr(\alpha_{\pm}\rho_{\pm}u_{\pm})&=0,\label{CnEq}\\
		\tder(\alpha_{\pm}\rho_{\pm}u_{\pm})+\dvr(\alpha_{\pm}\rho_{\pm}u_{\pm}\otimes u_{\pm})+\nabla\left(\alpha_{\pm}\mathfrak P_{\pm}(\rho_{\pm},\vartheta)\right)&\nonumber \\-P_I\nabla\alpha_{\pm}-\dvr\left(\alpha_{\pm}\mathbb S(\vartheta,\nabla u_{\pm})\right)&=0,\label{MERPM}\\
		\tder(\alpha_{\pm}\rho_{\pm}\mathfrak e_{\pm}(\rho_{\pm},\vartheta))+\dvr(\alpha_{\pm}\rho_{\pm}\mathfrak e_{\pm}(\rho_{\pm},\vartheta)u_{\pm})+\dvr(\alpha_{\pm}q_{\pm}(\vartheta,\nabla\vartheta))&\nonumber\\+\alpha_{\pm}\mathfrak P_{\pm}(\rho_{\pm},\vartheta)\dvr u_{\pm}-\alpha_{\pm}\mathbb S(\vartheta,\nabla u_{\pm})\cdot\nabla u_{\pm}-P_I(v_I-u_{\pm})\cdot\nabla\alpha_{\pm}&=0,\label{ENeBal}\\
		\alpha_++\alpha_-&=1.\label{AlphaSum}
	\end{align}
	The subscript $_{\pm}$ above is used for the determination which species the quantity belongs  to. Considering $T>0$ and a bounded domain $\Omega\subset\eR^d$ we use the notation $Q_T$ for the time space cylinder $(0,T)\times\Omega$. With this notation $\alpha_{\pm}:Q_T\to[0,1]$, $\rho_\pm: Q_T\to[0,\infty)$, $u_{\pm}:Q_T\to\eR^d$ stand for the unkonwn volumic fraction, density and velocity of the $\pm$ spieces. Next, $\vartheta:Q_T\to (0,\infty)$ is the unknown velocity of the mixture. Given functions $\mathfrak P_{\pm}, \mathfrak e_{\pm}:[0,\infty)\times(0,\infty)\to [0,\infty)$, $q_{\pm}:(0,\infty)\times \eR^d$, $\mathbb S_{\pm}:(0,\infty)\times\eR^{d\times d}\to\eR^{d\times d}$ represent the pressure, internal energy, heat flux and viscous stress tensor of the $\pm$ species. The conveniently chosen quantities $P_I:Q_T\to[0,\infty]$ and $v_I:Q_T\to\eR^d$ stand for the pressure and velocity at the interface. In the modeling of multi fluid flows, there are many possibilities concerning the choice of the functions $P_I$ and $v_I$ with no predominant choice. For the Fourier law and Newtonian fluids the heat flux  and the viscous stress tensor read
	\begin{equation}\label{FS}
		\begin{split}
			q_\pm(\vartheta,\nabla\vartheta)=&-\kappa_\pm(\vartheta)\nabla\vartheta,\\
			\mathbb S_\pm(\vartheta,\nabla u_\pm)=&\mu_\pm(\vartheta)\left(\nabla u_\pm +(\nabla u_\pm)^\top-\frac{2}{d}\dvr u_\pm\mathbb I\right)+\eta_{\pm}(\vartheta)\dvr u\mathbb I,
		\end{split}
	\end{equation}
	where the given functions $\kappa_{\pm},\mu_{\pm}:(0,\infty)\to(0,\infty)$ and $\eta_\pm:(0,\infty)\to[0,\infty)$ denote the heat conductivity and the shear and bulk viscosities of the $\pm$ spieces. Dividing equation for internal energy \eqref{ENeBal} by $\vartheta$ one formally obtains the entropy balance equations for the $\pm$ species
	\begin{equation}\label{ENTBal}
		\begin{split}			\tder(\alpha_{\pm}\rho_{\pm}\mathfrak s_{\pm}(\rho_{\pm},\vartheta))+\dvr(\alpha_{\pm}\rho_{\pm}\mathfrak s_{\pm}(\rho_{\pm},\vartheta)u_{\pm})+\dvr\left(\frac{\alpha_{\pm}}{\vartheta}q_{\pm}(\vartheta,\nabla\vartheta)\right)&\\-\frac{\alpha_{\pm}}{\vartheta}\mathbb S_{\pm}(\vartheta,\nabla u_{\pm})\cdot\nabla u_{\pm}-\frac{\alpha_{\pm}}{\vartheta^2}q_{\pm}\cdot\nabla\vartheta-\frac{1}{\vartheta}P_I(v_I-u_{\pm})\cdot\nabla\alpha_{\pm}&=0.
		\end{split}
	\end{equation}
  In fact, the latter equality is equivalent to \eqref{ENeBal} for sufficiently smooth $\alpha_\pm, \rho_\pm, u_\pm$ and $\vartheta$. The specific entropies $\mathfrak s_\pm$ are defined via
	\begin{equation}\label{SpEntrDef}
		\mathfrak s_\pm(r,\vartheta)=\frac{4b}{3r}\vartheta^3+s_\pm(r,\vartheta)
	\end{equation}
	and $s_\pm$ satisfy
	\begin{equation}\label{SDef}
		\partial_\vartheta s_\pm=\frac{1}{\vartheta}\partial_\vartheta e_\pm,\ \partial_r s_\pm=-\frac{1}{r^2}\partial_\vartheta P_\pm.
	\end{equation}
	We suppose that the pressures $\mathfrak P_\pm$ take the form 
	\begin{equation*}
		\mathfrak P_\pm(r,\vartheta)=\frac{b}{3}\vartheta^4+ P_\pm(r,\vartheta),
	\end{equation*}
	the internal energies $\mathfrak e_\pm$ are expressed as
	\begin{equation*}
		\mathfrak e_\pm(r,\vartheta)=\frac{b}{r}\vartheta^4+e_\pm(r,\vartheta),
	\end{equation*}
	where $b$ is a positive number known as the Stefan-Boltzman constant in physics. The quantities $P_\pm$ and $e_\pm$ are linked via the Maxwell relation
	\begin{equation}\label{MaxwellE}
		P_\pm(r,\vartheta)=r^2\partial_r e_\pm(r,\vartheta)+\vartheta\partial_\vartheta P_\pm(r,\vartheta).
	\end{equation}
	The relations \eqref{SDef} and \eqref{MaxwellE} are equivalent to the Gibbs equations
	\begin{equation}\label{GibbsO}
		\vartheta \partial_r s_\pm=\partial_r e_\pm -\frac{P_\pm}{r^2},\ \vartheta \partial_\vartheta s_\pm=\partial_\vartheta e_\pm. 
	\end{equation}
    We now make several simplifying assumptions:
	\begin{equation*}
		\kappa_{\pm}=\kappa,\ \mu_\pm=\mu,\ \eta_{\pm}=\eta,\ v_I=u_\pm=u
	\end{equation*}
    and restrict ourselves to the case $d=3$.
	Accordingly, we assume that 
	\begin{equation}\label{STensDef}
		\mathbb S(\vartheta,\nabla u)=\mu(\vartheta)\left(\nabla u+(\nabla u)^\top-\frac{2}{3}\dvr u\mathbb I\right)+\eta(\vartheta)\dvr u\mathbb I.
	\end{equation}
	We follow \cite{KNAC} and introduce new pressure functions $\mathsf P_\pm$ via
	\begin{equation}\label{MPDef}
		\mathsf P_\pm(f_\pm(\alpha_\pm)\alpha_\pm r,\vartheta)=\alpha_\pm P_\pm(r,\vartheta)
	\end{equation}
	with some functions $f_\pm$ whose properties will be specified later.
	
	Taking into consideration the above described simplifications and also the notation
	\begin{equation*}
		\alpha=\alpha_+,\ \rho=\alpha\rho_+,\ z=(1-\alpha)\rho_-
	\end{equation*}
	system  \eqref{AlphaTE}--\eqref{MERPM}, \eqref{AlphaSum} and \eqref{ENTBal} is transformed into
	\begin{equation}\label{BNS}
		\begin{split}
			\tder\alpha+u\cdot\nabla \alpha&=0,\\
			\tder\rho+\dvr(\rho u)&=0, \\
			\tder z+\dvr(z u)&=0, \\
			\tder\left((\rho+z)u\right)+\dvr\left((\rho+z)u\otimes u\right)&		+\nabla\left(\frac{b}{3}\vartheta^4+\mathsf P_+(f_+(\alpha)\rho,\vartheta)+\mathsf P_-(f_-(1-\alpha)z,\vartheta)\right)\\&=\dvr(\mathbb S(\vartheta,\nabla u)),\\
			\tder\left(\rho \mathfrak s_+\left(\frac{\rho}{\alpha},\vartheta\right)+z \mathfrak s_-\left(\frac{z}{1-\alpha},\vartheta\right)\right)&+\dvr\left(\left(\rho \mathfrak s_+\left(\frac{\rho}{\alpha},\vartheta\right)+z \mathfrak s_-\left(\frac{z}{1-\alpha},\vartheta\right)\right)u\right)\\-\dvr\left(\frac{\kappa(\vartheta)}{\vartheta}\nabla\vartheta\right)&=\frac{1}{\vartheta}\left(\mathbb S(\vartheta,\nabla u)\cdot\nabla u+\frac{\kappa(\vartheta)}{\vartheta}|\nabla\vartheta|^2\right)
		\end{split}
	\end{equation}
	in $Q_T$. The system is endowed with the initial conditions
	\begin{equation}\label{BNSID}
		\alpha(0,\cdot)=\alpha_0,\ \rho(0,\cdot)=\rho_0,\ z(0,\cdot)=z_0,\ (\rho+z)u(0,\cdot)=(\rho_0+z_0)u_0,\ \vartheta(0,\cdot)=\vartheta_0
	\end{equation}
	and with the boundary conditions
	\begin{equation}\label{TempBC}
		\nabla\vartheta\cdot n=0\text{ on }(0,T)\times\partial\Omega
	\end{equation}
	and either
	\begin{equation}\label{VelBC}
		u\cdot n=0\text{ and }\mathbb S(\vartheta,\nabla u)n\times n=0\text{ or }u=0\text{ on }(0,T)\times\partial\Omega
	\end{equation}
	for the case of the complete slip, the no slip respectively.

    Next, we explain the way how we transform system~\eqref{BNS} to a system that is better suited for the direct existence analysis. We define the following new quantitities 
	\begin{equation}\label{NQ}
		\mathsf{e}_{\pm}(f_\pm(\alpha_\pm)\alpha_\pm r,\vartheta)=\frac{1}{f_\pm(\alpha_{\pm})} e_{\pm}(r,\vartheta),\ \mathsf{s}_{\pm}(f_\pm(\alpha_\pm)\alpha_\pm r,\vartheta)=\frac{1}{f_\pm(\alpha_{\pm})} s_{\pm}(r,\vartheta).
	\end{equation}
	The properties of functions $f_\pm$, $\mathsf e_\pm$, $\mathsf s_\pm$ and $\mathsf P_\pm$ will be specified later. Introducing the changes of variables $\mathfrak r=f_+(\alpha)\rho=f_+(\alpha)\alpha\rho_+$, $\mathfrak z=f_-(1-\alpha)z=f_-(1-\alpha)(1-\alpha)\rho_-$ we get
	\begin{equation}\label{COV}
		\begin{split}
			\mathsf{e}_{+}(\mathfrak r,\vartheta)=&\frac{1}{f_+(\alpha)} e_{+}\left(\frac{\mathfrak r}{f_{+}(\alpha)\alpha},\vartheta\right),\\ 
			\mathsf{s}_{+}(\mathfrak r,\vartheta)=&\frac{1}{f_+(\alpha)} s_{+}\left(\frac{\mathfrak r}{f_+(\alpha)\alpha},\vartheta\right),\\ 	
			\mathsf{P}_{+}(\mathfrak r,\vartheta)=&\alpha_+ P_+\left(\frac{\mathfrak r}{f_+(\alpha)\alpha},\vartheta\right),\\
			\mathsf{e}_{-}(\mathfrak z,\vartheta)=&\frac{1}{f_-(1-\alpha)} e_{-}\left(\frac{\mathfrak z}{f_{-}(1-\alpha)(1-\alpha)},\vartheta\right),\\	
			\mathsf{s}_{-}(\mathfrak z,\vartheta)=&\frac{1}{f_-(1-\alpha)} s_{-}\left(\frac{\mathfrak z}{f_-(1-\alpha)(1-\alpha)},\vartheta\right),\\	
			\mathsf{P}_{-}(\mathfrak z,\vartheta)=&(1-\alpha) P_-\left(\frac{\mathfrak z}{f_-(1-\alpha)(1-\alpha)},\vartheta\right).
		\end{split}
	\end{equation}
	The specific internal energies $\mathfrak e_\pm$ and the entropies $\mathfrak s_\pm$ take the form
	\begin{equation}\label{TEneEnt}
		\begin{split}
			\mathfrak e_+(\mathfrak r, \alpha,\vartheta)=& \frac{b}{\mathfrak r}\vartheta^4\alpha+\mathsf e_+(\mathfrak r,\vartheta),\\
			\mathfrak e_-(\mathfrak z,\alpha,\vartheta)=& \frac{b}{\mathfrak z}\vartheta^4(1-\alpha)+\mathsf e_-(\mathfrak z,\vartheta),\\
			\mathfrak s_+(\mathfrak r, \alpha,\vartheta)=& \frac{4b}{3\mathfrak r}\vartheta^3\alpha+\mathsf s_+(\mathfrak r,\vartheta),\\
			\mathfrak s_-(\mathfrak z, \alpha,\vartheta)=& \frac{4b}{3\mathfrak z}\vartheta^3(1-\alpha)+\mathsf s_-(\mathfrak z,\vartheta).
		\end{split}
	\end{equation}
	
	We notice that the interface pressure $P_I$ disappears from the new system because $\alpha_++\alpha_-=1$. We note that \eqref{GibbsO} and \eqref{COV} imply that the Gibbs relations for quantities $\mathsf s_\pm$, $\mathsf e_\pm$ and $\mathsf P_\pm$ have the form
	\begin{equation}\label{GibbsT}
		\begin{split}
			\partial_{\mathfrak r} \mathsf e_+=\vartheta\partial_{\mathfrak r}\mathsf s_++\frac{\mathsf P_+}{\mathfrak r^2}&,\ \partial_\vartheta\mathsf e_+=\vartheta\partial_\vartheta\mathsf s_+\\
			\partial_{\mathfrak z}\mathsf e_-=\vartheta\partial_{\mathfrak z}\mathsf s_-+\frac{\mathsf P_-}{\mathfrak z^2}&,\ \partial_\vartheta\mathsf e_-=\vartheta\partial_\vartheta\mathsf s_-.
		\end{split}
	\end{equation}
	Introducing $\Sigma=\rho+z$ and $\mathcal P(\mathfrak r,\mathfrak z,\vartheta)=\frac{b}{3}\vartheta^4+\mathsf P_+(\mathfrak r,\vartheta)+\mathsf P_-(\mathfrak z,\vartheta)$ system \eqref{BNS} is transformed into 
	\begin{equation}\label{ABS}
		\begin{split}
			\tder \xi+\dvr(\xi u)&=0,\\
			\tder \mathfrak r+\dvr(\mathfrak ru)&=0,\\
			\tder \mathfrak z+\dvr(\mathfrak zu)&=0,\\
			\tder \Sigma+\dvr(\Sigma u)&=0,\\
			\tder \left(\Sigma u\right)+\dvr\left(\Sigma u\otimes u\right)+\nabla \mathcal P(\mathfrak r,\mathfrak z,\vartheta)=\dvr \mathbb S(\vartheta,\nabla u)&=0,\\
			\tder \left(\left(\mathfrak r\mathfrak s_+\left(\mathfrak r,\vartheta\right)+\mathfrak z\mathfrak s_-\left(\mathfrak z,\vartheta\right)\right)u\right)+\dvr\left(\left(\mathfrak r\mathfrak s_+\left(\mathfrak r,\vartheta\right)+\mathfrak z\mathfrak s_-\left(\mathfrak z,\vartheta\right)\right)u\right)&\\+\dvr\left(\frac{\kappa(\vartheta)}{\vartheta}\nabla\vartheta\right)-\frac{1}{\vartheta}\left(\mathbb S(\vartheta,\nabla u)\cdot\nabla u+\frac{\kappa(\vartheta)}{\vartheta}|\nabla\vartheta|^2\right)&=0
		\end{split}
	\end{equation}
	in $Q_T$ with initial conditions
	\begin{equation}\label{InCond}
		\begin{split}
			\xi(0,\cdot)=&\alpha_0\rho_+(0,\cdot),\\
			\mathfrak r(0,\cdot)=&f_+(\alpha_0)\alpha_0\rho_+(0,\cdot),\\
			\mathfrak z(0,\cdot)=&f_-(1-\alpha_0)(1-\alpha_0)\rho_-(0,\cdot),\\
			\Sigma u(0,\cdot)=&(\alpha\rho_++(1-\alpha)\rho_-)u(0,\cdot),\\
			\left(\mathfrak r\mathfrak s_+(\mathfrak r,\vartheta)+	\mathfrak z\mathfrak s_-(\mathfrak z,\vartheta)\right)(0,\cdot)=&\frac{4b}{3}\vartheta_0^3+\alpha_0\rho_{+,0}\alpha_0s_+(\rho_{+,0},\vartheta_0)+(1-\alpha_0)\rho_{-,0}\alpha_0s_-(\rho_{-,0},\vartheta_0)
		\end{split}
	\end{equation}and boundary conditions
	\begin{equation*}
		\nabla\vartheta\cdot n=0,\ u=0\text{ or }u\cdot n=0\text{ and } (\mathbb Sn)\times n=0\text{ on }(0,T)\times\partial\Omega.		
	\end{equation*}
	We point out that considering the continuity equations for $\xi$ and $\mathfrak r$ with \eqref{InCond}$_{1,2}$ we can recover the original variable $\alpha$. Indeed, the quotient $\frac{\xi}{\mathfrak r}$ satisfies at least formally the transport equation. Hence $\alpha=f_+^{-1}\left(\frac{\xi}{\mathfrak r}\right)$ for $f_+$ sufficiently regular satisfies at least formally the transport equation as well. The reason why the explicit dependence on $f_+^{-1}\left(\frac{\xi}{\mathfrak r}\right)$ is not highlighted in \eqref{ABS}$_6$ and \eqref{InCond}$_5$ as well is that the sums $\mathfrak r\mathfrak e_++\mathfrak z\mathfrak e_-$ or $\mathfrak r\mathfrak s_++\mathfrak z\mathfrak s_-$ are independent of $f_+^{-1}\left(\frac{\xi}{\mathfrak r}\right)$ due to the assumed structure of $\mathfrak s_\pm$, $\mathfrak e_\pm$ specified in \eqref{TEneEnt}. 
	Taking this fact into consideration we define 
	\begin{equation}\label{CrlESDef}
		\begin{split}
			\mathcal E(\mathfrak r,\mathfrak z,\vartheta)=&\frac{b}{3}\vartheta^4+\mathfrak r\mathsf e_+(\mathfrak r,\vartheta)+\mathfrak z\mathsf e_-(\mathfrak z,\vartheta),\\
			\mathcal S(\mathfrak r,\mathfrak z,\vartheta)=&\frac{4b}{3}\vartheta^3+\mathfrak r\mathsf s_+(\mathfrak r,\vartheta)+\mathfrak z\mathsf s_-(\mathfrak z,\vartheta).
		\end{split}
	\end{equation}

\begin{remark}
    For purely mathematical reasons we are not able to show existence of a solution satisfying \eqref{BNS}$_6$ with the equality sign but only with "$\geq$". In order to circumvent this issue, we include a nonnegative term in definition of a solution that in a certain sense dominates the term $\frac{1}{\vartheta}\left(\mathbb S(\vartheta,\nabla u)\cdot\nabla u+\frac{\kappa(\vartheta)}{\vartheta}|\nabla\vartheta|^2\right)$, which is nonnegative itself, cf. \eqref{STensDef} and \eqref{KGrowth}.
\end{remark}

Let us add some bibliographical remarks. In the last few years the study of compressible bi-fluid models has drawn immense interest. For the existence of global weak solutions to one-dimensional compressible bi-fluid models with large initial data we refer the readers to \cite{EvKar, EvWeZhu}. In the articles \cite{Evkar2, yaoZhu, yaoZhu2} the authors investigate one dimensional bi-fluid models with a singular pressure law. Moreover, we quote the article \cite{LiZat}, where the authors prove existence, uniqueness and stability of global weak solutions to a one dimensional bi-fluid model with a common velocity field and algebraic pressure closure.
 
Global existence results on compressible bi-fluid models in multi-dimension are limited and the theory has been quite recently developed.
In the fundamental work \cite{VaWeYu19}, the authors establish global existence of weak solutions for a bi-fluid model with a pressure law of the form $P(\rho,Z)=\rho^{\gamma}+Z^{\beta},$ where $\gamma>\frac{9}{5},$ $\beta\ge 1$ and the densities are comparable. The proof in \cite{VaWeYu19} relies on {\it a new compactness} of the quantity $\frac{Z}{\rho}$ which further allows for a variable reduction in the pressure law.  Improvements on the result of \cite{VaWeYu19} is obtained in \cite{NovPok20} where the authors are able to incorporate more intricate non-monotone pressure functions.
 In both the articles \cite{NovPok20, VaWeYu19} the densities are comparable. { Concerning more general solutions so-called dissipative solution or general boundary conditions we refer to \cite{BNN,KKNN}.

 Related to this discussion we quote here a very recent article \cite{Wen}, where the author considers a bi-fluid system in a time-independent domain of class $C^{2+s},$ $s>0$ with a pressure law of the form $P(\rho,Z)\sim\rho^{\gamma}+Z^{\beta},$ $\gamma,\beta\geq\frac{9}{5},$ and without any domination/ comparison of the densities involved. The result of \cite{Wen} extends the one proved in \cite{VaWeYu19} by allowing transition to each single phase flow, meaning that one of the phases can vanish in a point while the other can persist. 
 Let us mention also another recent article \cite{BrZat} in which the authors prove the
global existence theory of weak solutions for a bi-fluid Stokes equations on the d-dimensional torus for $d=2,3.$ The proof of \cite{BrZat} relies on the Bresch-Jabin’s new compactness tools for compressible Navier-Stokes equations. }

	Prior to finishing this section, we describe the strategy of the proof of the main result  that facilitates reader's orientation in this paper. As was already mentioned above, the first step in the existence proof for system \eqref{BNS}, i.e. the proof of Theorem~\ref{Thm:BNSEx}, is the transformation of \eqref{BNS} into \eqref{ABS}. The obvious advantage of this approach is the Navier--Stokes--Fourier like structure of the transformed system. Moreover, the fact that one substitutes the transport equation for $\alpha$ by the continuity equation for $\xi$ allows for the application of the available theory for continuity equations for passages from one approximation level to another. 
	
	At this moment, the natural approach is to deal with the existence result for \eqref{ABS}, which is of independent interest. The global in time existence of weak solutions to \eqref{ABS} is announced in Theorem~\ref{Thm:ABSEx} under hypotheses identified in Section~\ref{Sec:Hyp}. The majority of the paper content is devoted to the proof of Theorem~\ref{Thm:ABSEx}. In the part of the proof involving the construction of a solution to approximate system we adapt the corresponding construction in the  monofluid case from \cite[Section 3.3]{FeNo09}. In the proof of the almost compactness result for the ratio of density approximations that are pointwisely comparable and in the proof of a compactness result for approximations of the temperature we adapt the approach from \cite{KNAC}.
	
	The proof of Theorem~\ref{Thm:ABSEx} consists of the following steps:
	\begin{itemize}
		\item Solving the regularized form of \eqref{ABS} with two small parameters $\varepsilon>0$ characterizing the dissipation in all continuity equations and $\delta>0$ characterizing the artificial pressure added in \eqref{ABS}$_5$ and considering the balance of internal energy instead of the entropy production equation \eqref{ABS}$_6$. The existence of a solution to the approximate problem is shown via the Faedo--Galerkin scheme in the spirit of monofluid situation treated in \cite[Section 3.3]{FeNo09}. First, a solution to the Faedo--Galerkin approximation of the regularized momentum equation is shown to exist locally in time by the Schauder fixed point theorem. Next, deriving suitable uniform estimates of the solution allows for extending the solution to the whole given time interval. Then the balance of internal energy is converted to the entropy production equation provided that the approximate temperature is positive. Eventually, the limit passage to infinity with the Faedo--Galerkin parameter is performed.
		
		\item Performing the passage $\varepsilon\to 0_+$, i.e. letting the artificial viscosity in continuity equations tend to zero, relies on improved uniform integrability of sequences of  densities approximations, the compensated compactness arguments for showing the compactness of temperature approximations, the almost compactness of the ratio of densities approximations and a variant of the effective viscous flux identity.
		\item Performing the passage $\delta\to 0_+$, i.e. letting the artificial pressure term in the momentum equation tend to zero, follows the same procedure as in the previous step with minor differences caused by a lower exponent characterizing the uniform integrability of density approximations.
	\end{itemize}
	
	Having Theorem~\ref{Thm:ABSEx} proved we use it for showing the global in time existence of a weak solution to a system \eqref{ABS} created by a transformation of Baer--Nunziato one velocity system \eqref{BNS} with initial conditions involving the initial conditions for \eqref{BNS} in a suitable way. Then applying the almost uniqueness result for renormalized bounded solutions to the transport equations established in \cite{Nov20} we can conclude the proof of Theorem~\ref{Thm:BNSEx}.

	This introductory section is finished with some notations. The standard notation is used for the Lebesgue and Sobolev spaces equipped by the standard norms $\|\cdot\|_{L^p(\Omega)}$, $\|\cdot\|_{W^{k,p}(\Omega)}$ respectively. Furthermore, we use function spaces $W^{1,2}_0(\Omega)$ and $W^{1,2}_n(\Omega)$ defined on a bounded domain $\Omega\subset\eR^d$ with the Lipschitz boundary as
    \begin{equation}\label{SobDef}
        \begin{split}
            W^{1,2}_0(\Omega)=&\overline{C^\infty_c(\Omega)}^{\|\cdot\|_{W^{1,2}}},\\
            W^{1,2}_n(\Omega)=&\{u\in W^{1,2}(\Omega):\ u\cdot n=0\text{ on }\partial\Omega\}.
        \end{split}    
    \end{equation}
    We notice that $C^\infty_c(\Omega)$ stands for a space of smooth functions with compact support in $\Omega$ and the function $u$ in the expression $u\cdot n$ on $\partial\Omega$ is understood in the sense of traces.
    By $L^p(0,T;X)$ we denote the Bochner space of $p$--integrable functions on $(0,T)$ with values in $X$ and $C([0,T];X)$ stands for the space of continuous functions on $[0,T]$ with values in $X$ equipped by norms $\|\cdot\|_{L^p(0,T;X)}$, $\|\cdot\|_{C([0,T];X)}$ respectively. The frequently used vector space $C_w([0,T];X)$ consists of functions $f\in L^\infty(0,T;X)$ for which $t\mapsto\langle \eta, f(t)\rangle$ is for any $\eta\in X^*$ continuous on $[0,T]$. The universal generic constants are denoted by $c$, $\overline c$, $\underline c$ and can take different values even in the same formula.
 
	\Needspace{11\baselineskip}
	\section{Hypotheses}\label{Sec:Hyp}
	We gather here all the assumptions on data and constitutive functions.
	\begin{enumerate}
		\item Regularity of initial data
		\begin{align*}
			\mathfrak r_0\in L^{\gamma_+}(\Omega), \gamma_+\geq\frac{9}{5},\ \int_\Omega\mathfrak r_0>0,\ \mathfrak z_0\in L^{\gamma_-}(\Omega)\text{ if }\gamma_->\gamma_+,\ \int_\Omega \mathfrak z_0>0,\\
			(\mathfrak r_0,\mathfrak z_0)(x)\in\overline{\mathcal O_{\underline a,\overline a}}, (\mathfrak r_0+\mathfrak z_0,\Sigma_0)(x)\in \overline{\mathcal O_{\underline b,\overline b}},(\mathfrak r_0,\xi_0)(x)\in\overline{\mathcal O_{\underline d,\overline d}}\text{ for a.a. }x\in\Omega,\\
			(\Sigma u)_0\in L^1(\Omega),\frac{|(\Sigma u)_0|^2}{\Sigma_0}\in L^1(\Omega), \vartheta_0\in L^4(\Omega),\log\vartheta_0\in L^2(\Omega),\\
			\mathfrak r_0 \mathfrak s_+\left(\mathfrak r_0,\vartheta_0\right)+\mathfrak z_0\mathfrak s_-\left(\mathfrak z_0,\vartheta_0\right)\in L^1(\Omega),\\
			\mathfrak r_0 \mathfrak e_+\left(\mathfrak r_0,\vartheta_0\right)+\mathfrak z_0\mathfrak e_-\left(\mathfrak z_0,\vartheta_0\right)\in L^1(\Omega),
		\end{align*}
		where
		\begin{equation}\label{MODef}
			\mathcal O_{\underline a,\overline a}=\{(r,z)\in\eR^2:r\in(0,\infty), \underline a r< z<\overline a r\}
		\end{equation} 
		and
		\begin{equation*}
			0<\underline a<\overline a<\infty, 0<\underline b<\overline b<\infty, 0<\underline d<\overline d<\infty.
		\end{equation*}
		\item Transport coefficients $\mu,\eta, \kappa\in C^1([0,\infty))$ satisfy
		\begin{alignat}{2}
			\overline c^{-1}(1+\vartheta)&\leq \mu(\vartheta)&&\leq \overline c(1+\vartheta),\label{MuGr}\\
			0&\leq \eta(\vartheta)&&\leq \overline c(1+\vartheta),\label{EGr}\\
			\overline c^{-1}(1+\vartheta)^\beta&\leq \kappa(\vartheta)&&\leq \overline c(1+\vartheta)^\beta,\ \beta>\frac{4}{3}\label{KGrowth}.
		\end{alignat}
		\item Structure and regularity of pressure, internal energy and specific entropy\\
		We collect the assumptions on the quantities $\mathsf P_{\pm}$, $\mathsf e_{\pm}$, $\mathsf s_{\pm}$ appearing in Section~\ref{Sec:Intr}, cf. \eqref{MPDef} and \eqref{COV}.
		\begin{itemize}
			\item Regularity
			\begin{equation*}
				\mathsf P_{\pm},\mathsf e_{\pm}\in C([0,\infty)^2)\cap C^2((0,\infty)^2),\ \mathsf P_{\pm}(0,0)=0,\ \mathsf e_{\pm}(0,0)=0;
			\end{equation*}
			\item Thermodynamic stability conditions
			\begin{equation}\label{ThStab}
				\partial_r \mathsf P_{\pm}(r,\vartheta)>0,\ \partial_\vartheta \mathsf e_{\pm}(r,\vartheta)>0\text{ for all }(r,\vartheta)\in (0,\infty)^2;
			\end{equation}
			\item Structure of specific entropy
			\begin{equation}\label{SpEntStr}
				\mathsf s_{\pm}(r,\vartheta)=\bar s_{\pm}\log\vartheta-\tilde s_{\pm}\log r+\mathscr{s}_{\pm}(r,\vartheta)\text{ for some constants }\bar s_{\pm}, \tilde s_{\pm}\geq 0;
			\end{equation}
			\item Structure of the pressure, the precise meaning of \eqref{MPDef}\\
			There are $f_{\pm}\in C^1([0,1];[0,\infty))$ and $\mathsf P_{\pm}:[0,\infty)^2\to[0,\infty)$ such that 
			\begin{equation*}
				\alpha_{\pm}P_\pm(r,\vartheta)=\mathsf P_\pm(f_\pm(\alpha_\pm)\alpha_\pm r,\vartheta)\text{ for all }\alpha_\pm\in [0,1], r\in [0,\infty), \vartheta>0.
			\end{equation*}
			Moreover, the function $p:[0,\infty)\times[\underline a,\overline a]\times (0,\infty)\to \eR$ defined as 
			\begin{equation*}
				p(r,\zeta,\vartheta)=\mathsf P_+(r,\vartheta)+\mathsf P_-(r\zeta,\vartheta)
			\end{equation*}
			is for any $(\zeta,\vartheta)\in [\underline a,\overline a]\times(0,\infty)$ a nondecreasing function of $r$ on $[0,\infty)$.
		\end{itemize}
		\item Growth conditions for the pressure functions, the internal energies and the specific entropies
		\begin{itemize}
			\item Growth conditions for the internal energies $\mathsf e_{\pm}$
			\begin{equation}\label{IntEnGrowth}
				\overline c^{-1}(r^{\gamma_\pm-1}-1)\leq \mathsf e_{\pm}(r,\vartheta)\leq \overline c(r^{\gamma_\pm-1}+d^e_\pm \vartheta^{\omega^e_\pm})\text{ for all }(r,\vartheta)\in(0,\infty)^2,
			\end{equation}
			where $\gamma_\pm>0$ and $\omega^e_\pm>0$ satisfy
			\begin{equation*}
				\frac{1}{\gamma-1}+\frac{\omega^e_\pm}{p_\beta}<1.
			\end{equation*}
			We denote $\gamma=\max\{\gamma_+,\gamma_-\}\geq \overline \gamma_\beta$ where $\overline\gamma_\beta$ solves
			\begin{equation}\label{OGBPBDef}
				\overline\gamma_\beta+\min\left\{\frac{2\overline\gamma_\beta}{3}-1,\overline\gamma_\beta\left(\frac{1}{2}-\frac{1}{p_\beta}\right)\right\}=2\text{ with }p_\beta=\beta+\frac{8}{3}.
			\end{equation}
			Let us notice that $\overline\gamma_\beta\in\left(\frac{9}{5},2\right)$.
			\item Growth conditions for the specific entropies $\mathscr s_\pm$
			\begin{equation}\label{EntGr}
				|\mathscr s_{\pm}|\leq \overline c(r^{\gamma^s_\pm-1}+d^s_\pm \vartheta^{ \omega^s_\pm})\text{ for all }(r,\vartheta)\in(0,\infty)^2,
			\end{equation}
			where $\gamma^s_\pm,\omega^s_\pm>0$ and $\gamma=\max\{\gamma^s_+,\gamma^s_-\}$, $\omega^s=\max\{d^s_+\omega^s_+,d^s_-\omega^s_-\}$ obey
			\begin{equation*}
				\frac{\omega^s}{4}+\frac{\gamma^s}{\gamma-1}<\frac{5}{6}.
			\end{equation*} 
			\item Growth conditions for the pressures $\mathsf P_
			\pm$
			\begin{equation}\label{PressGrowth}
				\overline c^{-1}(r^{\gamma_\pm}-1)\leq \mathsf P_\pm(r,\vartheta)\leq \overline c(r^{\gamma_\pm}+d^P_\pm r^{\gamma^P_\pm}\vartheta^{\omega^P_\pm})\text{ for all }(r,\vartheta)\in(0,\infty)^2,
			\end{equation}
			where $\gamma^P_\pm,\omega_\pm^P>0$ and $\gamma^P=\max\{d^P_+\gamma^P_+,d^P_-\gamma^P_-\}$, $\omega^P=\max\{d^P_+\omega^P_+,d^P_-\omega^P_-\}$ obey
			\begin{equation*}
				\frac{\gamma^P}{\gamma}+\frac{\omega^P}{p_\beta}<1.
			\end{equation*}
			If $\gamma=\overline{\gamma}_\beta$ we assume that
			\begin{equation}\label{PressDecomp}
				p(\rho,\zeta,\vartheta)=\pi(\zeta,\vartheta)\rho^\gamma+\mathscr p(\rho,\zeta,\vartheta)
			\end{equation} 
			for some $\pi\in L^\infty((\underline a,\overline a)\times(0,\infty))$ such that
			\begin{equation*}
				\essinf_{(\underline a,\overline a)\times(0,\infty)}\pi\geq \underline\pi>0
			\end{equation*}
			and some $\mathscr p\in C([0,\infty)\times[\underline a,\overline a]\times [0,\infty))$ such that
			the mapping $\rho\mapsto \mathscr p(\rho,\zeta,\vartheta)$ is for all $(\zeta,\vartheta)$ continuous and nondecreasing on $[0,\infty)$.
			\item Growth conditions for the entropy and the internal energy variations
			\begin{align}
				r|\partial_r\mathscr s_\pm|(r,\vartheta)\leq& \overline c(r^{\underline \Gamma^s-1}+r^{\overline\Gamma^s-1})(1+\overline {d}^s_\pm\vartheta)^{\overline\omega^s},\label{EntRVar}\\
				r|\partial_\vartheta\mathscr s_\pm|(r,\vartheta)\leq &\overline c(1+r^{\tilde\Gamma^s})(1+\tilde {d}^s_\pm\vartheta)^{\tilde\omega^s},\label{EntTVar}\\
				r|\partial_r \mathsf e_\pm|(r,\vartheta)\leq& \overline c(r^{\underline \Gamma^e-1}+r^{\overline\Gamma^e-1})(1+\overline {d}^e_\pm\vartheta)^{\overline\omega^e}\label{EneRVar},\\
				r|\partial_\vartheta \mathsf e_\pm|(r,\vartheta)\leq &\overline c(1+r^{\tilde\Gamma^e})(1+\tilde {d}^e_\pm\vartheta)^{\tilde\omega^e}\label{EneTVar}
			\end{align}
			for all $(r,\vartheta)\in(0,\infty)^2$. In the above inequalities we consider $\underline\Gamma^a\in[0,1)$, $\overline\Gamma^a\in (0,\overline\gamma)$, $\tilde\Gamma^a\in(0,\overline\gamma)$ such that
			\begin{equation*}
				\frac{\max\{\underline\Gamma^a,\overline\Gamma^a\}}{\overline\gamma}+\frac{\overline\omega^a}{p_\beta}<1,\ \frac{\tilde\Gamma^a}{\overline\gamma}+\frac{\tilde\omega^a+1}{p_\beta}<1,
			\end{equation*}
			where the supersript $a$ stand for $e$ and $s$, $\overline\omega^a=\max\{d^a_+\overline\omega^a_+,d^a_-\overline\omega^a_-\}$, 
			\begin{equation}\label{OVGDef}
				\overline{\gamma}=\gamma+\gamma_{BOG},\ \gamma_{BOG}=\min\left\{\frac{2}{3}\gamma-1,\left(\frac{1}{2}-\frac{1}{p_\beta}\right)\gamma\right\}.
			\end{equation}

			\item Growth conditions for the pressure variations
			\begin{equation}\label{PressVarGr}
				\begin{split}
					|\partial_r\mathsf P_\pm|(r,\vartheta)\leq &\overline c(r^{\underline \Gamma^P-1}+r^{\overline \Gamma^P-1})(1+\overline d^P_\pm\vartheta)^{\overline\omega^P},\\
					|\partial_\vartheta\mathsf P_\pm|(r,\vartheta)\leq &\overline c(1+r^{\tilde \Gamma^P})(1+\vartheta)^{\tilde\omega^P_\pm}
				\end{split}
			\end{equation}
			for all $(r,\vartheta)\in (0,\infty)^2$ with some $\underline\Gamma^P\in[0,1)$, $\overline\Gamma^P\in (0,\gamma)$, $\tilde\Gamma^P\in(0,\gamma)$ such that
			\begin{equation*}
				\frac{\max\{\underline\Gamma^P,\overline\Gamma^P\}}{\overline\gamma}+\frac{\overline\omega^P}{p_\beta}<1,\ \frac{\tilde\Gamma^P}{\gamma}+\frac{\tilde\omega^P+1}{p_\beta}<1,
			\end{equation*}
			where $\overline\omega^P=\max\{d^P_+\overline\omega^P_+,d^P_-\overline\omega^P_-\}$, $\tilde\omega^P=\max\{d^P_+\tilde\omega^P_+,d^P_-\tilde\omega^P_-\}$ and $\overline\gamma$ is given by \eqref{OVGDef}.
		\end{itemize}
	\end{enumerate} 
 
	\begin{remark}
    The exponents $\gamma_\pm$, $\omega^e_\pm$, $\overline\gamma_\beta$, $\gamma^s_\pm$, $\omega^s_\pm, \gamma^P_\pm$, $\omega^P_\pm$, $\underline\Gamma^s$, $\overline\Gamma^s$, $\overline\omega^s$, $\tilde\Gamma^s$, $\overline\Gamma^s$, $\tilde\omega^s$, $\underline\Gamma^e$, $\overline\Gamma^e$, $\overline\omega^e$, $\tilde\omega^e$, $\underline\Gamma^P$, $\overline\Gamma^P$, $\overline\omega^P$, $\tilde\omega^P$ and the inequalities relating them are used for obtaining bound on approximating sequences independent of any regularizing parameters in Section~\ref{SubSec:DEst}.
    \end{remark}
    \begin{remark}
	    The number $p_\beta$ is introduced in \eqref{OGBPBDef} for a specification of an integrability exponent of $\vartheta$ that is later used for the interpolation with $L^\infty(0,T;L^4(\Omega))$ regularity of $\vartheta$ obtained from the radiative part of the internal energies, cf. \eqref{TEneEnt}.
        We obviously have
        \begin{equation}\label{PBLBound}
            p_\beta>4
        \end{equation}
        due to \eqref{KGrowth}.
	\end{remark}
 
	\begin{remark}
		For a couple $(g,h)\in [0,\infty]^2$ and all $\mathfrak a\in\eR$ we define
		\begin{equation}\label{DivConv}
			g/_{\mathfrak a}h
			=\begin{cases}
				g/h&\text{ if }h>0,\\
				\mathfrak a&\text{ if }h=0.
			\end{cases}
		\end{equation}		
	\end{remark}

	\section{Preliminaries}
	The following two lemmas will be useful for estimates of the full norm in Sobolev spaces provided only the estimate of the norm of gradient, its symmetric traceless part respectively. For the proof of the Poincar\'e type inequality in the following lemma we refer to \cite[Lemma 3.2]{Feireisl04}. 
	\begin{lemma}\label{Lem:Poincare}
		Let $\Omega\subset\eR^3$ be a bounded Lipschitz domain, $v\in W^{1,2}(\Omega)$ and nonnegative function $\rho$ fulfill
		\begin{equation*}
			\int_\Omega\rho^\gamma\leq K,\ \int_\Omega\rho\geq M>0
		\end{equation*}
		with some $\gamma>1$. Then there is $c=c(K,M,\gamma)>0$ such that
		\begin{equation*}
			\|v\|_{W^{1,2}(\Omega)}\leq c\left(\|\nabla v\|_{L^2(\Omega)}+\int_\Omega\rho|v|\right).
		\end{equation*}
	\end{lemma}
	The ensuing lemma that deals with the generalized Korn--Poincar\'e inequality is taken from \cite[Theorem 10.17]{FeNo09}.
	\begin{lemma}\label{Lem:KornPoinc}
		Let $\Omega\subset\eR^3$ be a bounded Lipschitz domain, $v\in W^{1,2}(\Omega)$ and nonnegative function $\rho$ fulfill
		\begin{equation*}
			\int_\Omega\rho^\gamma\leq K,\ \int_\Omega\rho\geq M>0
		\end{equation*}
		with some $\gamma>1$. Then there is $c=c(K,M,\gamma)>0$ such that
		\begin{equation*}
			\|v\|_{W^{1,2}(\Omega)}\leq c\left(\left\|\nabla v+(\nabla v)^\top-\frac{2}{3}\dvr v\mathbb I\right\|_{L^2(\Omega)}+\int_\Omega\rho|v|\right).
		\end{equation*}
	\end{lemma}

	Next, we state the lemma connecting families of parametrized measures and the weak convergence. For its proof see \cite[Theorem 6.2]{Ped97}.  
	\begin{lemma}\label{Lem:YMeas}
		Let $Q$ be a domain in $\eR^N$, $\{v_n\}$, $v_n:Q\to\eR^M$ be a weakly convergent sequence of functions in $L^1(Q)$. Then there is a nonrelabeled subsequence $\{v_n\}$ and a parametrized family $\{\nu_y\}_{y\in Q}$ of probability measures on $\eR^M$ depending measurably on $y\in Q$, meaning that 
		$\nu_y$ is a probability measure for a.a. $y\in Q$ and
		\begin{equation*}
			y\mapsto \int_{\eR^M}\phi(\lambda)\mathrm d\nu_y(\lambda)=\langle\nu_y,\phi\rangle\text{ is measurable on }Q,
		\end{equation*}
		possessing the property:
		
		For any Caratheodory function $\Phi:Q\times \eR^M\to\eR$, meaning that
		$\lambda\mapsto\Phi(x,\lambda)$ is continuous on $\eR^M$ for a.a. $x\in Q$ and $x\mapsto\Phi(x,\lambda)$ is measurable on $Q$ for all $\lambda\in\eR^M$, such that
		\begin{equation*}
			\Phi(\cdot,v_n)\rightharpoonup \bar\Phi\text{ in }L^1(Q)
		\end{equation*} 
		it follows that 
		\begin{equation}\label{WLId}
			\bar\Phi(y)=\int_{\eR^M}\Phi(y,z)\mathrm d\nu_y(z)\text{ for a.a. }y\in Q. 
		\end{equation}
		The family $\{\nu_y\}$ associated to a weakly convergent sequence $\{v_n\}$ in $L^1(Q)$ is termed Young measure. 
	\end{lemma}
	We use several times the following consequence of the latter assertion.
	\begin{lemma}\label{Lem:LimProd}
		Let $Q$ be a domain in $\eR^N$, $\{Z_n\}$, $Z_n:Q\to\eR^M$, $\{\theta_n\}$, $\theta_n:Q\to\eR$ and $\{(Z_n,\theta_n)\}$ be weakly convergent sequences in $L^1(Q)$. Moreover, let
		\begin{equation}\label{ProdLimId}
			\overline{b(Z)g(\theta)}=\overline{b(Z)}\ \overline{g(\theta)}
		\end{equation}  
		hold for any $b\in C^1(\eR^M)$ with $\nabla b\in C_c(\eR^M)$ and any $g\in C^1_b(\eR)$, where $\overline{b(Z)g(\theta)}$ denotes the weak limit of the composed sequence $\{b(Z_n)g(\theta_n)\}$ in $L^1(Q)$ and $\overline{b(Z)}$, $\overline{g(\theta)}$ stand for the weak limits of $\{b(Z_n)\}$, $\{g(\theta_n)\}$ respectively, in $L^1(Q)$. Then
		\begin{equation}\label{LimPr}
			\overline{\Phi(\cdot,Z)G(\theta)}=\overline{\Phi(\cdot,Z)}\ \overline{G(\theta)}
		\end{equation}
		holds for any Caratheodory function $\Phi:Q\times\eR^M\to\eR$ and any $G\in C(\eR)$ such that the sequences of compositions $\{\Phi(\cdot,Z_n)\}$, $\{G(\theta_n)\}$ and $\{\Phi(\cdot,Z_n)G(\theta_n)\}$ converge weakly in $L^1(Q)$.
		\begin{proof}
			Lemma~\ref{Lem:YMeas} guarantees the existence of parametrized Young measures $\{\nu^{(Z,\theta)}_y\}$, $\{\nu^{Z}_y\}$ and $\{\nu^{\theta}_y\}$ associated to the sequences $\{(Z_n,\theta_n)\}$, $\{Z_n\}$, $\{\theta_n\}$. Employing Lemma~\ref{Lem:YMeas} in combination with \eqref{ProdLimId} we get
			\begin{equation*}
				\int_{\eR^{M+1}}b(\lambda)g(\eta)\mathrm d\nu^{(Z,\theta)}_{y}(\lambda,\eta)=\overline{b(Z)g(\theta)}=\overline{b(Z)}\ \overline{g(\theta)}=\int_{\eR^{M}}b(\lambda)\mathrm d\nu^{Z}_{y}(\lambda)\int_\eR g(\eta)\mathrm d\nu^{\theta}_{y}(\eta)
			\end{equation*}
			implying the decomposition
			\begin{equation*}
				\nu^{(Z,\theta)}_y(A\times B)=\nu_y^{Z}(A)\nu_{y}^{\theta}(B)\text{ for any pair of open sets }A\subset\eR^M,\ B\subset\eR.
			\end{equation*}
			Consequently, fixing $\Phi$ and $G$ from the assumptions of the lemma we have
			\begin{equation*}
				\begin{split}
					\overline{\Phi(\cdot,Z)G(\theta)}(y)&=\int_{\eR^{M+1}}\Phi(y,\lambda)G(\eta)\mathrm d\nu^{(Z,\theta)}_y(\lambda,\eta)=\int_{\eR^{M+1}}\Phi(y,\lambda)G(\eta)\mathrm d\nu^{Z}_y(\lambda)\mathrm d\nu^{\theta}_y(\eta)\\&=\int_{\eR^{M}}\Phi(y,\lambda)\mathrm d\nu^{Z}_y(\lambda)\int_{\eR}G(\eta)\mathrm d\nu^{\theta}_y(\eta)=\overline{\Phi(\cdot,Z)}(y)\overline{G(\theta)}(y),
				\end{split}
			\end{equation*}
			where we also used the Fubini theorem, i.e. we have shown \eqref{LimPr}.
		\end{proof}
	\end{lemma}
	The following lemma concerns the weak limits of products of monotone operator. The assertion follows immediately from \cite[Theorem 10.19]{FeNo09}
	\begin{lemma}\label{Lem:MonWConv}
		Let $Q$ be a domain in $\eR^N$, $P,G:Q\times[0,\infty)\to\eR$ be functions such that $s\mapsto P(y,s)$, $s\mapsto G(y,s)$ are nondecreasing and continuous on $[0,\infty)$ for a.a. $y\in Q$. Assume that $\{s_n\}\subset L^1(Q)$ is a sequence such that
		\begin{equation*}
			\begin{split}
				P(\cdot,s_n)&\rightharpoonup\overline{P(\cdot,s)},\\
				G(\cdot,s_n)&\rightharpoonup\overline{G(\cdot,s)},\\
				P(\cdot,s_n)G(\cdot,s_n)&\rightharpoonup\overline{P(\cdot,s)G(\cdot,s)}
			\end{split}
		\end{equation*}
		in $L^1(Q)$.
		\begin{enumerate}
			\item Then
			\begin{equation*}
				\overline{P(\cdot,s)}\ \overline{G(\cdot,s)}\leq \overline{P(\cdot,s)G(\cdot,s)}.
			\end{equation*}	
			\item If $P:Q\times[0,\infty)\to\eR$, $G(s)=s$ and $\overline{P(s)s}=\overline{P(s)}s$ then $\overline{P(s)}=P(s)$
		\end{enumerate}
	\end{lemma}
	We will use several times the following variant of celebrated Div--Curl lemma, see \cite[Section 10.13]{FeNo09} for the proof.
	\begin{lemma}\label{Lem:DivCurl}
		Let $Q\subset\eR^N$ be an open set. Assume that
		\begin{equation*}
			\begin{alignedat}{2}
				U_n\rightharpoonup U\text{ in }L^p(Q;\eR^N),\\
				V_n\rightharpoonup V\text{ in }L^q(Q;\eR^N),
			\end{alignedat}
		\end{equation*}
		where $\frac{1}{p}+\frac{1}{q}=\frac{1}{r}<1$.
		In addition, let the sequnces $\{\mathrm{Div}\ U_n\}$ and $\{\mathrm{Curl}\  V_n\}$ are precompact in $W^{-1,s}(Q)$ for a certain $s>1$, where $\mathrm{Div}\  U=\sum^N_{j=1}\partial_{j}U_j$ and $\mathrm {Curl}\ V=(\partial_iV_j-\partial_jV_i)_{i,j=1}^N$.
		Then 
		\begin{equation*}
			U_n\cdot V_n\rightharpoonup U\cdot V\text{ in }L^r(Q;\eR^N).
		\end{equation*}
	\end{lemma}
	During the analysis we will work with the pseudodifferential operators $\mathfrak R=\nabla\otimes\nabla\Delta^{-1}$  and $\mathfrak U=\nabla\Delta^{-1}$ being defined in terms of the standard Fourier transform $\mathfrak F$ in the following way
	\begin{equation}\label{RUOpDef}
		\mathfrak U_j(v)=-\mathfrak F^{-1}\left(\frac{i\xi_j}{|\xi|^2}\mathfrak F(v)(\xi)\right),\ \mathfrak R_{ij}(v)=\mathfrak F^{-1}\left(\frac{\xi_i\xi_j}{|\xi|^2}\mathfrak F(v)(\xi)\right).
	\end{equation}
	The next lemma may be seen as a consequence of the Div--Curl lemma involving a commutator of the Riesz operator $\mathfrak R$, see \cite[Theorem 10.27]{FeNo09}.
	\begin{lemma}\label{DCRiesz}
		Let $u_n, v_n:\eR^N\to\eR^N$ be such that
		\begin{equation*}
			u_n\rightharpoonup u\text{ in }L^p(\eR^N),\ v_n\rightharpoonup v\text{ in }L^q(\eR^N),
		\end{equation*}
		where $\frac{1}{p}+\frac{1}{q}=\frac{1}{s}<1$. Then
		\begin{equation*}
			u_n\cdot\mathfrak R_{ij}(v_n)-\mathfrak R_{ij}(u_n)\cdot v_n\rightharpoonup u\cdot\mathfrak R_{ij}(v)-\mathfrak R_{ij}(u)\cdot v\text{ in }L^s(\eR^N).
		\end{equation*}
		
	\end{lemma}
	The next lemma states a commutator type estimate for operators $\mathfrak R_{ij}$ in the spirit of Coifman, Meyer \cite{CoMe75}, see \cite[Theorem 10.28]{FeNo09} for the detailed proof.
	\begin{lemma}\label{Lem:ComEst}
		Let $w\in W^{1,p}(\eR^N)$ and $V\in L^q(\eR^N;\eR^N)$ be given with
		\begin{equation*}
			1<r<N,\ \frac{1}{r}-\frac{1}{N}+\frac{1}{p}<1.
		\end{equation*}
		Then there exists $\beta>0$ and $s=s(p,q)>1$ such that
		\begin{equation*}
			\|\mathfrak R_{ij}(w V_j)-w\mathfrak R_{ij}(V_j)\|_{W^{\beta,s}(\eR^N)}\leq c(p,q)\|w\|_{W^{1,r}(\Omega)}\|V\|_{L^p(\eR^N;\eR^N)}.
		\end{equation*}
	\end{lemma}
	We end this preliminary section with a nowadays basic compactness result for families of continuity equations, which originates in \cite{VaWeYu19}. The presented version comes from \cite[Proposition 7]{NovPok20}. By $X^{1,2}$ used in the rest of this section we mean either the space $W^{1,2}_0(\Omega)$ or $W^{1,2}_n(\Omega)$ defined in \eqref{SobDef}.
	\begin{lemma}\label{Lem:AlmComp}
		\
		\begin{enumerate}
			\item Let $u_n\in L^2(0,T;X^{1,2})$, $(r_n,z_n)\in \overline{\mathcal O_{\underline a,\overline a}}\cap\left(\left(C([0,T];L^1(\Omega))\right)^2\cap \left(L^2(Q_T)\right)\right)^2$
			for $\mathcal O_{\underline a,\overline a}$ defined in \eqref{MODef}. Suppose that
			\begin{equation*}
				\sup_{n\in\eN}\left(\|r_n\|_{L^\infty(0,T;L^\gamma(\Omega))}+\|r_n\|_{L^2(Q_T)}+\|u_n\|_{L^2(0,T;W^{1,2}(\Omega))}\right)<\infty
			\end{equation*}
			with $\gamma>\frac{6}{5}$ and both couples $(r_n,u_n)$, $(z_n,u_n)$ fulfill the continuity equation
			\begin{equation*}
				\tder \rho+\dvr(\rho u)=0\text{ in }\mathcal D'(Q_T).
			\end{equation*}	
			Then there exists a nonrelabeled subsequence such that
			\begin{equation*}
				\begin{alignedat}{2}
					(r_n,z_n)&\to (r,z)&&\text{ in }(C_w([0,T];L^\gamma(\Omega)))^2,\\
					u_n&\rightharpoonup u&&\text{ in }L^2(0,T;W^{1,2}(\Omega)),
				\end{alignedat}
			\end{equation*}
			where 
			\begin{equation*}
				(r,z)\in \overline{\mathcal O_{\underline a,\overline a}}\cap \left(L^2(0,T;L^2(\Omega))\right)^2\cap \left(L^\infty(0,T;L^\gamma(\Omega))\right)^2\cap \left(C([0,T];L^1(\Omega))\right)^2
			\end{equation*}
			and $(r,z)$ fulfills the time integrated continuity equation up to the boundary in the renormalized sense, i.e. 
			\begin{equation*}
				\int_\Omega b(\rho)\phi(t)-\int_\Omega b(\rho_0)\phi(0)=\int_0^t\int_\Omega b(\rho)\left(\tder\phi+u\cdot\nabla \phi\right)+\left( b(\rho)-\rho b'(\rho) \right)\dvr u\phi
			\end{equation*}
			for all $t\in [0,T]$, all $\phi\in C^1([0,T]\times\overline\Omega)$ and all $b\in C^1([0,\infty))$ with $b'\in L^\infty((0,\infty))$.
			\item Suppose in addition to assumptions of the first item that 
			\begin{equation*}
				\lim_{n\to\infty}\int_\Omega r_n(0,\cdot)\zeta_n^2(0,\cdot)=\int_\Omega r(0,\cdot)\zeta^2(0,\cdot).
			\end{equation*}
			We define in agreement with \eqref{DivConv} for all $t\in[0,T]$ and $\mathfrak a\in [\underline a,\overline a]$
			\begin{equation*}
				\zeta_n(t,x)=z_n(t,x)/_{\mathfrak a} r_n(t,x),\ \zeta(t,x)=z(t,x)/_{\mathfrak a} r(t,x).
			\end{equation*}
			Then $\zeta_n,\zeta\in C([0,T];L^q(\Omega))$, $1\leq q<\infty$ and for all $t\in[0,T]$ $\underline a\leq \zeta_n\leq\overline a$, $\underline a\leq \zeta\leq\overline a$ for a.a. $x\in\Omega$. Moreover, both couples fulfill the time integrated pure transport equation up to the boundary, i.e.
			\begin{equation*}
				\int_\Omega \rho\phi(t)-\int_\Omega \rho_0\phi(0)=\int_0^t\int_\Omega\rho\left(\tder\phi+u\cdot\nabla\phi+\dvr u\phi\right)
			\end{equation*}
			for all $t\in[0,T]$ and all $\phi\in C^1([0,T]\times\overline\Omega)$.
			
			Finally, we get
			\begin{equation}\label{AlmConv}
				\lim_{n\to\infty}\int_\Omega(r_n|\zeta_n-\zeta|^2)(t,\cdot)=0
			\end{equation}
			for all $t\in[0,T]$.
		\end{enumerate}
	\end{lemma}
	The following lemma deals with passages between two continuity equations and a transport equation, for a more general version and its proof see \cite[Proposition 3.3]{Nov20}.
	
	\begin{lemma}\label{lem:ConToTr}
		\
		\begin{enumerate}
			\item Let $\rho,z\in L^2(Q_T)$, $(\rho,z)\in\overline{O_{\underline a,\overline a}}$, $u\in L^2(0,T;X^{1,2})$ be such that $(\rho,u)$ and $(z,u)$ satisfy the continuity equation
			\begin{equation}\label{ContEqDist}
				\int_0^T\int_\Omega r(\tder\phi+u\cdot\nabla\phi)=0\text{ for any }\phi\in C^\infty_c(Q_T).
			\end{equation}
			Let us define 
			\begin{equation*}
				s_a(t,x)=z(t,x)/_a\rho(t,x)\text{ for all }t\in[0,T]\text{ and a.a. }x\in\Omega
			\end{equation*}
			in accordance with the convention from \eqref{DivConv}. Then for any $a\geq 0$ $s_a$ fulfills the time integrated transport equation up to the boundary
			\begin{equation*}
				\int_\Omega s_a\phi(t,\cdot)-\int_\Omega s_a\phi(0,\cdot)=\int_0^t\int_\Omega s_a(\tder\phi+u\cdot\nabla \phi+\dvr u\phi)=0
			\end{equation*}
			for all $t\in[0,T]$ any $\phi\in C^1([0,T]\times\overline\Omega)$.
			\item Let $s^1,s^2\in L^\infty(Q_T)$, $s^1,s^2\geq 0$ in $Q_T$ and $u\in L^2(0,T;W^{1,2}_n(\Omega))$ be such that $(s^1,u)$ and $(s^2,u)$ fulfill the transport equation in the form 
			\begin{equation}\label{TrEqDist}
				\int_0^T\int_\Omega s(\tder\phi+u\cdot\nabla\phi+\dvr u \phi)=0
			\end{equation}
			for any $\phi\in C^1_c(Q_T)$. Let $\rho\in L^2(Q_T)\cap L^\infty(0,T;L^\gamma(\Omega))$ with some $\gamma>1$, $\rho\geq 0$ in $Q_T$ be such that $(\rho,u)$ fulfills continuity equation \eqref{ContEqDist}. Then $\rho s^1,\rho s^2\in C([0,T];L^1(\Omega))$ and $(\rho s^i,u)$ fulfills the time integrated continuity equation up to the boundary
			\begin{equation*}
				\int_\Omega \rho s^i\phi(t,\cdot)-\int_\Omega \rho s^i\phi(0,\cdot)=\int_0^t\int_\Omega\rho s^i(\tder\phi+u\cdot\nabla\phi),\ i=1,2
			\end{equation*}
			for any $t\in[0,T]$ and any $\phi\in C^1_c([0,T]\times\overline\Omega)$.
		\end{enumerate}
	\end{lemma}
	The following lemma is devoted to almost uniqueness property of solutions to renormalized transport equations, see \cite[Proposition 3.4]{Nov20} for more details and the proof.
	\begin{lemma}\label{Lem:AlmUniq}
		Let $u\in L^2(0,T;X^{1,2})$. Let $s^i\in L^\infty(Q_T)$, $s^i\geq 0$ $i=1,2$ be two solutions of transport equation \eqref{TrEqDist} such that $s^i\in C([0,T];L^1(\Omega))$. If moreover $s^1(0,\cdot)=s^2(0,\cdot)$ then 
		\begin{equation*}
			s^1(\tau,x)=s^2(\tau,x)\text{ for all }\tau\in[0,T] \text{ and a.a. }x\in\{\rho(\tau,\cdot)>0\},
		\end{equation*}
		where $\rho\geq 0$ is any weak solution to the continuity equation \eqref{ContEqDist} in the class $C^1([0,T];L^1(\Omega))\cap L^2(Q_T)\cap L^\infty(0,T;L^p(\Omega))$ for some $p>1$.
	\end{lemma}
	
	\section{Existence of a weak solution}
 Let us begin with the precise definition of a solution to system \eqref{BNS}, \eqref{ABS} respectively. We recall that $X^{1,2}$ stands either for the space $W^{1,2}_0(\Omega)$ or $W^{1,2}_n(\Omega)$ defined in \eqref{SobDef}.
	\begin{mydef}\label{DefWSBNS}
		The quintet $(\alpha,\rho, z, u,\vartheta)$ is a bounded energy weak solution to \eqref{BNS} if
		\begin{enumerate}
			\item $(\alpha,\rho, z, u,\vartheta)$ possesses the following regularity
			\begin{gather*}
					(\rho,z,\alpha)\in C_w([0,T];L^\gamma(\Omega)),\ \rho,z\geq 0\text{ a.e. in }Q_T,\\
					(\rho,z)(t,x)\in \overline{\mathcal O_{\underline a,\overline a}},\ \underline\alpha\leq \alpha(t,x)\leq \overline\alpha\text{ for all }t\in[0,T]\text{ and a.a. }x\in\Omega,\\
					u\in L^2(0,T;X^{1,2}), (\rho+z)|u|^2\in L^\infty(0,T;L^1(\Omega)), (\rho+z)u\in C_w([0,T];L^q(\Omega)),\\
					\mathsf P_{+}(f_+(\alpha)\rho,\vartheta),\mathsf P_{-}(f_-(1-\alpha)z,\vartheta)\in L^1(Q_T),\\
					\vartheta\in L^\infty(0,T;L^4(\Omega))\cap L^2(0,T;W^{1,2}(\Omega)),\log\vartheta\in L^2(0,T;W^{1,2}(\Omega))
			\end{gather*}
			with some $\gamma,q>1$.
			\item the continuity equation
			\begin{equation}\label{ConEqO}
				\int_\Omega r\psi(t)-\int_\Omega r_0\psi(0)=\int_0^t\int_\Omega r(\tder\psi+u\cdot\nabla\psi)
			\end{equation} 
			is fulfilled for any $t\in[0,T]$, $\psi\in C^1([0,T]\times\overline\Omega)$ with $r$ standing for $\rho$ and $z$.
			\item the transport equation
			\begin{equation}\label{TrEqO}
				\int_\Omega \alpha\psi(t)-\int_\Omega \alpha_0\psi(0)=\int_0^t\int_\Omega\left( \alpha(\tder\psi+u\cdot\nabla\psi)-\alpha\dvr u\psi\right)
			\end{equation} 
			is fulfilled for any $t\in[0,T]$, $\psi\in C^1([0,T]\times\overline\Omega)$.
			\item the momentum equation
			\begin{equation}\label{MomEqO}
				\begin{split}
					\int_\Omega (\rho+z)u\varphi(t)-&\int_\Omega(\rho_0+ z_0) u_0\varphi(0)=\int_0^t\int_\Omega\Biggl(\rho+z) u\cdot\tder\varphi+(\rho+z) u\otimes u\cdot\nabla\varphi\\
                    &+\left(\frac{b}{3}\vartheta^4+\mathsf P_+(f_+(\alpha)\rho,\vartheta)+\mathsf P_-(f_-(1-\alpha)z,\vartheta)\right)\dvr\varphi-\mathbb S(\vartheta,\nabla u)\cdot\nabla \varphi\Biggr)
				\end{split}
			\end{equation}
			is fulfilled for any $t\in[0,T]$ and $\varphi\in C^1([0,T]\times\overline{\Omega})$, where either $\varphi=0$ or $\varphi\cdot n=0$ on $(0,T)\times\partial\Omega$;
			\item the energy equality 
			\begin{equation}\label{EnerEqO}
				\begin{split}
					&\int_{\Omega}\left(\frac{1}{2}(\rho+z)|u|^2+\rho\mathfrak{  e}_+\left(\frac{\rho}{\alpha},\vartheta\right)+z \mathfrak{e}_-\left(\frac{z}{1-\alpha},\vartheta\right)\right)(t)\\
					&=\int_\Omega \left(\frac{1}{2}(\rho_0+z_0)|u_0|^2+\rho_0 \mathfrak{e}_+\left(\frac{\rho_0}{\alpha_0},\vartheta_0\right)+z_0 \mathfrak{e}_-\left(\frac{z_0}{1-\alpha_0},\vartheta_0\right)\right)
				\end{split}
			\end{equation}
			holds for a.a $t\in(0,T)$;
			\item the balance of entropy holds in the following sense: There is $\sigma\in (C([0,T]\times\overline\Omega))^*$ such that
			\begin{equation}\label{SigDomO}
				\langle\sigma,\phi\rangle\geq \int_0^T\int_\Omega\left(\frac{1}{\vartheta}\mathbb S(\vartheta,\nabla u)\cdot\nabla u+\frac{\kappa(\vartheta)}{\vartheta^2}|\nabla\vartheta|^2\right)\phi
			\end{equation}
			for any $\phi\in C([0,T]\times\overline\Omega)$, $\phi\geq 0$ and
			\begin{equation}\label{EntBalO}
				\begin{split}
					&\int_\Omega \left(\rho \mathfrak {s}_+\left(\frac{\rho}{\alpha},\vartheta\right)+z \mathfrak{s}_-\left(\frac{z}{1-\alpha},\vartheta\right)\right)\phi(t)-\int_\Omega\left(\rho_0 \mathfrak{s}_+\left(\frac{\rho_0}{\alpha_0},\vartheta_0\right)+z_0 \mathfrak{s}_-\left(\frac{z_0}{1-\alpha_0},\vartheta_0\right)\right)\phi(0)\\
					&=\int_0^t\int_\Omega\left(\left(\rho \mathfrak s_+\left(\frac{\rho}{\alpha},\vartheta\right)+z \mathfrak s_-\left(\frac{z}{1-\alpha},\vartheta\right)\right)\left(\tder\phi+u\cdot\nabla\phi\right)-\frac{\kappa(\vartheta)}{\vartheta}\nabla\vartheta\cdot\nabla\phi\right)+\langle\sigma_t,\phi\rangle,
				\end{split}
			\end{equation}
			where we denoted
			\begin{equation*}
				\langle\sigma_t,\phi\rangle=\int_{[0,t]\times\overline\Omega}\phi(\tau,x)\mathrm d\mu_\sigma(\tau,x)
			\end{equation*}
			and $\mu_\sigma$ is the unique Borel measure associated to $\sigma$ by the Riesz representation theorem.
			
		\end{enumerate}
	\end{mydef}
	\begin{mydef}\label{DefWSAbs} 
		The sixtet $(\xi,\mathfrak r,\mathfrak z,\Sigma, u,\vartheta)$ is called a bounded energy weak solution to \eqref{ABS} if
		\begin{enumerate}
			\item \begin{gather*}
					\xi, \mathfrak r, \mathfrak z, \Sigma\geq 0\text{ a.e. in }Q_T,\ \mathfrak r, \mathfrak z, \Sigma\in C_w([0,T];L^\gamma(\Omega))\text{ for some }\gamma>1,\\
					(\mathfrak r,\mathfrak z)(t,x)\in \overline{\mathcal O_{\underline a,\overline a}}, (\mathfrak r+\mathfrak z,\Sigma)(t,x)\in \overline{\mathcal O_{\underline b,\overline b}}, (\mathfrak r,\xi)(t,x)\in \overline{\mathcal O_{\underline d,\overline d}}\text{ for all }t\in[0,T]\text{ and a.a. }x\in\Omega, \\
					u\in L^2(0,T;X^{1,2}), \mathcal P(\mathfrak r,\mathfrak z,\vartheta)\in L^1(Q_T), \\
					\Sigma|u|^2\in L^\infty(0,T;L^1(\Omega)), \Sigma u\in C_w([0,T];L^q(\Omega))\text{ for some }q>1,\\
					\vartheta\in L^\infty(0,T;L^4(\Omega)), \vartheta,\log\vartheta\in L^2(0,T;W^{1,2}(\Omega));
				\end{gather*}
			\item the continuity equation
			\begin{equation}\label{ConEq}
				\int_\Omega r\psi(t)-\int_\Omega r_0\psi(0)=\int_0^t\int_\Omega r(\tder\psi+u\cdot\nabla\psi)
			\end{equation} 
			is fulfilled for any $t\in[0,T]$, $\psi\in C^1([0,T]\times\overline\Omega)$ with $r$ standing for $\xi,\mathfrak r,\mathfrak z$ and $\Sigma$;
			\item the momentum equation
			\begin{equation}\label{MomEqSig}
				\int_\Omega \Sigma u\varphi(t)-\int_\Omega(\Sigma u)_0\varphi(0)=\int_0^t\int_\Omega\left(\Sigma u\cdot\tder\varphi+\Sigma u\otimes u\cdot\nabla\varphi+\mathcal P(\mathfrak r,\mathfrak z,\vartheta)\dvr\varphi-\mathbb S(\vartheta,\nabla u)\cdot\nabla \varphi\right)
			\end{equation}
			is fulfilled for any $t\in[0,T]$ and $\varphi\in C^1([0,T]\times\overline{\Omega})$, where either $\varphi=0$ or $\varphi\cdot n=0$ on $(0,T)\times\partial\Omega$;
			\item the energy equality 
			\begin{equation}\label{EnerEq}
				\int_{\Omega}\left(\frac{1}{2}\Sigma|u|^2+\mathcal E(\mathfrak r,\mathfrak z,\vartheta)\right)(t)=\int_\Omega \left(\frac{|(\Sigma u)_0|^2}{2\Sigma_0}+\mathcal E(\mathfrak r_0,\mathfrak z_0,\vartheta_0)\right)
			\end{equation}
			holds for a.a $t\in(0,T)$;
			\item the balance of entropy holds in the following sense: There is $\sigma\in (C([0,T]\times\overline\Omega))^*$ such that
			\begin{equation}\label{SigDom}
				\langle\sigma,\phi\rangle\geq \int_0^T\int_\Omega\left(\frac{1}{\vartheta}\mathbb S(\vartheta,\nabla u)\cdot\nabla u+\frac{\kappa(\vartheta)}{\vartheta^2}|\nabla\vartheta|^2\right)\phi
			\end{equation}
			for any $\phi\in C([0,T]\times\overline\Omega)$, $\phi\geq 0$ and
			\begin{equation}\label{EntBal}
				\int_\Omega\mathcal S(\mathfrak r,\mathfrak z,\vartheta)\phi(t)-\int_\Omega\mathcal S(\mathfrak r_0,\mathfrak z_0,\vartheta_0)\phi(0)=\int_0^t\int_\Omega\left( \mathcal S(\mathfrak r,\mathfrak z,\vartheta)\left(\tder\phi+u\cdot\nabla\phi\right)-\frac{\kappa(\vartheta)}{\vartheta}\nabla\vartheta\cdot\nabla\phi\right)+\langle\sigma_t,\phi\rangle,
			\end{equation}
			where we denoted
			\begin{equation*}
				\langle\sigma_t,\phi\rangle=\int_{[0,t]\times\overline\Omega}\phi(\tau,x)\mathrm d\mu_\sigma(\tau,x)
			\end{equation*}
			and $\mu_\sigma$ is the unique Borel measure associated to $\sigma$ by the Riesz representation theorem.
		\end{enumerate}	
	\end{mydef}
	Having the bounded energy weak solution to problem \eqref{BNS} and \eqref{ABS} introduced we can state the main existence results.
	\begin{theorem}\label{Thm:ABSEx}
		Let $T>0$, $\Omega\subset\eR^3$ be a bounded domain of class $C^2$ and the hypotheses in Section~\ref{Sec:Hyp} be satisfied.  If $\gamma=\overline{\gamma_\beta}$ we assume $\frac{1}{\gamma+1}+\frac{1}{p_\beta}\leq \frac{1}{2}$ additionaly, where $\overline{\gamma_\beta}$ and $p_\beta$ are defined in \eqref{OGBPBDef}. Then problem \eqref{ABS} admits at least one bounded energy weak solution in the sense of Definition~\ref{DefWSAbs} 
	\end{theorem}
	\begin{theorem}\label{Thm:BNSEx}
		Let $T>0$, $\Omega\subset\eR^3$ be a bounded domain of class $C^2$. Let the hypothesis from Section~\ref{Sec:Hyp} be satisfied. Suppose that $f_+, f_-$ are strictly monotone and non-vanishing on $[0,1]$. Let $\gamma_+\geq\frac{9}{5}$, $\gamma_->0$
		\begin{equation}\label{CondCompO}
			\begin{split}
				0<\rho_0(x),\ 0<\underline\alpha\leq\alpha_0(x)\leq\overline\alpha<1,\ (f(\alpha_0)\rho_{0},f_-(1-\alpha_0)z_0)(x)\in \overline{\mathcal O_{\underline a,\overline a}}\text{ for a.a. }x\in\Omega,\\ \alpha_0\in L^\infty(\Omega),\ \rho_0 \in L^{\gamma_+}(\Omega),\ \int_\Omega\rho_0>0,\ z_0\in L^{\gamma_-}(\Omega)\text{ if }\gamma_->\gamma_+, \int_\Omega z_0>0,\\
				(\rho_0+z_0)|u_0|^2\in L^1(\Omega), \vartheta_0\in L^4(\Omega), \log\vartheta_0\in L^2(\Omega),\\
				\rho_0 s_+\left(\frac{\rho_0}{\alpha_0},\vartheta_0\right)+z_0s_-\left(\frac{z_0}{1-\alpha_0},\vartheta_0\right)\in L^1(\Omega),\\
				\rho_0 e_+\left(\frac{\rho_0}{\alpha_0},\vartheta_0\right)+z_0e_-\left(\frac{z_0}{1-\alpha_0},\vartheta_0\right)\in L^1(\Omega)
			\end{split}
		\end{equation}
		be given. If $\gamma=\overline{\gamma_\beta}$ we assume $\frac{1}{\gamma+1}+\frac{1}{p_\beta}\leq \frac{1}{2}$ additionaly, where $\overline{\gamma_\beta}$ and $p_\beta$ are defined in \eqref{OGBPBDef}. Then problem \eqref{BNS} admits at least one bounded energy weak solution in the sense of Definition~\ref{DefWSBNS}.
	\end{theorem}

	\section{Proof of Theorem~\ref{Thm:ABSEx}}
	\subsection{Initial data regularization}
	We begin with the description of the necessary regularization of initial data $\mathfrak r_{0}$, $\mathfrak z_0$, $\Sigma_0$, $\xi_0$ and $\vartheta_0$. Considering that $r_0$ stands for $\mathfrak r_{0}$, $\mathfrak z_0$, $\Sigma_0$, $\xi_0$ we suppose that the approximation $r_{0,\delta}$ of $r_0$ fulfills
	\begin{equation}\label{InDataReg}
		r_{0,\delta}\in C^{2,\nu}(\overline\Omega), \nu>0, \inf_{\Omega} r_{0,\delta}>0,\ \nabla r_{0,\delta}\cdot n=0\text{ on }(0,T)\times\partial\Omega,
	\end{equation}
	the approximate initial momenta are defined via
	\begin{equation}\label{SUZDDef} 
		(\Sigma u)_{0,\delta}(x)=\begin{cases}(\Sigma u)_0(x)& \text{ if }\Sigma_0(x)\leq\Sigma_{0,\delta}(x)\\
			0&\text{ else }\end{cases}
	\end{equation}
	and the approximate initial temperature obeys
	\begin{equation}\label{InTempReg}
		\vartheta_{0,\delta}\in W^{1,2}(\Omega)\cap L^\infty(\Omega), \essinf_\Omega\vartheta_{0,\delta}>0,\ \nabla\vartheta_{0,\delta}\cdot n=0\text{ on }(0,T)\times\partial\Omega.
	\end{equation}
	Next, the approximations of initial data are supposed to satisfy
	\begin{equation}\label{InDConv}
		(\mathfrak r_{0,\delta},\mathfrak z_{0,\delta},\Sigma_{0,\delta},\xi_{0,\delta})\to (\mathfrak r_{0},\mathfrak z_{0},\Sigma_{0},\xi_0) \text{ in }L^\gamma(\Omega)^4
	\end{equation}
	and $|\{\Sigma_{0,\delta}<\Sigma_0\}|\to 0$ as $\delta\to 0_+$.
	Moreover, we suppose that
	\begin{equation}\label{InDatConv}
		\begin{alignedat}{2}
			\int_\Omega\frac{|(\Sigma u)_{0,\delta}|^2}{\Sigma_{0,\delta}}&\to \int_\Omega\frac{|(\Sigma u)_0|^2}{\Sigma_0},\\
			\delta\int_\Omega\left(\mathfrak r^\Gamma_{0,\delta}+\mathfrak z^\Gamma_{0,\delta}\right)&\to 0,\\
			\delta\int_\Omega\left(\mathfrak r_{0,\delta}+\mathfrak z_{0,\delta}\right)\vartheta_{0,\delta}&\to 0,\\
			\delta\int_\Omega\left(\mathfrak r_{0,\delta}+\mathfrak z_{0,\delta}\right)\log\vartheta_{0,\delta}&\to 0,\\
			\mathcal S(\mathfrak r_{0,\delta},\mathfrak z_{0,\delta},\vartheta_{0,\delta})&\to\mathcal S(\mathfrak r_{0},\mathfrak z_{0},\vartheta_{0})&&\text{ in }L^1(Q_T),\\
			\mathcal E(\mathfrak r_{0,\delta},\mathfrak z_{0,\delta},\vartheta_{0,\delta})&\to\mathcal E(\mathfrak r_{0},\mathfrak z_{0},\vartheta_{0})&&\text{ in }L^1(Q_T).
		\end{alignedat}
	\end{equation}
	The above specified properties of approximations are satisfied by 
	\begin{equation*}
		\begin{split}
			r_{0,\delta}=&\omega_\delta*(r_0 \tau_{2\delta})+\delta,\\
			\vartheta_{0,\delta}=&T_\frac{1}{\delta}(\vartheta_0),
		\end{split}
	\end{equation*}
	where $\omega_\delta$ is the standard mollifier with radius $\delta>0$, $0\leq \tau_{2\delta}\leq 1$, $\tau_{2\delta}\in C^\infty_c(\Omega)$ is defined via
	\begin{equation*}
		\tau_{2\delta}(x)=\begin{cases}
			1&\text{ if }\dist(x,\partial\Omega)\geq 2\delta,\\
			0&\text{ if }\dist(x,\partial\Omega)\leq \delta
		\end{cases}
	\end{equation*}
	and $T_{\frac{1}{\delta}}$ is the cut-off function defined in \eqref{TKDef}.
	
	\subsection{Regularized system of equations}
	We consider the following system of equations. The functions $\xi$, $\mathfrak r$, $\mathfrak z$, $\Sigma$ satisfy the following regularized continuity equation
	\begin{equation}\label{RReg}
		\tder r+\dvr(r u)=\varepsilon\Delta r\text{ in }Q_T,
	\end{equation}
	equipped with the homgeneous Neumann boundary condition 
	\begin{equation}\label{RNeumann}
		\nabla r\cdot n=0\text{ on }(0,T)\times\partial\Omega
	\end{equation} 
	and initial conditions 
	\begin{equation}\label{RZInit}
		r(0,\cdot)=r_{0,\delta}(\cdot)\text{ in }\Omega.
	\end{equation} 
	
	The regularized momentum equation in $Q_T$ reads
	\begin{equation*}
		\tder \left(\Sigma u\right)+\dvr\left(\Sigma u\otimes u\right)+\nabla \mathcal P_\delta(\mathfrak r,\mathfrak z,\vartheta)=\dvr \mathbb S(\vartheta,\nabla u)-\varepsilon\nabla\Sigma\nabla u,
	\end{equation*}
	with $u=0$ or $u\cdot n=0$ and $\mathbb S(\vartheta,\nabla u)n\times n=0$ on $(0,T)\times\partial\Omega$ and the initial condition $u(0,\cdot)=u_0$ in $\Omega$. Following the strategy from \cite[Section 3]{FeNo09} we consider the regularized internal energy equation
	\begin{equation}\label{IntEnEq}
		\begin{split}
			&\tder \mathcal E_\delta\left(\mathfrak r,\mathfrak z,\vartheta\right)+\dvr\left(\mathcal E_\delta\left(\mathfrak r,\mathfrak z,\vartheta\right) u\right)+\mathcal P\left(\mathfrak r,\mathfrak z,\vartheta\right)\dvr u\\
			&=\mathbb
			S_\delta(\vartheta,\nabla u)\cdot\nabla u+\varepsilon\left(\partial^2_{rr}h_\delta(\mathfrak r,\mathfrak z)|\nabla\mathfrak r|^2+2\partial^2_{rz}h_\delta(\mathfrak r,\mathfrak z)\nabla\mathfrak r\cdot\nabla\mathfrak z+\partial^2_{zz}h_\delta(\mathfrak r,\mathfrak z)|\nabla\mathfrak z|^2\right)+\Delta \mathcal K_\delta(\vartheta)+\frac{\delta}{\vartheta^2}-\varepsilon\vartheta^{\Lambda+1},
		\end{split}
	\end{equation}
	with the homogeneous Neumann boundary condition
	\begin{equation}\label{IENeum} 
		\nabla \vartheta\cdot n=0\text{ on }\partial\Omega
	\end{equation}
	and the initial condition $\vartheta(0,\cdot)=\vartheta_{0,\delta}(\cdot)$ satisfying \eqref{InTempReg}.
	The following notation is used in the equations above
	\begin{align} 
		\mathcal E_\delta(\mathfrak r,\mathfrak z,\vartheta)=&\mathcal E(\mathfrak r,\mathfrak z,\vartheta)+\delta(\mathfrak r+\mathfrak z)\vartheta,\label{MEDDef}\\
        \Lambda\geq&\max\{4,\omega^P_{\pm},\overline{\omega}^e\},\label{LambdaDef}\\
		\Gamma>&\max\{6,\Lambda+1,2\gamma_\pm,2(\gamma^P_\pm+\omega^P_\pm),2\gamma^s_\pm,2(1+\omega^s_\pm),6(\gamma_\pm-1), 2\omega^e_\pm-1,2\overline\omega^e,2(\gamma^P_\pm-1),\\&2(\omega^P_\pm-1),\overline\Gamma^e-1\},\\
		\mathcal K_\delta(\vartheta)=&\int_0^\vartheta \kappa_\delta(s)\ds,\ \kappa_\delta(s)=\kappa(s)+\delta(s^\Gamma+\frac{1}{s}),\label{KappaDDef}\\
		\mathcal P_\delta(r,z,\vartheta)=&\mathcal P(r,z,\vartheta)+\delta\left(r^\Gamma+z^\Gamma+r^2+z^2\right),\label{PDDef}\\
		h_\delta(r,z)=&\frac{\delta}{\Gamma-1}\left(r^\Gamma+ z^\Gamma\right)+\delta(r^2+z^2).\label{HDDef}
	\end{align}
	
	\subsection{Existence of $\xi$, $\mathfrak r$, $\mathfrak z$ and $\vartheta$ for a fixed finite dimensional $u$ and fixed $\varepsilon,\delta>0$}\label{Sec:Appr}
	Considering a smooth velocity field, namely $u\in C([0,T];C^{2,\nu}(\overline\Omega))$, we deal with the existence of regular solutions to \eqref{RReg} and \eqref{IntEnEq}. The existence of unique solutions to regularized continuity equations and their properties needed for the analysis here are taken from \cite[Lemma 3.1]{FeNo09}. For the precise defintion of the space $X_n$ appearing in this section see Section \ref{LocSolG}.
	\begin{lemma}\label{Lem:ContApp}
		Let $\Omega\subset\eR^d$ be a bounded domain of class $C^{2,\nu}$, $\nu\in(0,1)$ and a vector field $u\in C([0,T];X^n)$ be given. Let the regularity of $r_{0,\delta}$ be specified in \eqref{InDataReg}. Then for any $\varepsilon>0$ there is a unique $r=r_u$ satisfying \eqref{RReg}--\eqref{RZInit}
		and
		\begin{equation}\label{VDef}
			r_u\in V:=\left\{r\in C([0,T];C^{2,\nu}(\overline\Omega)),\tder r\in C([0,T];C^{0,\nu}(\overline\Omega))\right\}
		\end{equation}
		Moreover, the mapping $u\mapsto r_u$ maps bounded sets in $C([0,T];X_n)$ into bounded sets in $V$ and is continuous with values in $C^1([0,T]\times\overline\Omega)$.
		
		Finally,
		\begin{align*}
			\inf_{\Omega}r_{0}\exp\left(-\int_0^t\|\dvr u\|_{L^\infty(\Omega)}\right)\leq r(t,x)\leq \sup_{\Omega}r_{0}\exp\left(\int_0^t\|\dvr u\|_{L^\infty(\Omega)}\right)
		\end{align*}
		for all $t\in[0,T]$, $x\in\Omega$.
	\end{lemma}

	\begin{lemma}\label{Lem:IEE}
		Let $\Omega\subset\eR^d$, $d=2,3$ be a bounded domain of class $C^{2,\nu}$, $\nu\in(0,1)$. Let $u\in C([0,T];X_n)$ be a given vector field and $\mathfrak r=\mathfrak r_u$, $\mathfrak z=\mathfrak z_u$ be the unique solutions to the regularized continuity equations constructed in Lemma~\ref{Lem:ContApp} corresponding to initial datum $\mathfrak r_{0,\delta}$, $\mathfrak z_{0,\delta}$ respectively, 
		Further let
		\begin{enumerate}
			\item the initial datum $\vartheta_{0,\delta}\in W^{1,2}(\Omega)\cap L^\infty(\Omega)$ satisfy $\essinf_{\Omega} \vartheta_{0,\delta}>0$;
			\item the constitutive functions ${\mathsf e}_{\pm}$, $\mathsf s_{\pm}$, $\mathsf{P}_{\pm}$ and the coefficients $\mu$, $\eta$ and $\kappa$ obey the structural assumptions from Section~\eqref{Sec:Hyp}.
		\end{enumerate}
		Then the problem \eqref{IntEnEq} with fixed $u$, $\mathfrak r_u$, $\mathfrak z_u$ possesses a unique strong solution $\vartheta=\vartheta_u$ which belongs to the regularity class
		\begin{equation*}
			Y=\{\vartheta\in L^\infty(0,T;W^{1,2}(\Omega)\cap L^\infty(\Omega));\frac{1}{\vartheta}\in L^\infty(Q_T),\ \tder\vartheta\in L^2(Q_T),\ \Delta\mathcal K_\delta(\vartheta)\in L^2(Q_T)\}
		\end{equation*}
		where $\mathcal K_\delta$ is given by \eqref{KappaDDef}. Moreover, the mapping $u\mapsto\vartheta_u$ maps bounded sets in $C([0,T];X_n)$ into bounded sets in $Y$ and is continuous with values in $L^2(0,T;W^{1,2}(\Omega))$.
	\end{lemma}
	\subsection{Approximate internal energy equation}
	The existence of a solution $\vartheta$ to \eqref{IntEnEq} is shown via the approximating procedure which we explain in detail. 
	
	\subsubsection{The solvability of the $\omega$--approximate internal energy equation}
	
	The construction of the solution to \eqref{IntEnEq} requires some preparatory work which we now present. Let $\mathfrak r,\mathfrak z$ be a solution to regularized continuity equation constructed in Lemma~\ref{Lem:ContApp} for a given velocity $u$. We extend the functions $\mathfrak r,\mathfrak z\in C([0,T];C^{2,\nu}(\overline\Omega))\cap C^1([0,T];C^{0,\nu}(\overline\Omega))$ and $u\in C([0,T];X_n)$ continuously to $\mathfrak r,\mathfrak z\in C(\eR;C^{2,\nu}(\overline\Omega))\cap C^1(\eR;C^{0,\nu}(\overline\Omega))$ and $u\in C(\eR;X_n)$ such that $\supp \mathfrak r,\mathfrak z\subset (-2T,2T)\times\overline\Omega$, $\supp u\subset (-2T,2T)\times\overline\Omega$.
	
	Moreover, we consider a sequence $\{\vartheta_{0}^\omega\}\subset C^{2,\nu}(\overline\Omega)$ such that $\vartheta^\omega_0\to \vartheta_{0,\delta}$ in $W^{1,2}(\Omega)\cap L^\infty(\Omega)$ and $\inf_{\Omega}\vartheta^\omega_0>\underline\vartheta_0$ as $\omega\to 0_+$.
	In accordance with \eqref{CrlESDef} we introduce necessary notations
	\begin{align}
		\mathcal E_{\delta,\omega}(r,z,\vartheta)=&[\langle\mathcal E\rangle]^\omega(r, z,\theta_\omega(\vartheta))+b\theta^4_\omega(\vartheta)+\delta( r+z)\vartheta,\nonumber\\
		\{D_\vartheta \mathcal E\}_{\delta,\omega}(r, z,\vartheta)=&[\langle\partial_\vartheta \mathcal E\rangle]^\omega(r, z,\vartheta)+4b\frac{\vartheta^4}{\sqrt{\vartheta^2+\omega^2}}+\delta(r+z),\nonumber\\
		\kappa_{\delta,\omega}(\vartheta)=&[\langle\kappa\rangle]^\omega(\theta_\omega(\vartheta))+\delta\left(\theta_\omega^\Gamma(\vartheta)+\frac{1}{\sqrt{\vartheta^2+\omega^2}}\right)\label{kappaDO},\\
		\mathcal K_{\delta,\omega}(\vartheta)=&\int_1^\vartheta\kappa_{\delta,\omega}(s)\ds,\nonumber\\
		\mathcal P_\omega(r,z,\vartheta)=&[\langle \mathcal P\rangle]^\omega(r,z,\theta_\omega(\vartheta))+\frac{b}{3}\theta^4_\omega(\vartheta),\nonumber\\
		G_\delta(t,x)=&\left(\partial^2_{rr}h_\delta(\mathfrak r,\mathfrak z)|\nabla \mathfrak r|^2+2\partial^2_{rz}h_\delta(\mathfrak r,\mathfrak z)\nabla\mathfrak r\cdot\nabla\mathfrak z+\partial^2_{zz}h_\delta(\mathfrak r,\mathfrak z)|\nabla z|^2\right)(t,x),\ G_{\delta,\omega}=G_\delta^\omega,\nonumber\\
		\mathbb S_{\delta,\omega}(\vartheta,\nabla u^\omega)=&\left(\langle\mu\rangle^\omega(\theta_\omega(\vartheta))\right)\left(\nabla u^\omega+((\nabla u)^\top)^\omega-\frac{2}{3}\dvr u^\omega\mathbb I\right)+\langle\eta\rangle^\omega(\theta_\omega(\vartheta))\dvr u^\omega\mathbb I,\nonumber 
	\end{align}
	where
	\begin{align}
		\theta_\omega(\vartheta)=&\frac{\sqrt{\vartheta^2+\omega^2}}{1+\omega\sqrt{\vartheta^2+\omega^2}},\label{{DefThetaO}}\\
		\langle f\rangle(s)=&\begin{cases} f(s)&\text{ if }s\in(0,\infty)\\
			\max\{\inf_{(0,\infty)}f,0\}&\text{ else}.
		\end{cases}\nonumber
	\end{align}
	We consider the mollification operator $b\mapsto b^\omega$ for $\omega\in(0,1)$ given by the convolution with a mollifier having the support in the ball with radius $\omega$. We note that this operator is applied to all independent variables in the case of functions $\langle \mathcal E\rangle$, $\langle \partial_\vartheta\mathcal E\rangle$, $\langle \mathcal P\rangle$, $\langle\mu\rangle$, $\langle\eta\rangle$, $\langle\kappa\rangle$ and to the variable $t$ only if the functions $\mathfrak r$, $\mathfrak z$, $u$ and $G$ are concerned. By virtue of \eqref{KGrowth} we conclude
	\begin{equation}\label{AppPositive}
		\begin{split}
			&\kappa_{\delta,\omega}(\vartheta)\geq \overline c^{-1} >0,\\ &\{D_\vartheta \mathcal E\}_{\delta,\omega}(\mathfrak r,\mathfrak z,\vartheta)>\delta\inf_{Q_T}(\mathfrak r+\mathfrak z)>0
		\end{split}
	\end{equation}
	for all $\vartheta\in \eR$ as $\mathfrak r+\mathfrak z$ solves the regularized continuity equation, cf. Lemma~\ref{Lem:ContApp}. The approximate internal energy equation for $\omega>0$ reads
	\begin{equation}\label{VarthetaOmPr}
		\begin{alignedat}{2}
			\{D_\vartheta \mathcal E\}_{\delta,\omega}(\mathfrak r^\omega,\mathfrak z^\omega,\vartheta)\tder\vartheta+\dvr\left(\mathcal E_{\delta,\omega}(\mathfrak r^\omega,\mathfrak z^\omega,\vartheta)u^\omega\right)-\Delta\mathcal K_{\delta,\omega}(\vartheta)=&-\partial_{\mathfrak r}\mathcal E_{\delta,\omega}(\mathfrak r^\omega,\mathfrak z^\omega,\vartheta)\tder \mathfrak r^\omega&&\\-\partial_{\mathfrak z}\mathcal E_{\delta,\omega}(\mathfrak r^\omega,\mathfrak z^\omega,\vartheta)\tder \mathfrak z^\omega&+\mathbb S_{\delta,\omega}(\vartheta,\nabla u^\omega)\cdot\nabla u^\omega
			+\varepsilon G_{\delta,\omega}\\ -\mathcal P_\omega(\mathfrak r^\omega,\mathfrak z^\omega,\vartheta){\color{red} \dvr u^\omega}
    &-\frac{\delta}{\vartheta^2+\omega^2}-\varepsilon\theta^{\Lambda+1}_\omega(\vartheta)&&\text{ in }Q_T,\\
			\nabla\vartheta\cdot n=&0&&\text{ on }(0,T)\times\partial\Omega,\\
			\vartheta(0,\cdot)=&\vartheta_{0,\omega}(\cdot)&&\text{ in }\Omega.
		\end{alignedat}
	\end{equation}
	
	The following lemma concerns the solvability of equation \eqref{VarthetaOmPr}.
	\begin{lemma}\label{Lem:VarThetaOmEx}
		Let $\Omega\subset\eR^3$ be a bounded domain of class $C^{2,\nu}$, $\nu\in(0,1)$. Suppose that the data are regularized as mentioned at the beginning of this subsection. Moreover, let $\mathfrak r=\mathfrak r_u$, $\mathfrak z=\mathfrak z_u$ be constructed in Lemma \ref{Lem:ContApp} for a given $u\in C([0,T];X_n)$.  Then for any $\omega\in(0,1)$ there exists $\vartheta=\vartheta_\omega$ possessing the regularity
		\begin{equation}\label{VThetaReg}
			\vartheta_\omega\in C([0,T];C^{2,\nu}(\overline\Omega))\cap C^1([0,T]\times\overline\Omega),\ \tder \vartheta_\omega\in C^{0,\frac{\nu}{2}}([0,T];C(\overline\Omega))
		\end{equation}
		and satisfying \eqref{VarthetaOmPr}.
		\begin{proof}
			Thanks to \eqref{AppPositive}$_2$ we can rewrite equation \eqref{VarthetaOmPr} for the unknown $\vartheta$ in the following quasilinear parabolic form
			\begin{equation*}
				\begin{split}
					\tder\vartheta -\sum_{i,j=1}^3 a_{ij}(t,x,\vartheta)\partial_{ij}\vartheta+b(t,x,\vartheta,\nabla\vartheta)=&0\text{ in }Q_T,\\
					\sum_{i,j=1}^3a_{ij}\partial_i\vartheta n_j=&0\text{ on }(0,T)\times\partial\Omega,\\
					\vartheta(0,\cdot)=&\vartheta_{0,\omega}\text{ in }\Omega,
				\end{split}
			\end{equation*}
			where
			\begin{equation*}
				a_{ij}(t,x,\vartheta)=\frac{\kappa_{\delta,\omega}(\vartheta)}{\{D_\vartheta \mathcal E\}_{\delta,\omega}(\mathfrak r^\omega,\mathfrak z^\omega,\vartheta)}\delta_{ij},\ i,j=1,2,3
			\end{equation*}
			and 
			\begin{equation*}
				\begin{split}
					&b(t,x,\vartheta,H)\\
					&=\frac{1}{\{D_\vartheta\mathcal E\}_{\delta,\omega}(\mathfrak r^\omega,\mathfrak z^\omega,\vartheta)}\left(-\kappa'_{\delta,\omega}(\vartheta)|H|^2+\partial_{\mathfrak r}\mathcal E_{\delta,\omega}(\mathfrak r^\omega,\mathfrak z^\omega,\vartheta)\left(\tder \mathfrak r^\omega+u^\omega\cdot\nabla\mathfrak r^\omega\right)\right.\\
					&\left.\quad+\partial_{\mathfrak z}\mathcal E_{\delta,\omega}(\mathfrak r^\omega,\mathfrak z^\omega,\vartheta)(\tder \mathfrak z^\omega+u^\omega\cdot\nabla \mathfrak z^\omega)+\partial_\vartheta\mathcal E_{\delta,\omega}(\mathfrak r^\omega,\mathfrak z^\omega,\vartheta)(H\cdot u^\omega)\right.\\&\left.\quad+\mathcal E_{\delta,\omega}\mathfrak r^\omega,\mathfrak z^\omega,\vartheta)\dvr u^\omega+\mathcal P_\omega(\mathfrak r^\omega,\mathfrak z^\omega,\vartheta)\dvr u^\omega-\mathbb S_{\delta,\omega}(\vartheta,\nabla u^\omega)\cdot\nabla u^\omega-\varepsilon G_{\delta,\omega}+\frac{\delta}{\vartheta^2+\omega^2}+\varepsilon\theta^{\Lambda+1}_\omega(\vartheta)\right).
				\end{split}
			\end{equation*}
			By the well known theory for quasilinear parabolic equations, cf. \cite[Theorems 7.2, 7.3, 7.4]{LaSoUr68} or the summary in \cite[Theorem 10.24]{FeNo09}, we can conclude the existence of a unique $\vartheta$ possessing the regularity specified in \eqref{VThetaReg} and satisfying \eqref{VarthetaOmPr}.
		\end{proof}
	\end{lemma}

	\begin{lemma}\label{Lem:SubSupIneq}
		Let $u\in C([0,T];X_n)$, $\mathfrak r,\mathfrak z\in C([0,T];C^2(\overline\Omega))$ with $\inf_{Q_T}(\mathfrak r+\mathfrak z)>0$ be given and $R^\omega,Z^\omega,u^\omega$ be defined via the convolution with a mollifier applied on the $t$ variable.
		Let $\underline\vartheta$ be a subsolution to \eqref{VarthetaOmPr}, i.e., $\underline\vartheta$ satisfies \eqref{VarthetaOmPr}$_1$ with the equality sign replaced by "$\leq$" and $\overline\vartheta$ be a supersolution to \eqref{VarthetaOmPr}, i.e., $\overline\vartheta$ satisfies \eqref{VarthetaOmPr}$_1$ with the equality sign replaced by "$\geq$". Suppose that $\underline\vartheta$ and $\overline\vartheta$ belong to the regularity class
		\begin{equation}\label{WClass}
			W=\{\vartheta\in L^2(0,T;W^{1,2}(\Omega));\tder\vartheta, \Delta\mathcal K_{\delta,\omega}(\vartheta)\in L^2(Q_T);
			0<\essinf_{Q_T}\vartheta\leq\esssup_{Q_T}\vartheta<\infty\}
		\end{equation}
		and
		\begin{equation}\label{InitIneq}
			\underline\vartheta(0,\cdot)\leq\overline\vartheta(0,\cdot)\text{ a.e. in }\Omega.
		\end{equation}
		Then
		\begin{equation}\label{SubSupIneq}
			\underline\vartheta\leq\overline\vartheta\text{ a.e. in }Q_T.
		\end{equation}
		\begin{proof}
			Taking the difference of inequalities satisfied by $\underline\vartheta$, $\overline\vartheta$ respectively, multiplied on $\sgnp(\underline\vartheta-\overline\vartheta)$
			where
			\begin{equation*}
				\sgnp(y)=\begin{cases}
					0&\text{ if }y\leq 0,\\
					1&\text{ if }y>0
				\end{cases}
			\end{equation*}
			we get after obvious manipulations
			\begin{equation}\label{SSIneq}
				\begin{split}
					&\tder \left(\mathcal E_{\delta,\omega}(\mathfrak r^\omega,\mathfrak z^\omega,\underline\vartheta)-\mathcal E_{\delta,\omega}(\mathfrak r^\omega,\mathfrak z^\omega,\overline\vartheta)\right)\sgnp(\underline\vartheta-\overline\vartheta)\\
					&\quad+\dvr((\mathcal E_{\delta,\omega}(\mathfrak r^\omega,\mathfrak z^\omega,\underline\vartheta)-\mathcal E_{\delta,\omega}(\mathfrak r^\omega,\mathfrak z^\omega,\overline\vartheta))u^\omega)\sgnp(\underline\vartheta-\overline\vartheta)-\Delta\left(\mathcal K_{\delta,\omega}(\underline\vartheta)-\mathcal K_{\delta,\omega}(\overline\vartheta)\right)\sgnp(\underline\vartheta-\overline\vartheta)\\
					&\quad+\Bigl(\left([\langle\partial_\vartheta \mathcal E\rangle]^\omega(\mathfrak r^\omega,\mathfrak z^\omega,\underline\vartheta)-\partial_\vartheta[\langle\mathcal E\rangle]^\omega(\mathfrak r^\omega,\mathfrak z^\omega,\theta_\omega(\underline\vartheta))\theta'_\omega(\underline\vartheta)\right)\tder\underline\vartheta\\
					&\quad\quad-\left([\langle\partial_\vartheta \mathcal E\rangle]^\omega(\mathfrak r^\omega,\mathfrak z^\omega,\overline\vartheta)-\partial_\vartheta[\langle\mathcal E\rangle]^\omega(\mathfrak r^\omega,\mathfrak z^\omega,\theta_\omega(\overline\vartheta))\theta'_\omega(\overline\vartheta)\right)\tder\overline\vartheta\Bigr)\sgnp(\underline\vartheta-\overline\vartheta)\\
					&\quad+4b\left(\left(\frac{\underline\vartheta^4}{\sqrt{\underline\vartheta^2+\omega^2}}-\theta^3_\omega(\underline\vartheta)\theta_\omega'(\underline\vartheta)\right)\tder\underline\vartheta-\left(\frac{\overline\vartheta^4}{\sqrt{\overline\vartheta^2+\omega^2}}-\theta_\omega^3(\overline\vartheta)\theta'_\omega(\overline\vartheta)\right)\tder\overline\vartheta\right)\sgnp(\underline\vartheta-\overline\vartheta)\\
					&\leq \left(F_\omega(t,x,\underline\vartheta)-F_\omega(t,x,\overline\vartheta)\right)\sgnp(\underline\vartheta-\overline\vartheta)
				\end{split}
			\end{equation}
			where
			\begin{align*}
				F_\omega(t,x,\vartheta)=&\mathbb S_{\delta,\omega}(\vartheta,\nabla u^\omega(t,x))\cdot\nabla u^\omega(t,x)+\varepsilon G_{\delta,\omega}(t,x)-\mathcal P_\omega(\mathfrak r^\omega(t,x),\mathfrak z^\omega(t,x),\vartheta) \dvr u^{\omega}
    +\frac{\delta}{\vartheta^2+\omega^2}-\varepsilon\theta^{\Lambda+1}_\omega(\vartheta).
			\end{align*}
    
			As the functions $\vartheta\mapsto \mathcal E_{\delta,\omega}(\mathfrak r^\omega,\mathfrak z^\omega,\vartheta)$ due to \eqref{ThStab}, $\vartheta\mapsto \theta_\omega^4(\vartheta)$, $\vartheta\mapsto \mathcal K_{\delta,\omega}(\vartheta)$ and $\vartheta\mapsto h(\vartheta), g(\vartheta), f(\vartheta)$, where $h$ stands for a primitive function to $\frac{\vartheta^4}{\sqrt{\vartheta^2+\omega^2}}$, $g$ for a primitive function to $[\langle\partial_\vartheta\mathcal E\rangle]^\omega(\mathfrak r^\omega,\mathfrak z^\omega,\vartheta)$, $f$ for a primitive function to $\partial_\vartheta[\langle\mathcal E\rangle]^\omega(\mathfrak r^\omega,\mathfrak z^\omega,\theta_\omega(\vartheta))\theta'_\omega(\vartheta)$, are nondecreasing functions due to \eqref{ThStab}, we get 
			\begin{equation}\label{SGNPEq}
				\begin{split}
					&\sgnp(\mathcal E_{\delta,\omega}(\mathfrak r^\omega,\mathfrak z^\omega,\underline\vartheta))-\mathcal E_{\delta,\omega}(\mathfrak r^\omega,\mathfrak z^\omega,\overline\vartheta))=\sgnp(\theta_\omega^4(\underline\vartheta)-\theta_\omega^4(\overline\vartheta))=\sgnp(\mathcal K_{\delta,\omega}(\underline\vartheta)-\mathcal K_{\delta,\omega}(\overline\vartheta))\\
					&=\sgnp(h(\underline\vartheta)-h(\overline\vartheta))=\sgnp(g(\underline\vartheta)-g(\overline\vartheta))=\sgnp(f(\underline\vartheta)-f(\overline\vartheta))=\sgnp(\underline\vartheta-\overline\vartheta).
				\end{split}
			\end{equation}
			Using the following observations for $w\in W^{1,2}(Q_T)$
			\begin{equation*}
				\tder|w|^+=\sgnp w\tder w,\ \nabla |w|^+=\sgnp w\nabla w\text{ a.e. in }Q_T
			\end{equation*}
			with $|z|^+=\max\{0,z\}=z\sgnp z$ standing for the positive part of $z\in\eR$ 
			we get by integrating \eqref{SSIneq} over $Q_t$ for an arbitrary $t\in(0,T)$ and taking into consideration assumption~\eqref{InitIneq} that
			\begin{equation}\label{IntIneq1}
				\begin{split}
					&\int_\Omega |\mathcal E_{\delta,\omega}(\mathfrak r^\omega,\mathfrak z^\omega,\underline\vartheta)-\mathcal E_{\delta,\omega}(\mathfrak r^\omega,\mathfrak z^\omega,\overline\vartheta)|^++\int_0^t\int_\Omega \nabla |\mathcal E_{\delta,\omega}(\mathfrak r^\omega,\mathfrak z^\omega,\underline\vartheta)-\mathcal E_{\delta,\omega}(\mathfrak r^\omega,\mathfrak z^\omega,\overline\vartheta)|^+\cdot u^\omega\\
					&\quad -\int_0^t\int_\Omega\left( \Delta\mathcal K_{\delta,\omega}(\underline\vartheta)- \Delta\mathcal K_{\delta,\omega}(\overline\vartheta) \right)\sgnp(\mathcal K_{\delta,\omega}(\underline\vartheta)-\mathcal K_{\delta,\omega}(\underline\vartheta))\\
					&\leq \int_0^t\int_\Omega \left|F_\omega(t,x,\underline\vartheta)-F_\omega(t,x,\overline\vartheta)\right|\sgnp(\underline\vartheta-\overline\vartheta)-(\mathcal E_{\delta,\omega}(\mathfrak r^\omega,\mathfrak z^\omega,\underline\vartheta)-\mathcal E_{\delta,\omega}(\mathfrak r^\omega,\mathfrak z^\omega,\overline\vartheta))\dvr u^\omega\sgnp(\underline\vartheta-\overline\vartheta).
				\end{split}
			\end{equation}
			Integrating the second term by parts and employing the boundary condition for $u$, 
			taking into consideration that
			\begin{equation*}
				\int_\Omega \left( \Delta\mathcal K_{\delta,\omega}(\underline\vartheta)- \Delta\mathcal K_{\delta,\omega}(\overline\vartheta) \right)\sgnp(\mathcal K_{\delta,\omega}(\underline\vartheta)-\mathcal K_{\delta,\omega}(\underline\vartheta))\leq 0,
			\end{equation*}
			cf.\cite[p. 155]{Feireisl04} and the fact that $F_\omega$ is Lipschitz continuous with respect to $\vartheta$, which is implied by the continuous differentiability of $\vartheta\mapsto \mathcal P_\omega(r,z,\vartheta)+\mathbb S_{\delta,\omega}(\vartheta,\nabla u)\cdot\nabla u+\varepsilon G_{\delta,\omega}+\frac{\delta}{\vartheta^2+\omega^2}-\varepsilon\theta_\omega^{\Lambda+1}(\vartheta)$ on the interval $[m,M]$, where $m=\min\{\essinf_{Q_T}\underline\vartheta,\essinf_{Q_T}\overline\vartheta\}$, $M=\max\{\esssup_{Q_T}\underline\vartheta,\esssup_{Q_T}\overline\vartheta\}$, we conclude from \eqref{IntIneq1}
			\begin{equation}\label{IntIneq2}
				\begin{split}
					\int_\Omega |\mathcal E_{\delta,\omega}(\mathfrak r^\omega,\mathfrak z^\omega,\underline\vartheta)-\mathcal E_{\delta,\omega}(\mathfrak r^\omega,\mathfrak z^\omega,\overline\vartheta)|^+\leq &\int_0^t\int_\Omega \left(|\dvr u^\omega||\mathcal E_{\delta,\omega}(\mathfrak r^\omega,\mathfrak z^\omega,\underline\vartheta)-\mathcal E_{\delta,\omega}(\mathfrak r^\omega,\mathfrak z^\omega,\overline\vartheta)|^+\right.\\&\left.+ c|\underline\vartheta-\overline\vartheta|\sgnp(\underline\vartheta-\overline\vartheta)\right)
				\end{split}
			\end{equation}
			with $c$ depending on $\|u\|_{C([0,T];X_n)}$,$\|\mathfrak r\|_{C^1([0,T]\times\overline\Omega)}$,$\|\mathfrak z\|_{C^1([0,T]\times\overline\Omega)}$, $m$ and $M$.
			The next task is to show that 
			\begin{equation}\label{PPEq}
				\underline c|\underline\vartheta-\overline\vartheta|^+\leq |\mathcal E_{\delta,\omega}(\mathfrak r,\mathfrak z,\underline\vartheta)-\mathcal E_{\delta,\omega}(\mathfrak r, \mathfrak z,\underline\vartheta)|^+\leq \overline c|\underline\vartheta-\overline\vartheta|^+.
			\end{equation}
			From the definition of $\mathcal E_{\delta,\omega}$, the assumed continuous differentiability of $\mathsf{e}_{\pm}$ and $\theta_\omega$ with respect to $\vartheta$ we get for some $\vartheta\in (m,M)$ 
			\begin{equation*}
				(\underline\vartheta-\overline\vartheta)\partial_\vartheta \mathcal E_{\delta,\omega}(\mathfrak r,\mathfrak z,\vartheta)=\mathcal E_{\delta,\omega}(\mathfrak r,\mathfrak z,\underline\vartheta)-\mathcal E_{\delta,\omega}(\mathfrak r,\mathfrak z,\overline\vartheta).
			\end{equation*}
			Denoting $Q=[\min_{[0,T]\times\overline\Omega}\mathfrak r, \max_{[0,T]\times\overline\Omega}\mathfrak r]\times[\min_{[0,T]\times\overline\Omega}\mathfrak z,\max_{[0,T]\times\overline\Omega}\mathfrak z]\times[m,M]$ it follows that
			\begin{equation*}
				\begin{split}
					\sgnp(\underline\vartheta-\overline\vartheta)(\underline\vartheta-\overline\vartheta)\min_{(r,z,\vartheta)\in Q}\partial_\vartheta \mathcal E_{\delta,\omega}(r,z,\vartheta)&=(\mathcal E_{\delta,\omega}(\mathfrak r,\mathfrak z,\underline\vartheta)-\mathcal E_{\delta,\omega}(\mathfrak r,\mathfrak z,\overline\vartheta))\sgnp(\underline\vartheta-\overline\vartheta)\\
					&\leq\sgnp(\underline\vartheta-\overline\vartheta)(\underline\vartheta-\overline\vartheta)\max_{(r,z,\vartheta)\in Q}\partial_\vartheta \mathcal E_{\delta,\omega}(r,z,\vartheta).
				\end{split}
			\end{equation*}
			As $\partial_\vartheta \mathcal E_{\delta,\omega}(\mathfrak r,\mathfrak z,\vartheta)$ is continuous on $Q$, we conclude \eqref{PPEq} by \eqref{SGNPEq}. Applying \eqref{PPEq} in \eqref{IntIneq2} yields
			\begin{equation*}
				\int_\Omega|\underline\vartheta-\overline\vartheta|^+(t,\cdot)\leq c\int_0^t\int_\Omega |\underline\vartheta-\overline\vartheta|^+
			\end{equation*}
			for any $t\in(0,T)$. Employing the Gronwall lemma we infer $|\underline\vartheta-\overline\vartheta|^+=0$ a.e. in $Q_T$, which implies \eqref{SubSupIneq}. 	
		\end{proof}
	\end{lemma}
	
	\begin{lemma}\label{Lem:TOmBounds}
		Let $\mathfrak r^\omega, \mathfrak z^\omega,u^\omega$ be as in Lemma~\ref{Lem:SubSupIneq} and $\vartheta_\omega$ be a strong solution to \eqref{VarthetaOmPr} belonging to the class $W$ defined in \eqref{WClass} with the initial datum $\vartheta_{0,\delta}$ satisfying 
		\begin{equation*}
			0<\underline\vartheta_0=\essinf_{\Omega}\vartheta_{0,\delta}\leq\esssup_{\Omega}\vartheta_{0,\delta}=\overline\vartheta_0<\infty
		\end{equation*}	
		with some constants $\underline\vartheta_0,\overline\vartheta_0$. Then there exist constants $\underline\vartheta,\overline\vartheta$ depending on $u,\mathfrak r, \mathfrak z,\varepsilon,\delta$ satisfying 
		\begin{equation*}
			0<\underline{\vartheta}\leq\underline\vartheta_{0}\leq \overline\vartheta_{0}\leq\overline{\vartheta}
		\end{equation*}
		and 
		\begin{equation*}
			\underline{\vartheta}\leq \vartheta_\omega\leq\overline{\vartheta}\text{ a.e. in }Q_T
		\end{equation*}
		for any $\omega\in(0,\underline\vartheta)$.
		\begin{proof}
			For purposes of this proof we denote
			\begin{equation*}
				\begin{split}
					M(\mathfrak r,\mathfrak z,\vartheta, u)=&\mathcal P_\omega(\mathfrak r,\mathfrak z,\vartheta)\dvr u+\partial_{\mathfrak r}\mathcal E_{\delta,\omega}(\mathfrak r,\mathfrak z,\vartheta)\left(\tder \mathfrak r+u\cdot\nabla \mathfrak r\right)+\partial_{\mathfrak z}\mathcal E_{\delta,\omega}(\mathfrak r,\mathfrak z,\vartheta)\left(\tder \mathfrak z+u\cdot\nabla \mathfrak z\right)\\&+\mathcal E_{\delta,\omega}(\mathfrak r,\mathfrak z,\vartheta)\dvr u-\mathbb S_{\delta,\omega}(\vartheta,\nabla u)\cdot\nabla u-\varepsilon G_{\delta,\omega}.
				\end{split}
			\end{equation*}
			Let us consider a constant $\tilde\vartheta
			\geq 0$. Then denoting
			\begin{align*}
				R=&\|\mathfrak r\|_{C^1([0,T]\times\overline\Omega)},\\
				Z=&\|\mathfrak z\|_{C^1([0,T]\times\overline\Omega)},\\
				U=&\|u\|_{C([0,T];X_n)},\\
				\underline r=&\inf_{Q_T}{\mathfrak r},\\
				\underline z=&\inf_{Q_T}{\mathfrak z}
			\end{align*}
			we get using the notations from Section~\ref{Sec:Hyp}
			\begin{align*}
				&\|M(\mathfrak r,\mathfrak z,\tilde \vartheta,u)\|_{C([0,T];C^1(\overline\Omega))}\\
				&\leq c\bigl(\theta_\omega^4(\tilde\vartheta)+(1+R^{\gamma_+}+d^P_+R^{\gamma^P_+} \theta^{\omega^P_+}_\omega(\tilde\vartheta))+(1+Z^{\gamma_-}+d^P_-Z^{\gamma^P_-}\theta_\omega^{\omega^P_-}(\tilde\vartheta))\bigr)U\\
				&\quad+\left(\delta(R+Z)\tilde\vartheta\right)(R+Z)(1+U)+c (\underline r^{-1}\theta^4_\omega(\vartheta)+1+R^{\gamma^e_+-1}+d^e_+\theta^{\omega^e_+}_\omega(\tilde\vartheta))R(1+U)\\
				&\quad+c(\underline z^{-1}\theta^4_\omega(\vartheta)+1+Z^{\gamma^e_--1}+d^e_-\theta_\omega^{\omega^e_-}(\tilde\vartheta))Z(1+U)\\
				&\quad+c(R^{\underline\Gamma^e-1}+R^{\overline\Gamma^e-1})(1+\overline d^e_+\theta_\omega(\tilde\vartheta))^{\overline\omega^e}R(1+U)+c\left(Z^{\underline\Gamma^e-1}+Z^{\overline\Gamma^e-1}\right)(1+\overline d^e_-\theta_\omega(\tilde\vartheta))^{\overline\omega^e}Z(1+U)\\
				&\quad+c(1+c(1+\theta_\omega(\tilde\vartheta))U^2+\varepsilon\delta(R^\Gamma+Z^\Gamma)=\mathcal M(R,Z,\tilde\vartheta,U)
			\end{align*}
			where we used the identity
			\begin{align*}
				&\partial_{\mathfrak r} \mathcal E(\mathfrak r,\mathfrak z,\vartheta)(\tder \mathfrak r+u\cdot\nabla \mathfrak r)+\partial_{\mathfrak z} \mathcal E(\mathfrak r,\mathfrak z,\vartheta)(\tder \mathfrak z+u\cdot\nabla \mathfrak z)+\mathcal E(\mathfrak r,\mathfrak z,\vartheta)\dvr u\\
				&=\left(\tilde{\mathfrak e}_+\left(\mathfrak r,\vartheta\right)+\mathfrak r\partial_{\mathfrak r}\tilde{\mathfrak e}_+\left(\mathfrak r,\vartheta\right)\right)(\tder \mathfrak r+\dvr (\mathfrak r u))+\left(\tilde{\mathfrak e}_-\left(\mathfrak z,\vartheta\right)+\mathfrak z\partial_{\mathfrak z}\tilde{\mathfrak e}_-\left(\mathfrak z,\vartheta\right)\right)(\tder \mathfrak z+\dvr(\mathfrak zu)).
			\end{align*}
			We notice that the mapping $\vartheta\mapsto \theta_\omega(\vartheta)$ is increasing on $(0,\infty)$ with values in $\left(\frac{\omega}{1+\omega^2},\frac{1}{\omega}\right)$ for $\omega\in (0,1)$. Moreover, we have  $\theta_\omega(\vartheta)\leq \sqrt 2$ for $\vartheta\in (0,1)$ and $\omega\in (0,1)$. 
			
			Obviously, considering $\varepsilon,\delta<1$ we can find $\bar\omega\in(0,1)$ such that 
			\begin{equation*}
				\frac{\delta}{\bar\omega}\geq\varepsilon \sqrt 2^{\Lambda+1}+\mathcal M(R,Z,\tilde\vartheta,U).
			\end{equation*}
			Then we get for $\underline\vartheta=\sqrt{\frac{\bar\omega}{2}}$ and any $\omega\in(0,\underline\vartheta)$ that
			$\frac{\delta}{\underline\vartheta^2+\omega^2}\geq\varepsilon\theta_\omega(\vartheta)^{\Lambda+1}+M(\mathfrak r^\omega,\mathfrak z^\omega,\alpha^\omega,\underline\vartheta_\omega,u^\omega)$ in $Q_T$ implying 
			that 
			\begin{equation*}
				\begin{alignedat}{2}
					\{D_\vartheta \mathcal E\}_{\delta,\omega}(\mathfrak r^\omega,\mathfrak z^\omega,\underline\vartheta)\tder\underline\vartheta+\dvr\left(\mathcal E_{\delta,\omega}(\mathfrak r^\omega,\mathfrak z^\omega,\underline\vartheta)u^\omega\right)-\Delta\mathcal K_{\delta,\omega}(\underline\vartheta)\leq&-\partial_{\mathfrak r}\mathcal E_{\delta,\omega}(\mathfrak r^\omega,\mathfrak z^\omega,\alpha^\omega,\underline\vartheta)\tder \mathfrak r^\omega&&\\-\partial_{\mathfrak z}\mathcal E_{\delta,\omega}(\mathfrak r^\omega,\mathfrak z^\omega,\underline\vartheta)\tder \mathfrak z^\omega&+\mathbb S_{\delta,\omega}(\underline\vartheta,\nabla u^\omega)\cdot\nabla u^\omega\\
					-\mathcal P_\omega(\mathfrak r^\omega,\mathfrak z^\omega,\underline\vartheta)\dvr u^\omega&+\frac{\delta}{\underline\vartheta^2+\omega^2}+\varepsilon G_{\delta,\omega}-\varepsilon\theta^{\Lambda+1}_\omega(\underline\vartheta)&&\text{ in }Q_T,
				\end{alignedat}
			\end{equation*}
			i.e., $\underline\vartheta$ is a subsolution to \eqref{VarthetaOmPr} and $\vartheta_\omega\geq\underline\vartheta$. Taking into account 
			the monotonicity of $\vartheta\mapsto\theta_\omega(\vartheta)$ and the fact that the exponent $\Lambda$ is chosen as a maximum of powers of $\theta_\omega(\vartheta)$ in the expression $\mathcal M(R,Z,\tilde\vartheta,U)$ we can find a sufficiently large $\overline\vartheta>1$ such that 
			\begin{equation*}
				\frac{\delta}{\overline\vartheta^2+\omega^2}\leq \frac{\delta}{\overline\vartheta^2}\leq \varepsilon \theta^{\Lambda+1}_\omega(\overline\vartheta)-\mathcal M(R,Z,\overline\vartheta,U).
			\end{equation*}
			The latter inequality guarantees that $\frac{\delta}{\overline\vartheta^2+\omega^2}\leq \varepsilon \theta^{\Lambda+1}_\omega(\overline\vartheta)+M(\mathfrak r,\mathfrak z,\overline\vartheta,u)$ in $Q_T$ implying that 
			\begin{equation*}
				\begin{alignedat}{2}
					\{D_\vartheta \mathcal E\}_{\delta,\omega}(\mathfrak r^\omega,\mathfrak z^\omega,\overline\vartheta)\tder\overline\vartheta+\dvr\left(\mathcal E_{\delta,\omega}\mathfrak r^\omega,\mathfrak z^\omega, \overline\vartheta)u^\omega\right)-\Delta\mathcal K_{\delta,\omega}(\overline\vartheta)\geq&-\partial_{\mathfrak r}\mathcal E_{\delta,\omega}(\mathfrak r^\omega,\mathfrak z^\omega,\overline\vartheta)\tder \mathfrak r^\omega&&\\-\partial_{\mathfrak z}\mathcal E_{\delta,\omega}(\mathfrak r^\omega,\mathfrak z^\omega,\overline\vartheta)\tder \mathfrak z^\omega+\mathbb S_{\delta,\omega}(\overline\vartheta,\nabla u^\omega)\cdot\nabla u^\omega\\
					-\mathcal P_\omega(\mathfrak r^\omega,\mathfrak z^\omega,\overline\vartheta)\dvr u^\omega&+\frac{\delta}{\overline\vartheta^2+\omega^2}+\varepsilon G_{\delta,\omega}-\varepsilon\theta^{\Lambda+1}_\omega(\overline\vartheta)&&\text{ in }Q_T,
				\end{alignedat}
			\end{equation*}
			i.e., $\overline\vartheta$ is a supersolution to \eqref{VarthetaOmPr} and $\vartheta_\omega\leq\overline\vartheta$.
		\end{proof}
	\end{lemma}
	
	\begin{lemma}\label{Lem:OmUnifEst}
		Let the assumptions of Lemma~\ref{Lem:VarThetaOmEx} be fulfilled. Let $\{\vartheta_\omega\}_{\omega\leq \underline\vartheta}$, where $\underline\vartheta$ is given by Lemma~\ref{Lem:TOmBounds}, be a family of solutions to \eqref{VarthetaOmPr}. Then  
		\begin{equation}\label{OmUnifEst1}
			\|\frac{1}{\vartheta_\omega}\|_{L^\infty(Q_T)}+\|\vartheta_\omega\|_{L^\infty(Q_T)}\leq h_1(\delta,\|\mathfrak r\|_{C^1([0,T]\times\overline\Omega)},\|\mathfrak z\|_{C^1([0,T]\times\overline\Omega)},\|u\|_{C([0,T];X_n)}),
		\end{equation}
		\begin{equation}\label{OmUnifEst2}
			\begin{split}
				&\|\tder\vartheta_\omega\|^2_{L^2(0,T;L^2(\Omega))}+\esssup_{(0,T)}\|\vartheta_\omega\|^2_{W^{1,2}(\Omega)}+\|\mathcal K_{\delta,\omega}(\vartheta_\omega)\|^2_{L^2(0,T;W^{2,2}(\Omega))}\\&\leq h_2(\delta,\|\mathfrak r\|_{C^1([0,T]\times\overline\Omega)},\|\mathfrak z\|_{C^1([0,T]\times\overline\Omega)}, \|\vartheta_{0,\delta}\|_{W^{1,2}(\Omega)},\inf_{Q_T}(\mathfrak r+\mathfrak z)^{-1}).
			\end{split}
		\end{equation}
		The functions $h_1$ and $h_2$ are bounded on bounded sets.
		
		Moreover, there is sequence $\{\vartheta_{\omega_n}\}\subset\{\vartheta_\omega\}_{\omega\in(0,\underline\vartheta)}$ and $\vartheta$ in $W$, see \eqref{WClass} for the definition of $W$, such that
		\begin{equation}\label{VTOmConv}
			\begin{alignedat}{2}
				\tder \vartheta_{\omega_n}&\rightharpoonup\tder\vartheta&&\text{ in }L^2(0,T;L^2(\Omega)),\\
				\vartheta_{\omega_n}&\rightharpoonup^*\vartheta&&\text{ in }L^\infty(0,T;W^{1,2}(\Omega)),\\
				\mathcal K_{\delta,\omega_n}(\vartheta_{\omega_n})&\rightharpoonup \mathcal K_{\delta}(\vartheta)&&\text{ in }L^2(0,T;W^{2,2}(\Omega))
			\end{alignedat}
		\end{equation} 
		as $\omega_n\to 0_+$
		and $\vartheta$ is a solution to \eqref{IntEnEq}.
		
		\begin{proof}
			We observe that the uniform bound in \eqref{OmUnifEst1} follows directly from Lemma~\ref{Lem:TOmBounds}. We proceed with the proof of the bound in \eqref{OmUnifEst2}. Multiplying equation \eqref{VarthetaOmPr}$_1$ on $\vartheta_\omega$ and integrating by parts yield
			\begin{equation*}
				\begin{split}
					&\frac{1}{2}\int_\Omega \{D_\vartheta\mathcal E\}_{\delta,\omega} (\mathfrak r^\omega,\mathfrak z^\omega,\vartheta_\omega)\tder|\vartheta_\omega|^2+\int_\Omega\kappa_{\delta,\omega}(\vartheta_\omega)|\nabla \vartheta_\omega|^2=\int_{\Omega}\mathcal E_{\delta,\omega}(\mathfrak r^\omega,\mathfrak z^\omega,\alpha^\omega,\vartheta_\omega)u^\omega\cdot\nabla\vartheta_\omega\\
					&+\int_\Omega\Bigl(-\partial_{\mathfrak r}\mathcal E_{\delta,\omega}(\mathfrak r^\omega,\mathfrak z^\omega,\vartheta_\omega)\tder\mathfrak r^\omega-\partial_{\mathfrak z}\mathcal E_{\delta,\omega}(\mathfrak r^\omega,\mathfrak z^\omega,\vartheta_\omega)\tder \mathfrak z^\omega+\mathbb S_{\delta,\omega}(\vartheta_\omega,\nabla u^\omega)\cdot\nabla u^\omega\\
					&\quad-\mathcal P_\omega(\mathfrak r^\omega,\mathfrak z^\omega,\vartheta_\omega)\dvr u^\omega+\varepsilon G_{\delta,\omega}-\frac{\delta}{\vartheta_\omega^2+\omega^2}-\varepsilon\theta^{\Lambda+1}_\omega(\vartheta_\omega)\Bigr)\vartheta_\omega.
				\end{split}
			\end{equation*}
			Using \eqref{AppPositive}, the Young inequality on the first term on the right hand side and \eqref{OmUnifEst1} we conclude
			\begin{equation}\label{FiEst}
				\begin{split}
					&\|\vartheta_\omega\|^2_{L^\infty(0,T;L^2(\Omega))}+\|\vartheta_\omega\|^2_{L^2(0,T;W^{1,2}(\Omega))}\\&\leq h_1(\|\mathfrak r\|_{C^1([0,T]\times\overline\Omega)},\|\mathfrak z\|_{C^1([0,T]\times\overline\Omega)},\|u\|_{C([0,T];X_n)},(\inf_{Q_T}(\mathfrak r+\mathfrak z))^{-1}),
				\end{split}
			\end{equation}
			where $h_1$ is bounded on bounded sets.
			
			Multiplying \eqref{VarthetaOmPr}$_1$ on $\tder\mathcal K_{\delta,\omega}(\vartheta_\omega)=\kappa_{\delta,\omega}(\vartheta_\omega)\tder\vartheta_\omega$ we get
			\begin{equation*}
				\begin{split}
					&\frac{\mathrm{d}}{\dt} \int_\Omega |\nabla \mathcal K_{\delta,\omega}(\vartheta_\omega)|^2+\int_\Omega\kappa_{\delta,\omega}(\vartheta_\omega)\{D_\vartheta\mathcal E_{\delta,\omega}\}(\mathfrak r^\omega,\mathfrak z^\omega,\vartheta_\omega)|\tder\vartheta_\omega|^2=-\int_\Omega \partial_\vartheta \mathcal E_{\delta,\omega}(\mathfrak r^\omega,\mathfrak z^\omega,\vartheta_\omega)u^\omega\tder\vartheta_\omega\nabla\mathcal K_{\delta,\omega}(\vartheta_\omega)\\
					&+\int_\Omega\Bigl(-\partial_{\mathfrak r}\mathcal E_{\delta,\omega}(\mathfrak r^\omega,\mathfrak z^\omega,\vartheta_\omega)(\tder\mathfrak r^\omega+u^\omega\cdot\nabla \mathfrak r^\omega)-\partial_{\mathfrak z}\mathcal E_{\delta,\omega}(\mathfrak r^\omega,\mathfrak z^\omega,\vartheta_\omega)(\tder \mathfrak z^\omega+u^\omega\cdot\nabla \mathfrak z^\omega)-\mathcal E_{\delta,\omega}(\mathfrak r^\omega,\mathfrak z^\omega,\vartheta_\omega)\dvr u^\omega
					\\
					&\quad+\mathbb S_{\delta,\omega}(\vartheta,\nabla u^\omega)\cdot\nabla u^\omega-\mathcal P_\omega(\mathfrak r^\omega,\mathfrak z^\omega,\vartheta_\omega)\dvr u^\omega+\varepsilon G_{\delta,\omega}-\frac{\delta}{\vartheta_\omega^2+\omega^2}-\varepsilon\theta^{\Lambda+1}_\omega(\vartheta_\omega)\Bigr)\tder\vartheta_\omega\kappa_{\delta,\omega}(\vartheta_\omega).
				\end{split}
			\end{equation*}
			Applying again \eqref{AppPositive} and the Young inequality along with \eqref{OmUnifEst1} for estimates on the right hand side we conclude by the Gronwall lemma
			\begin{equation*}
				\begin{split}
					&\|\nabla\mathcal K_{\delta,\omega}(\vartheta_\omega)\|^2_{L^\infty(0,T;L^2(\Omega))}+\|\tder\vartheta_\omega\|^2_{L^2(0,T;L^2(\Omega))}\\
					&\leq \tilde h_1(\|\mathfrak r\|_{C^1([0,T]\times\overline\Omega)},\|\mathfrak z\|_{C^1([0,T]\times\overline\Omega)},\|u\|_{C([0,T];X_n)},(\inf_{Q_T}(\mathfrak r+\mathfrak z))^{-1}),
				\end{split}
			\end{equation*}
			where $\tilde h_1$ is bounded on bounded sets. We point out that the first term on the left hand side of the latter inequality bounds $\|\nabla\vartheta_\omega\|^2_{L^\infty(0,T;L^2(\Omega))}$ as $\nabla\mathcal K_{\delta,\omega}(\vartheta)=\nabla\vartheta\kappa_{\delta,\omega}(\vartheta)$ and \eqref{AppPositive} holds. Hence we get 
			\begin{equation}\label{SeEst}
				\begin{split}
					&\|\vartheta_\omega\|^2_{L^\infty(0,T;W^{1,2}(\Omega))}+\|\tder\vartheta_\omega\|^2_{L^2(0,T;L^2(\Omega))}\\
					&\leq \widetilde h_1(\|\mathfrak r\|_{C^1([0,T]\times \overline\Omega)},\|\mathfrak z\|_{C^1([0,T]\times \overline\Omega)},\|u\|_{C([0,T];X_n)},(\inf_{Q_T}(\mathfrak r+\mathfrak z))^{-1}).
				\end{split}
			\end{equation}
			Going back to \eqref{VarthetaOmPr}$_1$ we see that the terms on the right hand side are bounded in $L^\infty(Q_T)$ due to \eqref{OmUnifEst1}. Moreover, the first two terms on the left hand side are bounded in $L^2(0,T;L^2(\Omega))$ as a consequence of \eqref{FiEst} and \eqref{SeEst} implying that 
			\begin{equation*}
				\begin{split}
					&\|\Delta\mathcal K_{\delta,\omega}(\vartheta_\omega)\|_{L^2(0,T;L^2(\Omega))}\\&\leq \tilde h_2((\|\mathfrak r\|_{C^1([0,T]\times\overline\Omega)},\|\mathfrak z\|_{C^1([0,T]\times\overline\Omega)},\|u\|_{C([0,T];X_n)},(\inf_{Q_T}(\mathfrak r+\mathfrak z))^{-1})),
				\end{split}
			\end{equation*}
			where $\tilde h_2$ is bounded on bounded sets. We notice that $r+z$ solves the regularized continuity equation with an initial condition that is bounded from below by a positive constant. Hence $\inf_{\Omega}(\mathfrak r+\mathfrak z)>0$ by Lemma~\ref{Lem:ContApp}. Using the elliptic regularity theory, we conclude \eqref{OmUnifEst2}.
			
			It remains to show that there is a sequence $\{\vartheta_{\omega_n}\}$ possessing a limit $\vartheta$ as $\omega_n\to 0_+$ that solves \eqref{IntEnEq}.
			To this end we note that as a consequence of the Aubin-Lions lemma, estimates \eqref{OmUnifEst1} and \eqref{OmUnifEst2} we can consider a sequence $\{\vartheta_{\omega_n}\}$ such that
			\begin{equation}\label{OConv}
				\begin{alignedat}{2}
					\vartheta_{\omega_n} &\to \vartheta&&\text{ in }L^p(Q_T)\text{ for any }p\in[1,\infty)\text{ and a.e. in }Q_T,\\
					\vartheta_{\omega_n} &\rightharpoonup \vartheta&&\text{ in }L^2(0,T;W^{1,2}(\Omega)),\\
					\tder\vartheta_{\omega_n} &\rightharpoonup \tder\vartheta&&\text{ in }L^2(0,T;L^2(\Omega)).
				\end{alignedat}
			\end{equation}
			Obviously, we can select a nonrelabeled subsequence $\{\vartheta_{\omega_n}\}$ that satisfies \eqref{VTOmConv}. Moreover, as $\chi_{(1,\vartheta_{\omega_n})}\to \chi_{(1,\vartheta)}$ and $\kappa_{\delta,\omega_n}(\vartheta_{\omega_n})\to \kappa_{\delta}(\vartheta)$ in $L^p(Q_T)$ for any $p\in[1,\infty)$ by \eqref{OConv}$_1$, \eqref{OmUnifEst1} and the dominated convergence theorem we get
			\begin{equation*}
				\mathcal K_{\delta,\omega_n}(\vartheta_{\omega_n})=\int_{\eR}\chi_{(1,\vartheta_{\omega_n})}\kappa_{\delta,\omega_n}(s)\ds\to \int_{\eR}\chi_{(1,\vartheta)}\kappa_{\delta}(s)\ds=\mathcal K_{\delta}(\vartheta)\text{ a.e. in }Q_T.
			\end{equation*}
			Using definitions of $\mathcal K_{\delta,\omega_n}$ and $\kappa_{\delta,\omega_n}$ in \eqref{kappaDO} and the bounds in \eqref{OmUnifEst1} we get the $L^\infty$--bound on $\mathcal K_{\delta,\omega}(\vartheta_{\omega})$ that is independent of $\omega$. Hence the dominated convergence theorem yields
			\begin{equation*}
				\mathcal K_{\delta,\omega_n}(\vartheta_{\omega_n})\to \mathcal K_{\delta}(\vartheta)\text{ in }L^p(Q_T)\text{ for any }p\in[1,\infty).
			\end{equation*}
			Eventually, from the latter convergence and the estimate in \eqref{OmUnifEst2} we conclude 
			\begin{equation*}
				\mathcal K_{\delta,\omega_n}(\vartheta_{\omega_n})\rightharpoonup 	\mathcal K_{\delta}(\vartheta)\text{ in }L^2(0,T;W^{2,2}(\Omega)).
			\end{equation*}
			Considering \eqref{VarthetaOmPr} with $\vartheta=\vartheta_{\omega_n}$ we can pass to the limit $\omega_n\to 0_+$. In order to perform the passage in the first term on left hand side we observe that $\{D_\vartheta\mathcal E\}_{\delta,\omega_n}(\mathfrak r^{\omega_n},\mathfrak z^{\omega_n},\vartheta_{\omega_n})\to \partial_\vartheta\mathcal E_{\delta}(\mathfrak r,\mathfrak z,\vartheta)$ first a.e. in $Q_T$ and then in $L^p(Q_T)$ for any $p\in[1,\infty)$ by the dominated convergence theorem. Taking into account also \eqref{OConv}$_3$ we have
			\begin{equation*}
				\{D_\vartheta\mathcal E\}_{\delta,\omega_n}(\mathfrak r^{\omega_n},\mathfrak z^{\omega_n},\vartheta_{\omega_n})\tder\vartheta_{\omega_n}\rightharpoonup\partial_\vartheta\mathcal E_{\delta}(\mathfrak r,\mathfrak z,\vartheta)\tder\vartheta\text{ in }L^q(Q_T)\text{ for any }q\in[1,2).
			\end{equation*}
			Similarly, employing \eqref{OConv}$_2$ we conclude
			\begin{equation*}
				\dvr(\mathcal E_{\delta,\omega}(\mathfrak r^{\omega_n},\mathfrak z^{\omega_n},\vartheta_{\omega_n})u^{\omega_n})\rightharpoonup 	\dvr(\mathcal E_{\delta}(\mathfrak r,\mathfrak z,\vartheta)u)\text{ in }	L^q(Q_T) \text{ for any }q\in[1,2),
			\end{equation*}
			and by \eqref{OConv}$_1$
			\begin{alignat*}{2}
				\partial_b\mathcal E_{\delta,\omega}(\mathfrak r^{\omega_n},\mathfrak z^{\omega_n},\vartheta_{\omega_n})\tder b^{\omega_n}&\to \partial_b\mathcal E_{\delta,\omega}(\mathfrak r,\mathfrak z,\vartheta)\tder b&&\text{ in } L^p(Q_T),\\
				\mathbb S_{\delta,\omega_n}(\vartheta_{\omega_n},\nabla u^{\omega_n})\cdot\nabla u^{\omega_n}&\to \mathbb S_{\delta}(\vartheta,\nabla u)\cdot\nabla u&&\text{ in }L^p(Q_T),\\
				\mathcal P_{\delta,\omega_n}(\mathfrak r^{\omega_n},\mathfrak z^{\omega_n},\vartheta_{\omega_n})\dvr u^{\omega_n}&\to \mathcal P_\delta(\mathfrak r,\mathfrak z,\vartheta)\dvr u&&\text{ in }L^p(Q_T),\\
				\frac{\delta}{(\vartheta^{\omega_n})^2+\omega_n^2}+\varepsilon G_{\delta,\omega}-\varepsilon\theta^{\Lambda+1}_{\omega_n}(\vartheta_{\omega_n})&\to\frac{1}{\vartheta^2}+\varepsilon G_{\delta}-\varepsilon\vartheta^{\Lambda+1}&&\text{ in }L^p(Q_T)	
			\end{alignat*}
			for any $p\in[1,\infty)$ with $b$ standing for $\mathfrak r$ or $\mathfrak z$. Having the above convergences at hand we perform the passage to the limit $\omega_n\to 0_+$ in \eqref{VarthetaOmPr} to conclude that the limit function $\vartheta$ satisfies \eqref{IntEnEq}. 
			
			Next, we notice that the uniform estimate of $\{\mathcal K_{\delta,\omega}(\vartheta_{\omega_n})\}$ in $L^2(0,T;W^{2,2}(\Omega))$ implies that $\{\vartheta_n\}$ is bounded in $L^2(0,T;W^{2,2}(\Omega))$ as $\kappa'_{\delta,\omega_n}(\vartheta_{\omega_n})$ is bounded from above and $\kappa_{\delta,\omega_n}(\vartheta_{\omega_n})$ satisfies \eqref{AppPositive}. Combining this fact with the continuity of the trace operator $\tr:W^{k,p}(\Omega)\to W^{k-\frac{1}{p},p}(\partial\Omega)$ for $k\in (0,\infty)$, $p\in[1,\infty)$ we conclude that for a nonrelabeled subsequence $0=\tr(\nabla\vartheta_{\omega_n})n\rightharpoonup \tr(\nabla\vartheta)n=0$, i.e., the limit function $\vartheta$ satisfies the homogeneous Neumann boundary condition.
		\end{proof}
	\end{lemma} 
	
	\begin{proof}[Proof of Lemma~\ref{Lem:IEE}]
		The existence of a unique solution $\vartheta$ belonging to the regularity class $Y$ follows immediately from Lemma~\ref{Lem:OmUnifEst}. The next immediate consequence of this lemma is that $\vartheta$ satisfies also
		\begin{equation}\label{UnEst1}
			\begin{split}
				&\|\frac{1}{\vartheta}\|_{L^\infty(Q_T)}+\|\vartheta\|_{L^\infty(Q_T)}\\
				&\leq h_1(\varepsilon,\delta,\|\mathfrak r\|_{C^1([0,T]\times\overline\Omega)},\|\mathfrak z\|_{C^1([0,T]\times\overline\Omega)},\|u\|_{C([0,T];X_n)}),
			\end{split}
		\end{equation}
		\begin{equation}\label{UnEst2}
			\begin{split}
				&\|\tder\vartheta\|^2_{L^2(0,T;L^2(\Omega))}+\esssup_{(0,T)}\|\vartheta\|^2_{W^{1,2}(\Omega)}+\|\mathcal K_{\delta}(\vartheta)\|^2_{L^2(0,T;W^{2,2}(\Omega))}\\&\leq h_2(\varepsilon,\delta,\|\mathfrak  r\|_{C^1([0,T]\times\overline\Omega)},\|\mathfrak z\|_{C^1([0,T]\times\overline\Omega)}, \|\vartheta_{0,\delta}\|_{W^{1,2}(\Omega)},\inf_{Q_T}(\mathfrak r+\mathfrak z)^{-1}).
			\end{split}
		\end{equation}
		Hence the fact that the mapping $u\mapsto\vartheta_u$ maps bounded sets in $C([0,T];X_n)$ into bounded sets in $Y$ follows with regard to Lemma~\ref{Lem:ContApp}. Next, considering a sequence $\{u_k\}\subset C([0,T];X_n)$, we have
		\begin{equation}\label{KBounds}
			\begin{alignedat}{2}
				\{\tder\vartheta_{u_{k}}\}&\text{ bounded in }&&L^2(0,T;L^2(\Omega)),\\
				\{\vartheta_{u_k}\}&\text{ bounded in }&&L^\infty(0,T;W^{1,2}(\Omega)),\\
				\{\mathcal K_{\delta}(\vartheta_{u_k})\}&\text{ bounded in }&&L^2(0,T;W^{2,2}(\Omega)).
			\end{alignedat}
		\end{equation}
		We note that $\{\tder\mathcal K_\delta(\vartheta_{u_k})\}$ is bounded in $L^2(0,T;L^2(\Omega))$ due to \eqref{KBounds} and the definition of $\mathcal K_\delta$ in \eqref{KappaDDef}. Hence with the help of the Aubin-Lions lemma we get the existence of $\{\vartheta_{u_{k_l}}\}\subset \{\vartheta_{u_{k}}\}$ such that
		\begin{equation}\label{KLConv}
			\begin{alignedat}{2}
				\tder \vartheta_{u_{k_l}}&\rightharpoonup\tder \vartheta &&\text{ in }L^2(0,T;L^2(\Omega)),\\
				\vartheta_{u_{k_l}}&\rightharpoonup^*\vartheta &&\text{ in }L^\infty(0,T;W^{1,2}(\Omega)),\\
				\mathcal K_\delta(\vartheta_{u_{k_l}})&\rightharpoonup\mathcal K_\delta(\vartheta)&&\text{ in }L^2(0,T;W^{2,2}(\Omega)),\\
				\vartheta_{u_{k_l}}&\to\vartheta &&\text{ in }L^p(0,T;L^p(\Omega))\text{ for any }p\in[1,\infty)\text{ and a.e. in }Q_T,\\
				\mathcal K_\delta(\vartheta_{u_{k_l}})&\to\mathcal K_\delta(\vartheta)&&\text{ in }L^2(0,T;W^{1,q}(\Omega))\text{ for any }q<2^*.
			\end{alignedat}
		\end{equation}
		Moreover, it follows that 
		\begin{equation}\label{NVarThetaConv}
			\nabla \vartheta_{u_{k_l}}\to \nabla \vartheta\text{ in }L^2(0,T;L^2(\Omega)).
		\end{equation}
		Indeed, by the definition $\nabla \vartheta_{u_{k_l}}=\frac{1}{\kappa_\delta(\vartheta_{u_{k_l}})}\nabla \mathcal K_\delta(\vartheta_{u_{k_l}})$ and $\frac{1}{\kappa_\delta(\vartheta_{u_{k_l}})}\to\frac{1}{\kappa_\delta(\vartheta)}$ in $L^s(0,T;L^s(\Omega))$ for any $s\in[1,\infty)$. We conclude \eqref{NVarThetaConv} by employing  \eqref{KLConv}$_5$.
		
		Having \eqref{KLConv} and \eqref{NVarThetaConv} we can pass to the limit $k_l\to \infty$ in equation \eqref{IntEnEq} with $u=u_{k_l}$, $\mathfrak r=\mathfrak r_{u_{k_l}}$, $\mathfrak z=\mathfrak z_{u_{k_l}}$ to conclude that the limit function $\vartheta$ is a unique solution to \eqref{IntEnEq} corresponding to $u$, i.e., $\vartheta=\vartheta_u$. We note that the uniqueness of $\vartheta$ can be shown by adopting the procedure from the proof of Lemma~\ref{Lem:SubSupIneq}. Accordingly, we have shown that an arbitrary sequence $\{u_k\}$ possesses a subsequence $\{u_{k_l}\}$ such that
		\begin{equation*}
			\vartheta_{u_{k_l}}\to\vartheta\text{ in }L^2(0,T;W^{1,2}(\Omega)),
		\end{equation*}
		i.e., the continuity of $u\mapsto\vartheta_u$ is proved.
	\end{proof}
 
    \Needspace{5\baselineskip}
	\subsection{Global existence of the approximate problem for $\varepsilon,\delta>0$}
	\subsubsection{Local solvability of Faedo--Galerkin approximations}\label{LocSolG}
	The goal of this subsection is the proof of existence a solution to the approximate problem on a possibly short time interval $(0,T_{n})$. More precisely, we seek $T_n\in(0,T]$ and $u\in C([0,T_n];X_n)$ satisfying the momentum balance in the sense of a Faedo--Galerkin approximation
	\begin{equation}\label{GalE1}
		\begin{split}
			\int_\Omega (\Sigma u)(t)\cdot\varphi-\int_\Omega \Sigma_{0,\delta}u_{0,\delta,n}\cdot\varphi=\int_0^t\int_\Omega &\left(\left(\Sigma(u\otimes u)-\mathbb S(\vartheta,\nabla u)\right)\cdot\nabla\varphi+\mathcal P_\delta(\mathfrak r,\mathfrak z,\vartheta)\dvr\varphi\right.\\&\left.-\varepsilon\nabla\Sigma\nabla u\cdot\varphi\psi\right)
		\end{split}
	\end{equation}
	for any $t\in(0,T_n)$, $\varphi\in X_n$, where $X_n=\mathrm{span}\{\Phi_i\}_{i=1}^n\subset C^{2,\nu}(\overline\Omega)$. We define 
	\begin{equation}\label{UDNDef}
		u_{0,\delta,n}=P_n\left(\frac{(\Sigma u)_{0,\delta}}{\Sigma_{0,\delta}}\right).
	\end{equation} 
	The functions $\{\Phi_i\}$ form an orthonormal basis in $L^2(\Omega)$ and satisfy either $\Phi_i\cdot n=0$ on $\partial\Omega$ in the case of the complete slip boundary condition or $\Phi_i=0$ on $\partial\Omega$ in the case of no-slip boundary condition. We note that $X_n$ is endowed with the Hilbert structure induced by the scalar product on $L^2(\Omega)$ and $P_n:L^2(\Omega)\to X_n$ denotes the orthogonal projection. Moreover, we assume that
	\begin{equation*}
		\cup_{n=1}^\infty X_n\text{ is dense in }X^{1,2},
	\end{equation*}
	where 
	\begin{equation}\label{XPDef}
		X^{1,2}=W^{1,2}_n(\Omega)\text{ or } X^{1,2}=W^{1,2}_0(\Omega)
	\end{equation}
    with $W^{1,2}_n(\Omega)$ and $W^{1,2}_0(\Omega)$  defined in \eqref{SobDef}. The functions $\xi$, $\Sigma$, $\mathfrak r$, $\mathfrak z$ satisfy equation \eqref{RReg}--\eqref{RZInit} with corresponding initial conditions $\xi_{0,\delta}, \Sigma_{0,\delta}$, $\mathfrak r_{0,\delta}$, $\mathfrak z_{0,\delta}$. 
	Moreover, we have
	\begin{equation}\label{PointEquiv}
		\underline a\mathfrak r\leq \mathfrak z\leq \overline a\mathfrak r
	\end{equation}
	as the differecens $\mathfrak z- \underline a\mathfrak r$ and $\overline a\mathfrak r-\mathfrak z$ solve the regularized continuity equations with positive initial conditions, cf. Lemma~\ref{Lem:ContApp}. In a complete analogy we get 
	\begin{equation}\label{PointEqS}
		\begin{split}
			\underline b\Sigma\leq \mathfrak r+\mathfrak z\leq\overline b\Sigma ,\\
			\underline d\mathfrak r\leq \xi\leq\overline d\mathfrak r
		\end{split}
	\end{equation}
	provided $\underline b\Sigma_{0,\delta}\leq \mathfrak r_{0,\delta}+\mathfrak z_{0,\delta}\leq \overline b\Sigma_{0,\delta}$, $\underline d\mathfrak r_{0,\delta}\leq \xi_{0,\delta}\leq\overline d\mathfrak r_{0,\delta}$ for some $0<\underline b<\overline b<\infty$ and $0<\underline d<\overline d<\infty$. $\vartheta$ solves \eqref{IntEnEq} with a given initial datum $\vartheta_{0,\delta}$ and the velocity $u$. For the purposes of this section we denote 
	\begin{equation*}
		\mathcal C=\{g\in C([0,T]\times\overline\Omega):\ \inf_{[0,T]\times\overline{\Omega}} g>0\}
	\end{equation*} 
	and define a mapping $J:\mathcal C\times X^*_n\to X_n$ as
	\begin{equation*}
		\int_\Omega g(t) J[g(t),\chi]\cdot\varphi=\langle\chi,P_n\varphi\rangle\text{ for any }\varphi\in L^2(\Omega),\ t\in [0,T].
	\end{equation*}  
	We note that the latter identity implies
	\begin{equation}\label{JEst}
		\begin{split}
			\|J[g(t),\chi]\|_{X_n}\leq&\frac{1}{\inf_{\overline\Omega}g(t)}\|\chi\|_{X_n^*}\\
			\|J[g_1(t),\chi]-J[g_2(t),\chi]\|_{X_n}\leq&\frac{1}{\inf_{\overline\Omega}g_1(t)\inf_{[0,T]\times\overline\Omega}g_2(t)}\|g_1(t)-g_2(t)\|_{L^\infty(\Omega)}\|\chi\|_{X_n^*}.
		\end{split}
	\end{equation}
	We define a mapping $\mathcal T$ on $C([0,T];X_n)$ as
	\begin{equation}\label{TMDeF}
		\mathcal T(u)(t)=J\left[\Sigma (t),\int_0^t M(\tau,\Sigma(\tau),\mathfrak r(\tau),\mathfrak z(\tau),\vartheta(\tau), u(\tau)\mathrm d{\tau} )+(\Sigma u)_0^*\right],
	\end{equation}
	where $\Sigma=\Sigma_u$, $\mathfrak r=\mathfrak r_u$, $\mathfrak z=\mathfrak z_u$ $\vartheta=\vartheta_u$ are constructed in Lemma~\ref{Lem:ContApp}, Lemma~\ref{Lem:IEE} respectively,
	\begin{equation*}
		\begin{split}
			&(\Sigma u)_{0,\delta}^*\in X_n^*,\ \langle (\Sigma u)_{0,\delta}^*;\varphi\rangle=\int_\Omega \Sigma_{0,\delta}u_{0,\delta,n}\cdot\varphi,\ M(\cdot,\Sigma,\mathfrak r,\mathfrak z,\vartheta,u)\in X_n^*,\\ 
			&\langle M(\cdot,\Sigma,\mathfrak r,\mathfrak z,\vartheta,u),\varphi\rangle=\int_\Omega \Sigma(u\otimes u)\cdot\nabla\varphi +\mathcal P_\delta (\mathfrak r,\mathfrak z,\vartheta)\dvr\varphi-\varepsilon\nabla\Sigma\nabla u\cdot\varphi-\mathbb S(\vartheta,\nabla u)\cdot \nabla\varphi\\&\text{ for any }\varphi\in X_n.
		\end{split}
	\end{equation*}
	Our task, the existence of a solution to \eqref{GalE1}, is solved once we show that $\mathcal T$ possesses a fixed point in a ball 
	\begin{equation*}
		B_{R,T_n}=\left\{u\in C([0,T_n];X_n):\ \|u\|_{C([0,T_n];X_n)}\leq R, u(0)=u_{0,\delta,n}\right\}
	\end{equation*} 
	for a suitably chosen $R$ and $T_n$. Let us estimate the expression on the right hand side of \eqref{TMDeF}. By \eqref{JEst}$_1$ it follows that
	\begin{equation}\label{JMEst}
		\begin{split}
			&\left\|J\left[\Sigma(t),\int_0^t M(\tau,\Sigma(\tau),\mathfrak r(\tau),\mathfrak z(\tau),\vartheta(\tau),u(\tau))\mathrm d\tau+(\Sigma u)_{0,\delta}^*\right]\right\|_{X_n}\\
			&\leq \frac{1}{\inf_{\overline\Omega}\Sigma(t)}\left(\sup_{\Omega}\Sigma_{0,\delta}\|u_{0,\delta,n}\|_{X_n}+\left\|\int_0^tM(\tau,\Sigma(\tau),\mathfrak r(\tau),\mathfrak z(\tau),\vartheta(\tau), u(\tau))\mathrm d\tau\right\|_{X_n^*}\right).
		\end{split}
	\end{equation}
	The straightforward calculation using Lemmas~\ref{Lem:ContApp} and \ref{Lem:IEE} yields
	\begin{equation}\label{MEst}
		\begin{split}
			&\left\|\int_0^tM(\tau,\Sigma(\tau),\mathfrak r(\tau),\mathfrak z(\tau),\vartheta(\tau), u(\tau))\mathrm d\tau\right\|_{X_n^*}\\
			&\leq c(\varepsilon,\delta)t\left( \|\vartheta\|_{L^\infty(0,T;L^\infty(\Omega))}\|u\|_{C([0,T];X_n)}+\|\Sigma\|_{C^1([0,T]\times\overline\Omega)}(1+\|u\|_{C([0,T];X_n)})\|u\|_{C([0,T];X_n)}\right.\\
			&\quad\left.+\|\mathfrak r\|^\Gamma_{C^1([0,T]\times\overline\Omega)}+\|\mathfrak z\|^\Gamma_{C^1([0,T]\times\overline\Omega)}+\|\vartheta\|^4_{L^\infty(0,T;L^4(\Omega))}+1\right).
		\end{split}
	\end{equation}
	Hence it follows for $u$ with $\|u\|_{C([0,T];X_n)}\leq R$ that
	\begin{equation*}
		\begin{split}
			&\left\|J\left[\Sigma(t),\int_0^t M(\tau,\Sigma(\tau),\mathfrak r(\tau),\mathfrak z(\tau),\vartheta(\tau),u(\tau))\mathrm d\tau+(\Sigma u)_{0,\delta}^*\right]\right\|_{X_n}\\
			&\leq \frac{c}{\inf_{\overline\Omega}\Sigma_{0,\delta}}e^{Rt}\left(\sup_{\overline\Omega}\Sigma_{0,\delta}\|u_{0,\delta,n}\|_{X_n}+th(R)\right)\text{ for all }t\in[0,T]
		\end{split}
	\end{equation*}
	using Lemma~\ref{Lem:ContApp}. We notice that $h$ is a positive function bounded on bounded sets. Choosing $R$ such that 
	\begin{equation*}
		R>2 \frac{c\sup_{\overline\Omega}\Sigma_{0,\delta}\|u_{0,\delta,n}\|_{X_n}}{\inf_{\overline\Omega}\Sigma_{0,\delta}}
	\end{equation*}
    we obtain 
    \begin{equation*}
        \|\mathcal Tu\|_{C([0,T];X_n)}\leq \frac{R}{2}e^{RT}+\frac{c}{\inf_{\overline\Omega}\Sigma_{0,\delta}}h(R)e^{RT}T.
    \end{equation*} Hence the choice of $T_n$ sufficiently small guarantees that $\mathcal T$ maps $B_{R,T_n}$ into itself.
	Next, we verify that $\mathcal T$ is continuous on $B_{R,T_n}$. To this end we consider a sequence $\{u_k\}$ such that $u_k\to u$ in $C([0,T];X_n)\cap B_{R,T_n}$ and by Lemmas~\ref{Lem:ContApp} and \ref{Lem:IEE} we have corresponding sequences $\{\Sigma_k\}$, $\{\mathfrak r_k\}$, $\{\mathfrak z_k\}$, $\{\vartheta_k\}$ such that $\Sigma_k\to\Sigma$, $\mathfrak r_k\to \mathfrak r$, $\mathfrak z_k\to\mathfrak z$ in $C^1([0,T]\times\overline \Omega)$ and $\vartheta\to\vartheta$ in $L^2(0,T;W^{1,2}(\Omega))$ with limits $\Sigma$, $\mathfrak r$, $\mathfrak z$ and $\vartheta$ corresponding to the limit velocity $u$. We want to show that 
	\begin{equation}\label{TukConv}
		\|\mathcal T(u_k)-\mathcal T(u)\|_{C([0,T_n];X_n)}\to 0.
	\end{equation}
	Obviously, we have
	\begin{equation}\label{TDiffEst}
		\begin{split}
			\|\mathcal T(u_k)-\mathcal T(u)\|_{C([0,T_n];X_n)}=&\|J[\Sigma_k,F_k]-J[\Sigma,F]\|_{C([0,T_n];X_n)}\\
			\leq& \|J[\Sigma_k,F_k]-J[\Sigma,F_k]\|_{C([0,T_n];X_n)}+\|J[\Sigma,F_k-F]\|_{C([0,T_n];X_n)}=I_k+II_k,
		\end{split}
	\end{equation}
	where we denoted 
	\begin{equation}\label{FkFDef}
		\begin{split}
			F_k(t)=&\int_0^tM(\tau,\Sigma_k(\tau),\mathfrak r_k(\tau),\mathfrak z_k(\tau),\vartheta_k(\tau),u_k(\tau))\mathrm d\tau+(\Sigma u)_0^*,\\
			F(t)=&\int_0^tM(\tau,\Sigma(\tau),\mathfrak r(\tau),\mathfrak z(\tau),\vartheta(\tau),u(\tau))\mathrm d\tau+(\Sigma u)_0^*.
		\end{split}
	\end{equation}
    Repeating the estimates leading to \eqref{MEst} we have
	\begin{equation}\label{FKEst}
		\begin{split}
			\|F_k(t)\|_{X_n^*}\leq& ct\left(\Sigma_{0,\delta}\|u_{0,\delta,n}\|_{X_n}+h(R)\right)\text{ for all }t\in[0,T_n].
		\end{split}
	\end{equation}
	Moreover, we obtain
	\begin{equation}\label{FkFDiff}
		\begin{split}
			&\|F_k(t)-F(t)\|_{X_n^*}\leq c(\varepsilon,\delta,n)\left(\|\mu\|_{C^1(\eR)}\|\vartheta_k-\vartheta\|_{L^2(0,T;W^{1,2}(\Omega))}\|u_k\|_{C([0,T];X_n)}\right.\\
			&\left.+\|\vartheta\|_{L^2(0,T;W^{1,2}(\Omega))}\|u_k-u\|_{C([0,T];X_n)}+\|\Sigma_k-\Sigma\|_{C^1([0,T]\times\overline\Omega)}(1+\|u_k\|_{C([0,T];X_n)})\|u_k\|_{C([0,T];X_n)}\right.\\&\left.+\|\Sigma\|_{C^1([0,T]\times\overline\Omega)}\left(\|u_k\|_{C([0,T];X_n)}+\|u\|_{C([0,T];X_n)}+1\right)\|u_k-u\|_{C([0,T];X_n)}\right.\\&\left.+\|D_{r,z,\vartheta}\mathcal P_\delta\|_{L^\infty(Q_T)}\left(\|\mathfrak r_k-\mathfrak r\|_{C^1([0,T]\times\overline\Omega)}+\|\mathfrak z_k-\mathfrak z\|_{C^1([0,T]\times\overline\Omega)}+\|\vartheta_k-\vartheta\|_{L^2(0,T;W^{1,2}(\Omega))}\right)\right).
		\end{split}
	\end{equation}
    The dependence of the constant $c(\varepsilon,\delta,n)$ on $n$ is due to the use of the equivalence of norms on $X_n$ in the latter estimate. We note that 
	\begin{equation*}
		\begin{split}
			\|D_{r,z,\vartheta}\mathcal P_\delta\|_{L^\infty(Q_T)}&\leq c(\delta)\left((1+\|\mathfrak r_k\|^\Gamma_{C^1([0,T]\times\overline\Omega)}+\|\mathfrak r\|^\Gamma_{C^1([0,T]\times\overline\Omega)}+\|\mathfrak z_k\|^\Gamma_{C^1([0,T]\times\overline\Omega)}+\|\mathfrak z\|^\Gamma_{C^1([0,T]\times\overline\Omega)})\right.\\
			&\left.\times(\|\vartheta_k\|_{L^\infty(Q_T)}+\|\vartheta\|_{L^\infty(Q_T)})+\|\vartheta_k\|^3_{L^\infty(Q_T)}+\|\vartheta\|^3_{L^\infty(Q_T)}\right)
		\end{split}
	\end{equation*}
	by the definition in \eqref{PDDef} and the hypothesis in \eqref{PressVarGr}. Hence by Lemmas~\ref{Lem:ContApp} and \ref{Lem:IEE} we have
	\begin{equation*}
		\|D_{r,z,\vartheta}\mathcal P_\delta\|_{L^\infty(Q_T)}\leq c,
	\end{equation*}
	where $c$ is independent of $k$. Consequently, it follows from \eqref{FkFDiff} that
	\begin{equation}\label{FkFconv}
		\sup_{t\in[0,T_n]}\|F_k(t)-F(t)\|_{X_n^*}\to 0.
	\end{equation}
	Returning back to \eqref{TDiffEst} we conclude $I_k\to 0$ by \eqref{JEst}$_2$ and \eqref{FKEst}. Moreover, we have by \eqref{JEst}$_1$
	\begin{equation*}
		II_k\leq \frac{1}{\inf_{[0,T]\times\overline\Omega}\Sigma}\|F_k-F\|_{C([0,T_n];X_n^*)}.
	\end{equation*} 
	Eventually, employing \eqref{FkFconv} we conclude \eqref{TukConv}.
	
	The next task is to verify that $\mathcal T$ is a compact operator on $B_{R,T_n}$.
	To this end we first fix arbitrary $t_1,t_2\in(0,T_n)$ and $u\in B_{R,T_n}$. Obviously, keeping the notation from \eqref{FkFDef}
	\begin{equation}\label{TTimeDiff}
		\|\mathcal T(u)(t_1)-\mathcal T(u)(t_2)\|_{X_n}\leq \|J[\Sigma(t_1),F(t_1)]-J[\Sigma(t_2),F(t_1)]\|_{X_n}+\|J[\Sigma(t_2),F(t_1)-F(t_2)]\|_{X_n}.
	\end{equation}
	By the ideas used for proving \eqref{JEst}$_2$, \eqref{JMEst} and \eqref{MEst} we conclude
	\begin{equation*}
		\|J[\Sigma(t_1),F(t_1)]-J[\Sigma(t_2),F(t_1)]\|_{X_n}\leq \frac{1}{(\inf_{[0,T]\times\overline\Omega}\Sigma)^2}\|\Sigma\|_{C^1([0,T]\times\overline\Omega)}(\sup_{\overline\Omega}\Sigma_{0,\delta}\|u_{0,\delta,n}\|_{X_n}+T_nh(R))|t_1-t_2|
	\end{equation*}
	and 
	\begin{equation*}
		\begin{split}
			\|J[\Sigma(t_2),F(t_1)-F(t_2)]\|_{X_n}\leq& \frac{1}{\inf_{[0,T]\times\overline\Omega}\Sigma}\left\|\int_{t_1}^{t_2} M(\tau,\Sigma(\tau),\mathfrak r(\tau),\mathfrak z(\tau),\vartheta(\tau),u(\tau))\mathrm d\tau\right\|_{X_n^*}\\
			\leq&\frac{1}{\inf_{[0,T]\times\overline\Omega}\Sigma} |t_1-t_2|h(R).
		\end{split}
	\end{equation*}
	As a consequence of \eqref{TTimeDiff} and the latter two inequalities, we infer that $\mathcal T(B_{R,T_n})$ consists of uniformly Lipschitz functions on $[0,T_n]$. In particular, by the Arzel\`a--Ascoli theorem $\mathcal T(B_{R,T_n})$ is a relatively compact subset of $C([0,T_n];X_n)$. Thus $\mathcal T$ is compact on $B_{R,T_n}$. Accordingly, the Leray--Schauder fixed point theorem yields the existence of a fixed point $u\in B_{R, T_n}$ for $\mathcal T$. As a consequence we get the existence of the $\Sigma$, $\mathfrak r$ $\mathfrak z$, $\alpha$, $\vartheta$, $u$ on a possibly short time interval $[0,T_n]$. Iterating this procedure as many time as necessary we reach $T_n=T$, i.e., we get the existence of the solution to \eqref{GalE1}, as long as there is a bound on $u$ independent of $n$. Finally, as we have already verified that $\mathcal T(B_{R,T_n})$ consists of uniformly in time Lipschitz functions, we know, in particular, that 
	\begin{equation}\label{TimeLVel} 
		u\in W^{1,\infty}([0,T_n];X_n).
	\end{equation}
	
	\subsubsection{Uniform estimates and global existence of Galerkin approximations}
	This subsection is devoted to the derivation of an uniform estimate of $\sup_{t\in(0,T_n)}\|u\|_{X_n}$ that is independent of $n$. Let us start with a necessary preparatory work.
 	Expanding the derivatives one has 	
    \begin{equation}\label{longiden1}
		\begin{split}
			&\tder \mathcal E_\delta(\mathfrak r,\mathfrak z,\vartheta)+\dvr\left(\mathcal E_\delta(\mathfrak r,\mathfrak z,\vartheta)u\right)\\
			&=\partial_r\mathcal E_\delta(\mathfrak r,\mathfrak z,\vartheta)\left(\tder\mathfrak r+u\cdot\nabla\mathfrak r\right)+\partial_z \mathcal E_\delta(\mathfrak r,\mathfrak z,\vartheta)\left(\tder\mathfrak z+	u\cdot\nabla \mathfrak z\right)+\partial_\vartheta\mathcal E_\delta(\mathfrak r,\mathfrak z,\vartheta)\left(\tder\vartheta+u\cdot\nabla\vartheta\right)+\mathcal E_\delta(\mathfrak r,\mathfrak z,\vartheta)\dvr u\\
			&=\mathsf e_{+,\delta}(\mathfrak r,\vartheta)(\tder\mathfrak r+\dvr(\mathfrak r u))+\mathfrak r\partial_r\mathsf e_{+,\delta}(\mathfrak r,\vartheta)(\tder\mathfrak r+u\cdot \nabla \mathfrak r)+\mathsf e_{-,\delta}(\mathfrak z,\vartheta)(\tder\mathfrak z+\dvr(\mathfrak z u))\\
            &\quad+\mathfrak z\partial_z\mathsf e_{-,\delta}(\mathfrak z,\vartheta)(\tder\mathfrak z+u\cdot \nabla \mathfrak z)+\left(4b\vartheta^3+\mathfrak r\partial_\vartheta\mathsf e_{+,\delta}(\mathfrak r,\vartheta)+\mathfrak z\mathsf e_{-,\delta}(\mathfrak z,\vartheta)\right)(\tder\vartheta+u\cdot\nabla\vartheta)+b\vartheta^4\dvr u.
            \end{split}
        \end{equation}
        Applying the Gibbs relations from \eqref{GibbsT} and the regularized continuity equation for $r=\mathfrak r$ and $r=\mathfrak z$ one obtains
        \begin{equation*}
            r\mathsf e_{\pm,\delta}(r,\vartheta)(\tder r+u\cdot\nabla r)=r\left(\vartheta\mathsf s_{\pm,\delta}(r,\vartheta)(\tder r+u\cdot\nabla r)+\frac{\mathsf P_{\pm}(r,\vartheta)}{r^2}(\varepsilon\Delta r-r\dvr u)\right).
        \end{equation*}
        Employing the latter identity in \eqref{longiden1} we arrive at 
	\begin{equation}\label{longiden}
		\begin{split}
			&\tder \mathcal E_\delta(\mathfrak r,\mathfrak z,\vartheta)+\dvr\left(\mathcal E_\delta(\mathfrak r,\mathfrak z,\vartheta)u\right)\\
			&=\varepsilon \left(\Delta \mathfrak r\left(\mathsf e_{+,\delta}(\mathfrak r,\vartheta)+\frac{\mathsf P_+\left(\mathfrak r,\vartheta\right)}{\mathfrak r}\right)+\Delta \mathfrak z\left(\mathsf e_{-,\delta}(\mathfrak z,\vartheta)+\frac{\mathsf P_-\left(\mathfrak z,\vartheta\right)}{\mathfrak z}\right)\right)-\left(\mathsf P_+(\mathfrak r,\vartheta)+\mathsf P_-(\mathfrak z,\vartheta)-b\vartheta^4\right)\dvr u\\
			&\quad+\vartheta\left(\tder\left(\frac{4b}{3}\vartheta^3\right)+u\cdot\nabla\left(\frac{4b}{3}\vartheta^3\right)+\mathfrak r\left(\tder\mathsf s_{+,\delta}(\mathfrak r,\vartheta)+u\cdot\nabla \mathsf s_{+,\delta}(\mathfrak r,\vartheta)\right)+\mathfrak z\left(\tder\mathsf s_{-,\delta}(\mathfrak z,\vartheta)+u\cdot\nabla \mathsf s_{-,\delta}(\mathfrak z,\vartheta)\right)\right)\\
			&=\vartheta\left(\tder\mathcal S_\delta(\mathfrak r,\mathfrak z,\vartheta)+\dvr\left(\mathcal S_\delta(\mathfrak r,\mathfrak z,\vartheta)u\right)\right)-\left(\frac{b}{3}\vartheta^4+\mathsf P_+(\mathfrak r,\vartheta)+\mathsf P_-(\mathfrak z,\vartheta)\right)\dvr u\\
			&\quad+\varepsilon \left(\Delta \mathfrak r\left(\mathsf e_{+,\delta}(\mathfrak r,\vartheta)-\vartheta\mathsf s_{+,\delta}(\mathfrak r,\vartheta)+\frac{\mathsf P_+\left(\mathfrak r,\vartheta\right)}{\mathfrak r}\right)+\Delta \mathfrak z\left(\mathsf e_{-,\delta}(\mathfrak z,\vartheta)-\vartheta\mathsf s_{-,\delta}(\mathfrak z,\vartheta)+\frac{\mathsf P_-\left(\mathfrak z,\vartheta\right)}{\mathfrak z}\right)\right),
		\end{split}
	\end{equation}
 
	where $\mathsf e_{\pm,\delta}=\mathsf e_\pm+\delta\vartheta$,
	\begin{equation}\label{TSDDef}
		\begin{split}
			\mathsf{s}_{\pm,\delta}(r,\vartheta)=& \mathsf{s}_{\pm}(r,\vartheta)+\delta\log\vartheta,\\
			\mathcal S_\delta(r,z,\vartheta)=&\frac{4b}{3}\vartheta^3+r\mathsf s_{+,\delta}(r,\vartheta)+z\mathsf s_{-,\delta}(z,\vartheta)
		\end{split}
	\end{equation}
	for $\mathsf s_{\pm}$ defined in \eqref{COV}.
	
	Using \eqref{longiden} in pointwise identity \eqref{IntEnEq} we get after the division of the result by $\vartheta$
	\begin{equation}\label{SDEq}
		\begin{split}
			&\tder\mathcal S_\delta(\mathfrak r,\mathfrak z,\vartheta)+\dvr\left(\mathcal S_\delta(\mathfrak r,\mathfrak z,\vartheta)u\right)-\dvr\left(\frac{\kappa_\delta(\vartheta)}{\vartheta}\nabla\vartheta\right)\\
			&=\frac{1}{\vartheta}\mathbb S(\vartheta,\nabla u)\cdot\nabla u+\frac{\kappa_\delta(\vartheta)}{\vartheta^2}|\nabla\vartheta|^2+\frac{\varepsilon}{\vartheta}\left(\partial^2_{rr}h_\delta(\mathfrak r,\mathfrak z)|\nabla\mathfrak r|^2+2\partial^2_{rz}h_\delta(\mathfrak r,\mathfrak z)\nabla\mathfrak r\cdot\nabla\mathfrak z+\partial^2_{zz}h_\delta(\mathfrak r,\mathfrak z)|\nabla\mathfrak z|^2\right)+\frac{\delta}{\vartheta^3}\\
            &\quad-\varepsilon\vartheta^{\Lambda}+\frac{\varepsilon}{\vartheta}\left(\Delta \mathfrak r\left(\vartheta\mathsf s_{+,\delta}\left(\mathfrak r,\vartheta\right)-\mathsf e_{+,\delta}\left(\mathfrak r,\vartheta\right)-\frac{\mathsf P_+\left(\mathfrak r,\vartheta\right)}{\mathfrak r}\right)+\Delta \mathfrak z\left(\vartheta\mathsf{ s}_{-,\delta}\left(\mathfrak z,\vartheta\right)-\mathsf{e}_{-,\delta}\left(\mathfrak z,\vartheta\right)-\frac{\mathsf P_-\left(\mathfrak z,\vartheta\right)}{\mathfrak z}\right)\right)
		\end{split}
	\end{equation}
	a.e. in $(0,T_n)\times\Omega$.
	Taking into consideration \eqref{TimeLVel}, we can differentiate \eqref{GalE1} with respect to time and obtain
	\begin{equation}\label{GalTDer}
		\begin{split}
			&\int_\Omega \tder \left(\Sigma u\right)\varphi+\int_\Omega \mathbb S\cdot\nabla\varphi-\int_\Omega \Sigma u\otimes u\cdot\nabla\varphi -\mathcal P_\delta\dvr \varphi\\
			&+\varepsilon \nabla \Sigma \nabla u\cdot\varphi=0\text{ a.e. in }(0,T_n)\text{ for all }\varphi\in X_n.
		\end{split}
	\end{equation}
	Setting $\varphi=u(t)$ in the latter identity for fixed $t\in(0,T_n)$ yields
	\begin{equation}\label{MomTest}
		\begin{split}
			0=&\frac{1}{2}\frac{\mathrm d}{\dt}\int_\Omega \Sigma |u|^2+\frac{1}{2}\int_\Omega \tder \Sigma|u|^2+\int_\Omega \mathbb S\cdot\nabla u+\frac{1}{2}\int_\Omega\dvr\left(\Sigma u\right) |u|^2 -\int_\Omega\mathcal P_\delta\dvr u-\frac{\varepsilon}{2} \int_\Omega\Delta \Sigma|u|^2\\
			=&\frac{1}{2}\frac{\mathrm d}{\dt}\int_\Omega \Sigma |u|^2+\int_\Omega \mathbb S\cdot\nabla u-\int_\Omega\mathcal P_\delta\dvr u.
		\end{split}
	\end{equation}
	thanks to the regularized continuity equation solved by $\Sigma$.
	
	We multiply the regularized continuity equation for $\mathfrak r$ on $\partial_r h_\delta$, see \eqref{HDDef} for the definition of $h_\delta$, the regularized continuity equation for $\mathfrak z$ on $\partial_z h_\delta$ respectively, and sum the resulting identities to obtain
	\begin{equation*}
		\tder h_\delta(\mathfrak r,\mathfrak z) + u\cdot\nabla h_\delta(\mathfrak r,\mathfrak z)+\left(\mathfrak r\partial_rh_\delta(\mathfrak r,\mathfrak z)  +\mathfrak z\partial_z h_\delta(\mathfrak r,\mathfrak z)\right)\dvr u  -\varepsilon \left(\partial_rh_\delta(\mathfrak r,\mathfrak z)\Delta\mathfrak r+\partial_zh_\delta(\mathfrak r,\mathfrak z)\Delta\mathfrak z\right)=0 \text{ in }Q_T. 
	\end{equation*}
	Integrating the latter identity over $\Omega$ for an arbitrary but fixed $t\in(0,T_n)$ yields
	\begin{equation}\label{IntHDExp}
		\begin{split}
			&\frac{\mathrm d}{\dt}\int_\Omega h_\delta(\mathfrak r,\mathfrak z) +\delta\int_\Omega \left(\mathfrak r^\Gamma+\mathfrak z^\Gamma+\mathfrak r^2+\mathfrak z^2\right)\dvr u+\varepsilon\int_\Omega\left(\partial^2_{rr}h_\delta(\mathfrak r,\mathfrak z)|\nabla\mathfrak r|^2+\partial^2_{zz}h_\delta(\mathfrak r,\mathfrak z)|\nabla\mathfrak z|^2\right)=0, 
		\end{split}
	\end{equation}
	where we also employed the integration by parts and the boundary conditions for $\mathfrak r$, $\mathfrak z$ and $u$. Summing \eqref{MomTest} and \eqref{IntHDExp}, taking into consideration the definition of $\mathcal P_\delta$ in \eqref{PDDef}, we arrive at
	\begin{equation}\label{AUId1}
		\begin{split}
			\frac{\mathrm d}{\dt}&\int_\Omega \left(\frac{1}{2}\Sigma |u|^2+h_\delta(\mathfrak r,\mathfrak z)\right)+\int_\Omega \mathbb S\cdot\nabla u-\int_\Omega\mathcal P\dvr u+\varepsilon\int_\Omega\left(\partial^2_{rr}h_\delta(\mathfrak r,\mathfrak z)|\nabla\mathfrak r|^2 z+\partial^2_{zz}h_\delta(\mathfrak r,\mathfrak z)|\nabla\mathfrak z|^2\right)=0.
		\end{split}
	\end{equation}
	We integrate \eqref{IntEnEq} over $\Omega$, apply the divergence theorem, \eqref{IENeum} and \eqref{KappaDDef}. Adding the result to \eqref{AUId1} and integrating over $(0,t)$ for an arbitrary but fixed $t\in(0,T_n)$ we get
	\begin{equation}\label{AuxIdEn}
		\begin{split}
			&\int_\Omega\left(\frac{1}{2}\Sigma|u|^2+\mathcal E_\delta(\mathfrak r,\mathfrak z,\vartheta)+h_\delta(\mathfrak r,\mathfrak z)\right)(t)\\
			&=\int_\Omega\left(\frac{1}{2}\Sigma_{0,\delta} |u_{0,\delta,n}|^2+\mathcal E_\delta(\mathfrak r_{0,\delta},\mathfrak z_{0,\delta},\vartheta_{0,\delta})+h_\delta(\mathfrak r_{0,\delta},\mathfrak z_{0,\delta})\right)+\int_0^t\int_\Omega\left( \frac{\delta}{\vartheta^2}-\varepsilon\vartheta^{\Lambda+1}\right).
		\end{split}
	\end{equation}
	
	Integrating \eqref{SDEq} over $Q_t$, multiplying this identity on an arbitrary positive constant $\Theta$ and subtracting the resulting identity from \eqref{AuxIdEn} we get
	\begin{equation}\label{EnBalG}
		\begin{split}
			&\int_\Omega\left(\frac{1}{2}\Sigma|u|^2+ H_{\delta,\Theta}\left(\mathfrak r,\mathfrak z,\vartheta\right)+h_\delta(\mathfrak r,\mathfrak z)\right)(t)\\
			&\quad+\Theta\int_0^t\int_\Omega\left(\frac{1}{\vartheta}\mathbb S\cdot\nabla u+\frac{\kappa_\delta(\theta)}{\vartheta^2}|\nabla\vartheta|^2+\frac{\delta}{\vartheta^3}+\frac{\varepsilon}{\vartheta}\left(\partial^2_{rr}h_\delta(\mathfrak r,\mathfrak z)|\nabla\mathfrak r|^2+\partial^2_{zz}h_\delta(\mathfrak r,\mathfrak z)|\nabla\mathfrak z|^2\right)\right)\\
			&\quad+\varepsilon\int_0^t\int_\Omega\vartheta^{\Lambda+1}\\
			&=\int_\Omega \left(\frac{1}{2}\Sigma u_{0,\delta}|u_{0,\delta,n}|^2+ H_{\delta,\Theta}\left(\mathfrak r_{0,\delta},\mathfrak z_{0,\delta},\vartheta_{0,\delta}\right)\right)+\int_\Omega h_\delta(\mathfrak r_{0,\delta},\mathfrak z_{0,\delta})+\int_0^t\int_\Omega \left(\frac{\delta}{\vartheta^2}+\varepsilon\Theta\vartheta^{\Lambda}\right)\\
			&\quad-\varepsilon\Theta\int_0^t\int_\Omega\frac{1}{\vartheta}\left(\Delta \mathfrak r\left(\vartheta\mathsf{ s}_{+,\delta}\left(\mathfrak r,\vartheta\right)-\mathsf{e}_{+,\delta}\left(\mathfrak r,\vartheta\right)-\frac{\mathsf P_+\left(\mathfrak r,\vartheta\right)}{\mathfrak r}\right)+\Delta \mathfrak z\left(\vartheta\mathsf{ s}_{-,\delta}\left(\mathfrak z,\vartheta\right)-\mathsf{e}_{-,\delta}\left(\mathfrak z,\vartheta\right)-\frac{\mathsf P_-\left(\mathfrak z,\vartheta\right)}{\mathfrak z}\right)\right),
		\end{split}
	\end{equation} 
	where 
	\begin{equation}\label{HelmDDef}
		\begin{split}
			H_{\delta,\Theta}(r,z,\vartheta)=&H_{\Theta}(r,z,\vartheta)+\delta (r+z) \left(\vartheta-\Theta\log\vartheta\right),\\
			H_{\Theta}(r,z,\vartheta)=&b\vartheta^4-\Theta\frac{4b}{3}\vartheta^3+r\mathsf e_{+}(r,\vartheta)-\Theta r\mathsf s_{+}(r,\vartheta)+z\mathsf e_{-}(z,\vartheta)-\Theta z\mathsf s_{-}(z,\vartheta).
		\end{split}
	\end{equation}
	Further we denote by $\mathcal I$ the integral multiplied on the factor $-\varepsilon\Theta$ in \eqref{EnBalG}. Let us focus our attention on $\mathcal I$. The relations from \eqref{GibbsT} imply
	\begin{equation*}
		\begin{split}
			&\nabla\left(\frac{1}{\vartheta}\left(\vartheta\mathsf s_{\pm,\delta}( r,\vartheta)-\mathsf e_{\pm,\delta}(r,\vartheta)-\frac{\mathsf P_\pm(r,\vartheta)}{r}\right)\right)\\
			&=\frac{1}{\vartheta^2}\left(\mathsf e_{\pm,\delta}(r,\vartheta)+ r\partial_r\mathsf e_\pm( r,\vartheta)\right)\nabla\vartheta-\frac{\partial_r\mathsf P_\pm( r,\vartheta)}{r\vartheta}\nabla r.
		\end{split}
	\end{equation*}
	Hence integrating by parts we obtain 
	\begin{equation*}
		\begin{split}
			\mathcal I=&\int_0^t\int_\Omega\frac{1}{\vartheta^2}\left(\mathsf e_{+,\delta}(\mathfrak r,\vartheta)+\mathfrak r\partial_r\mathsf e_+(\mathfrak r,\vartheta)\right)\nabla\vartheta\cdot\nabla\mathfrak r-\int_0^t\int_\Omega \frac{\partial_r\mathsf P_+(\mathfrak r,\vartheta)}{\mathfrak r\vartheta}|\nabla\mathfrak r|^2\\
			&+\int_0^t\int_\Omega\frac{1}{\vartheta^2}\left(\mathsf e_{-,\delta}(\mathfrak z,\vartheta)+\mathfrak z\partial_r\mathsf e_-(\mathfrak z,\vartheta)\right)\nabla\vartheta\cdot\nabla\mathfrak z-\int_0^t\int_\Omega \frac{\partial_r\mathsf P_-(\mathfrak z,\vartheta)}{\mathfrak z\vartheta}|\nabla\mathfrak z|^2=\sum_{i=1}^4 J_i.
		\end{split}
	\end{equation*}
	Obviously, the terms $J_2$ and $J_4$ are nonpositive. When estimating $J_1$ and $J_4$ we employ the hypotheses in \eqref{IntEnGrowth} and \eqref{EneRVar} and the Young inequality to conclude
	\begin{equation*}
		\begin{split}
			&\frac{1}{\vartheta^2}|(\mathsf e_\pm(r,\vartheta)+r\partial_r \mathsf e_\pm(r,\vartheta))\nabla\vartheta\cdot\nabla r|\leq c\frac{r^{\gamma_\pm-1}+d_\pm^e\vartheta^{\omega^e_\pm}+(r^{\underline\Gamma^e-1}+r^{\overline\Gamma^e-1})(1+\overline d^e_\pm\vartheta)^{\overline\omega^e}}{\vartheta^2}|\nabla r||\nabla\vartheta|\\
			&\leq \frac{\beta c(\Gamma)}{\vartheta}\left(\Gamma r^{\Gamma-2}+2\right)|\nabla r|^2+  c(\beta^{-1},\Gamma)\left(\vartheta^{\Gamma-2}+\frac{1}{\vartheta^3}\right)|\nabla\vartheta|^2
		\end{split}
	\end{equation*}
	with an arbitrary $\beta>0$ provided $\Gamma>\max\{2(\gamma_\pm-1),2\omega^e_\pm-1,2\overline\omega^e-1,\underline\Gamma^e+1,\overline\Gamma^e+1\}$.
	Having the estimates of the terms $J_1$ and $J_3$ at hand we choose first $\beta$ sufficiently small and subsequently $\varepsilon$ sufficiently small as well to obtain
	\begin{equation}\label{J1J3Est}
		\varepsilon\left(J_1+J_3\right)\leq \frac{1}{2}\int_0^t\int_\Omega \left(\frac{\varepsilon\delta}{\vartheta}(\Gamma \mathfrak r^{\Gamma-2}+2)|\nabla \mathfrak r|^2+\frac{\varepsilon\delta}{\vartheta}(\Gamma \mathfrak z^{\Gamma-2}+2)|\nabla \mathfrak z|^2\right)+\frac{\delta}{4}\int_0^t\int_\Omega|\nabla\vartheta|^2\left(\frac{1}{\vartheta}+\frac{1}{\vartheta^3}\right).
	\end{equation}
	Applying the Young inequality in the third integral on the right hand side of \eqref{EnBalG} we get
	\begin{equation*}
		\int_0^t\int_\Omega \left(\frac{\delta}{\vartheta^2}+\Theta\varepsilon\vartheta^\Lambda\right)\leq \frac{\delta\Theta}{2}\int_0^t\int_\Omega\frac{1}{\vartheta^3}+\frac{\varepsilon}{2}\int_0^t\int_\Omega\vartheta^{\Lambda+1}+cT|\Omega|\left(\delta\Theta^{-2}+\varepsilon\Theta^\Lambda\right)
	\end{equation*} 
	as $T_n\leq T$. Going back to \eqref{EnBalG} we immediately get for $t\in(0,T_n)$
	\begin{equation}\label{AuxEnIn}
		\begin{split}
			&\int_\Omega\left(\frac{1}{2}\Sigma|u|^2+ H_{\delta,\Theta}\left(\mathfrak r,\mathfrak z, \vartheta\right)+h_\delta(\mathfrak r,\mathfrak z)\right)(t)+\frac{\Theta}{4}\int_0^t\int_\Omega \sigma_{\varepsilon,\delta}+\frac{\varepsilon}{2}\int_0^t\int_\Omega\vartheta^{\Lambda+1}\\
			&\leq\int_\Omega \left(\frac{1}{2}\Sigma _{0,\delta}|u_{0,\delta,n}|^2+ H_{\delta,\Theta}\left(\mathfrak r_{0,\delta},\mathfrak z_{0,\delta},\vartheta_{0,\delta}\right)\right)+\int_\Omega h_\delta(\mathfrak r_{0,\delta},\mathfrak z_{0,\delta})\\
			&\quad+cT|\Omega|\left(\delta\Theta^{-2}+\varepsilon\Theta^\Lambda\right),
		\end{split}
	\end{equation} 
	where 
	\begin{equation}\label{SEDDef}
		\begin{split}
			\sigma_{\varepsilon,\delta}=&\left(\frac{1}{\vartheta}\mathbb S\cdot\nabla u+\frac{\kappa_\delta(\theta)}{\vartheta^2}|\nabla\vartheta|^2+\frac{\delta}{\vartheta^3}+\frac{\varepsilon}{\vartheta}\left(\partial^2_{rr}h_\delta(\mathfrak r,\mathfrak z)|\nabla\mathfrak r|^2+\partial^2_{zz}h_\delta(\mathfrak r,\mathfrak z)|\nabla\mathfrak z|^2\right)\right.\\
			&\left.+\frac{\partial_r\mathsf P_+(\mathfrak r,\vartheta)}{\mathfrak r\vartheta}|\nabla\mathfrak r|^2+\frac{\partial_r\mathsf P_-(\mathfrak z,\vartheta)}{\mathfrak z\vartheta}|\nabla\mathfrak z|^2\right).
		\end{split}
	\end{equation}
	Next, we focus on properties of the function $H_{\delta,\Theta}$ that help us to conclude suitable estimates on the norm of the velocity. The  differentiability of $\mathsf e_{\pm}$ and $\mathsf s_{\pm}$ allows us to compute
	\begin{equation}\label{HDer}
		\begin{split}
			\partial_r H_{\delta,\Theta}(r,z,\Theta)&=\mathsf{e}_{+,\delta}(r,\Theta)-\Theta\mathsf s_{+,\delta}(r,\Theta)+r\left(\partial_r\mathsf e_{+,\delta}(r,\Theta)-\Theta\partial_r\mathsf s_{+,\delta}(r,\Theta)\right),\\
			\partial_z H_{\delta,\Theta}(r,z,\Theta)&=\mathsf{e}_{-,\delta}(z,\Theta)-\Theta\mathsf s_{-,\delta}(z,\Theta)+z\left(\partial_z\mathsf e_{-,\delta}(z,\Theta)-\Theta\partial_z\mathsf s_{-,\delta}(z,\Theta)\right),\\
			\partial^2_{rr}H_{\delta,\Theta}(r,z,\Theta)&=\frac{1}{r}\partial_r \mathsf P_\pm(r,\Theta),\\
			\partial^2_{zz}H_{\delta,\Theta}(r,z,\Theta)&=\frac{1}{z}\partial_z \mathsf P_-(z,\Theta),\\
			\partial_\vartheta H_{\delta,\Theta}(r,z,\vartheta)&=4b\vartheta^2(\vartheta-\Theta)+\frac{r}{\vartheta}(\vartheta-\Theta)\partial_\vartheta\mathsf e_+(r,\vartheta)+\frac{z}{\vartheta}(\vartheta-\Theta)\partial_\vartheta\mathsf e_-(z,\vartheta)+\delta\frac{r}{\vartheta}\left(\vartheta-\Theta\right),
		\end{split}
	\end{equation}
	where we also used \eqref{GibbsT}. We point out that \eqref{HDer}$_{3,4}$, \eqref{ThStab} and also the obvious fact that $\partial^2_{rz}H_{\delta,\Theta}=\partial^2_{zr}H_{\delta,\Theta}=0$ imply that the function $(r,z)\mapsto H_{\delta,\Theta}(r,z,\Theta)$ is strictly convex. Consequently, we have
	\begin{equation}\label{HThBBound}
		H_{\delta,\Theta}(r,z,\Theta)-\partial_r H_{\delta,\Theta}(\overline r,\overline z,\Theta)(r-\overline r)-\partial_z H_{\delta,\Theta}(\overline r,\overline z,\Theta)(z-\overline z)-H_{\delta,\Theta}(\overline r,\overline z,\Theta)\geq 0
	\end{equation}
	for any $r,z\in (0,\infty)^2\setminus\{(\overline r,\overline z)\}$ with arbitrary $\overline r$, $\overline z$ and $\Theta$ positive.
	
	Thanks to the assumption on $\partial_\vartheta\mathsf e_\pm$ following from \eqref{ThStab} we deduce that for fixed $r$ and $z$ the function $\vartheta\mapsto H_{\delta,\Theta}$ is decreasing for $\vartheta<\Theta$, increasing for $\vartheta>\Theta$ and attains its global minimum at $\vartheta=\Theta$, in other words we have for fixed $r$, $z$ that
	\begin{equation}\label{HMin}
		H_{\delta,\Theta}(r,z,\vartheta)\geq H_{\delta,\Theta}(r,z,\Theta).
	\end{equation}
	Subtracting 
	\begin{equation*}
		\int_\Omega \left(\partial_r H_{\delta,\Theta}(\overline{\mathfrak r},\overline{\mathfrak z},\Theta)(\mathfrak r-\overline{\mathfrak r})+\partial_z H_{\delta,\Theta}(\overline{\mathfrak r},\overline{\mathfrak z},\Theta)(\mathfrak z-\overline{\mathfrak z})\right)(t) +\int_\Omega H_{\delta,\Theta}(\overline{\mathfrak r},\overline{\mathfrak z},\Theta)
	\end{equation*} 
	with $\overline {\mathfrak r}=\frac{\int_\Omega\mathfrak {r}_{0,\delta}}{|\Omega|}$, $\overline{\mathfrak z}=\frac{\int_\Omega\mathfrak {z}_{0,\delta}}{|\Omega|}$
	from the both sides of \eqref{AuxEnIn} yields
	\begin{equation}\label{EnIneqHDer}
		\begin{split}
			&\int_\Omega\Bigl(\frac{1}{2}\Sigma|u|^2+ H_{\delta,\Theta}\left(\mathfrak r,\mathfrak z,\vartheta\right)-\partial_r H_{\delta,\Theta}(\overline{\mathfrak r},\overline{\mathfrak z},\Theta)(\mathfrak r-\overline{\mathfrak r})-\partial_z H_{\delta,\Theta}(\overline{\mathfrak r},\overline{\mathfrak z},\Theta)(\mathfrak z-\overline{\mathfrak z})-H_{\delta,\Theta}(\overline{\mathfrak r},\overline{\mathfrak z},\Theta)\\
			&\quad+h_\delta(\mathfrak r,\mathfrak z)\Bigr)(t)+\frac{\Theta}{4}\int_0^t\int_\Omega \sigma_{\varepsilon,\delta}+\frac{\varepsilon}{2}\int_0^t\int_\Omega\vartheta^{\Lambda+1}\\
			&\leq\int_\Omega \left(\frac{1}{2}\Sigma_{0,\delta}|u_{0,\delta,n}|^2+ H_{\delta,\Theta}\left(\mathfrak r_{0,\delta},\mathfrak z_{0,\delta},\vartheta_{0,\delta}\right)\right)+\int_\Omega h_\delta(\mathfrak r_{0,\delta},\mathfrak z_{0,\delta})+cT|\Omega|(\delta\Theta^{-2}+\varepsilon\Theta^\Lambda)\\
			&\quad-\int_\Omega \left(H_{\delta,\Theta}(\overline{\mathfrak r},\overline{\mathfrak z},\Theta)\right),
		\end{split}
	\end{equation}
	as $\int_\Omega r(t)=\int_\Omega r_0$ for any $t\in(0,T)$ valid for $(r,r_0)=(\mathfrak r,\mathfrak r_{0,\delta})$ and $(r,r_0)=(\mathfrak z,\mathfrak z_{0,\delta})$ as well. The first term on the right hand side is bounded independently of $n$ since by the definition of $u_{0,\delta,n}$ in \eqref{UDNDef}, the continuity of $P_n$ and the definition of $(\Sigma u)_{0,\delta}$ in \eqref{SUZDDef} we have
	\begin{equation*}
		\int_\Omega \Sigma_{0,\delta}|u_{0,\delta,n}|^2\leq \sup_\Omega\Sigma_{0,\delta}\int_\Omega \frac{|(\Sigma u)_{0,\delta}|^2}{\Sigma_{0,\delta}^2}\leq\frac{\sup_\Omega\Sigma_{0,\delta}}{\inf_\Omega\Sigma_{0,\delta}}\int_{\Omega}\frac{|(\Sigma u)_{0}|^2}{\Sigma_0}.
	\end{equation*}
	The last term on the right hand side of \eqref{EnIneqHDer} is just a constant depending on $\overline{\mathfrak r}, \overline{\mathfrak z},\Theta$, $|\Omega|$ and $\delta$.  Combining \eqref{HThBBound} and \eqref{HMin} we conclude that 
	\begin{equation*}
		H_{\delta,\Theta}\left(\mathfrak r,\mathfrak z,\vartheta\right)-\partial_r H_{\delta,\Theta}(\overline{\mathfrak r},\overline{\mathfrak z},\Theta)(\mathfrak r-\overline{\mathfrak r})-\partial_z H_{\delta,\Theta}(\overline{\mathfrak r},\overline{\mathfrak z},\Theta)(\mathfrak z-\overline{\mathfrak z})\geq 0.
	\end{equation*}
	Accordingly, we deduce from \eqref{EnIneqHDer}
	\begin{equation}\label{NonnegKE}
		\sup_{t\in[0,T_n]}\int_{\Omega}\Sigma|u|^2\leq c.
	\end{equation}
	Hence as $\inf_{[0,T]\times\overline\Omega} \Sigma>0$, cf. Lemma~\ref{Lem:ContApp}, we infer 
	\begin{equation*}
		\sup_{t\in[0,T_n]}\|u\|_{X_n}\leq c
	\end{equation*}
	with the constant $c$ depending on the data, $\varepsilon$, $\delta$, $\Theta$. Therefore we can set $T_n=T$ for each $n\in\eN$.

	\subsubsection{Faedo--Galerkin limit}\label{Sec:NPass}
	
	We begin with the specification of the system of equations and its solution that is obtained after the limit passage $n\to\infty$ in \eqref{GalE1}, \eqref{IntEnEq} and \eqref{RReg}. Namely, the limit sixtet $(\xi,\Sigma,\mathfrak r,\mathfrak z,\vartheta,u)$ possesses the regularity
	\begin{equation}\label{EDSolReg}
            \begin{gathered}
			\xi,\Sigma,\mathfrak r,\mathfrak z\geq 0\text{ in }Q_T,\\ 
			\Sigma, \mathfrak r,\mathfrak z\in W^{1,1}(0,T;L^1(\Omega))\cap L^1(0,T;W^{2,1}(\Omega)),\ \Sigma\in L^2(0,T;W^{1,2}(\Omega)),\ u\in L^2(0,T;W^{1,2}(\Omega)),\\
			\vartheta\in L^\infty(0,T;L^4(\Omega)), \vartheta^{-1}\in L^3(0,T;L^3(\Omega)), \vartheta\in L^2(0,T;W^{1,2}(\Omega)),\\
			\mathcal S_\delta(\mathfrak r,\mathfrak z,\vartheta)\in L^1(\Omega).
            \end{gathered}
		\end{equation}
	$r\in\{\xi,\Sigma,\mathfrak r,\mathfrak z\}$ satisfies the aproximate continuity equation
	\begin{equation}\label{ApCE}
		\begin{split}
			\tder r+\dvr (ru)-\varepsilon\Delta r&=0\text{ a.e. in }Q_T,\\
			\nabla r\cdot n&=0\text{ on }(0,T)\times\partial\Omega,\\
			r(0)&=r_{0,\delta}.
		\end{split}
	\end{equation}
	The approximate balance of momentum equation is satisfied in the form
	\begin{equation}\label{ABM}
		\begin{split}
			&\int_\Omega \Sigma u\cdot\varphi(t)-\int_\Omega (\Sigma u)_{0,\delta}\cdot\varphi\\&=\int_0^t\int_\Omega \left(\Sigma u\cdot\tder\varphi+\Sigma(u\otimes u)\cdot\nabla\varphi-\mathbb S(\vartheta,\nabla u)\cdot\nabla\varphi+\mathcal  P_\delta(\mathfrak r,\mathfrak z,\vartheta)\dvr\varphi-\varepsilon\nabla\Sigma\nabla u\cdot\varphi\right)
		\end{split}
	\end{equation}
	for any $t\in[0,T]$ and any $\varphi\in C^1([0,T]\times\overline\Omega)$ with $\varphi\cdot n=0$ on $(0,T)\times\partial\Omega$ in the case of the complete slip boundary conditions or $\varphi=0$ on $(0,T)\times\partial\Omega$ in the case of the no--slip boundary conditions. The expressions $\mathbb S$ and $\mathcal P_\delta$ are defined in \eqref{STensDef}, \eqref{PDDef} respectively.
	
	The approximate entropy inequality is satisfied in the form
	\begin{equation}\label{AEI}
		\begin{split}
			&\int_\Omega \mathcal S_\delta(\mathfrak r,\mathfrak z,\vartheta)\phi(t)-\int_\Omega \mathcal S_\delta(\mathfrak r_0,\mathfrak z_0,\vartheta_0)\phi(0)+\int_0^t\int_\Omega\left(\frac{\kappa_\delta(\vartheta)}{\vartheta}\nabla\vartheta+\varepsilon R_\delta\right)\cdot\nabla\phi+\int_0^t\int_\Omega\sigma_{\varepsilon,\delta}\phi\\
			& \leq\int_0^t\int_\Omega\mathcal S_\delta(\mathfrak r,\mathfrak z,\vartheta)\left(\tder\phi+u\cdot\nabla\phi\right)+\varepsilon\int_0^t\int_\Omega\vartheta^\Lambda\phi
		\end{split}
	\end{equation}
	for almost all $t\in[0,T]$ and any $\phi\in C^1([0,T]\times\overline\Omega)$, $\phi\geq 0$, where $\mathcal S_\delta$ and $\kappa_\delta$ are defined in \eqref{TSDDef}$_2$, \eqref{KappaDDef} respectively, and
	\begin{equation}\label{SDEDef}
		\begin{split}
			\sigma_{\varepsilon,\delta}=&\frac{1}{\vartheta}\left(\mathbb S(\vartheta,\nabla u)\cdot\nabla u+\left(\frac{\kappa(\vartheta)}{\vartheta}+\frac{\delta}{2}\left(\vartheta^{\Gamma-1}+\frac{1}{\vartheta^2}\right)\right)|\nabla\vartheta|^2+\frac{\delta}{\vartheta^2}\right)\\
			&+\frac{\varepsilon}{\vartheta}\left(\partial^2_{rr}h_\delta(\mathfrak r,\mathfrak z)|\nabla\mathfrak r|^2+\partial^2_{zz}h_\delta(\mathfrak r,\mathfrak z)|\nabla\mathfrak z|^2\right)+\frac{\partial_r\mathsf P_+(\mathfrak r,\vartheta)}{\mathfrak r\vartheta}|\nabla \mathfrak r|^2+\frac{\partial_r\mathsf P_-(\mathfrak z,\vartheta)}{\mathfrak z\vartheta}|\nabla \mathfrak z|^2,\\
			R_\delta=&-\left(\vartheta\mathsf s_{+,\delta}(\mathfrak r,\vartheta) -\mathsf e_{+,\delta}-\frac{\mathsf P_+(\mathfrak r,\vartheta)}{\mathfrak r}\right)\frac{\nabla\mathfrak r}{\vartheta}-\left(\vartheta\mathsf s_{-,\delta}(\mathfrak z,\vartheta) -\mathsf e_{-,\delta}-\frac{\mathsf P_-(\mathfrak z,\vartheta)}{\mathfrak z}\right)\frac{\nabla\mathfrak z
			}{\vartheta}
		\end{split}
	\end{equation}
	The approximate energy balance
	\begin{equation}\label{AEB}
		\begin{split}
			&\int_\Omega\left(\frac{1}{2}\Sigma|u|^2+\mathcal E_\delta(\mathfrak r,\mathfrak z,\vartheta)+h_\delta(\mathfrak r,\mathfrak z)\right)(t)\\&=\int_{\Omega}\left(\frac{1}{2}\frac{|(\Sigma u)_{0}|^2}{\Sigma_{0,\delta}}+\mathcal E_\delta(\mathfrak r_{0,\delta},\mathfrak z_{0,\delta},\vartheta_{0,\delta})+h_\delta(\mathfrak r_{0,\delta},\mathfrak z_{0,\delta})\right)\\
			&\quad+\int_0^t\int_\Omega \left(\frac{\delta}{\vartheta^2}-\varepsilon\vartheta^\Lambda\right)\text{ holds for a. a. }t\in[0,T],
		\end{split}
	\end{equation}
	where $\mathcal E_\delta$ is defined in \eqref{MEDDef}.
	
	\subsubsection{Estimates independent of the dimension of Faedo-Galerkin approximations}\label{SubSec:NEst}
	Let $\{\xi_n,\Sigma_n,\mathfrak r_n,\mathfrak z_n,\vartheta_n, u_n\}$ be a sequence of sixtets constructed in the previous subsections. Namely, let inequality 
	\begin{equation}\label{AuxEN}
		\begin{split}
			&\int_\Omega\left(\Sigma_n|u_n|^2+\mathcal E_\delta(\mathfrak r_n,\mathfrak z_n,\vartheta_n)+\Theta\left|\mathcal S_\delta(\mathfrak r_n,\mathfrak z_n,\vartheta_n)\right|+h_\delta(\mathfrak r_n,\mathfrak z_n)\right)(t)\\
			&\quad+\frac{\varepsilon}{2}\int_0^t\int_\Omega\vartheta_n^{\Lambda+1}+\frac{\Theta}{4}\int_0^t\int_\Omega \left(\frac{1}{\vartheta_n}\left(\mathbb S(\vartheta_n,\nabla u_n)\cdot\nabla u_n+\left(\frac{\kappa(\vartheta_n)}{\vartheta_n}+\delta\left(\vartheta_n^{\Gamma-1}+\frac{1}{\vartheta_n^2}\right)\right)|\nabla\vartheta_n|^2\right)+\frac{\delta}{\vartheta_n^3}\right.\\
			&\quad\left.+\frac{\varepsilon}{\vartheta}\left(\partial^2_{rr}h_\delta(\mathfrak r_n,\mathfrak z_n)|\nabla\mathfrak r_n|^2+\partial^2_{zz}h_\delta(\mathfrak r_n,\mathfrak z_n)|\nabla\mathfrak z_n|^2\right)+\frac{\partial_r\mathsf P_+(\mathfrak r_n,\vartheta_n)}{\mathfrak r\vartheta}|\nabla\mathfrak r_n|^2+\frac{\partial_r\mathsf P_-(\mathfrak z_n,\vartheta_n)}{\mathfrak z\vartheta}|\nabla\mathfrak z_n|^2\right)\\
			&\leq c(\text{data},T,|\Omega|,\varepsilon,\delta,\Theta,\overline{\mathfrak r},\overline{\mathfrak z})
		\end{split}
	\end{equation}
	hold for any $t\in(0,T)$. The latter inequality follows from \eqref{EnIneqHDer} combined with the ensuing assertion concerning the coercivity of the function $H_{\delta,\Theta}$ defined in \eqref{HelmDDef}$_1$, cf. \cite[p. 1536]{KNAC},
	\begin{equation}\label{HCoer}
		\begin{split}
			&\frac{1}{2}\left(\mathcal E_\delta(r,z,\vartheta) +\Theta|\mathcal S_\delta(r,z,\vartheta)|\right)\\
			&\leq \frac{3}{2} H_{\delta,\Theta}(r,z,\vartheta)+|\partial_r H_{\delta,2\Theta}(\overline r,\overline z,2\Theta)(r-\overline r)+\partial_z H_{\delta,2\Theta}(\overline r,\overline z,\vartheta)(z-\overline z)+H_{\delta,2\Theta}(\overline r,\overline z,2\theta)|\\
			&\text{ holds for any } (r,z,\vartheta)\in(0,\infty)^3.
		\end{split}
	\end{equation}
	Indeed, the latter inequality is obviously satisfied if $\mathcal S_\delta(r,z,\vartheta)<0$. Otherwise, we get that $2H_{\delta,\Theta}(r,z,\vartheta)=\mathcal E_\delta(r,z,\vartheta)+H_{\delta,2\Theta}(r,z,\vartheta)$, $H_{\delta,\Theta}(r,z,\vartheta)=\Theta\mathcal S_\delta(r,z,\vartheta)+H_{\delta,2\Theta}(r,z,\vartheta)$ and 
	\begin{equation*}
		H_{\delta,2\Theta}(r,z,\vartheta)-\partial_r H_{\delta,2\Theta}(\overline r,\overline z,2\Theta)(r-\overline r)-\partial_z H_{\delta,2\Theta}(\overline r,\overline z,2\Theta)(z-\overline r)-H_{\delta,2\Theta}(\overline r,\overline z,2\Theta)\geq 0,
	\end{equation*}
	from \eqref{HThBBound} and \eqref{HMin}. Hence
	\begin{align*}
		\mathcal E_\delta(r,z,\vartheta)+\Theta\mathcal S_\delta(r,z,\vartheta)\leq &3H_{\delta,\Theta}(r,z,\vartheta)+2\left(\partial_r H_{\delta,2\Theta}(\overline r,\overline z,2\Theta)(r-\overline r)+\partial_z H_{\delta,2\Theta}(\overline r,\overline z,2\Theta)(z-\overline r)\right.\\ &\left. +H_{\delta,2\Theta}(\overline r,\overline z,2\Theta)\right)
	\end{align*}
	implying \eqref{HCoer}.
	It follows from \eqref{PointEquiv} and \eqref{PointEqS} that
	\begin{equation}\label{NPointEqS}
		\begin{split}
			\underline a\mathfrak r^n\leq \mathfrak z^n\leq \overline a\mathfrak r^n,\\
			\underline b\Sigma_n\leq \mathfrak r_n+\mathfrak z_n\leq\overline b\Sigma_n ,\\
			\underline d\mathfrak r_n\leq \xi_n\leq\overline d\mathfrak r_n.
		\end{split}
	\end{equation}
	We now make a collection of estimates that are necessary for further analysis. We deduce directly from \eqref{AuxEN}
	\begin{equation}\label{NEst}
		\begin{split}
			\|\sqrt{\Sigma_n}u_n\|_{L^\infty(0,T;L^2(\Omega))}&\leq c,\\
			\|\mathfrak r_n\|_{L^\infty(0,T;L^\Gamma(\Omega))}+\|\mathfrak z_n\|_{L^\infty(0,T;L^\Gamma(\Omega))}&\leq c,\\
			\|\vartheta_n^{-1}\mathbb S(\vartheta_n,\nabla u_n)\cdot\nabla u_n\|_{L^1(Q_T)}&\leq c,\\
			\|\vartheta_n\|_{L^\infty(0,T;L^4(\Omega))}+\|\vartheta_n\|_{L^{\Lambda+1}(Q_T)}&\leq c,\\
			\|\vartheta_n^{-1}\|_{L^3(Q_T)}&\leq c,\\
			\|\nabla \vartheta^\frac{\Gamma}{2}_n\|_{L^2(0,T;L^2(\Omega))}+\|\nabla \vartheta^{-\frac{1}{2}}_n\|&\leq c,\\
			\|\sqrt{\kappa(\vartheta_n)}\vartheta_n^{-1}\nabla\vartheta_n\|_{L^2(0,T;L^2(\Omega))}&\leq c,\\
			\|\mathcal E_\delta(\mathfrak r_n,\mathfrak z_n,\vartheta_n)\|_{L^\infty(0,T;L^1(\Omega))}+\|\mathcal S_\delta(\mathfrak r_n,\mathfrak z_n,\vartheta_n)\|_{L^\infty(0,T;L^1(\Omega))}&\leq c.
		\end{split}
	\end{equation}
	Next, we make a list of estimates following from \eqref{NEst}. By \eqref{NEst}$_{1,2}$, and \eqref{NPointEqS}$_{2}$ it follows that
	\begin{equation}\label{SNUNEst}
		\|\Sigma_nu_n\|_{L^\infty(0,T;L^\frac{2\Gamma}{\Gamma+1}(\Omega))}\leq c.
	\end{equation} 
	Estimate \eqref{NEst}$_3$, the definition of $\mathbb S$ in \eqref{STensDef} and \eqref{MuGr} yield
	\begin{equation}\label{L2GrUnEst}
		\|\nabla u_n+(\nabla u_n)^\top-\frac{2}{3}\dvr u_n\mathbb I\|_{L^2(Q_T)}\leq c.
	\end{equation} 
	Then we apply Lemma~\ref{Lem:KornPoinc} to conclude
	\begin{equation}\label{GKPN}
		\|u_n(t)\|_{W^{1,2}(\Omega)}\leq c(\underline M, K, \Gamma)\|(\nabla u_n+(\nabla u_n)^\top-\frac{2}{3}\dvr u_n\mathbb I)(t)\|_{L^2(\Omega)}+\int_\Omega \Sigma_n|u_n|(t),
	\end{equation}
	provided $\int_\Omega \Sigma_n^\Gamma(t)\leq K$ and $\int_\Omega\Sigma_n(t)\geq \underline M>0$. Let us note that the existence of the constant $K$ follows by \eqref{NPointEqS}$_1$ and \eqref{NEst}$_2$. Moreover, we can set $\underline M=\int_\Omega\Sigma_{0,\delta}>0$ as $\int_\Omega\Sigma_n(t)=\int_\Omega\Sigma_{0,\delta}$ for any $t\in[0,T]$ since $\Sigma_n$ solves the regularized continuity equation. Employing estimates \eqref{NEst}$_{1,2}$, \eqref{L2GrUnEst} in \eqref{GKPN} we deduce
	\begin{equation}\label{UNBSEst}
		\|u_n\|_{L^2(0,T;W^{1,2}(\Omega))}\leq c
	\end{equation}
	and
	\begin{equation}\label{UNBSLEst}
		\|u_n\|_{L^2(0,T;L^6(\Omega))}\leq c
	\end{equation}
	by the Sobolev embedding $W^{1,2}(\Omega)\hookrightarrow L^6(\Omega)$.
	Moreover, by \eqref{NEst}$_7$ and \eqref{KGrowth} we conclude
	\begin{equation}\label{NLThetaNEst}
		\|\nabla \log\vartheta_n\|_{L^2(Q_T)}\leq c,
	\end{equation}
	which in a combination with \eqref{NEst}$_{6}$ yields
	\begin{equation*}
		\begin{split}
			&\|\nabla\vartheta_n\|_{L^2(0,T;L^2(\Omega))}\leq \|\nabla\vartheta_n\chi_{\{\vartheta_n\leq 1\}}\|_{L^2(0,T;L^2(\Omega))}+\|\nabla\vartheta_n\chi_{\{\vartheta_n> 1\}}\|_{L^2(0,T;L^2(\Omega))}\\
			&\leq \|\nabla\log\vartheta_n\|_{L^2(0,T;L^2(\Omega))}+\|\nabla\vartheta^\frac{\Gamma}{2}_n\|_{L^2(0,T;L^2(\Omega))}\leq c.
		\end{split}
	\end{equation*}
	Hence we get in the same manner
	\begin{equation*}
		\|\nabla \vartheta_n^\nu\|_{L^2(0,T;L^2(\Omega))}\leq c, \text{ for any }\nu\in \left[1,\frac{\Gamma}{2}\right].
	\end{equation*}
	Using Lemma~\ref{Lem:Poincare} and \eqref{NEst}$_4$ we obtain
	\begin{equation}\label{TNBEst}
		\|\vartheta^\nu_n\|_{L^2(0,T;W^{1,2}(\Omega))}\leq c\text{ for any }\nu\left[1,\frac{\Gamma}{2}\right].
	\end{equation}
	Accodingly, using the Sobolev embedding $W^{1,2}(\Omega)\hookrightarrow L^6(\Omega)$ we have
	\begin{equation}\label{TNGEst}
		\|\vartheta_n\|_{L^\Gamma(0,T;L^{3\Gamma}(\Omega))}\leq c,
	\end{equation}
	which combined with \eqref{NEst}$_5$ implies
	\begin{equation}\label{LTNEst}
		\|\log\vartheta_n\|_{L^q(Q_T)}\leq c\text{ for any }q\in[1,\infty).
	\end{equation}
	
	We return back to \eqref{AUId1}, where we integrate over $(0,T)$, use the fact that $\mathbb S(\vartheta_n,\nabla u_n)\cdot\nabla u_n\geq 0$ and estimate using \eqref{PressGrowth}, \eqref{NEst}$_2$, \eqref{TNGEst}
	\begin{equation*}
		\begin{split}	&\varepsilon\delta\int_0^T\int_\Omega\left((\Gamma\mathfrak r_n^{\Gamma-2}+2)|\nabla\mathfrak r_n|^2+(\Gamma\mathfrak z_n^{\Gamma-2}+2)|\nabla\mathfrak z_n|^2\right)\leq
			\left|\int_0^T\int_\Omega\mathcal P(\mathfrak r_n,\mathfrak z_n,\vartheta_n)\dvr u_n\right|\\
            &\quad+\int_\Omega\left(\frac{1}{2}\Sigma_{0,\delta}|u_{0,\delta,n}|^2+h_\delta(\mathfrak r_{0,\delta},\mathfrak z_{0,\delta})\right)\\
			&\leq c\int_0^T\int_\Omega(\mathfrak r_n^{\gamma_+}+\mathfrak z_n^{\gamma_-}+\mathfrak r_n^{\gamma^P_+}\vartheta_n^{\omega^P_+}+\mathfrak z_n^{\gamma^P_-}\vartheta_n^{\omega^P_-}+\vartheta_n^4)|\dvr u_n|+c\\
			&\leq c\left(1+\|\mathfrak r_n\|^{\gamma_+}_{L^\Gamma(Q_T)}+\|\mathfrak z_n\|^{\gamma_-}_{L^\Gamma(Q_T)}+(\|\mathfrak r_n\|^{\gamma^P_+}_{L^\Gamma(\Omega)}+\|\mathfrak z_n\|^{\gamma^P_-}_{L^\Gamma(\Omega)})\|\vartheta_n\|_{L^\Gamma(Q_T)}\right)\|\dvr u_n\|_{L^2(Q_T)}
		\end{split}
	\end{equation*}
	provided $\Gamma>\max\{2\gamma_\pm, 2(\gamma^P_\pm+\omega^P_\pm)\}$.
	As the right hand side of the latter inequality is bounded independently of $n$ due to \eqref{NEst}$_2$ and \eqref{UNBSEst}, we conclude 
	\begin{equation}\label{DensNEst}
		\|\mathfrak r_n\|_{L^2(0,T;W^{1,2}(\Omega))}+\|\mathfrak r^\frac{\Gamma}{2}_n\|_{L^2(0,T;W^{1,2}(\Omega))}+\|\mathfrak z_n\|_{L^2(0,T;W^{1,2}(\Omega))}+\|\mathfrak z^\frac{\Gamma}{2}_n\|_{L^2(0,T;W^{1,2}(\Omega))}\leq c.
	\end{equation}
	Moreover, we have for $t\in(0,T)$ and $r_n=\xi_n$ or $r_n=\Sigma_n$ from \eqref{RReg}
	\begin{equation*}
		\frac{1}{2}\|r_n(t)\|^2+\varepsilon\|\nabla r_n\|_{L^2(Q_T)}=\frac{1}{2}\|r_{0,\delta}\|^2_{L^2(\Omega)}-\frac{1}{2}\int_0^t\int_\Omega r_n^2\dvr u_n.
	\end{equation*}
	Hence by \eqref{NPointEqS}$_2$, \eqref{NEst}$_{2}$ and \eqref{UNBSEst} we conclude 
	\begin{equation}\label{XSNEst}
		\|\xi_n\|_{L^2(0,T;W^{1,2}(\Omega))}+\|\Sigma_n\|_{L^2(0,T;W^{1,2}(\Omega))}\leq c.
	\end{equation}
	In order to obtain further estimates on the sequences $\{\mathfrak r_n\}$, $\{\mathfrak z_n\}$, $\{\Sigma_n\}$ and $\{\xi_n\}$ we exploit the $L^p$--$L^q$ theory for parabolic equations. Indeed, we can rewrite \eqref{RReg} with $(r,u)=(r_n,u_n)$, where $r_n=\mathfrak r_n$ or $r_n=\mathfrak z_n$ or $r_n=\Sigma_n$ in the form
	\begin{equation}\label{RNPE}
		\tder r_n-\varepsilon\Delta r_n=-\left(\nabla r_n\cdot u_n+r_n\dvr u_n\right)\text { in }Q_T.
	\end{equation}
	The up to now available estimates, namely \eqref{UNBSLEst} and \eqref{DensNEst} imply
	\begin{equation}\label{TTL1B}
		\|\nabla r_n\cdot u_n\|_{L^1(0,T;L^\frac{3}{2}(\Omega))}\leq c.
	\end{equation}
	Hence the right hand side of \eqref{RNPE} is bounded in $L^1(0,T;L^\frac{3}{2}(\Omega))$ and an additional bound with a better time integrability is necessary.
	To this end we derive the estimate
	\begin{equation}\label{NSqrRNEst}
		\varepsilon\int_0^T\int_\Omega\frac{|\nabla r_n|^2}{r_n}\leq c,\text{ for }r_n=\mathfrak r_n\text{ or }r_n=\mathfrak z_n\text{ or }r_n=\Sigma_n.
	\end{equation}
	
	Multiplying \eqref{RReg} for $(r,u)=(r_n,u_n)$ on $G'(r_n)$ for some differentiable function $G$, integrating over $Q_t$ for a fixed  $t\in[0,T]$ and performing the integration by parts we obtain 
	\begin{equation}\label{CERegRen}
		\int_\Omega G(r_n)(t)-\int_{\Omega} G(r_{0,\delta})+\int_0^t\varepsilon\int_\Omega G''(r_n)|\nabla r_n|^2=\int_0^t\int_\Omega\left(G(r_n)-G'(r_n)r_n\right)\dvr u_n\text{ for all }t\in[0,T].
	\end{equation}
	Hence setting $G(r_n)=r_n\log r_n$ in the latter identity we conclude \eqref{NSqrRNEst}.
	
	Estimates \eqref{NPointEqS}, \eqref{NEst}$_1$ and \eqref{NSqrRNEst} yield
	\begin{equation}\label{SL1B}
		\|\nabla r_n\cdot u_n\|_{L^2(0,T;L^1(\Omega))}\leq \left\|\frac{\nabla r_n}{\sqrt{r_n}}\right\|_{L^2(0,T;L^2(\Omega))}\|\sqrt{r_n}u_n\|_{L^\infty(0,T;L^2(\Omega))}\leq c\text{ for } r_n=\Sigma_n\text{ or }r_n=\mathfrak r_n\text{ or }r_n=\mathfrak z_n.
	\end{equation}
	Interpolating \eqref{TTL1B} and \eqref{SL1B} we conclude 
	\begin{equation}\label{NRNUNEst}
		\|\nabla r_n\cdot u_n\|_{L^{q(p)}(0,T;L^p(\Omega))}\leq c, \text{ for any }p\in\left(1,\frac{3}{2}\right)\text{ and }q(p)\in(1,2).
	\end{equation}
	For $p$ and $q(p)$ from the latter estimates we deduce for the solution $r_n$ to \eqref{RNPE}
	\begin{equation}\label{TDerSecDerEst}
		\|\tder r_n\|_{L^{q(p)}(0,T;L^p(\Omega))}+\|\nabla^2 r_n\|_{L^{q(p)}(0,T;L^p(\Omega))}\leq c,
	\end{equation}
	using the maximal $L^p$--$L^q$ regularity theory for parabolic equations, see the summaries in \cite[Section 10.14]{FeNo09} or \cite[Section 7.6.1]{NovStr04} or more details in \cite[Chapter III]{Amann95}. 
	
	In order to proceed we collect several estimates concerning terms in the approximate entropy balance equation. To this end we rewrite \eqref{SDEq} in the form
	\begin{equation}\label{AEntDiv}
		\tder\mathcal S_\delta(\mathfrak r_n,\mathfrak z_n,\vartheta_n)+\dvr(R_n^{(1)})=R_n^{(2)}+R_n^{(3)}
	\end{equation}
	with
	\begin{equation}\label{RTerms}
		\begin{split}
			R_n^{(1)}=&\mathcal S_\delta(\mathfrak r_n,\mathfrak z_n,\vartheta_n)u_n-\frac{\kappa_\delta(\vartheta_n)}{\vartheta_n}\nabla\vartheta_n\\
			&-\varepsilon\left(\left(\vartheta_n\mathsf s_{+,\delta}(\mathfrak r_n,\vartheta_n)-\mathsf e_{+,\delta}(\mathfrak r_n,\vartheta_n)-\frac{\mathsf P_+(\mathfrak r_n,\vartheta_n)}{\mathfrak r_n}\right)\frac{\nabla \mathfrak r_n}{\vartheta_n}\right.\\
			&\left.+\left(\vartheta_n\mathsf s_{-,\delta}(\mathfrak z_n,\vartheta_n)-\mathsf e_{-,\delta}(\mathfrak z_n,\vartheta_n)-\frac{\mathsf P_-(\mathfrak z_n,\vartheta_n)}{\mathfrak z_n}\right)\frac{\nabla \mathfrak z_n}{\vartheta_n}\right),\\
			R^{(2)}_n=&\frac{1}{\vartheta_n}\left(\mathbb S(\vartheta_n,\nabla u_n)\cdot\nabla u_n+\left(\frac{\kappa(\vartheta_n)}{\vartheta_n}+\delta\left(\vartheta_n^{\Gamma-1}+\frac{1}{\vartheta^2_n}\right)\right)|\nabla\vartheta_n|^2+\frac{\delta}{\vartheta^2_n}\right)\\
			&+\frac{\varepsilon\delta}{\vartheta_n}\left(\left(\Gamma\mathfrak r_n^{\Gamma-2}+2\right)|\nabla\mathfrak r_n|^2\right)+\left(\left(\Gamma\mathfrak z_n^{\Gamma-2}+2\right)|\nabla\mathfrak z_n|^2\right)\\
			&+\varepsilon\left(\frac{\partial_r\mathsf P_+(\mathfrak r_n,\vartheta_n)}{\mathfrak r_n\vartheta_n}|\nabla\mathfrak r_n|^2+\frac{\partial_z\mathsf P_-(\mathfrak z_n,\vartheta_n)}{\mathfrak z_n\vartheta_n}|\nabla\mathfrak z_n|^2\right),\\
			R_n^{(3)}=&-\varepsilon\left(\mathsf e_{+,\delta}(\mathfrak r_n,\vartheta_n)+\mathfrak r_n\partial_r\mathsf e_{+}(\mathfrak r_n,\vartheta_n
			)\right)\frac{\nabla\mathfrak r_n\cdot\nabla\vartheta_n}{\vartheta_n^2}\\
			&-\varepsilon\left(\mathsf e_{-,\delta}(\mathfrak z_n,\vartheta_n)+\mathfrak z_n\partial_r\mathsf e_{-}(\mathfrak z_n,\vartheta_n
			)\right)\frac{\nabla\mathfrak z_n\cdot\nabla\vartheta_n}{\vartheta_n^2}-\varepsilon\vartheta_n^\Lambda.
		\end{split}
	\end{equation}
	We continue with estimates of terms appearing in $R_n^{(1)}$. To this end we use \eqref{SpEntStr} and \eqref{EntGr} to obtain
	\begin{equation*}
		\begin{split}
			&|\mathcal S_\delta (\mathfrak r_n,\mathfrak z_n,\vartheta_n)|\leq c(1+\vartheta^3+(\mathfrak r_n+\mathfrak z_n)|\log\vartheta_n|+\mathfrak r_n|\log \mathfrak r_n|+\mathfrak z_n|\log\mathfrak z_n|+\mathfrak r_n^{\gamma^s_+}+\mathfrak z_n^{\gamma^s_-}+ \mathfrak r_n\vartheta_n^{\omega^s_+}+\mathfrak z_n\vartheta_n^{\omega^s_-}).
		\end{split}
	\end{equation*}
	Hence by \eqref{NEst}$_2$ and \eqref{TNBEst} we conclude
	\begin{equation}\label{RNZNEntEst}
		\|\mathcal S_\delta (\mathfrak r_n,\mathfrak z_n,\vartheta_n)\|_{L^q(Q_T)}\leq c
	\end{equation}
	provided $q\leq\min\left\{\frac{\Gamma}{3},\frac{\Gamma}{\gamma^s_\pm},\frac{\Gamma}{1+\omega^s_\pm}\right\}$.
	Combining the latter bound and \eqref{UNBSLEst} we get for some $p>1$
	\begin{equation}\label{RNZNEntUNEst}
		\|\mathcal S_\delta (\mathfrak r_n,\mathfrak z_n,\vartheta_n)u_n\|_{L^p(Q_T)}\leq c
	\end{equation}
	provided $\Gamma>\max\{6,2\gamma^s_\pm,2(1+\omega^s_\pm)\}$. Taking into consideration \eqref{KappaDDef} we get from \eqref{AuxEN}
	\begin{equation}\label{SqKDThThNTh}
		\left\|\frac{\sqrt{\kappa_{\delta}(\vartheta_n)}}{\vartheta_n}\nabla\vartheta_n\right\|_{L^2(Q_T)}\leq c.
	\end{equation}
	Next, by estimates \eqref{NEst}$_{4,5}$ and \eqref{TNGEst} and the interpolation we get
	\begin{equation}\label{KappaDEst}
		\|\kappa_\delta(\vartheta_n)\|_{L^q(Q_T}\leq c
	\end{equation}
	for certain $q\in[1,\min\{\frac{\Gamma}{\beta},\frac{\Gamma+\frac{8}{3}}{\Gamma},3\})$. Combining the latter two estimates we conclude
	\begin{equation}\label{R1NEEntEst}
		\left\|\frac{\kappa_{\delta}(\vartheta_n)}{\vartheta_n}\nabla\vartheta_n\right\|_{L^p(Q_T)}\leq\|\kappa_\delta(\vartheta_n)\|^\frac{1}{2}_{L^q(Q_T)} \left\|\frac{\sqrt{\kappa_{\delta}(\vartheta_n)}}{\vartheta_n}\nabla\vartheta_n\right\|_{L^2(Q_T)}
	\end{equation}
	for $p=p(q)\in[1,\frac{6}{5})$. 
	By \eqref{SpEntStr} and \eqref{EntGr} we have 
	\begin{equation*}
		|\mathsf{s}_{\pm,\delta}(r_n,\vartheta_n)\nabla r_n|\leq c(|\log\vartheta_n|+|\log r_n|+r_n^{\gamma^s_{\pm}-1}+\vartheta_n^{\omega^s_\pm})|\nabla r_n|.
	\end{equation*}
	Hence employing \eqref{TNBEst}, \eqref{LTNEst} and \eqref{DensNEst} 
	\begin{equation}\label{EntNDensEst}
		\|\mathsf{s}_{\pm,\delta}(r_n,\vartheta_n)\nabla r_n\|_{L^q(Q_T)}\leq c
	\end{equation}
	provided $q\leq\min\left\{\frac{2\Gamma}{\Gamma+2\left(\gamma^s_\pm-1\right)},\frac{2\Gamma}{\Gamma+2\omega^s_\pm}\right\}$. 
	By \eqref{IntEnGrowth} we get
	\begin{equation*}
		\left|\frac{\mathsf e_{\pm,\delta}(r_n,\vartheta_n)}{\vartheta_n}\nabla r_n\right|\leq c\left(\vartheta_n^{\omega^e_\pm-1}+\frac{r_n^{\gamma^e_\pm-1}}{\vartheta_n}\right)|\nabla r_n|.
	\end{equation*}
	Consequently, by estimates \eqref{NEst}$_5$, \eqref{TNGEst} and \eqref{DensNEst} we deduce
	\begin{equation}\label{R1NEEneEst}
		\left\|\frac{\mathsf e_{\pm,\delta}(r_n,\vartheta_n)}{\vartheta_n}\nabla r_n\right\|_{L^q(Q_T)}\leq c
	\end{equation}
	provided $q\leq\min\left\{\frac{6\Gamma}{5\Gamma+6(\gamma_\pm-1)},\frac{2\Gamma}{\Gamma+2\left(\omega^e_\pm-1\right)}\right\}$. According to \eqref{PressGrowth}, it follows that
	\begin{equation*}
		\left|\frac{\mathsf P_\pm(r_n,\vartheta_n)}{r_n\vartheta_n}\nabla r_n\right|\leq c\left(\frac{r_n^{\gamma_\pm-1}}{\vartheta_n}+r_n^{\gamma^P_\pm-1}\vartheta_n^{\omega^P_\pm-1}\right)|\nabla r_n|.
	\end{equation*}
	Using again estimates \eqref{NEst}$_5$, \eqref{TNGEst} and \eqref{DensNEst} we obtain
	\begin{equation}\label{R1NEPEst}
		\left\|\frac{\mathsf P_\pm(r_n,\vartheta_n)}{r_n\vartheta_n}\nabla r_n\right\|_{L^q(Q_T)}\leq c
	\end{equation}
	provided $q\leq\min\{\frac{2\Gamma}{\Gamma+2(\gamma^P_\pm-1)},\frac{2\Gamma}{\Gamma+2(\omega^P_\pm-1)}\}$.
	As a consequence of \eqref{RNZNEntUNEst}--\eqref{R1NEPEst} we have
	\begin{equation}\label{R1NLqEst}
		\|R^{(1)}_n\|_{L^q(Q_T)}\leq c\text{ for a certain }q>1.
	\end{equation}
	We will further need the estimate of quantities involving $\vartheta_n$ and derivatives of $u_n$
	\begin{equation}\label{TNUNEst}
		\begin{split}
			\left\|\sqrt{\frac{\mu(\vartheta_n)}{\vartheta_n}}\left(\nabla u_n+(\nabla u_n)^\top-\frac{2}{3}\dvr u_n\mathbb I\right)\right\|_{L^2(Q_T)}&\leq c,\\
			\left\|\sqrt{\frac{\eta(\vartheta_n)}{\vartheta_n}}\dvr u_n\right\|_{L^2(Q_T)}&\leq c.
		\end{split}
	\end{equation}
	
	At last we obtain by \eqref{IntEnGrowth}, \eqref{TNGEst} and \eqref{DensNEst}
	\begin{equation}\label{NEDEst}
		\|\mathcal E_\delta(\mathfrak r_n,\mathfrak z_n,\vartheta_n)\|_{L^q(Q_T)}\leq c,\text{ for some }q>1.
	\end{equation}
	\subsubsection{Convergences as $n\to\infty$}\label{SubSec:NPass}
	Based on estimates \eqref{NPointEqS}, \eqref{NEst}, \eqref{UNBSEst} and \eqref{TNBEst} we infer the existence of a subsequence, which we do not relabel, and limits $\mathfrak r,\mathfrak z, u, \vartheta$ such that
	\begin{equation}\label{NWeakly}
		\begin{alignedat}{2}
			(\xi_n,\Sigma_n,\mathfrak r_n,\mathfrak z_n)&\rightharpoonup^*(\xi,\Sigma,\mathfrak r,\mathfrak z)&&\text{ in }L^\infty(0,T;L^\Gamma(\Omega)),\\
			u^n&\rightharpoonup u&&\text{ in }L^2(0,T;W^{1,2}(\Omega)),\\
			\vartheta_n&\rightharpoonup\vartheta&&\text{ in }L^2(0,T;W^{1,2}(\Omega)).
		\end{alignedat}
	\end{equation}
	Obviously, by \eqref{NWeakly}$_1$ we get from \eqref{NPointEqS}
	\begin{equation}\label{EPointEqS}
		\begin{split}
			\underline a\mathfrak r\leq \mathfrak z\leq \overline a\mathfrak r,\\
			\underline b\Sigma\leq \mathfrak r+\mathfrak z\leq\overline b\Sigma,\\
			\underline d\mathfrak r\leq \xi\leq\overline d\mathfrak r.
		\end{split}
	\end{equation}
	We deduce 
	\begin{equation}\label{SNUNCWeak}
		\Sigma_nu_n \to \Sigma u \text{ in }C_w([0,T];L^\frac{2\Gamma}{\Gamma+1}(\Omega))
	\end{equation}
	using bound \eqref{SNUNEst} by means of the Arzel\`a--Ascoli theorem as 
	it follows from \eqref{GalE1} that
	the set of functions $\{t\mapsto\int_\Omega\Sigma_n u_n(t)\varphi\}$ is equicontinuous and bounded in $C([0,T])$ for any fixed $\varphi\in\cup_{n=1}^\infty X_n$ and $\cup_{n=1}^\infty X_n$ is dense in $L^\frac{2\Gamma}{\Gamma+1}(\Omega)$.
	Moreover, the compact embedding $W^{1,2}(\Omega)$ to $L^{(\frac{2\Gamma}{\Gamma+1})'}(\Omega)$ valid for $\Gamma\geq\frac{3}{2}$ yields the convergence 
	\begin{equation*}
		\Sigma_n u_n\to\Sigma u\text{ in }C([0,T];W^{-1,2}(\Omega)),
	\end{equation*}
	which in combination with \eqref{NWeakly}$_2$ implies
	\begin{equation}\label{ConvTNWeakly}
		\Sigma_n u_n\otimes u_n\rightharpoonup \Sigma u\otimes u\text{ in }L^1(Q_T)
	\end{equation}
	and 
	\begin{equation}\label{NKEWeakly}
		\Sigma_n|u_n|^2\rightharpoonup \Sigma|u|^2\text{ in }L^1(Q_T).
	\end{equation}
	Returning back to \eqref{DensNEst} and \eqref{XSNEst} we conclude for nonrelabeled subsequences
	\begin{equation}\label{XNSNRNZNLWWeakly}
		(\xi_n,\Sigma_n,\mathfrak r_n,\mathfrak z_n)\rightharpoonup(\xi,\Sigma,\mathfrak r,\mathfrak z)\text{ in }L^2(0,T;W^{1,2}(\Omega)).
	\end{equation}
	We obtain from \eqref{TDerSecDerEst} again for nonrelabeled subsequences
	\begin{equation}\label{TDerSecDerWeakly}
		\begin{alignedat}{2}
			(\tder\xi_n,\tder\Sigma_n,\tder\mathfrak r_n,\tder\mathfrak z_n)\rightharpoonup&(\tder\xi,\tder\Sigma,\tder\mathfrak r,\tder\mathfrak z)&&\text{ in }L^{q(p)}(0,T;L^p(\Omega)),\\
			(\nabla^2\xi_n,\nabla^2\Sigma_n,\nabla^2\mathfrak r_n,\nabla^2\mathfrak z_n)\rightharpoonup&(\nabla^2\xi,\nabla^2\Sigma,\nabla^2\mathfrak r,\nabla^2\mathfrak z)&&\text{ in }L^{q(p)}(0,T;L^p(\Omega)).
		\end{alignedat}
	\end{equation}
	Moreover, by standard Aubin--Lions compactness arguments one deduces from \eqref{TDerSecDerEst} for nonrelabeled subsequences that
	\begin{equation}\label{DensPointwisely}
		(\xi_n,\Sigma_n,\mathfrak r_n,\mathfrak z_n)\to(\xi,\Sigma,\mathfrak r,\mathfrak z)\text{ in }L^{q(p)}(0,T;W^{1,p}(\Omega))\text{ and a.e. in }Q_T.
	\end{equation}
	Combining \eqref{NPointEqS}, \eqref{NEst}$_2$, \eqref{NWeakly}$_2$ and \eqref{DensPointwisely}  we get
	\begin{equation}\label{DVRWeakly}
		(\dvr(\Sigma_nu_n),\dvr(\mathfrak r_nu_n),\dvr(\mathfrak z_nu_n))\rightharpoonup(\dvr(\Sigma u),\dvr(\mathfrak ru),\dvr(\mathfrak zu))\text{ in }L^1(Q_T).
	\end{equation}
	Using \eqref{TDerSecDerWeakly}, the continuity of the trace operator, the homogeneous Neumann boundary conditions for $\Sigma_n$, $\mathfrak r_n$, $\mathfrak z_n$ and the Stokes theorem yields
	\begin{equation*}
		(\nabla\xi\cdot n,\nabla\Sigma\cdot n,\nabla\mathfrak r\cdot n,\nabla \mathfrak z\cdot n)=0 \text{ on }(0,T)\times\partial\Omega.
	\end{equation*}
	By \eqref{NWeakly}$_2$, \eqref{DVRWeakly} and \eqref{TDerSecDerWeakly} we conclude that the limits $(\xi,\Sigma,\mathfrak r,\mathfrak z)$ fulfill \eqref{ApCE}. 
	Considering $G(z)=z^2$ in \eqref{CERegRen}, employing convergences \eqref{DensPointwisely}, \eqref{NWeakly}$_2$ and multiplying \eqref{CERegRen} on the corresponding limit function $\Sigma$, $\mathfrak r$ or $\mathfrak z$ and integrating the resulting identity over $Q_T$, $t\in(0,T]$ we get
	\begin{equation*}
		\lim_{n\to\infty}\|\nabla r_n\|^2_{L^2(Q_T)}=\|\nabla r\|^2_{L^2(Q_T)}
	\end{equation*}
 for $(r_n,r)\in \{(\Sigma_n,\Sigma),(\mathfrak r_n,\mathfrak r),(\mathfrak z_n,\mathfrak z)\}$.
	Hence taking into account \eqref{XNSNRNZNLWWeakly}
	we get in particular
	\begin{equation}\label{NDensStrong}
		(\nabla\xi_n,\nabla\Sigma_n,\nabla\mathfrak r_n,\nabla\mathfrak z_n)\to(\nabla\xi,\nabla\Sigma,\nabla\mathfrak r,\nabla\mathfrak z)\text{ in }L^2(Q_T).
	\end{equation}
	\subsubsection{Pointwise convergence of $\vartheta_n$ and consequences}
	We intend to apply the Div-Curl lemma reported in Lemma~\ref{Lem:DivCurl} on sequences $\{U_n\}$ and $\{V_n\}$,
	where
	\begin{equation*}
		\begin{split}
			U_n=&(\mathcal S_\delta (\mathfrak r_n,\mathfrak z_n,\vartheta_n), R_n^{(1)}),\\
			V_n=&(T_k(\vartheta_n),0,0,0).
		\end{split}
	\end{equation*}
	The function $T_k$ is a truncation of the identity defined as
	\begin{equation}\label{TKDef}
		T_k(z)=kT\left(\frac{z}{k}\right),\ z\in[0,\infty), k>1,
	\end{equation}
	for some $T\in C^1([0,\infty))$, $T$ increasing and concave on $[0,\infty)$, $T(z)=z$ on $[0,1]$, $T(z)\to 2$ as $z\to\infty$. Obviously, $\{U_n\}$ is weakly compact in $L^q(Q_T)$ by \eqref{RNZNEntEst} and \eqref{R1NLqEst}, whereas $\{T_k(\vartheta_n)\}$ is weakly compact in $L^{q'}(Q_T)$ for some $q>1$. As $\mathrm{Curl}\ T_k(\vartheta_n)$ is an antisymmetric matrix consisting of linear combinations of components of the vector $T'_k(\vartheta_n)\nabla\vartheta_n$, thus involving only the spatial derivatives of $\nabla\vartheta_n$, bound \eqref{TNBEst} implies $\|\mathrm{Curl}\ T_k(\vartheta_n)\|_{L^2(Q_T)}\leq c$.
	According to \eqref{AEntDiv} we obtain
	\begin{equation*}
		\mathrm{Div}\ U_n=R_n^{(2)}+R_n^{(3)}.
	\end{equation*}
	Consequently, $\|\mathrm{Div}\ U_n\|_{L^1(Q_T)}\leq c$ due to \eqref{J1J3Est} and \eqref{AuxEN}. Hence employing the compact embeddings $W^{1, q}(Q_T)$ to $L^\infty(Q_T)$ for $q>3$ and $W^{1,2}(Q_T)$ into $L^2(Q_T)$ we have $\{\mathrm{Div}\ U_n\}$ and $\{\mathrm{Curl}\ V_n\}$ precompact in $W^{-1,s}(Q_T)$  provided $s\in[1,\frac{4}{3})$.
	Using the notation $\overline F(y)$ for an $L^1(Q_T)$--weak limit of the sequence $\{F(y_n)\}$ we conclude by the Div-Curl lemma
	\begin{equation}\label{NLimitsIdent}
		\frac{4b}{3}\overline{\vartheta^3T_k(\vartheta)}+\overline{\mathfrak S_\delta(\mathfrak r,\mathfrak z,\vartheta)T_k(\vartheta)}=\frac{4b}{3}\overline{\vartheta^3}\ \overline{T_k(\vartheta)}+\overline{\mathfrak S_\delta(\mathfrak r,\mathfrak z,\vartheta)}\ \overline{T_k(\vartheta)},
	\end{equation}
	where $\mathfrak S_\delta(r,z,\vartheta)=r\mathsf s_{+,\delta}(r,\vartheta)+z\mathsf s_{-,\delta}(z,\vartheta)$.
	Let us show that
	\begin{equation}\label{SDelLimIneq}
		\overline{\mathfrak S_\delta(\mathfrak r,\mathfrak z,\vartheta)T_k(\vartheta)}\geq \overline{\mathfrak S_\delta(\mathfrak r,\mathfrak z,\vartheta)}\ \overline{T_k(\vartheta)}.
	\end{equation}
	Pointing out that $x\mapsto \mathfrak S_\delta(r,z,T_k^{-1}(x))$ is nondecreasing on $[0,\infty)$, which follows from the definition of $\mathsf s_{\pm,\delta}$ in \eqref{TSDDef}$_1$, thermodynamic stability condition \eqref{ThStab} and the fact that $(T_k^{-1})'>0$, we get
	\begin{equation*}
		(\mathfrak S_\delta(\mathfrak r_n,\mathfrak z_n,T_k^{-1}(T_k(\vartheta)))-\mathfrak S_\delta(\mathfrak r_n,\mathfrak z_n,T_k^{-1}(\overline{T_k(\vartheta)})))(T_k(\vartheta_n)-\overline{T_k(\vartheta)})\geq 0\text{ a.e. in }Q_T.
	\end{equation*}
	Hence passing to the weak limit on the left hand side of the latter inequality we arrive at
	\begin{equation}\label{EntMonIneq}
		\overline{\mathfrak S_\delta(\mathfrak r,\mathfrak z,\vartheta)T_k(\vartheta)}-\overline{\mathfrak S_\delta(\mathfrak r,\mathfrak z,\vartheta)}\ \overline {T_k(\vartheta)}\geq \overline{\mathfrak S_\delta(\mathfrak r,\mathfrak z,T_k^{-1}(\overline {T_k(\vartheta)})T_k(\vartheta)}-\overline{\mathfrak S_\delta(\mathfrak r,\mathfrak z,T_k^{-1}(\overline {T_k(\vartheta)})}\ \overline{T_k(\vartheta)}.
	\end{equation}
	Obviously, by \eqref{SpEntStr}, \eqref{EntGr}, \eqref{NEst}$_2$, \eqref{DensPointwisely} and the Vitali convergence theorem
	\begin{equation}\label{SDelNConv}
		\mathfrak S_\delta(\mathfrak r_n,\mathfrak z_n,T_k^{-1}(\overline{T_k(\vartheta)}))\to \mathfrak S_\delta(\mathfrak r,\mathfrak z,T_k^{-1}(\overline{T_k(\vartheta)}))\text{ in }L^q(Q_T)\text{ for some }q>1.
	\end{equation}
	Next, as 
	\begin{equation*}
		T_k(\vartheta_n)\rightharpoonup \overline{T_k(\vartheta)}\text{ in }L^{q'}(Q_T).
	\end{equation*}
	\eqref{SDelLimIneq} follows from \eqref{EntMonIneq}. Having \eqref{SDelLimIneq} at hand we conclude from \eqref{NLimitsIdent}
	\begin{equation*}
		\overline{\vartheta^3T_k(\vartheta)}\leq \overline{\vartheta^3}\ \overline{T_k(\vartheta)}.
	\end{equation*}
	By Lemma~\ref{Lem:MonWConv} $(1)$ we get the opposite inequality. Hence we arrive at
	\begin{equation}\label{ProdTKTh}
		\overline{\vartheta^3T_k(\vartheta)}= \overline{\vartheta^3}\ \overline{T_k(\vartheta)}.
	\end{equation}
	Next, by the weak lower semicontinuity of $L^1$--norm we have
	\begin{equation*}
		\lim_{k\to\infty}\|\overline{T_k(\vartheta)}-\vartheta\|_{L^1(Q_T)}\leq \lim_{k\to\infty}\liminf_{n\to\infty}\|T_k(\vartheta_n)-\vartheta_n\|_{L^1(Q_T)}\leq 2\lim_{k\to\infty}\sup_{n\in\eN}\int_{\{\vartheta_n\geq k\}}\vartheta_n=0,
	\end{equation*}
	and
	\begin{equation*}
		\lim_{k\to\infty}\|\overline{\vartheta^3T_k(\vartheta)}-\overline{\vartheta^4}\|_{L^1(Q_T)}\leq \lim_{k\to\infty}\liminf_{n\to\infty}\|\vartheta_n^3T_k(\vartheta_n)-\vartheta_n^4\|_{L^1(Q_T)}\leq 2\lim_{k\to\infty}\sup_{n\in\eN}\int_{\{\vartheta_n\geq k\}}\vartheta_n^4=0,
	\end{equation*} 
	where we also took into account the bound on $\{\vartheta_n\}$ in $L^\infty(0,T;L^4(\Omega))\cap L^2(0,T;L^6(\Omega))$ following from \eqref{NEst}$_4$ and \eqref{TNBEst}. By the latter two convergences combined with \eqref{ProdTKTh} we deduce
	\begin{equation*}
		\overline{\vartheta^4}=\overline{\vartheta^3}\vartheta
	\end{equation*}
	Applying Lemma~\ref{Lem:MonWConv} $(2)$ we conclude 
	\begin{equation*}
		\overline{\vartheta^3}=\vartheta^3
	\end{equation*}
	and the convergence 
	\begin{equation}\label{ThetaNPointwisely}
		\vartheta_n\to\vartheta\text{ a.e. in }Q_T
	\end{equation}
	follows at least for a nonrelabeled subsequence as $s\mapsto s^3$ is strictly convex on $(0,\infty)$, cf. \cite[Theorem 10.20]{FeNo09}.
	Now we deal with convergences of terms depending on $\vartheta_n$. First, using \eqref{DensPointwisely}, \eqref{ThetaNPointwisely}, \eqref{PressGrowth}, \eqref{PDDef}, bound \eqref{TNGEst} and the interpolation using bounds \eqref{NEst}$_{2}$ and \eqref{DensNEst} we get by the Vitali convergence theorem
	\begin{equation}\label{NPressConv}
		\mathcal P_\delta(\mathfrak r_n,\mathfrak z_n,\vartheta_n)\to \mathcal P_\delta(\mathfrak r,\mathfrak z,\vartheta)\text{ in }L^1(Q_T).
	\end{equation} 
	In particular, we have
	\begin{equation}\label{NHDConv}
		h_\delta(\mathfrak r_n,\mathfrak z_n)\to h_\delta(\mathfrak r,\mathfrak z)\text{ in }L^1(Q_T).
	\end{equation} 
	We observe that by \eqref{MuGr}, \eqref{EGr}, \eqref{ThetaNPointwisely} and estimate \eqref{TNGEst}
	\begin{equation}\label{VThMuNConv}
		\begin{alignedat}{2}
			\mu(\vartheta_n)&\to\mu(\vartheta)&&\text{ in }L^q(Q_T),\\
			\eta(\vartheta_n)&\to\eta(\vartheta)&&\text{ in }L^q(Q_T)\text{ for any }q<\Gamma.
		\end{alignedat}
	\end{equation}
	Combining the latter convergences with \eqref{NWeakly}$_2$ we conclude
	\begin{equation}\label{STensDelNConv}
		\mathbb S(\vartheta_n,\nabla u_n)\rightharpoonup 	\mathbb S(\vartheta,\nabla u)\text{ in }L^1(Q_T).
	\end{equation}
	Next, convergences \eqref{DensPointwisely}, \eqref{ThetaNPointwisely} combined with bound \eqref{NEDEst} imply
	\begin{equation}\label{NIntEnL1Str}
		\mathcal E_\delta(\mathfrak r_n,\mathfrak z_n,\vartheta_n) \to \mathcal E_\delta(\mathfrak r,\mathfrak z,\vartheta)\text{ in }L^1(Q_T)
	\end{equation}
	by the Vitali convergence theorem. Similarly, but using bound \eqref{RNZNEntEst}
	we deduce
	\begin{equation}\label{NEntConv}
		\mathcal S_\delta(\mathfrak r_n,\mathfrak z_n,\vartheta_n)\to \mathcal S_\delta(\mathfrak r,\mathfrak z,\vartheta)\text{ in }L^2(0,T;L^2(\Omega)),
	\end{equation}
	which combined with \eqref{NWeakly}$_2$ yields
	\begin{equation}\label{NEntCConv}
		\mathcal S_\delta(\mathfrak r_n,\mathfrak z_n,\vartheta_n)u_n\to \mathcal S_\delta(\mathfrak r,\mathfrak z,\vartheta)u\text{ in }L^1(0,T;L^1(\Omega)).
	\end{equation}
	Using \eqref{KGrowth}, \eqref{ThetaNPointwisely}, bounds \eqref{NEst}$_5$ and \eqref{TNGEst} and the Vitali convergence theorem yield
	\begin{equation*}
		\frac{\kappa(\vartheta_n)}{\vartheta_n}\to\frac{\kappa(\vartheta)}{\vartheta}\text{ in }L^2(Q_T).
	\end{equation*}
	The latter convergence and \eqref{NWeakly}$_3$ imply
	\begin{equation}\label{FirstWeakCon}
		\frac{\kappa(\vartheta_n)}{\vartheta_n}\nabla\vartheta_n\rightharpoonup\frac{\kappa(\vartheta)}{\vartheta}\nabla\vartheta\text{ in }L^1(Q_T).
	\end{equation}
	Employing \eqref{NEst}$_4$, \eqref{TNBEst}, \eqref{TNGEst}, \eqref{ThetaNPointwisely} and the interpolation between $L^\infty(0,T;L^4(\Omega))$ and $L^\Gamma(0,T;L^{3\Gamma}(\Omega))$ we have
	\begin{equation}\label{SecWeakConv}
		\vartheta_n^{\Gamma-1}\nabla\vartheta_n\rightharpoonup \vartheta^{\Gamma-1}\nabla\vartheta\text{ in }L^1(Q_T).
	\end{equation}
	Indeed, the above mentioned bounds guarantee the estimate of $\{\nabla\vartheta^\Gamma_n\}$ in $L^p(Q_T)$ for some $p>1$ and the existence of a nonrelabeled subsequence weakly converging in the latter space accordingly. In order to identify a weak limit one uses \eqref{ThetaNPointwisely}. Similarly, using \eqref{NEst}$_{5,6}$ and \eqref{ThetaNPointwisely} we get
	\begin{equation}\label{ThirdWeakCon}
		\frac{1}{\vartheta_n^2}\nabla\vartheta_n\rightharpoonup \frac{1}{\vartheta^2}\nabla\vartheta\text{ in }L^\frac{3}{2}(Q_T).
	\end{equation} 
	Hence taking into condsideration \eqref{KappaDDef} we conclude 
	\begin{equation}\label{KdVNOVNNVNweakly}
		\frac{\kappa_\delta(\vartheta_n)}{\vartheta_n}\nabla\vartheta_n\rightharpoonup 	\frac{\kappa_\delta(\vartheta)}{\vartheta}\nabla\vartheta\text{ in }L^1(Q_T)
	\end{equation}
	by \eqref{FirstWeakCon}--\eqref{ThirdWeakCon}.
	Next, using convergences \eqref{DensPointwisely}, \eqref{NDensStrong} and \eqref{ThetaNPointwisely} and estimates \eqref{EntNDensEst}--\eqref{R1NEPEst} we get
	\begin{equation}\label{NTermsNRConv}
		\begin{split}
			&\frac{1}{\vartheta_n}\left(\vartheta_n\mathsf s_{\pm,\delta}(r_n,\vartheta_n)-\mathsf e_{\pm,\delta}(r_n,\vartheta_n)-\frac{\mathsf P_\pm(r_n,\vartheta_n)}{r_n}\right)\nabla r_n\\
			&\to \frac{1}{\vartheta}\left(\vartheta\mathsf s_{\pm,\delta}(r,\vartheta)-\mathsf e_{\pm,\delta}(r,\vartheta)-\frac{\mathsf P_\pm(r,\vartheta)}{r}\right)\nabla r\text{ in }L^1(Q_T)
		\end{split}
	\end{equation}
	by the Vitali convergence theorem. Taking into account \eqref{MuGr}, \eqref{EGr} and bound \eqref{NEst}$_5$ we obtain by the Vitali convergence theorem
	\begin{equation*}
		\sqrt{\frac{\mu(\vartheta_n)}{\vartheta_n}}\to \sqrt{\frac{\mu(\vartheta)}{\vartheta}},\ \sqrt{\frac{\eta(\vartheta_n)}{\vartheta_n}}\to \sqrt{\frac{\eta(\vartheta)}{\vartheta}}\text{ in }L^2(Q_T).
	\end{equation*}
	Having \eqref{TNUNEst} at hand we combine the latter convergences with \eqref{NWeakly}$_2$ to obtain
	\begin{equation}\label{SDNWeakly}
		\begin{alignedat}{2}
			\sqrt{\frac{\mu(\vartheta_n)}{\vartheta_n}}\left(\nabla u_n+(\nabla u_n)^\top-\frac{2}{3}\dvr u_n\mathbb I\right)\rightharpoonup &\sqrt{\frac{\mu(\vartheta)}{\vartheta}}\left(\nabla u+(\nabla u)^\top-\frac{2}{3}\dvr u\mathbb I\right)&&\text{ in }L^2(Q_T),\\
			\sqrt{\frac{\eta(\vartheta_n)}{\vartheta_n}}\dvr u_n\rightharpoonup &\sqrt{\frac{\eta(\vartheta)}{\vartheta}}\dvr u&&\text{ in }L^2(Q_T).
		\end{alignedat}
	\end{equation}
	Next, we have 
	\begin{equation}\label{FrSqrtKDCon}
		\frac{\sqrt{\kappa_\delta(\vartheta_n)}}{\vartheta_n}\nabla\vartheta_n\rightharpoonup\frac{\sqrt{\kappa_\delta(\vartheta)}}{\vartheta}\nabla\vartheta\text{ in }L^2(Q_T).
	\end{equation}
	Indeed, employing \eqref{ThetaNPointwisely} combined with \eqref{KappaDEst} the Vitali convergence theorem yields 
	\begin{equation}\label{SqrtKDConv}		\sqrt{\kappa_\delta(\vartheta_n)}\to\sqrt{\kappa_\delta(\vartheta)}\text{ in }L^2(Q_T),
	\end{equation}
	whereas
	\begin{equation}\label{NLogConv}
		\frac{1}{\vartheta_n}\nabla\vartheta_n=\nabla\log\vartheta_n\rightharpoonup \nabla\log\vartheta=\frac{1}{\vartheta}\nabla\vartheta\text{ in }L^2(Q_T)
	\end{equation}
	due to \eqref{NLThetaNEst} and the convergence 
	\begin{equation*}
		\log\vartheta_n\to\log\vartheta\text{ in }L^q(Q_T)\text{ for any }q\in [1,\infty),
	\end{equation*}
	which follows by the Vitali convergence theorem taking into consideration \eqref{LTNEst} and \eqref{ThetaNPointwisely}. Combining \eqref{SqKDThThNTh}, \eqref{SqrtKDConv} and \eqref{NLogConv} we conclude \eqref{FrSqrtKDCon}.
	
	We have 
	\begin{equation*}
		\sqrt{\frac{\Gamma r_n^{\Gamma-2}+2}{\vartheta_n}}\to \sqrt{\frac{\Gamma r^{\Gamma-2}+2}{\vartheta}}\text{ in }L^2(Q_T)
	\end{equation*}
	by \eqref{DensPointwisely}, \eqref{ThetaNPointwisely}, bounds \eqref{NEst}$_{2,5}$ and  the Vitali convergence theorem. Hence combining the latter convergence with \eqref{NDensStrong} and with regard to \eqref{AuxEN} we obtain 
	\begin{equation}\label{SqNDensWeakly}
		\sqrt{\frac{\Gamma r_n^{\Gamma-2}+2}{\vartheta_n}}\nabla r_n\rightharpoonup \sqrt{\frac{\Gamma r^{\Gamma-2}+2}{\vartheta}}\nabla r\text{ in }L^2(Q_T).
	\end{equation}
	By virtue of convergences \eqref{DensPointwisely}, \eqref{ThetaNPointwisely}, inequality \eqref{PressVarGr}$_1$ and bounds \eqref{NEst}$_{2,5}$ the Vitali convergence theorem yields
	\begin{equation*}
		\sqrt{\frac{\partial_r\mathsf P_\pm(r_n,\vartheta_n)}{r_n\vartheta_n}}\to\sqrt{\frac{\partial_r\mathsf P_\pm(r,\vartheta)}{r\vartheta}}\text{ in }L^2(Q_T).
	\end{equation*}
	Combining the latter convergence with \eqref{NDensStrong} and with regard to \eqref{AuxEN} we deduce
	\begin{equation}\label{SqPressNWeakly}
		\sqrt{\frac{\partial_r\mathsf P_\pm(r_n,\vartheta_n)}{r_n\vartheta_n}}\nabla r_n\rightharpoonup\sqrt{\frac{\partial_r\mathsf P_\pm(r,\vartheta)}{r\vartheta}}\nabla r\text{ in }L^2(Q_T).
	\end{equation}
	By \eqref{TNGEst} and \eqref{ThetaNPointwisely} it follows that 
	\begin{equation} \label{NLPowerConv} 
		\vartheta_n^{\Lambda+1}\to \vartheta^{\Lambda+1}\text{ in }L^1(Q_T).
	\end{equation}
	The Fatou lemma and \eqref{ThetaNPointwisely} yield
	\begin{equation}\label{NTLIneq}
		\int_0^T\int_{\Omega}\frac{1}{\vartheta^3}\leq\liminf_{n\to\infty} \int_0^T\int_{\Omega}\frac{1}{\vartheta_n^3},
	\end{equation}
	which guarantees the positivity of $\vartheta$ a.e. in $Q_T$ as 
	the expression on the right hand side of the latter inequality is finite due to \eqref{NEst}$_5$. Moreover, taking into consideration \eqref{ThetaNPointwisely} and \eqref{NEst}$_5$ the Vitali convergence theorem yields
	\begin{equation} \label{NNPowerConv} 
		\vartheta_n^{-2}\to \vartheta^{-2}\text{ in }L^1(Q_T).
	\end{equation}
	Having the collection of necessary convergences we can show that the limit sixtet $(\xi,\mathfrak r,\mathfrak z,\Sigma,u,\vartheta)$ solves the approximated system \eqref{EDSolReg}--\eqref{AEI}. Obviously, the latter functions possess the regularity specified in \eqref{EDSolReg}. Next, the fact that the functions $\xi$, $\Sigma$, $\mathfrak r$ and $\mathfrak z$ solve \eqref{ApCE} follows using the convergence from \eqref{TDerSecDerWeakly} and \eqref{DVRWeakly} in the regularized continuity equation solved by elements of sequences $\{\xi_n\}$, $\{\Sigma_n\}$, $\{\mathfrak r_n\}$ and $\{\mathfrak z_n\}$. In order to show that \eqref{ABM} is fulfilled, we observe that it follows from \eqref{GalTDer} that
	\begin{equation*}
		\begin{split}
			&\int_\Omega (\Sigma_n u_n)(t)\varphi\psi(t)-\int_\Omega \Sigma_{0,\delta}u_{0,\delta,n}\cdot \varphi\psi(0)\\
			&=\int_0^t\int_\Omega \left(\Sigma_n u_n\cdot\tder\psi\varphi+\left(\Sigma_n(u_n\otimes u_n)-\mathbb S(\vartheta_n,\nabla u_n)\right)\cdot\nabla \varphi\psi+\mathcal  P_\delta(\mathfrak r_n,\mathfrak z_n,\vartheta_n)\dvr \varphi\psi-\varepsilon\nabla\Sigma_n\nabla u_n\cdot \varphi\psi\right)
		\end{split}
	\end{equation*}
	holds for any $t\in(0,T)$, $\psi\in C^1([0,T])$ and $\varphi\in \cup_{k=1}^lX_k$, $l\in\eN$. Passing first to the limit $n\to\infty$ in the latter identity using \eqref{NWeakly}$_2$, \eqref{SNUNCWeak}, \eqref{ConvTNWeakly}, \eqref{NDensStrong}, \eqref{NPressConv},\eqref{STensDelNConv} and
	the fact that
	\begin{equation}\label{UDNConv}
		u_{0,\delta, n}=P_n\left(\frac{(\Sigma u)_{0,\delta}}{\Sigma_{0,\delta}}\right)\to \frac{(\Sigma u)_{0,\delta}}{\Sigma_{0,\delta}}\text{ in }L^2(\Omega)    
	\end{equation}
	provided $\frac{(\Sigma u)_{0,\delta}}{\Sigma_{0,\delta}}\in L^2(\Omega)$, which directly follows from \eqref{SUZDDef}, the assumption $\frac{|(\Sigma u)_0|^2}{\Sigma_0}\in L^1(\Omega)$ and $\inf_{\Omega}\Sigma_{0,\delta}> 0$. Performing then the limit passage 
	$l\to\infty$ during which we employ the density of $C^{1}([0,T];\cup_{k=1}^\infty X_k)$ in $C^{1}([0,T];X^{1,2})$ for $X^{1,2}$ being defined in \eqref{XPDef}, we conclude \eqref{ABM}.
	
	In order to obtain approximate entropy inequality \eqref{AEI} we return back to 
	\eqref{AEntDiv} and estimate the first two terms in $R_n^{(3)}$ in the same way as in \eqref{J1J3Est}. Multiplying the resulting inequality on an arbitrary but fixed $\phi\in C^1([0,T]\times\overline\Omega)$, $\phi\geq 0$, integrating over $Q_t$ we arrive at
	\begin{equation}\label{NEntIneq}
		\begin{split}
			&\int_\Omega \mathcal S_\delta(\mathfrak r_n,\mathfrak z_n,\vartheta_n)\phi(t)-\int_\Omega \mathcal S_\delta(\mathfrak r_{0,\delta},\mathfrak z_{0,\delta},\vartheta_{0,\delta})\phi(0)+\int_0^t\int_\Omega\left(\frac{\kappa_\delta(\vartheta_n)}{\vartheta_n}\nabla\vartheta_n-\varepsilon R_n\right)\cdot\nabla\phi+\int_0^t\int_\Omega\sigma_{\varepsilon,\delta,n}\phi\\&\leq\int_0^t\int_\Omega\left(\mathcal S_\delta(\mathfrak r_n,\mathfrak z_n,\vartheta_n)(\tder\phi+u_n\cdot\nabla\phi)+\varepsilon\vartheta_n^\Lambda\phi\right)
		\end{split}
	\end{equation}
	with 
	\begin{equation*}
		\begin{split}
			\sigma_{\varepsilon,\delta,n}=&\frac{1}{\vartheta_n}\left(\mathbb S(\vartheta_n,\nabla u_n)\cdot\nabla u_n+\left(\frac{\kappa(\vartheta_n)}{\vartheta_n}+\delta\left(\vartheta_n^{\Gamma-1}+\frac{1}{\vartheta^2_n}\right)\right)|\nabla\vartheta_n|^2+\frac{\delta}{\vartheta^2_n}\right)\\
			&+\frac{\varepsilon\delta}{\vartheta_n}\left(\left(\Gamma\mathfrak r_n^{\Gamma-2}+2\right)|\nabla\mathfrak r_n|^2\right)+\left(\left(\Gamma\mathfrak z_n^{\Gamma-2}+2\right)|\nabla\mathfrak z_n|^2\right)\\
			&+\varepsilon\left(\frac{\partial_r\mathsf P_+(\mathfrak r_n,\vartheta_n)}{\mathfrak r_n\vartheta_n}|\nabla\mathfrak r_n|^2+\frac{\partial_z\mathsf P_-(\mathfrak z_n,\vartheta_n)}{\mathfrak z_n\vartheta_n}|\nabla\mathfrak z_n|^2\right),
		\end{split}
	\end{equation*}
	\begin{align*}
		R_n=&\left(\vartheta_n\mathsf s_{+,\delta}(\mathfrak r_n,\vartheta_n)-\mathsf e_{+,\delta}(\mathfrak r_n,\vartheta_n)-\frac{\mathsf P_+(\mathfrak r_n,\vartheta_n)}{\mathfrak r_n}\right)\frac{\nabla\mathfrak r_n}{\vartheta_n}\\
  &+\left(\vartheta_n\mathsf s_{-,\delta}(\mathfrak z_n,\vartheta_n)-\mathsf e_{-,\delta}(\mathfrak z_n,\vartheta_n)-\frac{\mathsf P_-(\mathfrak z_n,\vartheta_n)}{\mathfrak z_n}\right)\frac{\nabla\mathfrak z_n}{\vartheta_n}.
	\end{align*}
	In order to pass to the limit $n\to\infty$ in the first term on the left hand side of \eqref{NEntIneq} we employ \eqref{NEntConv}, whereas for the passage in the third term one applies \eqref{KdVNOVNNVNweakly} and \eqref{NTermsNRConv}. Performing the limit passage $n\to\infty$ in the fourth term on the left hand side of \eqref{NEntIneq} we obtain
	\begin{equation*}
		\int_0^t\int_\Omega \sigma_{\varepsilon,\delta}\phi\leq \liminf_{n\to\infty}\int_0^t\int_\Omega\sigma_{\varepsilon,\delta,n}\phi.
	\end{equation*}
	To this end we employ the weak lower semicontinuity of convex functionals and convergences \eqref{SDNWeakly}, \eqref{SqNDensWeakly}, \eqref{SqPressNWeakly},  and \eqref{NTLIneq}.
	Finally, in the terms on the right hand side of \eqref{NEntIneq} we use \eqref{NEntCConv} and \eqref{NLPowerConv}. Moreover, it follows from \eqref{AEI} that
	\begin{equation*}
		\begin{split}
			\int_0^T\int_\Omega \sigma_{\varepsilon,\delta}\phi&\leq \varepsilon\int_0^T\int_\Omega\vartheta^\Lambda\phi+\int_0^T\int_\Omega \mathcal S_\delta(\mathfrak r,\mathfrak z,\vartheta)\left(\tder \phi+u\cdot\nabla\phi\right)+\int_\Omega\mathcal S_\delta(\mathfrak r_{0,\delta}, \mathfrak z_{0,\delta},\vartheta_{0,\delta})\phi(0)\\&-\int_\Omega\mathcal S_\delta(\mathfrak r, \mathfrak z,\vartheta)\phi(T)-\int_0^T\int_\Omega\left(\frac{\kappa_\delta(\vartheta)}{\vartheta}\nabla\vartheta+\varepsilon R\right)\nabla\phi\text{ for any }\phi\in C^1([0,T]\times\overline\Omega)
		\end{split}
	\end{equation*}
	i.e. the expression on the right hand side of the latter inequality can be understood as a nonnegative linear form defined on $C^1([0,T]\times\overline\Omega)$, which we denote $\tilde\sigma_{\varepsilon,\delta}$. Consequently, by means of the Riesz representation theorem there exists a regular, nonegative Borel measure $\mu_{\tilde\sigma_{\varepsilon,\delta}}$ on $[0,T]\times\overline\Omega$ such that
	\begin{equation}\label{SigmEpsDom} \langle\tilde\sigma_{\varepsilon,\delta},\phi\rangle=\int_{[0,T]\times\overline\Omega}\phi(t,x)\mathrm d\mu_{\tilde\sigma_{\varepsilon,\delta}}(t,x)\geq\int_{[0,T]\times\overline\Omega}\sigma_{\varepsilon,\delta}\phi\text{ for any }\phi\in C^1([0,T]\times\overline\Omega), \phi\geq 0
	\end{equation}
	and
	\begin{equation}\label{AEBalMes}
		\begin{split}
			&\int_\Omega\mathcal S_\delta(\mathfrak r, \mathfrak z,\vartheta)\phi(t)-\int_\Omega\mathcal S_\delta(\mathfrak r_{0,\delta}, \mathfrak z_{0,\delta},\vartheta_{0,\delta})\phi(0)\\
			&=\int_0^t\int_\Omega \mathcal S_\delta(\mathfrak r,\mathfrak z,\vartheta)\left(\tder \phi+u\cdot\nabla\phi\right)+\int_0^t\int_\Omega\left(\frac{\kappa_\delta(\vartheta)}{\vartheta}\nabla\vartheta+\varepsilon R\right)\nabla\phi+\langle\tilde\sigma_{\varepsilon,\delta}(t),\phi\rangle- \varepsilon\int_0^t\int_\Omega\vartheta^\Lambda\phi
		\end{split}
	\end{equation}
	for any $\phi\in C^1([0,T]\times\overline\Omega)$ where
	\begin{equation*}
		\langle\tilde \sigma_{\varepsilon,\delta}(t),\varphi\rangle=\int_{[0,t]\times\overline\Omega}\varphi(\tau,x)\mathrm d\mu_{\tilde \sigma_{\varepsilon,\delta}}(t,x).
	\end{equation*}
	
	Eventually, we show that approximate energy balance \eqref{AEB} is satisfied. To this end we fix $t\in[0,T]$, multiply \eqref{AuxIdEn} with $\Sigma_n$, $u_n$, $\mathfrak r_n$, $\mathfrak z_n$, $\vartheta_n$ on $\psi(t)\in C([0,T])$ and integrate over $(0,T)$. For the limit passage $n\to\infty$ on the left hand side of the resulting identity we apply \eqref{NKEWeakly}, \eqref{NIntEnL1Str} and \eqref{NHDConv}, whereas for the limit passage on the right hand side we employ convergences \eqref{NLPowerConv}, \eqref{NNPowerConv} and \eqref{UDNConv}.

	\subsection{Limit passage $\varepsilon\to 0_+$}\label{Sec:EpsPass}
	Let $\delta>0$ be arbitrary but fixed and $\{(\xi_\varepsilon,\mathfrak r_\varepsilon,\mathfrak z_\varepsilon,\Sigma_\varepsilon,u_\varepsilon,\vartheta_\varepsilon)\}$ be a family of solutions to \eqref{EDSolReg}--\eqref{AEB}. The goal of this subsection is to show that there is a limit $(\xi,\mathfrak r,\mathfrak z,\Sigma,u,\vartheta)$ of the latter sequence such that: 
	
	The renormalized continuity equation is satisfied in the form
	\begin{equation}\label{DRCE}
		\int_\Omega r\phi(t)-\int_\Omega r_{0,\delta}\phi(0)=\int_0^t\int_\Omega r(\tder\phi+u\cdot\nabla\phi),
	\end{equation}
	for $r\in\{\xi, \Sigma, \mathfrak r, \mathfrak z\}$, for all $\phi\in C^1([0,T]\times\overline\Omega)$ and any $t\in[0,T]$.
	
	The approximate balance of momentum equation is satisfied in the form
	\begin{equation}\label{DMomBal}
		\int_\Omega\Sigma u\cdot\varphi(t)-\int_\Omega(\Sigma u)_{0,\delta}\cdot\varphi(0)=\int_0^t\int_\Omega \left(\Sigma u\cdot\tder\varphi+\Sigma u\otimes u\cdot\nabla \varphi+\mathcal P_\delta(\mathfrak r,\mathfrak z,\vartheta)\dvr\varphi-\mathbb S(\vartheta,\nabla u)\cdot\nabla\varphi\right)
	\end{equation}
	for any $t\in[0,T]$ and all $\varphi\in C^1([0,T]\times\overline\Omega)$ such that either
	\begin{equation*}
		\varphi\cdot n=0\text{ or }\varphi=0 \text{ on }(0,T)\times\partial\Omega
	\end{equation*}
	where $\mathbb S$, $\mathcal P_\delta$ are defined in \eqref{STensDef}, \eqref{PDDef} respectively.
	
	There is a positive linear functional $\sigma_\delta\in (C([0,T]\times\overline\Omega))^*$ such that
	\begin{equation}\label{SigmDDef}
		\langle\sigma_\delta,\phi\rangle\geq \int_0^T\int_\Omega\left(\frac{1}{\vartheta}\left(\mathbb S(\vartheta,\nabla u)\cdot\nabla u+\left(\frac{\kappa(\vartheta)}{\vartheta}+\frac{\delta}{2}\left(\vartheta^{\Gamma-1}+\frac{1}{\vartheta^2}\right)\right)|\nabla\vartheta|^2+\frac{\delta}{\vartheta^2}\right)\right)\phi
	\end{equation}
	for any $\phi\in C([0,T]\times\overline\Omega)$, $\phi\geq 0$ and any $t\in[0,T]$.
	
	The approximate entropy balance identity holds in the form 
	\begin{equation}\label{DAEBal}
		\begin{split}
			&\int_\Omega \mathcal S_\delta(\mathfrak r,\mathfrak z,\vartheta)\psi(t)-\int_\Omega \mathcal S_\delta(\mathfrak r_{0,\delta},\mathfrak z_{0,\delta},\vartheta_{0,\delta})\psi(0)\\
			&=\int_0^t\int_\Omega \mathcal S_\delta(\mathfrak r,\mathfrak z,\vartheta)\left(\tder\psi+u\cdot\nabla\psi\right)+\int_0^t\int_\Omega\frac{\kappa_\delta(\vartheta)}{\vartheta}\nabla\vartheta\cdot\nabla\psi+\langle\sigma_{\delta,t},\psi\rangle
		\end{split}
	\end{equation}
	for all $\psi\in C^{1}([0,T]\times\overline\Omega)$, where 
	\begin{equation*}
		\langle\sigma_{\delta,t},\psi\rangle=\int_{[0,t]\times\overline\Omega}\psi(\tau,x)\mathrm d\mu_{\sigma_\delta}(\tau,x)
	\end{equation*}
	and $\mu_{\sigma_\delta}$ is the unique regular Borel measure on $[0,T]\times\overline\Omega$ associated to the functional $\sigma_\delta$ by the Riesz representation theorem.
	
	The approximate energy balance holds
	\begin{equation}\label{DAEneBal}
		\begin{split}
			&\int_\Omega\left(\frac{1}{2}\Sigma|u|^2+\mathcal E_\delta(\mathfrak r,\mathfrak z,\vartheta)+h_\delta(\mathfrak r,\mathfrak z)\right)(t)\\
			&=\int_\Omega\left(\frac{1}{2}\frac{|(\Sigma u)_{0,\delta}|^2}{\Sigma_{0,\delta}}+\mathcal E_\delta\left(\mathfrak r_{0,\delta},\mathfrak z_{0,\delta}, \vartheta_{0,\delta}+h_\delta(\mathfrak r_{0,\delta},\mathfrak z_{0,\delta})\right)\right)+\int_0^t\int_\Omega\frac{\delta}{\vartheta^2}\text{ for a.a. }t\in(0,T).
		\end{split}
	\end{equation}
	\subsubsection{Estimates uniform in $\varepsilon$}\label{Sec:EpsEst}
	Subtracting $\Theta$--multiple of \eqref{AEBalMes} from approximate energy balance \eqref{AEB} for $\Sigma_\varepsilon,\mathfrak r_\varepsilon,\mathfrak z_\varepsilon,u_\varepsilon,\vartheta_\varepsilon$ we obtain
	\begin{equation}
		\begin{split}
			&\int_\Omega \left(\frac{1}{2}\Sigma_\varepsilon|u_\varepsilon|^2+H_{\delta,\Theta}(\mathfrak r_\varepsilon,\mathfrak z_\varepsilon,\vartheta_\varepsilon)+h_\delta(\mathfrak r_\varepsilon,\mathfrak z_\varepsilon)\right)(t)+\Theta\tilde\sigma_{\varepsilon,\delta}([0,t]\times\overline\Omega)+\varepsilon\int_0^t\int_\Omega \vartheta_\varepsilon^{\Lambda+1}\\
			&=\int_\Omega\left(\frac{1}{2}\frac{|(\Sigma u)_{0,\delta}|^2}{\Sigma_{0,\delta}}+H_{\delta,\Theta}(\mathfrak r_{0,\delta},\mathfrak z_{0,\delta},\vartheta_{0,\delta})+h_\delta(\mathfrak r_{0,\delta},\mathfrak z_{0,\delta})\right)+\int_0^t\int_\Omega\left(\frac{\delta}{\vartheta^2}+\varepsilon\Theta \vartheta_\varepsilon^\Lambda\right)\text{ for a.a. }t\in[0,T].
		\end{split}
	\end{equation}
	We proceed in a complete analogy first to the proof of \eqref{AuxEN} and then similarly to \eqref{NEst} conclude the following estimates
	\begin{equation}\label{EpsEst}
		\begin{split}
			\|\sqrt{\Sigma_\varepsilon}u_\varepsilon\|_{L^\infty(0,T;L^2(\Omega))}&\leq c,\\
			\|\mathfrak r_\varepsilon\|_{L^\infty(0,T;L^\Gamma(\Omega))}+\|\mathfrak z_\varepsilon\|_{L^\infty(0,T;L^\Gamma(\Omega))}&\leq c,\\
			\|\vartheta_\varepsilon^{-1}\mathbb S(\vartheta_\varepsilon,\nabla u_\varepsilon)\cdot\nabla u_\varepsilon\|_{L^1(Q_T)}&\leq c,\\
			\|\vartheta_\varepsilon\|_{L^\infty(0,T;L^4(\Omega))}+\varepsilon^\frac{1}{\Lambda+1}\|\vartheta_\varepsilon\|_{L^{\Lambda+1}(Q_T)}&\leq c,\\
			\|\vartheta_\varepsilon^{-1}\|_{L^3(0,T;L^3(\Omega))}+\|\nabla\vartheta_\varepsilon^\frac{\Gamma}{2}\|_{L^2(0,T;L^2(\Omega))}+\|\nabla\vartheta_\varepsilon^{-\frac{1}{2}}\|_{L^2(0,T;L^2(\Omega))}&\leq c,\\
			\|\sqrt{\kappa(\vartheta_\varepsilon)}\vartheta_\varepsilon^{-1}\nabla \vartheta_\varepsilon\|_{L^2(0,T;L^2(\Omega))}&\leq c,\\
			\tilde{\sigma}_{\varepsilon,\delta}([0,T]\times\overline\Omega)&\leq c,\\
			\|\mathcal S_\delta(\mathfrak r_\varepsilon,\mathfrak z_\varepsilon,\vartheta_\varepsilon)\|_{L^\infty(0,T;L^1(\Omega))}&\leq c.		
		\end{split}
	\end{equation}
	By \eqref{EPointEqS} we have
	\begin{equation}\label{EpsPointIneq}
		\begin{split}
			\underline a\mathfrak r_\varepsilon\leq \mathfrak z_\varepsilon\leq\overline a\mathfrak r_\varepsilon,\\
			\underline b\Sigma_\varepsilon\leq \mathfrak r_\varepsilon+\mathfrak z_\varepsilon\leq\overline b\Sigma_\varepsilon,\\
			\underline d\mathfrak r_\varepsilon\leq\xi_\varepsilon\leq\overline d\mathfrak r_\varepsilon.
		\end{split}
	\end{equation}
	Moreover, as in Section~\ref{SubSec:NEst} we obtain 
	\begin{equation}\label{FEpsEst}
		\begin{split}
			\|u_\varepsilon\|_{L^2(0,T;W^{1,2}(\Omega))}&\leq c,\\
			\|\vartheta_\varepsilon^\nu\|_{L^2(0,T;W^{1,2}(\Omega))}&\leq c\text{ for any }\nu\in\left[1,\frac{\Gamma}{2}\right],\\
			\|\nabla\log\vartheta_\varepsilon\|_{L^2(0,T;L^2(\Omega))}&\leq c.
		\end{split}
	\end{equation}
	It follows from \eqref{EpsEst}$_8$ that $\overline s_\pm\|\mathfrak r_\varepsilon\log\vartheta_\varepsilon\|_{L^\infty(0,T;L^1(\Omega))}\leq c$. The latter inequality, \eqref{EpsEst}$_2$ and \eqref{FEpsEst}$_3$ combined with the fact that $\int_\Omega\mathfrak r_\varepsilon=\int_\Omega\mathfrak r_{0,\delta}>0$ yields
	\begin{equation}\label{LogEpsEst}
		\|\log\vartheta_\varepsilon\|_{L^2(0,T;W^{1,2}(\Omega))}\leq c.
	\end{equation}	
	Repeating the arguments leading to \eqref{DensNEst} it follows that
	\begin{equation}\label{DensEpsEst}
		\varepsilon^\frac{1}{2}\|\mathfrak r_\varepsilon\|_{L^2(0,T;W^{1,2}(\Omega))}+\varepsilon^\frac{1}{2}\|\mathfrak r_\varepsilon^\frac{\Gamma}{2}\|_{L^2(0,T;W^{1,2}(\Omega))}+\varepsilon^\frac{1}{2}\|\mathfrak z_\varepsilon\|_{L^2(0,T;W^{1,2}(\Omega))}+\varepsilon^\frac{1}{2}\|\mathfrak z_\varepsilon^\frac{\Gamma}{2}\|_{L^2(0,T;W^{1,2}(\Omega))}\leq c.
	\end{equation}
	Using approximate continuity equation \eqref{ApCE} for $\Sigma_\varepsilon$ and $\xi_\varepsilon$, employing \eqref{EpsEst}$_2$ and \eqref{EpsPointIneq}$_{2,3}$ we get
	\begin{equation}\label{GradEpsEst}
		\begin{split}
			\varepsilon^\frac{1}{2}\|\nabla\Sigma_\varepsilon\|_{L^2(Q_T)}\leq c,\\
			\varepsilon^\frac{1}{2}\|\nabla\xi_\varepsilon\|_{L^2(Q_T)}\leq c.
		\end{split}
	\end{equation} 
	It also follows that
	\begin{equation}\label{SDelEpsEst}
		\begin{split}
			\|\mathcal S_\delta(\mathfrak r_\varepsilon,\mathfrak z_\varepsilon,\vartheta_\varepsilon)\|_{L^p(Q_T)}\leq c,\text{ for some }p>2
		\end{split}
	\end{equation}
	by \eqref{SpEntStr}, \eqref{EntGr}, \eqref{EpsEst}$_4$, \eqref{FEpsEst}$_{2}$ and \eqref{LogEpsEst}. Then employing \eqref{FEpsEst}$_1$ we get
	\begin{equation}\label{SDelUEpsEst}
		\|\mathcal S_\delta(\mathfrak r_\varepsilon,\mathfrak z_\varepsilon,\vartheta_\varepsilon)u_\varepsilon\|_{L^q(Q_T)}\leq c\text{ for some }q>1
	\end{equation}
	provided $\Gamma>6$. Next, \eqref{EpsEst}$_{4,5,6}$ imply
	\begin{equation}\label{FracKDelEst}
		\left\|\frac{\kappa_\delta(\vartheta_\varepsilon)}{\vartheta_\varepsilon}\nabla \vartheta_\varepsilon\right\|_{L^q(Q_T)}\leq c\text{ for some }q>1.
	\end{equation} 
	Setting 
	\begin{equation}\label{REpsDef}
		R_\varepsilon=\left(\vartheta_\varepsilon\mathsf s_{+,\delta}(\mathfrak r_\varepsilon,\vartheta_\varepsilon)-\mathsf e_{+,\delta}(\mathfrak r_\varepsilon,\vartheta_\varepsilon)-\frac{\mathsf P_+(\mathfrak r_\varepsilon,\vartheta_\varepsilon)}{\mathfrak r_\varepsilon}\right)\frac{\nabla\mathfrak r_\varepsilon}{\vartheta_\varepsilon}+\left(\vartheta_\varepsilon\mathsf s_{-,\delta}(\mathfrak z_\varepsilon,\vartheta_\varepsilon)-\mathsf e_{-,\delta}(\mathfrak z_\varepsilon,\vartheta_\varepsilon)-\frac{\mathsf P_-(\mathfrak z_\varepsilon,\vartheta_\varepsilon)}{\mathfrak z_\varepsilon}\right)\frac{\nabla\mathfrak z_\varepsilon}{\vartheta_\varepsilon}
	\end{equation}
	we have 
	\begin{equation}\label{REpsEst}
		\varepsilon^\frac{1}{2}\|R_\varepsilon\|_{L^p(Q_T)}\leq c\text{ for some }p>1.
	\end{equation}
	Furthermore, we get employing \eqref{MuGr}, \eqref{EGr} and bounds \eqref{EpsEst}$_3$, \eqref{FEpsEst}$_2$ that
	\begin{equation}\label{STensDelEpsEst}
		\|\mathbb S(\vartheta_\varepsilon,\nabla u_\varepsilon)\|_{L^\frac{2\Gamma}{\Gamma+2} (Q_T)}\leq c.
	\end{equation}
	
	Next, we proceed to a uniform pressure estimate. To this end we consider in \eqref{ABM} a test function of the form $\varphi=\phi\mathcal B\left(\mathfrak r_\varepsilon-\frac{1}{|\Omega|}\int_\Omega\mathfrak r_\varepsilon\right)$, where $\phi\in C^1((0,T))$ and $\mathcal B$ is the Bogovskii (or the inverse of the divergence) operator, see \cite[Section 10.5]{FeNo09} for its precise definition and properties. As $\mathfrak r_\varepsilon$ fulfills the approximate continuity equation we have
	\begin{equation*}
		\tder\varphi=-\mathcal B\left(\dvr \left(\mathfrak r_\varepsilon u_\varepsilon-\varepsilon\nabla\mathfrak r_\varepsilon\right)\right).
	\end{equation*}
	Employing properties of the operator $\mathcal B$ we conclude
	\begin{equation*}
		\begin{split}
			\|\varphi(t,\cdot)\|_{W^{1,\Gamma(\Omega)}}&\leq c\|\mathfrak r_\varepsilon(t,\cdot)\|_{L^\Gamma(\Omega)},\\
			\|\tder\varphi(t,\cdot)\|_{L^2(\Omega)}&\leq c\|(\mathfrak r_\varepsilon u_\varepsilon-\varepsilon\nabla \mathfrak r_\varepsilon)(t,\cdot)\|_{L^2(\Omega)}.
		\end{split}
	\end{equation*}
	We obtain from \eqref{ABM} using \eqref{EpsEst}, \eqref{FEpsEst}, \eqref{GradEpsEst} and \eqref{STensDelEpsEst} the bound
	\begin{equation}\label{EpsTestBog}
		\int_0^T\phi\int_\Omega\mathcal P_\delta(\mathfrak r_\varepsilon,\mathfrak z_\varepsilon,\vartheta_\varepsilon)\mathfrak r_\varepsilon\leq c.
	\end{equation}
	It follows from \eqref{PDDef} using \eqref{EpsPointIneq} and \eqref{EpsTestBog} that 
	\begin{equation}\label{EpsBetDensInteg}
		\|\mathfrak r_\varepsilon\|_{L^{\Gamma+1}(Q_T)}+\|\mathfrak z_\varepsilon\|_{L^{\Gamma+1}(Q_T)}\leq c.
	\end{equation}
	Eventually, taking into account \eqref{PressGrowth} and \eqref{PDDef} we get by  \eqref{FEpsEst}$_2$ and \eqref{EpsBetDensInteg} that
	\begin{equation}\label{EpsPressEst}
		\|\mathcal P_\delta(\mathfrak r_\varepsilon,\mathfrak z_\varepsilon,\vartheta_\varepsilon)\|_{L^q(Q_T)}\leq c\text{ for some }q>1.
	\end{equation}
	\subsubsection{Convergences as $\varepsilon\to 0_+$}
	Proceeding in the exactly same way as in Section~\ref{SubSec:NPass} we deduce from the bounds in \eqref{EpsEst} and \eqref{FEpsEst} using \eqref{EpsPointIneq} that up to a nonrelabeled subsequence
	\begin{equation}\label{EpsWeakly}
		\begin{alignedat}{2}
			(\xi_\varepsilon,\Sigma_\varepsilon,\mathfrak r_\varepsilon,\mathfrak z_\varepsilon)&\rightharpoonup^*(\xi,\Sigma,\mathfrak r,\mathfrak z) &&\text{ in }L^\infty(0,T;L^\Gamma(\Omega)),\\
			u_\varepsilon&\rightharpoonup u&&\text{ in }L^2(0,T;W^{1,2}(\Omega)),\\
			\vartheta_\varepsilon&\rightharpoonup\vartheta&&\text{ in }L^2(0,T;W^{1,2}(\Omega))
		\end{alignedat}
	\end{equation}
	and 
	\begin{equation}\label{MomentuEpsWeakly}
		\begin{alignedat}{2}
			\Sigma_\varepsilon u_\varepsilon&\to \Sigma u&&\text{ in }C_w([0,T];L^\frac{2\Gamma}{\Gamma+1}(\Omega)),\\
			\Sigma_\varepsilon u_\varepsilon\otimes u_\varepsilon&\rightharpoonup \Sigma u\otimes u&&\text{ in }L^1(Q_T),\\
			\Sigma_\varepsilon |u_\varepsilon|^2&\rightharpoonup \Sigma |u|^2&&\text{ in }L^1(Q_T).
		\end{alignedat}
	\end{equation}
	It directly follows from \eqref{EpsEst}$_7$ that
	\begin{equation}\label{MeasEpsConv}
		\tilde\sigma_{\varepsilon,\delta}\rightharpoonup^* \sigma_\delta\text{ in } \left(C([0,T]\times\overline\Omega)\right)^*.
	\end{equation}
	Moreover, as \eqref{EpsWeakly}$_1$ is available and we know that each element of the quartet $(\xi_\varepsilon,\Sigma_\varepsilon,\mathfrak r_\varepsilon,\mathfrak z_\varepsilon)$ fulfills the integral formulation of the approximate continuity equation, we can deduce 
	\begin{equation}\label{EpsCWeakly}
		(\xi_\varepsilon,\Sigma_\varepsilon,\mathfrak r_\varepsilon,\mathfrak z_\varepsilon)\to(\xi,\Sigma,\mathfrak r,\mathfrak z) \text{ in }C_w([0,T];L^\Gamma(\Omega)).		
	\end{equation}
	Furthermore, similarly to the proof of \eqref{MomentuEpsWeakly}$_1$ we get
	\begin{equation}\label{FEpsWeakly}
		\begin{split}
			(\xi_\varepsilon u_\varepsilon,\mathfrak r_\varepsilon u_\varepsilon,\mathfrak z_\varepsilon u_\varepsilon)\rightharpoonup (\xi u,\mathfrak r u,\mathfrak zu)\text{ in }L^2(0,T;L^\frac{2\Gamma}{\Gamma+2}(\Omega))
		\end{split}
	\end{equation}
	using \eqref{EpsEst}$_2$, \eqref{EpsWeakly}$_2$ and \eqref{EpsCWeakly}.
	\subsubsection{Pointwise convergence of $\vartheta_\varepsilon$ and consequences}
	We now apply the Div--Curl lemma similarly to its application in Section~\ref{SubSec:NPass} in order to conclude
	\begin{equation}\label{ThetaEpsPointwisely}
		\vartheta_\varepsilon\to\vartheta\text{ a.e. in }Q_T.
	\end{equation}
	We denote
	\begin{equation*}
		\begin{split}
			U_\varepsilon=&\left(\mathcal S_\delta(\mathfrak r_\varepsilon,\mathfrak z_\varepsilon,\vartheta_\varepsilon),\mathcal S_\delta(\mathfrak r_\varepsilon,\mathfrak z_\varepsilon,\vartheta_\varepsilon)u_\varepsilon+\frac{\kappa_\delta(\vartheta_\varepsilon)}{\vartheta_\varepsilon}\nabla\vartheta_\varepsilon+\varepsilon R_\varepsilon\right),\\
			V_\varepsilon=&\left(T_k(\vartheta_\varepsilon),0,0,0\right)
		\end{split}
	\end{equation*}
	where $R_\varepsilon$ and $T_k$ are defined in \eqref{REpsDef}, \eqref{TKDef} respectively. Then using \eqref{AEBalMes} it follows that
	\begin{equation*}
		\langle\mathrm{Div}\ U_\varepsilon,\psi\rangle=\langle\tilde\sigma_{\varepsilon,\delta},\psi\rangle-\int_0^T\int_\Omega\varepsilon\vartheta_\varepsilon^\Lambda\psi\text{ for any }\psi\in C^\infty_c(Q_T).
	\end{equation*}
	By the compact embeddings $\left(C([0,T]\times\overline\Omega)\right)^*$ to $W^{-1,s}(Q_T)$ and $L^1(Q_T)$ to $W^{-1,s}(Q_T)$ for $s\in [1,\frac{4}{3})$ and the bounds in \eqref{EpsEst}$_7$ we conclude the relative compactness of $\{\mathrm{Div}\ U_\varepsilon\}$ in $W^{-1,s}(Q_T)$.	
	The direct calculation yields
	\begin{equation*}
		\mathrm{Curl}\ V_\varepsilon=T_k'(\vartheta_\varepsilon)\begin{pmatrix}
			0&\nabla\vartheta_\varepsilon\\
			(\nabla\vartheta_\varepsilon)^\top&0
		\end{pmatrix},
	\end{equation*}
	which is obviously bounded in $L^\infty(Q_T)$ and relatively compact in $W^{-1,s}(Q_T)$ subsequently. Next, using \eqref{SDelEpsEst}--\eqref{FracKDelEst} we have that $\{U_\varepsilon\}$ is bounded in $L^p(Q_T)$ for a certain $p>1$. Hence by Lemma~\ref{Lem:DivCurl} we conclude
	\begin{equation*}
		\overline{\mathcal S_\delta(\mathfrak r,\mathfrak z,\vartheta)T_k(\vartheta)}=\overline{\mathcal S_\delta(\mathfrak,\mathfrak z,\vartheta)}\ \overline{T_k(\vartheta)}.
	\end{equation*}
	Next, we proceed as in Section~\ref{SubSec:NPass} but with one difference. Indeed, as the analogue of \eqref{SDelNConv} is not available because of the lack of pointwise convergence of $\{\mathfrak r_\varepsilon\}$ and $\{\mathfrak z_\varepsilon\}$, we can still show
	\begin{equation}\label{SDelProdLim}
		\overline{\mathfrak S_\delta(\mathfrak r,\mathfrak z,\vartheta)T_k(\vartheta)}\geq \overline{\mathfrak S_\delta(\mathfrak r,\mathfrak z,\vartheta)}\ \overline{T_k(\vartheta)},
	\end{equation}
	i.e., the analogue of \eqref{SDelLimIneq}. Knowing already the monotonicity of $x\mapsto\mathfrak S_\delta(r,z,T_k^{-1}(x))$ we have the analogue of \eqref{EntMonIneq}, i.e.,
	\begin{equation}\label{SDelDiffIneq}
		\overline{\mathfrak S_\delta(\mathfrak r,\mathfrak z,\vartheta)T_k(\vartheta)}-\overline{\mathfrak S_\delta(\mathfrak r,\mathfrak z,\vartheta)}\ \overline {T_k(\vartheta)}\geq \overline{\mathfrak S_\delta(\mathfrak r,\mathfrak z,T_k^{-1}(\overline {T_k(\vartheta)})T_k(\vartheta)}-\overline{\mathfrak S_\delta(\mathfrak r,\mathfrak z,T_k^{-1}(\overline {T_k(\vartheta)})}\ \overline{T_k(\vartheta)}.
	\end{equation}
	To conclude \eqref{SDelProdLim} from \eqref{SDelDiffIneq} we apply Lemma~\ref{Lem:LimProd} and get
	\begin{equation*}
		\mathfrak S_\delta(\mathfrak r_\varepsilon,\mathfrak z_\varepsilon,T_k^{-1}(\overline{T_k(\vartheta)}))(T_k(\vartheta_\varepsilon)-\overline{T_k(\vartheta)})\rightharpoonup 0\text{ in }L^1(Q_T).
	\end{equation*}
	Hence it remains to show that 
	\begin{equation}\label{LimPodBG}
		\overline{b(\mathfrak r,\mathfrak z)G(\vartheta)}=\overline{b(\mathfrak r,\mathfrak z)}\ \overline{G(\vartheta)}
	\end{equation}
	for any $b\in C^1(\eR^2)$ with $\nabla b\in C_c(\eR^2)$ and any Lipschitz $G\in C^1(\eR)$. To this end we fix such a $b$ and find a sequence $\{b_n\}\subset C^2(\eR^2)$ such that 
	\begin{equation}\label{BNConv}
		b_n\to b\text{ in }C^1(\eR^2).
	\end{equation} 
	We fix $n\in\eN$ and multiply the approximate continuity equation \eqref{ApCE} for $\mathfrak r_\varepsilon$ on $\partial_rb_n(\mathfrak r_\varepsilon,\mathfrak z_\varepsilon)\phi$ and the approximate continuity equation \eqref{ApCE} for $\mathfrak z_\varepsilon$ on $\partial_zb_n(\mathfrak r_\varepsilon,\mathfrak z_\varepsilon)\phi$ with $\phi\in C^\infty_c(\Omega)$, sum both identities and integrate the sum over $\Omega$ to arrive at
	\begin{equation}
		\begin{split}
			&\frac{\mathrm{d}}{\mathrm d t}\int_\Omega b_n(\mathfrak r_\varepsilon,\mathfrak z_\varepsilon)\phi-\int_\Omega b_n(\mathfrak r_\varepsilon,\mathfrak z_\varepsilon)u_\varepsilon\cdot\nabla\phi+\varepsilon\int_\Omega\left(\partial^2_{rr}b_n(\mathfrak r,\mathfrak z)|\nabla \mathfrak r_\varepsilon|^2+2\partial^2_{rz}b_n(\mathfrak r_\varepsilon,\mathfrak z_\varepsilon)\nabla \mathfrak r_\varepsilon\cdot\mathfrak z_\varepsilon\partial^2_{zz}b_n(\mathfrak r,\mathfrak z)|\nabla \mathfrak z_\varepsilon|^2\right)\phi\\
			&+\varepsilon\int_\Omega \left(\partial_r b_n(\mathfrak r_\varepsilon,\mathfrak z_\varepsilon)\nabla \mathfrak r_\varepsilon+\partial_z b_n(\mathfrak z_\varepsilon,\mathfrak z_\varepsilon)\nabla \mathfrak z_\varepsilon\right)\cdot\nabla\phi+\int_\Omega\left(\mathfrak r_\varepsilon\partial_rb_n(\mathfrak r_\varepsilon,\mathfrak z_\varepsilon)+\mathfrak z_\varepsilon\partial_zb_n(\mathfrak r_\varepsilon,\mathfrak z_\varepsilon)-b_n(\mathfrak r_\varepsilon,\mathfrak z_\varepsilon)\right)\dvr u_\varepsilon\phi=0.
		\end{split}
	\end{equation}
	As a consequence of the latter identity we get the uniform boundedness and equicontinuity of $\{t\mapsto\int_\Omega b_n(\mathfrak r_\varepsilon,\mathfrak z_\varepsilon)\phi\}$ in $C([0,T])$. Hence by the Arzel\`a--Ascoli theorem there is a nonrelabeled subsequence such that
	\begin{equation*}
		b_n(\mathfrak r_\varepsilon,\mathfrak z_\varepsilon)\to\overline{b_n(\mathfrak r,\mathfrak z)}\text{ in }C_w([0,T];L^\Gamma(\Omega)).
	\end{equation*}
	Accordingly, we get 
	\begin{equation*}
		b_n(\mathfrak r_\varepsilon,\mathfrak z_\varepsilon)\to\overline{b_n(\mathfrak r,\mathfrak z)}\text{ in }C_w([0,T]; W^{-1,2}(\Omega)).
	\end{equation*}
	The regularity of $G$ and the bound in \eqref{FEpsEst}$_2$ imply
	\begin{equation*}
		G(\vartheta_\varepsilon)\to\overline{G(\vartheta)}\text{ in }L^2([0,T]; W^{1,2}(\Omega)).
	\end{equation*}
	Combining the latter two convergences yields
	\begin{equation}\label{AuxLimIdnt}
		\overline{b_n(\mathfrak r,\mathfrak z)G(\vartheta)}=\overline{b_n(\mathfrak r,\mathfrak z)}\ \overline{G(\vartheta)}.
	\end{equation}
	Finally, as $\{b(\mathfrak r_\varepsilon,z_\varepsilon)\}$ and $\{b(\mathfrak r_\varepsilon,z_\varepsilon)G(\vartheta_\varepsilon)\}$ are obviously bounded in $L^1(Q_T)$, the weak lower semicontinuity of $L^1$--norm and \eqref{BNConv} imply
	\begin{equation*}
		\begin{split}
			\|\overline{b_n(\mathfrak r,\mathfrak z)G(\vartheta)}-\overline{b(\mathfrak r,\mathfrak z)G(\vartheta)}\|_{L^1(Q_T)}\leq\|b_n-b\|_{C^1(\eR^2)}\liminf_{\varepsilon\to 0_+}\|G(\vartheta_\varepsilon)\|_{L^1(Q_T)}\to 0,\\
			\|(\overline{b_n(\mathfrak r,\mathfrak z)}\ -\overline{b(\mathfrak r,\mathfrak z)}) \overline{G(\vartheta)}\|_{L^1(Q_T)}\leq\|b_n-b\|_{C^1(\eR^2)}\|\overline{G(\vartheta)}\|_{L^1(Q_T)}\to 0.
		\end{split}
	\end{equation*}
	The latter convergences allow for the limit passage $n\to\infty$ in \eqref{AuxLimIdnt} to conclude \eqref{LimPodBG}. Having \eqref{LimPodBG} at hand we can proceed as in Section~\ref{SubSec:NPass} to end up with \eqref{ThetaEpsPointwisely}.
	We also get the positivity of the limit temperature $\vartheta$ a.e. in $Q_T$ following from
	\begin{equation*}
		\vartheta^{-3}\in L^1(Q_T)
	\end{equation*}
	as in Section~\ref{SubSec:NPass}. 
	\subsubsection{Derivation of the effective viscous flux identity}\label{Sec:EpsEffVicFl}
	We make a list of convergences of several sequences that will be used later.
	Employing \eqref{EpsWeakly}$_2$, \eqref{ThetaEpsPointwisely}, \eqref{MuGr}, \eqref{EGr},  bounds \eqref{FEpsEst}$_2$ and \eqref{STensDelEpsEst} we conclude
	\begin{equation}\label{SDelEpsWeak}
		\mathbb S(\vartheta_\varepsilon,\nabla u_\varepsilon)\rightharpoonup\mathbb S(\vartheta,\nabla u)\text{ in }L^1(Q_T).
	\end{equation}
	We know that
	\begin{equation}\label{Theta4EpsS}
		\vartheta_\varepsilon^4\to \vartheta^4\text{ in }L^1(Q_T)
	\end{equation}
	due to \eqref{FEpsEst}$_2$ and \eqref{ThetaEpsPointwisely}. Therefore denoting
	\begin{equation}\label{SmallPDef}
		p_\delta(r,\zeta,\theta)=
		\mathsf P_+(r,\theta)+\mathsf P_-(\zeta r,\theta)+\delta\left(r^\Gamma+(\zeta r)^\Gamma+r^2+(\zeta r)^2\right)
	\end{equation}
	we have for $\zeta_\varepsilon=\frac{\mathfrak z_\varepsilon}{\mathfrak r_\varepsilon}$ and $\zeta=\frac{\mathfrak z}{\mathfrak r}$
	\begin{equation}\label{SmPDef}
		\begin{split}
			p_\delta(\mathfrak r_\varepsilon, \zeta_\varepsilon,\vartheta_\varepsilon)=& p_\delta(\mathfrak r_\varepsilon, \zeta,\vartheta)+\mathsf P_+(\mathfrak r_\varepsilon,\vartheta_\varepsilon)+\mathsf P_-(\zeta_\varepsilon\mathfrak r_\varepsilon,\vartheta_\varepsilon)-\mathsf P_+(\mathfrak r_\varepsilon,\vartheta)-\mathsf P_-(\zeta\mathfrak r_\varepsilon,\vartheta)\\
			&+\delta(\mathfrak r_\varepsilon^\Gamma(\zeta_\varepsilon^\Gamma-\zeta^\Gamma)+\mathfrak r_\varepsilon^2(\zeta_\varepsilon^2-\zeta^2))\\=&p_\delta(\mathfrak r_\varepsilon,\zeta,\vartheta)+D_\varepsilon.
		\end{split}
	\end{equation}
	By means of the first order Taylor expansion and \eqref{PressVarGr} we get
	\begin{equation*}
		\begin{split}
			|D_\varepsilon|\leq& 
			c(1+\mathfrak r_\varepsilon^{\tilde\Gamma^P})(1+\vartheta_\varepsilon^{\tilde\omega^P_+}+\vartheta^{\tilde\omega^P_+}+\vartheta_\varepsilon^{\tilde\omega^P_-}+\vartheta^{\tilde\omega^P_-})|\vartheta_\varepsilon-\vartheta|\\
			&+c(1+\vartheta_\varepsilon^{\overline\omega^P}+\vartheta^{\overline\omega^P})(\mathfrak r_\varepsilon^{\underline\Gamma^P}+\mathfrak r_\varepsilon^{\overline\Gamma^P})|\zeta_\varepsilon-\zeta|+c(\mathfrak r^\Gamma_\varepsilon+\mathfrak r_\varepsilon^2)|\zeta_\varepsilon-\zeta|.
		\end{split}
	\end{equation*}
	Hence using bounds \eqref{FEpsEst}$_2$, \eqref{EpsBetDensInteg}, convergence \eqref{ThetaEpsPointwisely} and item (ii) of Lemma~\ref{Lem:AlmComp} we get
	\begin{equation*}
		\lim_{\varepsilon\to 0_+}\|D_\varepsilon\|_{L^1(Q_T)}=0.
	\end{equation*}
	Taking into account \eqref{EpsPressEst} and \eqref{Theta4EpsS} we deduce the existence of a weak limit of $\{p_\delta(\mathfrak r_\varepsilon,\zeta_\varepsilon,\vartheta_\varepsilon)\}$ in $L^1(Q_T)$. Going back to \eqref{SmPDef} we then have
	\begin{equation*}
		\overline{p_\delta(\mathfrak r,\zeta,\vartheta)}=\overline{\overline{p_\delta(\mathfrak r,\zeta,\vartheta)}},
	\end{equation*}
	where $\overline{\overline{p_\delta(\mathfrak r,\zeta,\vartheta)}}$ denotes a weak limit of the sequence $\{p_\delta(\mathfrak r_\varepsilon,\zeta,\vartheta)\}$ in $L^1(Q_T)$.
	We are now ready for the limit passage $\varepsilon\to 0_+$ in approximate balance of momentum \eqref{ABM}. Employing convergences \eqref{MomentuEpsWeakly}$_{1,2}$, \eqref{SDelEpsWeak}, \eqref{Theta4EpsS} and bounds \eqref{FEpsEst}$_1$, \eqref{GradEpsEst}$_1$ we arrive at
	\begin{equation}\label{MomEEpsLim}
		\begin{split}
			&\int_\Omega \Sigma u\cdot\varphi(t)-\int_\Omega (\Sigma u)_{0,\delta}\cdot\varphi(0)\\&=\int_0^t\int_\Omega \left(\Sigma u\cdot\tder\varphi+\Sigma(u\otimes u)\cdot\nabla\varphi-\mathbb S(\vartheta,\nabla u)\cdot\nabla\varphi+\left(\frac{b}{3}\vartheta^4+\overline{\overline{p_\delta(\mathfrak r,\zeta,\vartheta)}}\right)\dvr\varphi\right)
		\end{split}
	\end{equation}
	for all $t\in[0,T]$, $\varphi\in C^1([0,T]\times\overline \Omega)$ such that either $\varphi=0$ or $\varphi\cdot n=0$ on $(0,T)\times\partial\Omega$.
	Following the nowadays standard procedure of the proof of strong convergence of density approximations in the mono-fluid theory we derive so called effective viscous flux identity. As there are no essential differences compared with the mono-fluid case, cf. \cite[Section 3.6.5]{FeNo09}, we focus just on the most substantial steps of the proof. 
	We take as a test function in \eqref{ABM} 
	\begin{equation*}
		\varphi(t,x)=\psi(t)\phi(x)\mathfrak U\left(1_\Omega\mathfrak r_\varepsilon\right),\text{ where }\psi\in C^1_c((0,T)), \phi\in C^1_c(\Omega)
	\end{equation*}
	and the pseudodifferential operator $\mathfrak U$ is defined in \eqref{RUOpDef}. We note that $1_\Omega\mathfrak r_\varepsilon$ fulfills the extended approximated continuity equation on $(0,T)\times\eR^3$ and the extension is possible since $\mathfrak r_\varepsilon u_\varepsilon$ and $\nabla\mathfrak r_\varepsilon$ possess zero normal trace. This fact, properties of operators $\mathfrak U$ and $\mathfrak R$ defined in \eqref{RUOpDef} and straightforward manipulations result in
	\begin{equation}\label{MomEEpsUTest}
		\begin{split}
			&\int_0^T\psi\int_\Omega\phi p_\delta(\mathfrak r_\varepsilon,\zeta_\varepsilon,\vartheta_\varepsilon)1_\Omega\mathfrak r_\varepsilon-\int_0^T\psi\int_\Omega\phi\mathbb S(\vartheta_\varepsilon,\nabla u_\varepsilon)\cdot\mathfrak R(1_\Omega\mathfrak r_\varepsilon)\\
			&=\int_0^T\psi\int_\Omega (u_\varepsilon)_k\cdot\left(1_\Omega\mathfrak r_\varepsilon\mathfrak R_{jk}\left(\phi(\Sigma_\varepsilon u_\varepsilon)_j\right)-\phi(\Sigma_\varepsilon u_\varepsilon)_j\mathfrak R_{jk}(1_\Omega\mathfrak r_\varepsilon)\right)+\mathfrak J(\psi',\phi,\mathfrak r_\varepsilon,\mathfrak z_\varepsilon, \Sigma_\varepsilon,u_\varepsilon,\vartheta_\varepsilon).
		\end{split}
	\end{equation}
	Repeating the procedure for 
	\begin{equation*}
		\varphi=\psi\phi\mathfrak U(1_\Omega\mathfrak r)
	\end{equation*}
	in \eqref{MomEEpsLim} we arrive at
	\begin{equation}\label{MomEEpsLimUTest}
		\begin{split}
			&\int_0^T\psi\int_\Omega\phi \overline{\overline{p_\delta(\mathfrak r,\zeta,\vartheta)}}1_\Omega\mathfrak r-\int_0^T\psi\int_\Omega\phi\mathbb S(\vartheta,\nabla u)\mathfrak R(1_\Omega\mathfrak r)\\
			&=\int_0^T\psi\int_\Omega u_k\cdot\left(1_\Omega\mathfrak r\mathfrak R_{jk}\left(\phi(\Sigma u)_j\right)-\phi(\Sigma u)_j\mathfrak R_{jk}(1_\Omega\mathfrak r)\right)+\mathfrak J(\psi',\phi,\mathfrak r,\mathfrak z, \Sigma,u_\varepsilon,\vartheta).
		\end{split}
	\end{equation}
	We pass to the limit $\varepsilon\to 0_+$ using classical compactness arguments including the fact that $\mathfrak U:L^p(\Omega)\to W^{1,p}(\Omega)$ is a continuous linear operator for $p\in (1,\infty)$ to get
	\begin{equation*}
		\lim_{\varepsilon\to 0_+}\mathfrak J(\psi',\phi, \mathfrak r_\varepsilon,\mathfrak z_\varepsilon, \Sigma_\varepsilon, u_\varepsilon,\vartheta_\varepsilon)=\mathfrak J(\psi',\phi, \mathfrak r,\mathfrak z, \Sigma, u,\vartheta).
	\end{equation*}
	Hence subtracting \eqref{MomEEpsLimUTest} from \eqref{MomEEpsUTest} we get after the limit passage $\varepsilon\to 0_+$
	\begin{equation}\label{EpsLimIdent}
		\begin{split}		&\int_0^T\psi\int_\Omega\phi\left(\overline{\overline{p_\delta(\mathfrak r,\zeta,\vartheta)1_\Omega\mathfrak r}}-\overline{\overline{p_\delta(\mathfrak r,\zeta,\vartheta)}}\ 1_\Omega\mathfrak r\right)=\int_0^T\psi\int_\Omega\phi\left(\overline{\mathbb S(\vartheta,\nabla u)\cdot\mathfrak R(1_\Omega\mathfrak r)}-\mathbb S(\vartheta,\nabla u)\cdot\mathfrak R(1_\Omega\mathfrak r)\right)\\
			&+\lim_{\varepsilon\to 0_+}\int_0^T\psi\int_\Omega\left((u_\varepsilon)_k \cdot\left(1_\Omega\mathfrak r_\varepsilon\mathfrak R_{jk}(\phi(\Sigma_\varepsilon u_\varepsilon)_j)-\phi(\Sigma_\varepsilon u_\varepsilon)_j\mathfrak R_{jk}(1_\Omega\mathfrak r_\varepsilon)\right)-u_k\cdot\left(1_\Omega\mathfrak r\mathfrak R_{jk}(\phi(\Sigma u)_j)\right.\right.\\&\quad\quad\quad\left.\left.-\phi(\Sigma u)_j\mathfrak R_{jk}(1_\Omega\mathfrak r)\right)\right).
		\end{split}
	\end{equation}
	Our next task is to verify that the latter limit vanishes. To this end we apply Lemma~\ref{DCRiesz} along with \eqref{MomentuEpsWeakly} and \eqref{EpsCWeakly}, the continuity of $\mathfrak R$ on $L^p(\Omega)$ for any $p\in(1,\infty)$ to obtain
	\begin{equation}\label{EpsComConv}
		1_\Omega\mathfrak r_\varepsilon \mathfrak R_{jk}(\phi(\Sigma_\varepsilon u_\varepsilon)_j)-\phi(\Sigma_\varepsilon u_\varepsilon)_j\mathfrak R_{jk}(1_\Omega\mathfrak r_\varepsilon)\rightharpoonup 1_\Omega\mathfrak r\mathfrak R_{jk}(\phi(\Sigma u)_j)-\phi(\Sigma u)_j\mathfrak R_{jk}(1_\Omega\mathfrak r)
	\end{equation}
	in $L^\frac{2\Gamma}{\Gamma+3}(\Omega)$	for arbitrary but fixed $t\in[0,T]$. Moreover, by virtue of the compact embedding $L^\frac{2\Gamma}{\Gamma+3}(\Omega)$ into $W^{-1,2}(\Omega)$ for $\Gamma>\frac{9}{2}$ it follows that convergence \eqref{EpsComConv} is strong in $W^{-1,2}(\Omega)$ for any $t\in[0,T]$ and using the Vitali convergence theorem we get that convergence \eqref{EpsComConv} is in fact strong in $L^p(0,T;W^{-1,2}(\Omega))$ for any $p\in[1,\infty)$. Combining the latter convergence and \eqref{EpsWeakly}$_2$ we deduce from \eqref{EpsLimIdent} that 
	\begin{equation}\label{AlmEffViscFlEps}		\int_0^T\psi\int_\Omega\phi\left(\overline{\overline{p_\delta(\mathfrak r,\zeta,\vartheta)1_\Omega\mathfrak r}}-\overline{\overline{p_\delta(\mathfrak r,\zeta,\vartheta)}}\ 1_\Omega\mathfrak r\right)=\int_0^T\psi\int_\Omega\phi\left(\overline{\mathbb S(\vartheta,\nabla u)\cdot\mathfrak R(1_\Omega\mathfrak r)}-\mathbb S(\vartheta,\nabla u)\cdot\mathfrak R(1_\Omega\mathfrak r)\right).
	\end{equation}
	We focus our attention on the terms on the right hand side of the latter equality. Employing the identity 
	\begin{equation}\label{RSymGrIdent}
		\sum_{i,j=1}^3\mathfrak R_{ij}(\partial_i u_j+\partial_j u_i)=2\dvr u
	\end{equation}
	and the fact that $\mathfrak R_{ij}$ commutes in $L^p$--$L^{p'}$ duality, we obtain
	\begin{equation}\label{EpsLimShortIdent}
		\int_0^T\psi\int_\Omega\phi\overline{\mathbb S(\vartheta,\nabla u)\cdot\mathfrak R(1_\Omega\mathfrak r)}=\lim_{\varepsilon\to 0_+}\int_0^T\psi\int_\Omega\phi\left(\frac{4}{3}\mu(\vartheta_\varepsilon)+\eta(\vartheta_\varepsilon)\right)\mathfrak r_\varepsilon\dvr u_\varepsilon+\lim_{\varepsilon\to 0_+}\int_0^T\psi\int_\Omega \mathfrak r_\varepsilon\omega(\vartheta_\varepsilon,u_\varepsilon)
	\end{equation}
	and
	\begin{equation}\label{IdentLim}
		\int_0^T\psi\int_\Omega\phi\mathbb S(\vartheta,\nabla u)\cdot\mathfrak R(1_\Omega\mathfrak r)=\int_0^T\psi\int_\Omega\phi\left(\frac{4}{3}\mu(\vartheta)+\eta(\vartheta)\right)\mathfrak r\dvr u+\int_0^T\psi\int_\Omega \mathfrak r\omega(\vartheta,u),
	\end{equation}
	where
	\begin{equation}\label{omDef}
		\omega(\vartheta, u)=\sum_{i,j=1}^3\left(\mathfrak R_{ij}\left(\phi\mu(\vartheta)\left(\partial_i u_j+\partial_j u_i\right)\right)-\phi\mu(\vartheta)\mathfrak R_{i,j}\left(\partial_i u_j+\partial_j u_i\right)\right).
	\end{equation}
	The task now is to show that the second terms on the right hand side of \eqref{EpsLimShortIdent} and \eqref{IdentLim} are equal. To this end we first apply Lemma~\ref{Lem:ComEst} to infer
	\begin{equation}\label{OmEpsLWBound}
		\|\omega(\vartheta_\varepsilon, u_\varepsilon)\|_{L^1(0,T;W^{k,p}(\Omega))}\leq c\text{ for some }k\in(0,1), p>1 
	\end{equation}
	by virtue of bounds \eqref{FEpsEst}$_{1,2}$ which also guarantee that
	\begin{equation}\label{OmEpsLBound}
		\|\omega(\vartheta_\varepsilon,u_\varepsilon)\|_{L^q(Q_T)}\leq c\text{ for some }q>1.
	\end{equation}
	Considering sequences $\{U_\varepsilon\}$, $\{V_\varepsilon\}$ of vector fields
	\begin{equation*}
		U_\varepsilon=(\mathfrak r_\varepsilon,\mathfrak r_\varepsilon u_\varepsilon),\ V_\varepsilon=(\omega(\vartheta_\varepsilon,u_\varepsilon),0,0,0)
	\end{equation*}
	we have that $\{V_\varepsilon\}$ is bounded in $L^q(Q_T)$ $q>1$ due to \eqref{OmEpsLBound} and $\{\mathrm{Curl}\ V_\varepsilon\}$ is compact in $W^{-1,\tilde q}(Q_T)$ for some $\tilde q>1$ thanks to \eqref{OmEpsLWBound}. Obviously, $\{U_\varepsilon\}$ is bounded in $L^s(Q_T)$ for some $s>1$ due to \eqref{EpsEst}$_2$ and \eqref{FEpsEst}$_1$. We note that the choice of $s$ and $q$ is such that $\frac{1}{s}+\frac{1}{q}<1$. Moreover, $\{\mathrm{Div}\ U_\varepsilon\}$ is compact in $W^{-1,\tilde q}(Q_T)$ as a consequence of identity \eqref{ApCE} for $\mathfrak r_\varepsilon$ and $u_\varepsilon$, convergences \eqref{EpsCWeakly} and \eqref{FEpsWeakly} and bound \eqref{DensEpsEst}. Therefore by Lemma~\ref{Lem:DivCurl} we conclude
	\begin{equation}\label{OmREpscon}
		\omega(\vartheta_\varepsilon, u_\varepsilon)\mathfrak r_\varepsilon\rightharpoonup \omega(\vartheta, u)\mathfrak r\text{ in }L^1(Q_T)
	\end{equation}
	since in accordance with convergences \eqref{EpsWeakly}$_2$, \eqref{ThetaEpsPointwisely} respectively, and estimates \eqref{MuGr}, \eqref{EpsEst}$_4$ respectively, it follows that $\overline{\omega(\vartheta, u)}=\omega(\vartheta, u)$.
	Eventually, going back to \eqref{AlmEffViscFlEps}, where we employ \eqref{EpsLimShortIdent}, \eqref{IdentLim} and \eqref{OmREpscon}, we arrive at the desired effective viscous flux identity
	\begin{equation}\label{EffViscFlEps}
		\overline{\overline{p_\delta(\mathfrak r,\zeta,\vartheta)\mathfrak r}}-\overline{\overline{p_\delta(\mathfrak r,\zeta,\vartheta)}}\mathfrak r=\left(\frac{4}{3}\mu(\vartheta)+\eta(\vartheta)\right)\left(\overline{\mathfrak r\dvr u}-\mathfrak r\dvr u\right).
	\end{equation}
	\subsubsection{Pointwise convergence of $\mathfrak r_\varepsilon$ and consequences}\label{Sec:EpsThConv}
	The goal of this section is the justification of 
	\begin{equation}\label{REpsPoint}
		\mathfrak r_\varepsilon\to\mathfrak r\text{ a.e. in }Q_T.
	\end{equation}
	We begin with multiplacation of \eqref{ApCE} for $\mathfrak r_\varepsilon$ and $u_\varepsilon$ on $G'(\mathfrak r_\varepsilon)$, where $G$ is smooth and convex on $[0,\infty)$. We integrate the resulting identity over $Q_t$ and by parts and let $\varepsilon\to 0_+$ to obtain
	\begin{equation*}
		\int_\Omega \overline{G(\mathfrak r)}(t)+\int_0^t\int_\Omega\overline{\left(G'(\mathfrak r)\mathfrak r-G(\mathfrak r)\dvr u\right)}\leq \int_\Omega G(\mathfrak r_{0,\delta})\text{ for a.a. }t\in(0,T).
	\end{equation*}
	The special choice $G(r)=r\log r$ yields
	\begin{equation}\label{EpsLimRen}
		\int_\Omega \overline{\mathfrak r\log\mathfrak r}(t)+\int_0^t\int_\Omega\overline{\mathfrak r\dvr u}\leq \int_\Omega \mathfrak r_{0,\delta}\log r_{0,\delta}\text{ for a.a. }t\in(0,T).
	\end{equation}
	On the other hand, we know that thanks to the regularity of $\mathfrak r$ and $u$ this pair fulfills the renormalized continuity equation. In particular, we have
	\begin{equation}\label{RenLimits}
		\int_\Omega \mathfrak r\log \mathfrak r(t)+\int_0^t\int_\Omega \mathfrak r\dvr u=\int_\Omega \mathfrak r_{0,\delta}\log \mathfrak r_{0,\delta}\text{ for a.a. }t\in(0,T).
	\end{equation}
	The difference of \eqref{EpsLimRen} and \eqref{RenLimits} combined with \eqref{EffViscFlEps} yields
	\begin{equation*}
		\begin{split}
			\int_\Omega \left(\overline{\mathfrak r\log\mathfrak r}-\mathfrak r\log\mathfrak r\right)(t)&\leq\int_0^t\int_\Omega\left(\mathfrak r\dvr u-\overline{\mathfrak r\dvr u}\right)\\
			&=\int_0^t\int_\Omega\frac{1}{\frac{4}{3}\mu(\vartheta)+\eta(\vartheta)}\left(\overline{\overline{p_\delta(\mathfrak r,\zeta,\vartheta)}}\mathfrak r-\overline{\overline{p_\delta(\mathfrak r,\zeta,\vartheta)\mathfrak r}}\right)\leq 0\text{ for a.a. }t\in(0,T)
		\end{split}
	\end{equation*}
	by Lemma~\ref{Lem:MonWConv} as $r\mapsto p(r,\cdot,\cdot)$ is nondecreasing by assumption \eqref{ThStab} and the monotonicity of a power function. Taking into consideration the strict convexity of $s\mapsto s\log s$ on $[0,\infty)$ we get \eqref{REpsPoint}. Using \eqref{AlmConv} we conclude
	\begin{equation}\label{ZEpsPoint}
		\mathfrak z_\varepsilon\to\mathfrak z\text{ a.e. in }Q_T.
	\end{equation}
	Using bound \eqref{EpsBetDensInteg}, convergences \eqref{REpsPoint}, \eqref{ZEpsPoint} respectively, and the Vitali convergence theorem we conclude
	\begin{equation}\label{EpsDensStrongly}
		(\mathfrak r_\varepsilon,\mathfrak z_\varepsilon)\to (\mathfrak r,\mathfrak z)\text{ in }L^p(Q_T)\text{ for }p\in[1,\Gamma+1).
	\end{equation}
	Having convergences \eqref{REpsPoint}, \eqref{ZEpsPoint} respectively, knowing the growth \eqref{PressGrowth} and bound \eqref{EpsBetDensInteg} we apply the Vitali convergence theorem to obtain
	\begin{equation}\label{EpsPressStrongly}
		\overline{\overline{p_\delta(\mathfrak r,\zeta,\vartheta)}}=p_\delta(\mathfrak r,\zeta,\vartheta)
	\end{equation}
	from \eqref{SmallPDef}. Plugging this identity in \eqref{MomEEpsLim} we conclude \eqref{DMomBal}.
	\subsubsection{Passage $\varepsilon\to 0_+$ in the energy equality}
	Using pointwise convergences \eqref{ThetaEpsPointwisely}, \eqref{REpsPoint}, \eqref{ZEpsPoint}, growth estimate \eqref{IntEnGrowth}, bounds \eqref{FEpsEst}$_2$ and \eqref{EpsBetDensInteg} we conclude
	\begin{equation}\label{EpsIEConv}
		\mathcal E_\delta(\mathfrak r_\varepsilon,\mathfrak z_\varepsilon,\vartheta_\varepsilon)\to\mathcal E_\delta(\mathfrak r,\mathfrak z,\vartheta)\text{ in }L^p(Q_T)\text{ for some }p>1.
	\end{equation}
	Moreover, we deduce $\vartheta^{-3}\in L^1(Q_T)$ by \eqref{EpsEst}$_5$, implying $\vartheta\geq 0$ a.e. in $Q_T$, and
	\begin{equation} \label{NEpsPowerConv} 
		\vartheta_\varepsilon^{-2}\to \vartheta^{-2}\text{ in }L^1(Q_T).
	\end{equation}
	Employing convergences \eqref{MomentuEpsWeakly}$_3$, \eqref{EpsIEConv}, \eqref{NEpsPowerConv}, bound \eqref{EpsEst}$_4$ and the strong convergence of densities in \eqref{EpsDensStrongly} we pass to the limit $\varepsilon\to 0_+$ in \eqref{AEB} and obtain \eqref{DAEneBal}.
	\subsubsection{Limit passage $\varepsilon\to 0_+$ in the continuity equations} Using convergences \eqref{MomentuEpsWeakly}$_1$, \eqref{EpsCWeakly} and \eqref{FEpsWeakly} and bounds \eqref{DensEpsEst}, \eqref{GradEpsEst} respectively, in \eqref{ApCE} for $r\in\{\xi_\varepsilon,\Sigma_\varepsilon,\mathfrak r_\varepsilon,\mathfrak z_\varepsilon\}$ we obtain that the limit functions $(\mathfrak r,\mathfrak z,\Sigma,\xi)$ satisfy the continuity equations in the form specified in \eqref{DRCE}.
	\subsubsection{Passage $\varepsilon\to 0_+$ in the entropy balance}
	Using the same convergences as above combined with the bound \eqref{SDelEpsEst} we get
	\begin{equation}\label{SDelEpsConv}
		\mathcal S_\delta(\mathfrak r_\varepsilon,\mathfrak z_\varepsilon,\vartheta_\varepsilon)\to \mathcal S_\delta(\mathfrak r,\mathfrak z,\vartheta)\text{ in }L^p(Q_T)\text{ for some }p>2,
	\end{equation}
	which combined with \eqref{EpsWeakly}$_2$ yields
	\begin{equation}\label{SDelUEpsConv}
		\mathcal S_\delta(\mathfrak r_\varepsilon,\mathfrak z_\varepsilon,\vartheta_\varepsilon)u_\varepsilon\rightharpoonup \mathcal S_\delta(\mathfrak r,\mathfrak z,\vartheta)u\text{ in }L^q(Q_T)\text{ for some }q>1.
	\end{equation}
	As in Section~\ref{SubSec:NPass} we use \eqref{EpsEst}$_{4,5}$, convergences \eqref{EpsWeakly}$_3$ and \eqref{ThetaEpsPointwisely} to conclude
	\begin{equation}\label{KDNThEpsConv}
		\frac{\kappa_\delta(\vartheta_\varepsilon)}{\vartheta_\varepsilon}\nabla\vartheta_\varepsilon\rightharpoonup \frac{\kappa_\delta(\vartheta)}{\vartheta}\nabla\vartheta\text{ in }L^1(Q_T).
	\end{equation}
	
	Eventually, we employ convergences \eqref{MeasEpsConv}, \eqref{SDelEpsConv}, \eqref{SDelUEpsConv}, \eqref{KDNThEpsConv}, bounds \eqref{EpsEst}$_4$ and \eqref{REpsEst} to pass to the limit $\varepsilon\to 0_+$ in \eqref{AEBalMes} to conclude \eqref{DAEBal}. We note that 
	\begin{equation*}
		(\vartheta,u)\mapsto \int_0^T\int_\Omega \frac{1}{\vartheta_\varepsilon}\mathbb S(\vartheta_\varepsilon,\nabla u_\varepsilon)\cdot\nabla u_\varepsilon,\ \vartheta\mapsto\int_0^T\int_\Omega\frac{\kappa_\delta(\vartheta)}{\vartheta}|\nabla \vartheta|^2
	\end{equation*}
	are weak lower semicontinuous as was observed in Section~\ref{SubSec:NPass} and remaining terms in $\sigma_{\varepsilon,\delta}$ are nonnegative, cf. \eqref{SDEDef}$_1$, which explains how one passes from \eqref{SigmEpsDom} to \eqref{SigmDDef}.
	\subsection{Limit passage $\delta\to 0_+$}
	The last step in the existence proof is the limit passage in equations \eqref{DRCE}--\eqref{DAEneBal}. In comparison with the limit passage $\varepsilon\to 0_+$ there appear only few differences and many arguments are adapted almost unchanged. In particular, the control of density oscillations beyond the DiPerna--Lions theory is necessary, 
	\subsubsection{Estimates independent of $\delta$}\label{SubSec:DEst}
	Let us consider a family $\{(\xi_\delta,\mathfrak r_\delta,\mathfrak z_\delta,\Sigma_\delta,u_\delta,\vartheta_\delta)\}$ of solutions to \eqref{DRCE}--\eqref{DAEneBal}.
	In a usual way we take $\psi=1$ in \eqref{DAEBal} and subtract the $\Theta$--multiple of the resulting identity from \eqref{DAEneBal} to obtain the dissipation balance
	\begin{equation}\label{DissBalDel}
		\begin{split}
			&\int_\Omega\left(\frac{1}{2}\Sigma_\delta|u_\delta|^2+H_{\delta,\Theta}(\mathfrak r_\delta,\mathfrak z_\delta,\vartheta_\delta)+h_{\delta}(\mathfrak r_\delta,\mathfrak z_\delta)\right)(t)+\Theta\sigma_\delta([0,t]\times\overline\Omega)\\
			&=\int_\Omega\left(\frac{1}{2}\frac{|(\Sigma u)_{0,\delta}|^2}{\Sigma_{0,\delta}}+H_{\delta,\Theta}(\mathfrak r_{0,\delta},\mathfrak z_{0,\delta},\vartheta_{0,\delta})+h_\delta(\mathfrak r_{0,\delta},\mathfrak z_{0,\delta})\right)+\int_0^t\int_\Omega\frac{\delta}{\vartheta_\delta^2}\text{ for a.a. }t\in(0,T)
		\end{split}
	\end{equation}
	with the Helmholtz function $H_{\delta,\Theta}$ defined in \eqref{HelmDDef}. We notice that the term containing $\frac{\delta}{\vartheta_\delta^2}$ can be absorbed in the term $\frac{\delta}{\vartheta_\delta^3}$ in \eqref{SigmDDef} using the Young inequaity. Assuming that the first integral on the right hand side of \eqref{DissBalDel} is bounded uniformly with respect to $\delta>0$ we conclude as in Sections~\ref{SubSec:NEst} and \ref{Sec:EpsEst}
	\begin{equation}\label{DelEst}
		\begin{split}
			\|\sqrt{\Sigma_\delta} u_\delta\|_{L^\infty(0,T;L^2(\Omega))}&\leq c,\\
			\|\mathfrak r_\delta\|_{L^\infty(0,T;L^\gamma(\Omega))}+\|\mathfrak r_\delta\|_{L^\infty(0,T;L^\gamma(\Omega))}&\leq c,\\
			\|\mathfrak r_\delta\|_{L^\infty(0,T;L^\Gamma(\Omega))}+\|\mathfrak z_\delta\|_{L^\infty(0,T;L^\Gamma(\Omega))}&\leq c\delta^{-\frac{1}{\Gamma}},\\
			\|\vartheta_\delta\|_{L^\infty(0,T;L^4(\Omega))}&\leq c,\\
			\sigma_\delta([0,T]\times\overline{\Omega})&\leq c,\\
			\|\mathcal S_\delta(\mathfrak r_\delta,\mathfrak z_\delta,\vartheta_\delta)\|_{L^\infty(0,T;L^1(\Omega))}&\leq c.
		\end{split}
	\end{equation}
	As a consequence of \eqref{DelEst}$_5$ we get from \eqref{SigmDDef}
	\begin{equation}\label{FDelEst}
		\begin{split}
			\|\vartheta^{-1}_\delta\mathbb S(\vartheta_\delta,\nabla u_\delta)\cdot\nabla u_\delta\|_{L^1(Q_T)}&\leq c,\\
			\|\nabla \log \vartheta_\delta\|_{L^2(Q_T)}&\leq c,\\
			\|\nabla \vartheta_\delta^\frac{\beta}{2}\|_{L^2(Q_T)}&\leq c,\\
			\|\vartheta_\delta^{-1}\|_{L^3(Q_T)}&\leq c\delta^{-\frac{1}{3}},\\
			\|\nabla\vartheta_\delta^\frac{\Gamma}{2}\|_{L^2(Q_T)}+	\|\nabla\vartheta_\delta^{-\frac{1}{2}}\|_{L^2(Q_T)}&\leq c\delta^{-\frac{1}{2}}.
		\end{split}
	\end{equation}
	Moreover, we get as in Section~\ref{Sec:EpsPass} by the generalized Korn--Poincar\'e inequality
	\begin{equation}\label{UDelEst}
		\|u_\delta\|_{L^2(0,T;W^{1,2}(\Omega))}\leq c.
	\end{equation}	
	By Lemma~\ref{Lem:Poincare} we have from \eqref{DelEst}$_{2,4,6}$ and \eqref{FDelEst}$_{2}$
	\begin{equation}\label{LogThDelEst}
		\|\log\vartheta_\delta\|_{L^2(0,T;W^{1,2}(\Omega))}\leq c.
	\end{equation}
	Furthermore, by the Poincar\'e inequality we have due to \eqref{DelEst}$_4$ and \eqref{FDelEst}$_{3,5}$
	\begin{equation}\label{VDSob}
		\begin{split}\|\vartheta_\delta^\frac{\beta}{2}\|_{L^2(0,T;W^{1,2}(\Omega))}&\leq c,\\
			\|\vartheta_\delta^\frac{\Gamma}{2}\|_{L^2(0,T;W^{1,2}(\Omega))}&\leq c\delta^{-\frac{1}{2}}.
		\end{split}
	\end{equation}
	Hence using first the Sobolev embedding $W^{1,2}(\Omega)\hookrightarrow L^6(\Omega)$ and then the interpolation and \eqref{DelEst}$_4$ we obtain
	\begin{equation}\label{VThDeltaEst}
		\|\vartheta_\delta\|_{L^\beta(0,T;L^{3\beta}(\Omega))}+\|\vartheta_\delta\|_{L^{p_\beta}(Q_T)}\leq c, \text{ with }p_\beta=\frac{8}{3}+\beta.
	\end{equation}	
	By \eqref{EPointEqS} we get
	\begin{equation}\label{DPointEqS}
		\begin{split}
			\underline a\mathfrak r_\delta\leq \mathfrak z_\delta\leq\overline a\mathfrak r,\\
			\underline b\Sigma_\delta\leq \mathfrak r_\delta+\mathfrak z_\delta\leq\overline b\Sigma_\delta,\\
			\underline d\mathfrak r_\delta\leq \xi_\delta\leq \overline d\mathfrak r_\delta.
		\end{split}
	\end{equation}
    Taking into account \eqref{STensDef} we write
    \begin{equation*}
    \mathbb S(\vartheta_\delta,\nabla u_\delta)=\sqrt{\mu(\vartheta_\delta)\vartheta_\delta}\sqrt{\frac{\mu(\vartheta_\delta)}{\vartheta_\delta}}\left(\nabla u+(\nabla u)^\top -\frac{2}{3}\dvr u\mathbb I\right)+\sqrt{\eta(\vartheta_\delta)\vartheta_\delta}\sqrt{\frac{\eta(\vartheta_\delta)}{\vartheta_\delta}}\dvr u_\delta\mathbb I.
    \end{equation*}
	Using \eqref{FDelEst}$_1$ and \eqref{VDSob}$_1$ we get
	by \eqref{MuGr} and \eqref{EGr}
	\begin{equation*}
		\|\mathbb S(\vartheta_\delta,\nabla u_\delta)\|_{L^q(Q_T)}\leq c
	\end{equation*}
    with $q=\frac{2p_\beta}{2+p_\beta}>\frac{4}{3}$ due to \eqref{PBLBound}.
	Next, we use a test function of the form $\varphi=\phi\mathcal B\left(\mathfrak r_\delta^{\mathfrak c}-\frac{1}{|\Omega|}\int_\Omega \mathfrak r_\delta^{\mathfrak c} \right)$, where $\phi\in C^1_c((0,T))$, $\mathcal B$ is the Bogovskii operator and $\mathfrak c\leq\gamma_{BOG}$ with $\gamma_{BOG}$ given by \eqref{OVGDef} to obtain 
	\begin{equation*}
		\int_0^T\int_\Omega \mathcal P_\delta(\mathfrak r_\delta,\mathfrak z_\delta,\vartheta_\delta)\mathfrak r_\delta^{\mathfrak c}\leq c.
	\end{equation*}
	The latter bound, \eqref{PDDef} and \eqref{DPointEqS} imply
	\begin{equation}\label{DensDelBetIntegr}
		\|\mathfrak r_\delta\|_{L^{\overline{\gamma}}(Q_T)}+\|\mathfrak z_\delta\|_{L^{\overline{\gamma}}(Q_T)}+\delta^\frac{1}{\Gamma}\left(\|\mathfrak r_\delta\|_{L^{\overline{\Gamma}}(Q_T)}+\|\mathfrak z_\delta\|_{L^{\overline{\Gamma}}(Q_T)}\right)\leq c,\ \overline\gamma=\gamma+\gamma_{BOG},\ \overline\Gamma=\Gamma+\gamma_{BOG}
	\end{equation}
	due to \eqref{PDDef} and \eqref{PressGrowth}. By virtue of \eqref{VThDeltaEst} and \eqref{DensDelBetIntegr} we infer
	\begin{equation}\label{PressDelLqBound}
		\|\mathcal P_\delta(\mathfrak r_\delta,\mathfrak z_\delta,\vartheta_\delta)\|_{L^q(Q_T)}\leq c,\text{ for some }q>1.
	\end{equation}
	Employing \eqref{IntEnGrowth} we have 
	\begin{equation}\label{IntEnDelIntegr}
		\|\mathcal E(\mathfrak r_\delta,\mathfrak z_\delta,\vartheta_\delta)\|_{L^q(Q_T)}\leq c\text{ for some }q>1
	\end{equation}
	due to \eqref{VThDeltaEst} and \eqref{DensDelBetIntegr}.
	By \eqref{SpEntStr} and \eqref{EntGr} we have
	\begin{equation*}
		|\mathcal S(\mathfrak r_\delta,\mathfrak z_\delta,\vartheta_\delta)|\leq c(1+\vartheta_\delta^3+\mathfrak r_\delta|\log \mathfrak r_\delta|+\mathfrak z_\delta|\log\mathfrak z_\delta|+(\mathfrak r_\delta+\mathfrak z_\delta)|\log\vartheta_\delta|+\mathfrak r^{\gamma^s_+}_\delta+\mathfrak z^{\gamma^s_-}_\delta).
	\end{equation*}
	Hence using \eqref{DelEst}$_{2,3,4}$ and \eqref{LogThDelEst}
	\begin{equation}\label{SDelEst}
		\|\mathcal S(\mathfrak r_\delta,\mathfrak z_\delta,\vartheta_\delta)\|_{L^2(0,T;L^q(\Omega))}\leq c\text{ for some } q>\frac{6}{5}.
	\end{equation}
	We also have
	\begin{equation}\label{DelDLThEst}
		\|(\mathfrak r_\delta+\mathfrak z_\delta)\log\vartheta_\delta\|_{L^p(Q_T)}\leq c\delta^{-\frac{1}{\Gamma}}\text{ for }p\in(\frac{6}{5},2)
	\end{equation}
	by \eqref{DelEst}$_3$ and \eqref{LogThDelEst}.
	Moreover, 
	\begin{equation*}
		\begin{split}
			|\mathcal S(\mathfrak r_\delta,\mathfrak z_\delta,\vartheta_\delta) u_\delta|\leq c(|u_\delta|+|\vartheta_\delta|^3|u_\delta|+(\mathfrak r_\delta|\log \mathfrak r_\delta|+\mathfrak z_\delta |\log\mathfrak z_\delta|)|u_\delta|).
		\end{split}
	\end{equation*}
	Using \eqref{DelEst}$_{2,4}$, \eqref{UDelEst}, the Sobolev embedding $W^{1,2}(\Omega)\hookrightarrow L^6(\Omega)$ we get
	\begin{equation*}
		\begin{split}
			\|\mathcal S(\mathfrak r_\delta,\mathfrak z_\delta,\vartheta_\delta) u_\delta\|_{L^q(Q_T)}\leq c\text{ for some }q>1.
		\end{split}
	\end{equation*}
	Moreover, applying \eqref{DelEst}$_3$ we have
	\begin{equation}\label{DelEstDLTU}
		\begin{split}
			&\int_{\Omega}(\mathfrak r_\delta+\mathfrak z_\delta)|\log\vartheta_\delta|(1+|u_\delta|)\\
			&\leq (\|\mathfrak r_\delta\|_{L^\infty(0,T;L^\Gamma(\Omega))}+\|\mathfrak z_\delta\|_{L^\infty(0,T;L^\Gamma(\Omega))}) \|\log\vartheta_\delta\|_{L^2(0,T;L^2(\Omega))}\left(1+\|u_\delta\|_{L^2(0,T;L^6(\Omega))}\right)\leq c\delta^{-\frac{1}{\Gamma}}.
		\end{split}
	\end{equation}
	By \eqref{KGrowth} it follows that
	\begin{equation*}
		\frac{\kappa(\vartheta_\delta)}{\vartheta_\delta}|\nabla \vartheta_\delta|\leq c(|\nabla\log\vartheta_\delta|+\vartheta^\frac{\beta}{2}|\nabla\vartheta_\delta^\frac{\beta}{2}|).
	\end{equation*}
	Thanks to \eqref{LogThDelEst}, \eqref{VDSob}$_1$ and \eqref{VThDeltaEst} we have the bound
	\begin{equation*}
		\left\|\frac{\kappa(\vartheta_\delta)}{\vartheta_\delta}\nabla \vartheta_\delta\right\|_{L^q(Q_T)}\leq c\text{ with }q=\frac{2p_\beta}{\beta+p_\beta}.
	\end{equation*}
	We observe that
	\begin{equation*}
		\vartheta_\delta^{\Gamma -1}\nabla\vartheta_\delta=\frac{2}{\Gamma}\vartheta_\delta^\frac{1}{4}\vartheta_\delta^{\frac{\Gamma}{2}-\frac{1}{4}}\nabla \vartheta_\delta^\frac{\Gamma}{2}.
	\end{equation*}
	Hence we estimate using \eqref{DelEst}$_4$, \eqref{VDSob}$_2$ and the Sobolev embedding $W^{1,2}(\Omega)\hookrightarrow L^6(\Omega)$
	\begin{equation}\label{FDelTEst}
		\delta\|\vartheta_\delta^{\Gamma-1}\nabla\vartheta^\frac{\Gamma}{2}\|_{L^1(Q_T)}\leq\frac{2\delta}{\Gamma}\|\vartheta_\delta\|_{L^\infty(0,T;L^4(\Omega))}^\frac{1}{4}\|\vartheta_\delta^\frac{\Gamma}{2}\|^{\frac{\frac{\Gamma}{2}-\frac{1}{4}}{\frac{\Gamma}{2}}}_{L^2(0,T;L^{6}(\Omega))}\|\nabla\vartheta_\delta^\frac{\Gamma}{2}\|_{L^2(Q_T)}\leq c\delta^\frac{1}{4\Gamma}.
	\end{equation}
	Similarly, we get 
	\begin{equation}\label{SDelTEst}
		\delta\|\vartheta_\delta^{-2}\nabla\vartheta_\delta\|_{L^1(Q_T)}\leq c\delta\|\vartheta_\delta^{-1}\|_{L^3(Q_T)}\|\nabla\vartheta_\delta^{-\frac{1}{2}}\|_{L^2(Q_T)}\leq c\delta^\frac{1}{6}
	\end{equation}
	by \eqref{FDelEst}$_{4,5}$.
	
	\subsubsection{Convergences as $\delta\to 0_+$}
	As in previos subsections we get from \eqref{DelEst}$_{2,4,5}$ and \eqref{UDelEst}
	\begin{equation}\label{DelWConv}
		\begin{alignedat}{2}
			(\xi_\delta,\Sigma_\delta,\mathfrak r_\delta,\mathfrak z_\delta)&\rightharpoonup^*(\xi,\Sigma,\mathfrak r,\mathfrak z)&&\text{ in }L^\infty(0,T;L^\gamma(\Omega)),\\
			u_\delta&\rightharpoonup u&&\text{ in }L^2(0,T;W^{1,2}(\Omega)),\\
			\vartheta_\delta&\rightharpoonup^*\vartheta&&\text{ in }L^\infty(0,T;L^4(\Omega)),\\
			\sigma_\delta&\rightharpoonup^*\sigma&&\text{ in }(C([0,T]\times\overline\Omega))^*.
		\end{alignedat}
	\end{equation}
	We also obtain
	\begin{equation}\label{FDelWConv}
		\begin{alignedat}{2}
			(\xi_\delta,\Sigma_\delta,\mathfrak r_\delta,\mathfrak z_\delta)&\to(\xi,\Sigma,\mathfrak r,\mathfrak z)&&\text{ in }C_w([0,T];L^\gamma(\Omega)),\\
			\Sigma_\delta u_\delta&\to \Sigma u&&\text{ in }C_w([0,T];L^\frac{2\gamma}{\gamma+1}(\Omega)),\\
			\Sigma_\delta u_\delta\otimes u_\delta&\rightharpoonup \Sigma u\otimes u&&\text{ in }L^1(Q_T),\\
			\Sigma_\delta |u_\delta|^2&\rightharpoonup \Sigma |u|^2&&\text{ in }L^1(Q_T),\\
			(\mathfrak \xi_\delta u_\delta,\mathfrak r_\delta u_\delta,\mathfrak z_\delta u_\delta)&\rightharpoonup (\mathfrak \xi u,\mathfrak r u,\mathfrak z u)&&\text{ in }L^2(0,T;L^\frac{2\gamma}{\gamma+2}(\Omega)).
		\end{alignedat}
	\end{equation}
	\subsubsection{Pointwise convergence of $\vartheta_\delta$ and consequences}
	Proceeding as in Section~\ref{Sec:EpsThConv} we show that 
	\begin{equation}\label{ThetaDelPoint}
		\vartheta_\delta\to\vartheta\text{ a.e. in }Q_T.
	\end{equation}
	Indeed, by the Div--Curl lemma for 
	\begin{equation*}
		U_\delta=\left(\mathcal S_\delta(\mathfrak r_\delta,\mathfrak z_\delta,\vartheta_\delta),\mathcal S_\delta(\mathfrak r_\delta,\mathfrak z_\delta,\vartheta_\delta)u_\delta+\frac{\kappa_\delta(\vartheta_\delta)}{\vartheta_\delta}\nabla\vartheta_\delta\right), V_\delta=\left(T_k(\vartheta_\delta),0,0,0\right),
	\end{equation*}
	where 
	\begin{equation}
		\mathcal S_\delta(\mathfrak r_\delta,\mathfrak z_\delta,\vartheta_\delta)=\mathcal S(\mathfrak r_\delta,\mathfrak z_\delta,\vartheta_\delta)+\delta(\mathfrak r_\delta+\mathfrak z_\delta)\log\vartheta_\delta,
	\end{equation}
	it follows that
	\begin{equation}\label{LimDelProdEntr}
		\overline{\mathcal S(\mathfrak r,\mathfrak z,\vartheta)T_k(\vartheta)}=\overline{\mathcal S(\mathfrak r,\mathfrak z,\vartheta)}\ \overline{T_k(\vartheta)}
	\end{equation}
	as $\delta\left(\mathfrak r_\delta+\mathfrak z_\delta\right)\log\vartheta_\delta T_k(\vartheta_\delta)\to 0$ in $L^1(Q_T)$ implied by \eqref{DelDLThEst} yields 
	\begin{equation*}
		\mathcal S_\delta(\mathfrak r_\delta,\mathfrak z_\delta,\vartheta_\delta)T_k(\vartheta_\delta)\rightharpoonup \overline{\mathcal S(\mathfrak r,\mathfrak z,\vartheta)T_k(\vartheta)}\text{ in }L^1(Q_T).
	\end{equation*} 
	\begin{equation}\label{LimDeltaPr}
		\overline{b(\mathfrak r,\mathfrak z)G(\vartheta)}=\overline{b(\mathfrak r,\mathfrak z)}\ \overline{G(\vartheta)}
	\end{equation}
	for any $b\in C^1(\eR^2)$, $\nabla b\in C_c(\eR^2;\eR^2)$ and Lipschtz $G\in C^1(\eR)$. Fixing such an arbitrary $b$ and $G$, defining $Z_\delta=(\mathfrak r_\delta,\mathfrak z_\delta)$ and
	\begin{equation*}
		b_\delta=b(Z_\delta)
	\end{equation*}
	we obtain $b_\delta\in C([0,T];L^1(\Omega))$ as $(Z_\delta, u_\delta)$ satisfy the renormalized continuity equation with the renormalizing function $b$. Hence the sequence $\{t\mapsto\int_\Omega b_\delta\phi\}$ is equi--continuous and bounded in $C([0,T])$ for any $\phi\in C^1_c(\Omega)$. Using the Arzel\`a--Ascoli theorem and the density of $C^1_c(\Omega)$ in $L^{p'}(\Omega)$ $p'\in[1,\infty)$, we get at least for a nonrelabeled subsequence 
	\begin{equation*}
		b_\delta\to\overline{b(\mathfrak r,\mathfrak z)}\text{ in }C_w([0,T];L^p(\Omega)),\ p\in(1,\overline\gamma)
	\end{equation*}
	performing a diagonalization procedure. Accordingly, it follows that
	\begin{equation*}
		b_\delta\to\overline{b(\mathfrak r,\mathfrak z)}\text{ in }L^2(0,T;W^{-1,2}(\Omega))
	\end{equation*} 
	in view of the compact embedding $L^p(\Omega)$ into $W^{-1,2}(\Omega)$ valid for $p>\frac{6}{5}$. Obviously, we get
	\begin{equation*}
		G(\vartheta_\delta)\rightharpoonup \overline{G(\vartheta)}\text{ in }L^2(0,T;W^{1,2}(\Omega)).
	\end{equation*}
	Combining the latter two convergences we conclude \eqref{LimDeltaPr}. Then we can proceed as in Section~\ref{Sec:EpsThConv} to deduce 
	\begin{equation*}
		\overline{\mathfrak S(\mathfrak r,\mathfrak z,\vartheta)T_k(\vartheta)}\geq\overline{\mathfrak S(\mathfrak r,\mathfrak z,\vartheta)}\ \overline{T_k(\vartheta)} 
	\end{equation*}
	for $\mathfrak G(r,z,\vartheta)=r\mathsf s_+(r,\vartheta)+z\mathsf s_-(z,\vartheta)$. Moreover, we have
	\begin{equation*}
		\overline{\vartheta^3 T_k(\vartheta)}\geq \overline{\vartheta^3}\ \overline{T_k(\vartheta)}
	\end{equation*}
	by Lemma~\ref{Lem:MonWConv}. Combining the last two inequalities and \eqref{LimDelProdEntr} we get $\overline{\vartheta^3T_k(\vartheta)}=\overline{\vartheta^3}\ \overline{T_k(\vartheta)}$. We conclude 
	\begin{equation}\label{ThetaDeltaPoint}
		\vartheta_\delta\to\vartheta\text{ a.e. in }Q_T
	\end{equation}
	using the arguments leading to \eqref{ThetaEpsPointwisely}. Accordingly, bound \eqref{VThDeltaEst} implies
	\begin{equation}\label{ThetaDelStongly}
		\vartheta_\delta\to\vartheta\text{ in }L^p(Q_T),\ p\in[1,p_\beta).
	\end{equation}
	\subsubsection{Derivation of the effective viscous flux identity}
	As in Section~\ref{Sec:EpsEffVicFl} we obtain 
	\begin{equation}\label{SThNUDelConv}
		\mathbb S(\vartheta_\delta,\nabla u_\delta)\rightharpoonup \mathbb S(\vartheta,\nabla u)\text{ in }L^1(Q_T).
	\end{equation}
	We denote
	\begin{equation*}
		p(r,\zeta,\vartheta)=\mathsf P_+(r,\vartheta)+\mathsf P_-(\zeta r,\vartheta)
	\end{equation*}
    and defining $\zeta_\delta= \mathfrak z_\delta/_{\mathfrak a} \mathfrak r_\delta$ we realize that
	\begin{equation}\label{SmpId}
		p(\mathfrak r_\delta,\zeta_\delta,\vartheta_\delta)=p(\mathfrak r_\delta,\zeta,\vartheta)+D_\delta,
	\end{equation}
	where
	\begin{equation*}
		D_\delta=\mathsf P_+(\mathfrak r_\delta,\vartheta_\delta)+\mathsf P_-(\zeta_\delta \mathfrak r_\delta,\vartheta_\delta)-\mathsf P_+(\mathfrak r_\delta,\vartheta)-\mathsf P_-(\zeta_\delta \mathfrak r_\delta,\vartheta).
	\end{equation*}
	Employing the first order Taylor expansion, the assumptions on growth from \eqref{PressVarGr}, \eqref{AlmConv} for $\zeta_\delta$, $\zeta=\mathfrak z/_{\mathfrak a} \mathfrak r$ and convergence \eqref{ThetaDelStongly} we conclude
	\begin{equation*}
		\lim_{\delta\to 0_+}\|D_\delta\|_{L^1(Q_T)}=0.
	\end{equation*}
	Employing \eqref{PressDelLqBound}, the latter limit, convergence \eqref{ThetaDelStongly} in \eqref{SmpId} we deduce
	\begin{equation}\label{PDWeakLim}
		\mathcal P_\delta(\mathfrak r_\delta,\mathfrak z_\delta,\vartheta_\delta)\rightharpoonup \frac{b}{3}\vartheta^4+ \overline{\overline{p(\mathfrak r,\zeta,\vartheta)}}\text{ in }L^1(Q_T)
	\end{equation}
	as $\delta(\mathfrak r_\delta^\Gamma+\mathfrak z_\delta^\Gamma+\mathfrak r_\delta^2+\mathfrak z_\delta^2)\to 0$ in $L^1(Q_T)$ due to \eqref{DensDelBetIntegr}.
	
	As convergences \eqref{FDelWConv}$_{1,5}$ and \eqref{InDConv} are available, we can pass to the limit $\delta\to 0_+$ in \eqref{DRCE} with $u=u_\delta$ and $r$ standing for $\mathfrak r_\delta$ or $\mathfrak z_\delta$ or $\Sigma_\delta$ or $\xi_\delta$ to recover \eqref{ConEq}.
	
	Using in \eqref{DMomBal} convergences \eqref{FDelWConv}$_{2,3}$, \eqref{SThNUDelConv}, \eqref{PDWeakLim} and $(\Sigma u)_{0,\delta}\to (\Sigma u)_0$ in $L^1(\Omega)$ following from \eqref{SUZDDef} we arrive at
	\begin{equation}\label{MomEqDelLim}
		\int_\Omega \Sigma u\cdot\varphi (t)-\int_\Omega (\Sigma u)_0\cdot\varphi(0)=\int_0^t\int_\Omega \left(\Sigma u\cdot\tder\varphi+\left((\Sigma u\otimes u-\mathbb S(\vartheta,\nabla u)\right)\cdot \nabla\varphi+\left(\frac{b}{3}\vartheta^4+\overline{\overline{p(\mathfrak r,\mathfrak \zeta,\vartheta)}}\right)\dvr\varphi\right)
	\end{equation}
	for all $t\in[0,T]$ and any $\phi\in C^1([0,T]\times\overline\Omega)$.
	Similarly to Section~\ref{Sec:EpsEffVicFl} we set 
	\begin{equation*}
		\varphi=\psi\phi\mathfrak U(1_\Omega T_k(\mathfrak r_\delta)),\ \psi\in C^1_c((0,T)), \phi\in C^\infty_c(\Omega)
	\end{equation*}
	in \eqref{DMomBal} and 
	\begin{equation*}
		\varphi=\psi\phi\mathfrak U(1_\Omega T_k(\mathfrak r)),\ \psi\in C^1_c((0,T)), \phi\in C^\infty_c(\Omega)
	\end{equation*}
	in \eqref{MomEqDelLim} and perform the passage $\delta\to 0_+$ to obtain 
	\begin{equation}\label{DelPresIdent}
		\begin{split}
			&\int_0^T\psi\int_\Omega\phi\left(\overline{\overline {p(\mathfrak r,\zeta,\vartheta)T_k(\mathfrak r)}}-\overline{\overline {p(\mathfrak r,\zeta,\vartheta)}}\ \overline{T_k(\mathfrak r)}\right)=\\&\int_0^T\psi\int_\Omega\phi\left(\overline{\mathbb S(\vartheta,\nabla u)\cdot\mathfrak R(1_\Omega T_k(\mathfrak r))}-\mathbb S(\vartheta,\nabla u)\cdot\mathfrak R(\overline{1_\Omega T_k(\mathfrak r)})\right)\\
			&+\lim_{\delta\to 0_+}\int_0^T\psi\int_\Omega\left((u_\delta)_j\cdot(1_\Omega T_k(\mathfrak r_\delta)\mathfrak R_{ij}(\phi(\Sigma_\delta u_\delta)_i))-\phi(\Sigma_\delta u_\delta)_i\mathfrak R_{ij}(1_\Omega T_k(\mathfrak r_\delta))-u_j\cdot(1_\Omega\overline{T_k(\mathfrak r)}\mathfrak R_{ij}(\phi\Sigma u))\right.\\&\quad\quad\left.-\phi(\Sigma u)_i\mathfrak R_{ij}(1_\Omega \overline{T_k(\mathfrak r)})\right).
		\end{split}
	\end{equation}
	Then as $\mathfrak r_\delta$ fulfills the renormalized continuity equation with the renormalizing function $b(\mathfrak r_\delta)=T_k(\mathfrak r_\delta)$ we immediately get that $T_k(\mathfrak r_\delta)\to\overline{T_k(\mathfrak r)}$ in $C([0,T];L^q(\Omega))$ for any $q\in[1,\infty)$ and we conclude as in Section~\ref{Sec:EpsEffVicFl} by Lemma~\ref{DCRiesz} that
	\begin{equation*}
		\left(1_\Omega T_k(\mathfrak r_\delta)\mathfrak R_{ij}(\phi(\Sigma_\delta u_\delta)_i)-\phi(\Sigma_\delta u_\delta)_i\mathfrak R_{ij}(1_\Omega T_k(\mathfrak r_\delta))\right)(t)\rightharpoonup \left(1_\Omega\overline{T_k(\mathfrak r)}\mathfrak R_{ij}(\phi\Sigma u)-\phi(\Sigma u)_i\mathfrak R_{ij}(1_\Omega \overline{T_k(\mathfrak r)})\right)(t)
	\end{equation*}
	first in $L^\frac{2\gamma}{\gamma+1}(\Omega)$ for all $t\in[0,T]$. Using the compact embedding $L^\frac{2\gamma}{\gamma+1}(\Omega)$ in $W^{-1,2}(\Omega)$ for $\gamma>\frac{6}{5}$ we get that the latter convergence is strong in $W^{-1,2}(\Omega)$ for all $t\in[0,T]$. Consequently, by the Vitali convergence theorem we get that the convergence is strong in $L^p(0,T;W^{-1,2}(\Omega))$ for any $p\in[1,\infty)$ which in combination with \eqref{DelWConv}$_2$ justifies that the limit on the right hand side of \eqref{DelPresIdent} vanishes. Hence we arrive at
	\begin{equation}\label{AlmEffViscFlDel}\begin{split}
    &\int_0^T\psi\int_\Omega\phi\left(\overline{\overline{p(\mathfrak r,\zeta,\vartheta)T_k(\mathfrak r)}}-\overline{\overline{p(\mathfrak r,\zeta,\vartheta)}}\ T_k(\mathfrak r)\right)\\
    &=\int_0^T\psi\int_\Omega\phi\left(\overline{\mathbb S(\vartheta,\nabla u)\cdot\mathfrak R(1_\Omega T_k(\mathfrak r))}-\mathbb S(\vartheta,\nabla u)\cdot\mathfrak R(1_\Omega T_k(\mathfrak r))\right).
 \end{split}
	\end{equation}
	We focus our attention on the terms on the right hand side of the latter equality. As in Section~\ref{Sec:EpsEffVicFl} we have
	\begin{equation}\label{DelLimIdent}
    \begin{split}
		&\int_0^T\psi\int_\Omega\phi\overline{\mathbb S(\vartheta,\nabla u)\cdot\mathfrak R(1_\Omega T_k(\mathfrak r))}=\lim_{\delta\to 0_+}\int_0^T\psi\int_\Omega\phi\left(\frac{4}{3}\mu(\vartheta_\delta)+\eta(\vartheta_\delta)\right)T_k(\mathfrak r_\delta)\dvr u_\delta\\
  &+\lim_{\delta\to 0_+}\int_0^T\psi\int_\Omega T_k(\mathfrak r_\delta)\omega(\vartheta_\delta,u_\delta)
    \end{split}
	\end{equation}
	and
	\begin{equation}\label{IdentLimDel}
		\int_0^T\psi\int_\Omega\phi\mathbb S(\vartheta,\nabla u)\cdot\mathfrak R(1_\Omega \overline{T_k(\mathfrak r)})=\int_0^T\psi\int_\Omega\phi\left(\frac{4}{3}\mu(\vartheta)+\eta(\vartheta)\right)\overline{T_k(\mathfrak r)}\dvr u+\int_0^T\psi\int_\Omega \overline{T_k(\mathfrak r)}\omega(\vartheta,u),
	\end{equation}
	where $\omega(\theta, u)$ is defined in \eqref{omDef}.
	
	Next, according to Lemma~\ref{Lem:ComEst} we infer
	\begin{equation}\label{OmDelLWBound}
		\|\omega(\vartheta_\delta, u_\delta)\|_{L^1(0,T;W^{k,p}(\Omega))}\leq c\text{ for some }k\in(0,1), p>1 
	\end{equation}
	by virtue of bounds \eqref{UDelEst} and \eqref{VDSob}$_1$ which also guarantee that
	\begin{equation}\label{OmDelLBound}
		\|\omega(\vartheta_\delta,u_\delta)\|_{L^q(Q_T)}\leq c\text{ for some }q>1.
	\end{equation}
	Following the reasoning leading to \eqref{OmREpscon} we conclude
considering sequences $\{U_\delta\}$ and $\{V_\delta\}$ of the four dimensional vector fields
	\begin{equation*}
		\begin{split}
			U_\delta=&\left(T_k(\mathfrak r_\delta), T_k(\mathfrak r_\delta)u_\delta\right),\ V_\delta=\left(\omega(\vartheta_\delta,u_\delta),0,0,0\right)
		\end{split}
	\end{equation*}
	that
	\begin{equation}\label{OmRDelcon}
		\omega(\vartheta_\delta, u_\delta)T_k(\mathfrak r_\delta)\rightharpoonup \omega(\vartheta, u)\overline{T_k(\mathfrak r)}\text{ in }L^1(Q_T)
	\end{equation}
	since in accordance with convergences \eqref{DelWConv}$_2$, \eqref{ThetaDelStongly} respectively, and estimate \eqref{MuGr} it follows that $\overline{\omega(\vartheta, u)}=\omega(\vartheta, u)$.
	Eventually, going back to \eqref{AlmEffViscFlDel}, where we employ \eqref{DelLimIdent}, \eqref{IdentLimDel} and \eqref{OmRDelcon}, we arrive at the desired effective viscous flux identity
	\begin{equation}\label{EffViscFlDel}
		\overline{\overline{p(\mathfrak r,\zeta,\vartheta)T_k(\mathfrak r)}}-\overline{\overline{p(\mathfrak r,\zeta,\vartheta)}}\ \overline{T_k(\mathfrak r)}=\left(\frac{4}{3}\mu(\vartheta)+\eta(\vartheta)\right)\left(\overline{T_k(\mathfrak r)\dvr u}-\overline{T_k(\mathfrak r)}\dvr u\right).
	\end{equation}
	\subsubsection{Pointwise convergence of $\mathfrak r_\delta$ and consequences}
	Firstly, we notice that by \eqref{DelEst}$_2$ and \eqref{DensDelBetIntegr} we have
	\begin{equation}\label{TKLKBounds}
		\begin{split}
			\|T_k(\mathfrak r_\delta)\|_{L^q(Q_T)}+\|T_k(\mathfrak r)\|_{L^q(Q_T)}\leq c&\text{ for }q\in[1,\overline\gamma],\\
			\|L_k(\mathfrak r_\delta)\|_{L^\infty(0,T;L^q(\Omega))}\leq c&\text{ for }q\in[1,\gamma)
		\end{split}
	\end{equation}
	for a constant $c$ independent of $k,\delta$. We note that $T_k$ is defined in \eqref{TKDef} and
    \begin{equation*}
        L_k(r)=\int_1^r\frac{T_k(s)}{s^2}\mathrm ds
    \end{equation*}
 Next, we know that $(\mathfrak r_\delta, u_\delta)$ fulfills the renormalized continuity equation with renormalizing function $b(r)=L_k(r)$ and $(\mathfrak r, u)$ as well. Considering the difference of these equations with the test function being $1$ we get
	\begin{equation*}
		\int_\Omega \left(L_k(\mathfrak r_\delta)-L_k(\mathfrak r)\right)(t)=\int_0^t\int_\Omega\left(T_k(\mathfrak r)\dvr u-\overline{T_k(\mathfrak r)}\dvr u_\delta\right)+\int_0^t\int_\Omega\left(\overline{T_k(\mathfrak r)}-T_k(\mathfrak r_\delta)\right)\dvr u_\delta
	\end{equation*}
	for all $t\in[0,T]$. We note that $\overline{T_k(\mathfrak r)}\in L^q(Q_T)$ due to \eqref{TKLKBounds}$_1$. Performing the passage $\delta\to 0_+$ and employing effective viscous flux identity \eqref{EffViscFlDel} we have
	\begin{equation}\label{LKLimIdent}
		\begin{split}
			\int_\Omega \left(\overline{L_k(\mathfrak r)}-L_k(\mathfrak r)\right)(t)=&\int_0^t\int_\Omega\left(T_k(\mathfrak r)-\overline{T_k(\mathfrak r)}\right)\dvr u\\&+\int_0^t\int_\Omega\frac{1}{\frac{4}{3}\mu(\vartheta)+\eta(\vartheta)}\left(\overline{\overline{p(\mathfrak r,\zeta,\vartheta)}}\ \overline{T_k(\mathfrak r)}-\overline{\overline{p(\mathfrak r,\zeta,\vartheta)T_k(\mathfrak r)}}\right)
		\end{split}
	\end{equation}
	for all $t\in[0,T]$. Here $\overline{L_k(\mathfrak r)}$ is the weak limit of $\{L_k(\mathfrak r_\delta)\}$ in $C_w([0,T];L^q(\Omega))$, the function space is determined by \eqref{TKLKBounds}$_2$ and the fact that each $\mathfrak r_\delta$ satisfies the continuity equation in the renormalized form with $b=L_k$. Employing Lemma~\ref{Lem:MonWConv} and \eqref{ThStab} we deduce that the last integral on the right hand side of \eqref{LKLimIdent} is nonpositive implying 
	\begin{equation*}
		\int_\Omega \left(\overline{L_k(\mathfrak r)}-L_k(\mathfrak r)\right)(t)\leq\int_0^t\int_\Omega\left(T_k(\mathfrak r)-\overline{T_k(\mathfrak r)}\right)\dvr u
	\end{equation*}
	for all $t\in[0,T]$. Performing the passage $k\to\infty$ in the latter inequality we conclude
	\begin{equation*}
		\overline{\mathfrak r\log \mathfrak r}=\mathfrak r\log\mathfrak r\text{ a.e. in }Q_T
	\end{equation*}  
	provided we show that 
	\begin{equation}\label{CrucLim}
		\lim_{k\to\infty}\int_0^t\int_\Omega\left(T_k(\mathfrak r)-\overline{T_k(\mathfrak r)}\right)\dvr u=0.
	\end{equation}
	For $\gamma<\overline\gamma_\beta$ we get by virtue of \eqref{DensDelBetIntegr}, \eqref{TKLKBounds} and the interpolation that
	\begin{equation}\label{IntTkREst}
		\begin{split}
			\left|\int_0^t\int_\Omega\left(T_k(\mathfrak r)-\overline{T_k(\mathfrak r)}\right)\dvr u\right|&\leq \|T_k(\mathfrak r)-\overline{T_k(\mathfrak r)}\|_{L^2(Q_T)}\|\dvr u\|_{L^2(Q_T)}\\
			&\leq c\|\overline{T_k(\mathfrak r)}-T_k(\mathfrak r)\|^\lambda_{L^{\overline\gamma}(Q_T)}\|\overline{T_k(\mathfrak r)}-T_k(\mathfrak r)\|^{1-\lambda}_{L^1(Q_T)}\\
			&\leq c\|\overline{T_k(\mathfrak r)}-T_k(\mathfrak r)\|^{1-\lambda}_{L^1(Q_T)},
		\end{split}
	\end{equation}
	where $\lambda=\frac{\overline\gamma}{2(\overline\gamma-1)}$. Furthermore, we have using the weak lower semicontinuity of $L^1$--norm
	\begin{equation*}
		\begin{split}
			\|\overline{T_k(\mathfrak r)}-T_k(\mathfrak r)\|_{L^1(Q_T)}&\leq \|\overline{T_k(\mathfrak r)}-\mathfrak r\|_{L^1(Q_T)}+\|\mathfrak r-T_k(\mathfrak r)\|_{L^1(Q_T)}\\
			&\leq\liminf_{\delta\to 0_+}\|T_k(\mathfrak r_\delta)-\mathfrak r_\delta\|_{L^1(Q_T)}+\|\mathfrak r-T_k(\mathfrak r)\|_{L^1(Q_T)}\leq ck^{-q}
		\end{split}
	\end{equation*}
	for some $q>0$.	Performing the passage $k\to\infty$ we get $\|\overline{T_k(\mathfrak r)}-T_k(\mathfrak r)\|_{L^1(Q_T)}\to 0$. Hence going back to \eqref{IntTkREst} we conclude \eqref{CrucLim}. If $\gamma=\overline\gamma_\beta$ then $\overline{\gamma}=2$ and the interpolation above cannot be applied. Instead, we show that so called oscillation defect measure is bounded, i.e.,
	\begin{equation}\label{ODMBound}
		\sup_{k\geq 1}\limsup_{\delta\to 0_+}\|T_k(\mathfrak r_\delta)-T_k(\mathfrak r)\|_{L^{\overline{\gamma}_\beta+1}(Q_T)}<\infty.
	\end{equation}
	Then as $\overline\gamma_\beta+1>2$ we use $\|\overline{T_k(\mathfrak r)}-T_k(\mathfrak r)\|_{L^{\overline{\gamma}_\beta+1}(Q_T)}$ for the interpolation in \eqref{IntTkREst} and employ the fact that the latter norm is estimated by the oscillation defect measure, which as we already know is bounded. Let us show \eqref{ODMBound}. Employing the decomposition from \eqref{PressDecomp} in \eqref{EffViscFlDel} we obtain
	\begin{equation*}
		\begin{split}
			&\underline\pi\int_0^T\int_\Omega \left(\overline{\mathfrak r^{\overline\gamma_\beta} T_k(\mathfrak r)}-\overline{\mathfrak r^{\overline\gamma_\beta}}\ \overline{T_k(\mathfrak r)}\right)+\int_0^T\int_\Omega\left(\overline{\overline{\mathscr p(\mathfrak r,\zeta,\vartheta)T_k(\mathfrak r)}}-\overline{\overline{\mathscr p(\mathfrak r,\zeta,\vartheta)}}\ \overline{T_k(\mathfrak r)}\right)\\
			&\leq \int_0^T\int_\Omega\left(\frac{4}{3}\mu(\vartheta)+\eta(\vartheta)\right)\left(\overline{T_k(\mathfrak r)\dvr u}-T_k(\mathfrak r)\dvr u+\left(T_k(\mathfrak r)-\overline{T_k(\mathfrak r)}\right)\dvr u\right).
		\end{split}
	\end{equation*}
	As $r\mapsto\mathscr p(r,\zeta,\vartheta)$ is nondecreasing we get using Lemma~\ref{Lem:MonWConv} that the second integral on the left hand side of the latter inequality is nonnegative. Moreover, due to the convexity of $r\mapsto r^\gamma$, we conclude employing \eqref{MuGr} and \eqref{EGr} that
	\begin{equation*}
		\begin{split}
			&\left(\overline{\mathfrak r^{\overline\gamma_\beta}}-\mathfrak r^{\overline\gamma_\beta}\right)\left(T_k(\mathfrak r)-\overline{T_k(\mathfrak r)}\right)+\limsup_{\delta\to 0_+}\int_0^T\int_\Omega\left(\mathfrak r_\delta^{\overline\gamma_\beta}-\mathfrak r^{\overline\gamma_\beta}\right)\left(T_k(\mathfrak r_\delta)-T_k(\mathfrak r)\right)\\
            &=\int_0^T\int_\Omega \left(\overline{\mathfrak r^{\overline\gamma_\beta} T_k(\mathfrak r)}-\overline{\mathfrak r^{\overline\gamma_\beta}}\ \overline{T_k(\mathfrak r)}\right)\\
			&\leq
			c\limsup_{\delta\to 0_+}\int_0^T\int_\Omega(1+\vartheta)^{-1}\left(\frac{4}{3}\mu(\vartheta)+\eta(\vartheta)\right)\left(\left(T_k(\mathfrak r_\delta)-T_k(\mathfrak r)\right)\dvr u_\delta+\left(T_k(\mathfrak r)-T_k(\mathfrak r_\delta)\right)\dvr u\right)(1+\vartheta)\\
			&\leq c\limsup_{\delta\to 0_+}\|T_k(\mathfrak r_\delta)-T_k(\mathfrak r)\|_{L^{\overline{\gamma}_\beta+1}(Q_T)}\left(\|\dvr u_\delta\|_{L^2(Q_T)}+\|\dvr u\|_{L^2(Q_T)}\right)\|1+\vartheta\|_{L^{p_\beta}(Q_T)}.
		\end{split}
	\end{equation*}
	The first term on the first line of the latter inequality is nonnegative since $s\mapsto s^\gamma$ is convex and $s\mapsto T_k(s)$ is concave. Moreover, as $|T_k(s_1)-T_k(s_2)|\leq |s_1-s_2|$, $|s_1-s_2|^\gamma\leq |s_1^\gamma-s_2^\gamma|$, $\gamma>1$ we have
	
	\begin{equation*}
		\limsup_{\delta\to 0_+}\|T_k(\mathfrak r_\delta)-T_k(\mathfrak r)\|_{L^{\overline{\gamma}_\beta+1}(Q_T)}\leq c,
	\end{equation*}
	where $c$ is independent of $k$, due to \eqref{UDelEst} and \eqref{VThDeltaEst}. Hence \eqref{ODMBound} immediately follows.
	Again, using the strict convexity of $s\mapsto s\log s$ and \eqref{AlmConv} we conclude
	\begin{equation}\label{DensDelPointw}
		\mathfrak r_\delta\to \mathfrak r,\ \mathfrak z_\delta\to\mathfrak z\text{ a.e. in }Q_T
	\end{equation}
	implying 
	\begin{equation}\label{DensDelStrong}
		(\mathfrak r_\delta,\mathfrak z_\delta)\to (\mathfrak r,\mathfrak z)\text{ in }L^p(Q_T)\text{ for any }p\in[1,\overline{\gamma}).
	\end{equation}
	Obviously, convergences \eqref{ThetaDelPoint}, \eqref{DensDelPointw} yield $\overline{\overline{p(\mathfrak r,\zeta,\vartheta)}}=p(\mathfrak r,\zeta,\vartheta)$ in \eqref{MomEqDelLim}, i.e., we have shown that \eqref{MomEqSig} is satisfied.
	\subsubsection{Limit passage $\delta\to 0_+$ in the energy equality}
	Having convergences \eqref{ThetaDelPoint}, \eqref{DensDelPointw} and bound \eqref{IntEnDelIntegr} at hand we get
	\begin{equation*}
		\mathcal E(\mathfrak r_\delta,\mathfrak z_\delta,\vartheta_\delta)\to \mathcal E(\mathfrak r,\mathfrak z,\vartheta)\text{ in }L^1(Q_T).
	\end{equation*}
	Hence employing \eqref{FDelWConv}$_4$, the latter convergence, \eqref{InDatConv}$_{1,3,6}$, bounds \eqref{DelEst}$_3$ and \eqref{FDelEst}$_4$ in \eqref{DAEneBal} we conclude \eqref{EnerEq}.
	\subsubsection{Limit passage $\delta\to 0_+$ in the entropy balance}
	By virtue of \eqref{ThetaDelPoint} and \eqref{DensDelPointw} we deduce
	\begin{equation}\label{SDelStrongly}
		\mathcal S(\mathfrak r_\delta,\mathfrak z_\delta,\vartheta_\delta)\to \mathcal S(\mathfrak r,\mathfrak z,\vartheta)\text{ in }L^2(0,T;L^q(\Omega))\text{ for some }q>\frac{6}{5}
	\end{equation}
	using the Vitali convergence theorem due to \eqref{SDelEst}. Recalling that $\mathcal S_\delta(\mathfrak r_\delta,\mathfrak z_\delta,\vartheta_\delta)=\mathcal S(\mathfrak r_\delta,\mathfrak z_\delta,\vartheta_\delta)+\delta(\mathfrak r_\delta+\mathfrak z_\delta)\log\vartheta_\delta$ we have by virtue of \eqref{DelDLThEst} and \eqref{SDelStrongly}
	\begin{equation}\label{SDelDelStrongly}
		\mathcal S_\delta(\mathfrak r_\delta,\mathfrak z_\delta,\vartheta_\delta)\to \mathcal S(\mathfrak r,\mathfrak z,\vartheta)\text{ in }L^2(0,T;L^q(\Omega))\text{ for some }q>\frac{6}{5}.
	\end{equation}
	Moreover, taking into consideration the Sobolev embedding $W^{1,2}(\Omega)$ to $L^6(\Omega)$ and \eqref{DelWConv}$_2$ we infer
	\begin{equation}\label{SDelUDelWeakly}
		\mathcal S_\delta(\mathfrak r_\delta,\mathfrak z_\delta,\vartheta_\delta) u_\delta\rightharpoonup\mathcal S(\mathfrak r,\mathfrak z,\vartheta) u\text{ in }L^1(Q_T).
	\end{equation}
	Next, using convergence \eqref{ThetaDelPoint} and bound \eqref{VThDeltaEst} it follows that
	\begin{equation}\label{BHStrong}
		\vartheta_\delta^\frac{\beta}{2}\to\vartheta^\frac{\beta}{2}\text{ in }L^p(Q_T)\text{ for some }p>2.
	\end{equation}
	Hence we have due to \eqref{VDSob}$_1$
	\begin{equation*}
		\nabla \vartheta^\frac{\beta}{2}_\delta\rightharpoonup \nabla \vartheta^\frac{\beta}{2}\text{ in }L^2(Q_T)
	\end{equation*}
	and by virtue of \eqref{KGrowth}
	\begin{equation*}
		\kappa(\vartheta_\delta)\vartheta^{-\frac{\beta}{2}}_\delta\chi_{\{\vartheta_\delta\geq 1\}}\to \kappa(\vartheta)\vartheta^{-\frac{\beta}{2}}\chi_{\{\vartheta\geq 1\}}\text{ in }L^2(Q_T).
	\end{equation*}
	Combining the latter two convergences we get that 
	\begin{equation}\label{KThOThNThweakly}
		\frac{\kappa(\vartheta_\delta)}{\vartheta_\delta}\chi_{\{\vartheta_\delta\geq 1\}}\nabla\vartheta_\delta=\frac{2}{\beta}\kappa(\vartheta_\delta)\vartheta_\delta^{-\frac{\beta}{2}}\chi_{\{\vartheta_\delta\geq 1\}}\nabla(\vartheta^\frac{\beta}{2}_\delta)\rightharpoonup\frac{2}{\beta} \kappa(\vartheta)\vartheta^{-\frac{\beta}{2}}\chi_{\{\vartheta\geq 1\}}\nabla(\vartheta^\frac{\beta}{2})=\frac{\kappa(\vartheta)}{\vartheta}\chi_{\{\vartheta\geq 1\}}\nabla\vartheta\text{ in }L^1(Q_T).
	\end{equation}
	Moreover, we have
	\begin{equation}\label{KDelStrongly}
		\kappa(\vartheta_\delta)\chi_{\{\vartheta_\delta\leq 1\}}\to \kappa(\vartheta)\chi_{\{\vartheta\leq 1\}}\text{ in }L^p(Q_T)\text{ for any }p\in[1,\infty)
	\end{equation}
	by the Lebesgue dominated convergence theorem and \eqref{ThetaDelPoint}. Next, as a consequence of \eqref{LogThDelEst}, \eqref{ThetaDelPoint} and the Vitali convergence theorem we obtain
	\begin{equation*}
		\log\vartheta_\delta\to\log \vartheta\text{ in }L^1(Q_T)
	\end{equation*}
	and accordingly
	\begin{equation*}
		\nabla\log\vartheta_\delta\rightharpoonup \nabla\log\vartheta\text{ in }L^2(Q_T).
	\end{equation*}
	Combining the latter convergence and \eqref{KDelStrongly} we get
	\begin{equation}\label{KThOThSecWeakly}
		\frac{\kappa(\vartheta_\delta)}{\vartheta_\delta}\chi_{\{\vartheta_\delta\leq 1\}}\nabla\vartheta_\delta=\kappa(\vartheta_\delta)\chi_{\{\vartheta_\delta\leq 1\}}\nabla\log\vartheta_\delta\rightharpoonup \kappa(\vartheta)\chi_{\{\vartheta\leq 1\}}\nabla\log\vartheta=\frac{\kappa(\vartheta)}{\vartheta}\chi_{\{\vartheta\leq 1\}}\nabla\vartheta\text{ in }L^1(Q_T).
	\end{equation}
	The next task is the recovery of \eqref{EntBal}. To this end we employ convergences \eqref{InDatConv}$_{4,5}$, \eqref{DelWConv}$_4$, \eqref{SDelDelStrongly}, \eqref{SDelUDelWeakly}, \eqref{KThOThNThweakly}, \eqref{KThOThSecWeakly} and bounds \eqref{FDelEst}$_{4,5}$, \eqref{SDelTEst} respectively  to pass to the limit $\delta\to 0_+$ in \eqref{DAEBal}. We note that inequality \eqref{SigDom} holds. Indeed, going back to \eqref{SigmDDef} we observe that the $\delta$--dependent quantities are nonnegative and we already know that for the limit passage in the remaining terms the weak lower semicontinuity of convex functionals is sufficient, cf. Section~\ref{Sec:NPass}.
	\section{Proof of Theorem~\ref{Thm:BNSEx}}
	We set 
	\begin{equation*}
		F(\alpha)=\frac{1}{f_+(\alpha)},\ G(\alpha)=\frac{1}{f_-(1-\alpha)}
	\end{equation*} 
	and obtain that $F$ and $G$ are strictly monotone non-vanishing functions on $[0,1]$ by assumptions on $f_{\pm}$. We denote
	\begin{equation*}
		\begin{split}
			\underline F&=\min\{F(\underline\alpha), F(\overline\alpha)\}, \overline F=\max\{F(\underline\alpha),F(\overline\alpha)\},\\
			\underline G&=\min\{G(\underline\alpha), G(\overline\alpha)\}, \overline G=\max\{G(\underline\alpha),G(\overline\alpha)\}.
		\end{split}
	\end{equation*} 
	Obviously, $F:[\underline\alpha,\overline\alpha]\to[\underline F,\overline F]$  and $G:[\underline\alpha,\overline\alpha]\to[\underline G,\overline G]$ are $C^1$--diffeomorphisms. We intend to employ Theorem~\ref{Thm:ABSEx} with initial conditions 
	\begin{equation*}
		\begin{split}
			\xi_0=\rho_0,\ \mathfrak r_0=f_+(\alpha_0)\rho_{0},\ \mathfrak z_0=f_-(1-\alpha_0)z_{0},\ \Sigma_0 =\rho_0+z_0,\ u_0,\\
			\mathcal S(\mathfrak r,\mathfrak z,\vartheta)(0,\cdot)=\frac{4b}{3}\vartheta_0^3+\rho_{0}s_+\left(\frac{\rho_{0}}{\alpha_0},\vartheta_0\right)+z_{0}s_-\left(\frac{z_{0}}{1-\alpha_0},\vartheta_0\right).
		\end{split}
	\end{equation*}
	If $(\alpha_0,\rho_{+,0},\rho_{-,0})$ fulfill conditions in 
	\eqref{CondCompO}$_1$ then $(\xi_0,\mathfrak r_0.\mathfrak z_0,\Sigma_0)$ fulfill condition 
	\begin{equation}\label{CondCompTr}
		\underline a\mathfrak r_0\leq\mathfrak z_0\leq\overline a\mathfrak r_0,\ \min\{\underline F,\underline G\}\Sigma_0\leq \mathfrak r_0+\mathfrak z_0\leq \max\{\overline F,\overline G\}\Sigma_0, \  \underline F\mathfrak r_0\leq\xi\leq\overline F\mathfrak r_0.
	\end{equation}
	Hence Theorem~\ref{Thm:ABSEx} guarantees the global in time existence of a bounded energy weak solution $(\xi,\mathfrak r,\mathfrak z,\Sigma, u,\vartheta)$ to \eqref{ABS} in the sense of Definition~\ref{DefWSAbs}. 
	
	We define in accordance with convention \eqref{DivConv} $\alpha=F^{-1}\left(\xi/_a\mathfrak r \right)$, $\rho=F(\alpha)\mathfrak r$ and $z=G(\alpha)\mathfrak z$. Obviously, by virtue of Lemma~\ref{lem:ConToTr} $(1)$ the quantities $\xi/_a\mathfrak r$ and $\alpha$ fulfill transport equation \eqref{TrEqO} with initial conditions $\xi/_a\mathfrak r_0=F(\alpha_0)$, $\alpha(0)=\alpha_0$ respectively and $\alpha$ possesses the regularity specified in Definition~\ref{DefWSBNS}. Applying Lemma~\ref{lem:ConToTr} $(2)$ we infer that the quantities $\rho$ and $z$ possess the regularity specified in Definition~\ref{DefWSBNS}. The pairs $(\rho,u)$ and $(z,u)$ satisfy continuity equation \eqref{ConEqO} with initial condition $\rho(0)=F(\alpha_0)\mathfrak r_0$, $z(0)=G(\alpha_0)\mathfrak z_0$ respectively. We now show that
	\begin{equation}\label{SigmIdent}
		\Sigma(t,x)=(\rho+z)(t,x)\text{ for all }t\in[0,T]\text{ and a.a. }x\in\Omega.
	\end{equation}
	Firstly, we observe that $z/_a\rho$ satisfies the transport equation with the initial condition $\frac{z_{0}}{\rho_{0}}$. The pair $(T,u)$, where $T=\Sigma/_a\mathfrak r$ fulfills a transport equation with the initial datum
	\begin{equation}\label{TInDat}
		T(0)=F(\alpha_0)+F(\alpha_0)\frac{z_{0}}{\rho_{0}}
	\end{equation}
	according to Lemma~\ref{lem:ConToTr} $(1)$. We have that $(S,u)$, where $S=\rho/_a\mathfrak r+z/_a\mathfrak r$, satisfies the transport equation with the initial condition equal to the expression on the right hand side of \eqref{TInDat}. Hence it follows that $\mathfrak r T(t,x)=\mathfrak rS(t,x)$ for all $t\in[0,T]$ and a.a. $x\in\Omega$ by Lemma~\ref{Lem:AlmUniq}, which concludes \eqref{SigmIdent}. Accordingly, we observe that momentum equation \eqref{MomEqO} is fulfilled. Taking into account the definition of the energies $\mathfrak e_{\pm}$ and entropies $\mathfrak s_{\pm}$ in \eqref{TEneEnt} we conclude from \eqref{EnerEq} and \eqref{EntBal} that energy equality \eqref{EnerEqO} and entropy balance \eqref{EntBalO} are satisfied.

{\Large\textbf{Funding}}

This work has been supported by the Czech Science Foundation (GA\v CR) under grant GA22-01591S (for \v S.N. and M.K.). Moreover, \v S. N. and M.K. have been supported by  Praemium Academiæ of \v S. Ne\v casov\' a. Finally, the Institute of Mathematics of the Czech Academy of Sciences is supported by RVO:67985840.

{\Large\textbf{Acknowledgement}}

 The authors thank to Florian Oschmann for careful reading of the manuscript leading to the improvement of its readability and understandability as well.


\begin{thebibliography}{50}
		\bibitem{Amann95}
		Amann, H.: 
		\newblock Linear and quasilinear parabolic problems. {V}ol. {I}. Monographs in Mathematics 89, Birkh\"{a}user Boston, 1995.
		\bibitem{BaNu}
		Baer, M.R. and Nunziato, J.W.:
		\newblock A two--phase mixture theory for the deflagration--to--detonation transition in reactive granular materials,
		\newblock Int. J. Multiphase Flow 12, 861--889, 1986.
       \bibitem{BrZat}
    Bresch, D. and Mucha, P.B. and Zatorska, E.:
    \newblock Finite-energy solutions for compressible two-fluid Stokes system.
    \newblock Arch. Ration. Mech. Anal. 232(2), 987–1029, 2019.
		\bibitem{CoMe75}
		Coifman, R. R. and Meyer, Y.:
		\newblock On commutators of singular integrals and bilinear singular integrals, Trans. Amer. Math. Soc. 212, 315--331, 1975.
		
   \bibitem{EvKar}
   Evje, S. and Karlsen, K.~H.:
   \newblock Global existence of weak solutions for a viscous two-fluid model.
   \newblock J. Differ. Equ. 245(9), 2660–2703, 2008. 
  
     
     \bibitem{Evkar2}
     Evje, S. and Karlsen, K.~H.:
     \newblock Global weak solutions for a viscous liquid-gas model with singular pressure law.
     \newblock Commun. Pure Appl. Anal. 8, no. 6, 1867–1894, 2009.
   
   \bibitem{EvWeZhu}
   Evje, S. and Wen, H. and Zhu, C.:
   \newblock On global solutions to the viscous liquid-gas model with unconstrained transition to single-phase flow.
   \newblock  Math. Models Methods Appl. Sci. 27(2), 323–346, 2017.

  
  \bibitem{Feireisl04}
		Feireisl, E.:
		\newblock Dynamics of viscous compressible fluids.
		Oxford Lecture Series in Mathematics and its Applications, 26. Oxford University Press, Oxford, 2004.
		\bibitem{FeNo09}
		Feireisl, E. and Novotn\'y A.:
		\newblock Singular limits in thermodynamics of viscous fluids. Advances in Mathematical Fluid Mechanics.
		\newblock Birkhäuser Verlag, Basel, 2009.

  \bibitem{BNN}
 Jin, B. and Kwon, Y.-S. \v and Ne\v casov\' a, \v{S}. and Novotn\' y, A.: 
 \newblock Existence and stability of dissipative turbulent solutions to a simple bi-fluid model of compressible fluids. 
 \newblock J. Elliptic Parabol. Equ.
 (7), no. 2, 537–570, 2021.
    \bibitem{KKNN}
     Kra\v cmar, S. and  Kwon, Y.-S. and Ne\v casov\' a, \v S. and Novotn\' y, A.:
  \newblock Weak solutions for a bifluid model for a mixture of two compressible noninteracting fluids with general boundary data. 
  \newblock SIAM J. Math. Anal. 54 (1), 818–871, 2022.
		
		\bibitem{KNAC}
		Kwon, Y.-S. and Novotn\'y, A. and Arthur Cheng, C.H.:
		\newblock On weak solutions to a dissipative Baer--Nunziato--type system for a mixture of two compressible heat conducting gases,
		\newblock  Math. Models Methods Appl. Sci. 30 no. 8, 1517--1553, 2020. 
		
		\bibitem{LaSoUr68}
		Ladyzhenskaya, O.A. and Solonnikov V.A. and Uraltseva N.N.:
		\newblock Linear and quasi-linear equations of parabolic type. Translated from the Russian by S. Smith.,
		\newblock Trans. Math. Monographs. 23., Amer. Math. Soc., Providence, 1968. 

   \bibitem{LiZat}
   Li, Y. and Sun, Y. and Zatorska, E.:
   \newblock Large time behavior for a compressible two-fluid model with algebraic pressure closure and large initial data.
   \newblock Nonlinearity 33, no. 8, 4075–4094, 2020.

   \bibitem{Maltese}
   Maltese, D. and  Mich\'{a}lek, M. and Muha, P.B. and Novotn\'{y}, A. and Pokorn\'{y}, M. and Zatorska, E.:
   \newblock Existence of weak solutions for compressible Navier–Stokes equations with entropy transport.
   \newblock  J. Differ. Equ. Volume 261, Issue 8, 4448--4485, 2016.
   
		
		\bibitem{Nov20}
		Novotn\'{y}, A.:
		\newblock Weak solutions for a bi-fluid model for a mixture of two
		compressible non interacting fluids,
		\newblock Sci. China Math. 63 no. 12, 2399--2414, 2020.
		
		\bibitem{NovPok20}
		Novotn\'{y}, A. and Pokorn\'{y}, M.:
		\newblock Weak solutions for some compressible multicomponent fluid
		models,
		\newblock Arch. Ration. Mech. Anal. 235, 355--403, 2020.
		
		\bibitem{NovStr04}
		Novotn\'{y}, A. and Stra\v{s}kraba, I.:
		\newblock Introduction to the mathematical theory of compressible flow,
		Oxford Lecture Series in Mathematics and its Applications,
		\newblock Oxford University Press, Oxford, 2004.
		
		\bibitem{Ped97}
		Pedregal, P.:
		\newblock Parametrized Measures and Variational Principles,
		\newblock Progress in Nonlinear Differential Equations and their Applications, Vol. 30,
		Birkh\"auser Verlag, Basel, 1997.
		\bibitem{VaWeYu19}
		Vasseur, A. and Wen, H. and Yu, C.:
		\newblock Global weak solution to the viscous two--fluid model with finite energy,
		\newblock J. Math. Pures Appl. (9) 125, 247--282, 2019.
        
        \bibitem{Wen}
        Wen, H.:
        \newblock On global solutions to a viscous compressible two-fluid model with unconstrained transition to single-phase flow in three dimensions,
        \newblock Calc. Var. Partial Differential Equations 60, no. 4, Paper No. 158, 2021.
    
        \bibitem{yaoZhu}
        Yao, L. and Zhu, C.:
        \newblock Free boundary value problem for a viscous two-phase model with mass-dependent viscosity,
        \newblock  J. Differ. Equ. 247, no. 10, 2705–2739, 2009.
   
        \bibitem{yaoZhu2}
        Yao, L. and Zhu, C.:
        \newblock Existence and uniqueness of global weak solution to a two-phase flow model with vacuum,
        \newblock  Math. Ann. 349, no. 4, 903–928, 2011. 	
	\end{thebibliography}
\end{document}